%% file: short-weyl.tex
\documentclass[notitlepage,twoside,a4paper]{amsart}
\usepackage{mathrsfs}
\usepackage{amsmath,amssymb,enumerate}
\usepackage{epsfig,fancyhdr,color}
\usepackage{epstopdf}
\usepackage{amssymb}
\usepackage{amsmath,amsthm}
\usepackage{latexsym}
\usepackage{amscd}
\usepackage{psfrag}
\usepackage{graphicx}
\usepackage{epsf}
\usepackage[latin1]{inputenc}
\usepackage[all]{xy}
\usepackage{tikz}
\usepackage{prettyref}
\usepackage{subcaption}
\usepackage{float}
\usepackage{color}
\usepackage{array}
\usepackage{cite}
\usetikzlibrary{snakes}
\usepackage[totalwidth=17cm,totalheight=24cm]{geometry}

\newcommand{\II}{I\hspace{-0.1cm}I}
\newcommand{\III}{I\hspace{-0.1cm}I\hspace{-0.1cm}I}

\newtheorem{theorem}{\rm\bf Theorem}[section]
\newtheorem{proposition}[theorem]{\rm\bf Proposition}
\newtheorem{lemma}[theorem]{\rm\bf Lemma}
\newtheorem{corollary}[theorem]{\rm\bf Corollary}
\newtheorem{definition}[theorem]{\rm\bf Definition}
\newtheorem{remark}[theorem]{\rm\bf Remark}
\newtheorem{example}[theorem]{\rm\bf Example}

\newtheorem {claim}[theorem]{\rm\bf Claim}

\newtheorem{Step-nabla}{\rm\bf Step}

\newcommand{\C}{{\mathbb C}}
\newcommand{\CP}{{\mathbb CP}}
\newcommand{\N}{{\mathbb N}}
\newcommand{\DD}{{\mathbb D}}
\newcommand{\HH}{{\mathbb H}}
\newcommand{\R}{{\mathbb R}}
\newcommand{\RP}{{\mathbb {RP}}}
\newcommand{\Z}{{\mathbb Z}}

\newcommand{\bD}{\mathbb{D}}

\newcommand{\cI}{\mathcal{I}}
\newcommand{\cJ}{\mathcal{J}}
\newcommand{\cK}{\mathcal{K}}

\newcommand{\cX}{\mathcal{X}}

\newcommand{\cB}{{\mathcal B}}
\newcommand{\cC}{{\mathcal C}}
\newcommand{\cE}{{\mathcal E}}

\newcommand{\cU}{{\mathcal U}}
\newcommand{\cS}{{\mathcal S}}
\newcommand{\cV}{{\mathcal V}}
\newcommand{\cW}{{\mathcal W}}
\newcommand{\cCC}{{\mathcal{CC}}}
\newcommand{\cCP}{{\mathcal{CP}}}

\newcommand{\cQC}{{\mathcal{QC}}}

\newcommand{\cT}{{\mathcal T}}

\newcommand{\PSL}{\rm{PSL}}

\newcommand{\codim}{{\rm codim}}
\newcommand{\bdS}{{{\rm d}\mathbb{S}}}
\newcommand{\bHS}{\mathbb{HS}}
\newcommand{\nx}{n_{\bar x}}
\newcommand{\ny}{n_{\bar y}}

\newcounter{notes}%

\def\interieur#1{\mathord{\mathop{\kern 0pt #1}\limits^\circ}}

\title[]{The geometric data on the boundary of convex subsets of hyperbolic manifolds}

\author{Qiyu Chen}
\address{Qiyu Chen:
School of Mathematics, South China University of Technology,
510641, Guangzhou, P. R. China.}
\email{qiyuchen@scut.edu.cn}
\thanks{Q. Chen was partially supported by NSFC No: 12101244, Guangdong Basic and Applied Basic Research Foundation No: 2022A1515012225, and Guangzhou Science and Technology Program No: 202201010464.}

\author{Jean-Marc Schlenker}
\address{Jean-Marc Schlenker:
University of Luxembourg, FSTM, Department of Mathematics,
Maison du nombre, 6 avenue de la Fonte,
L-4364 Esch-sur-Alzette, Luxembourg}
\email{jean-marc.schlenker@uni.lu}
\thanks{J.-M. S. was partially supported by FNR project O20/14766753.}

\date{\today}

\begin{document}

\begin{abstract}
  Let $N$ be a geodesically convex subset in a convex co-compact hyperbolic manifold $M$ with incompressible boundary. We assume that each boundary component of $N$ is either a boundary component of $\partial_\infty M$, or a smooth, locally convex surface in $M$. We show that $N$ is uniquely determined by the boundary data defined by the conformal structure on the boundary components at infinity, and by either the induced metric or the third fundamental form on the boundary components which are locally convex surfaces. We also describe the possible boundary data. This provides an extension of both the hyperbolic Weyl problem and the Ahlfors-Bers Theorem.

  Using this statement for quasifuchsian manifolds, we obtain existence results for similar questions for convex domains $\Omega\subset \HH^3$ which meets the boundary at infinity $\partial_{\infty}\HH^3$ either along a quasicircle or along a quasidisk. The boundary data then includes either the induced metric or the third fundamental form in $\HH^3$, but also an additional ``gluing'' data between different components of the boundary, either in $\HH^3$ or in $\partial_\infty\HH^3$.

  \bigskip

\end{abstract}

\maketitle

\tableofcontents

\input{short-weyl1} 

\input{short-weyl2} 

\input{short-weyl3} 

\input{short-weyl4} 

\input{short-weyl5} 

\input{short-weyl6} 

\bibliographystyle{plain}
\bibliography{bibweyl}
\end{document}

%% file: short-weyl1.tex
\section{Introduction and main results}

\subsection{The classical Weyl problem in Euclidean and hyperbolic spaces}

The classical Weyl problem, going back to Hermann Weyl \cite{weyl1916uber}, asks whether any smooth metric of positive curvature on the sphere can be obtained uniquely as the induced metric on the boundary of a bounded convex subset in $\R^3$. A proof was obtained by Lewy \cite{lewy1935priori} for analytic metrics, by Nirenberg \cite{nirenberg} and Pogorelov \cite{pogorelov:deformation} for smooth metrics, and by Alexandrov \cite{alexandrov:intrinsic,AZ} for polyhedral metrics. The result was then further extended by Pogorelov \cite{Po} to general non-smooth metrics: the induced metric on the boundary of any bounded convex subset of $\R^3$ (even if non-smooth) is a geodesic distance of curvature $K\geq 0$ in the sense of Alexandrov, and each such metric on the sphere is realized on the boundary of a unique bounded convex body.

The same question was considered by Alexandrov \cite{alexandrov:intrinsic} and Pogorelov \cite{Po} for surfaces in the 3-dimensional hyperbolic space $\HH^3$. Given a bounded convex subset $\Omega\subset \HH^3$, the induced metric on its boundary has curvature $K\geq -1$. Conversely, Alexandrov and Pogorelov proved that any smooth metric of curvature $K>-1$ is induced on the boundary of a unique bounded convex subset of $\HH^3$ with smooth boundary. A similar result holds for bounded polyhedra, for which the induced metric on the boundary is a hyperbolic metric with cone singularities of angles less than $2\pi$ \cite{alex}.

In the hyperbolic space, there is also a {\em dual} Weyl problem, describing the possible third fundamental forms of the boundaries of bounded convex subsets. For a smooth surfaces $S\subset \HH^3$, the third fundamental form is classically defined for two tangent vectors $u,v$ to $S$ at a point $x$ by:
$$ \III(u,v) = I(\nabla_uN,\nabla_vN), $$
where $I$ is the induced metric on $S$, $\nabla$ is the Levi-Civita connection of $\HH^3$, and $N$ is a unit normal vector field to $S$. This third fundamental form can also be described as the induced metric on a dual surface in the de Sitter space (see Section \ref{ssc:dual}), and under this definition extends to the boundary of a convex domain with non-smooth boundary, for instance polyhedra in $\HH^3$. For polyhedra in $\HH^3$, the third fundamental form is also called the {\em dual metric}, e.g. in \cite{HR}.

If $\Omega\subset \HH^3$ is a convex bounded domain, the third fundamental form on its boundary is $CAT(1)$: it has curvature $K\leq 1$, and its closed geodesics have length $L>2\pi$ \cite{CD}. Conversely, any smooth metric on $S^2$ with curvature $K<1$ and closed geodesics of length $L>2\pi$ is realized as the third fundamental form on the boundary of a unique bounded convex subset in $\HH^3$ with smooth boundary, see \cite{these}. A similar result was proved by Rivin and Hodgson \cite{Ri,HR}: any spherical metric on $S^2$ with cone singularities of angles greater than  $2\pi$, with closed geodesics of length $L>2\pi$, is the dual metric of a unique bounded polyhedron in $\HH^3$. This description of the third fundamental forms on the boundary of bounded or ideal polyhedra is also related to the description of the dihedral angles of those polyhedra, see e.g. \cite{andreev,andreev-ideal,shu,rivin-annals,rousset1}.

\subsection{The Weyl problem for unbounded convex subsets in hyperbolic space}\label{subsec:unbounded convex domain}

We are interested here in the Weyl problem for convex subsets that are not bounded. More specifically, we will consider here what are perhaps the simplest cases of such a convex subset, of two possible types. The first case is a convex unbounded subset $\Omega\subset \HH^3$ such that
\begin{itemize}
\item the boundary of $\Omega$ in $\HH^3$ is an open disk, which we denote by $\partial \Omega$,
\item the ideal boundary of $\Omega$ in the ideal boundary $\partial_{\infty}\HH^3$ (identified with $\CP^1$) of $\HH^3$ is a closed disk, which we denote by $\partial_{\infty}\Omega$.
\end{itemize}

Clearly, in this setting, the induced metric (or third fundamental form) on $\partial\Omega$ does not uniquely determine $\Omega$ (up to isometries). For instance, if $\Omega$ is the hyperbolic convex hull of $\partial_\infty\Omega$ (assuming that $\partial_{\infty}\Omega$ is not a round disk), then $\partial\Omega$ is developable, and $(\partial\Omega,I)$ is always isometric to the hyperbolic plane. In the same setting, for any $K\in (-1,0)$, the boundary (in $\partial_{\infty}\HH^3$) of $\partial_\infty\Omega$ bounds a unique $K$-surface $\Sigma_K$ bounding a convex domain with asymptotic boundary $\partial_\infty\Omega$, see \cite{rosenberg-spruck}, and $\Sigma_K$ equipped with its induced metric is always isometric to the hyperbolic plane with the metric scaled by $1/|K|$. The same construction can also be used for the third fundamental form, since $\Sigma_K$ also has a third fundamental form of constant curvature $K^*=K/(K+1)$, so it is always isometric to the hyperbolic plane with the metric scaled by $1/|K^*|$.

The second type of unbounded convex subset of $\HH^3$ that we consider here is a convex domain $\Omega\subset \HH^3$ such that $\partial\Omega$ is the disjoint union of two complete disks, denoted here by $\partial_-\Omega$ and $\partial_+\Omega$, which ``meet'' at infinity along a Jordan curve. In the spirit of the Weyl problem and its dual, we would like to prescribe on $\partial_-\Omega$ and $\partial_+\Omega$ either the induced metric or the third fundamental form. However this data is again not sufficient to uniquely determine $\Omega$, as shown in \cite{convexhull} for metrics of constant curvature: one needs also to prescribe a ``gluing map'' at infinity between $\partial_-\Omega$ and $\partial_+\Omega$.

Our results are obtained only when the limit curve at infinity -- the boundary of $\partial_\infty\Omega$ in the first case, the common boundary of $\partial_-\Omega$ and $\partial_+\Omega$ in the second case -- is a quasi-circle, which corresponds to the gluing maps being quasi-symmetric. We can state those results as follows, with points (i) and (ii) corresponding to the first case above, and points (iii)-(v) corresponding to the second case. Before stating the result, however, we need a simple definition.

\begin{definition}\label{def:bounded conf}
  Let $h$ be a conformal metric on 
  an open disk $D\subset \CP^1$, 
  and let $h_0$ be the conformal hyperbolic metric on 
  $D$. We say that $h$ is {\em bounded} if there exists a constant $C>1$ such that at all points of 
  $D$, $h\leq Ch_0$.
\end{definition}

Note that if $h$ is complete and has curvature $K\geq -c$ for some $c>0$, then $h\geq (1/c)h_0$
(see \cite[Theorem 2.8]{mateljevic}, where the argument is based on \cite{cheng-yau}).

\begin{theorem} \label{tm:main-iv}
Let $D_-, D_+\subset\CP^1$ be disjoint quasi-disks with $\partial D_-=\partial D_+$. Let $\epsilon\in(0,1)$ and let $h_-$ (resp. $h_+$) be a smooth, complete, bounded conformal metric of curvature varying in $[-1+\epsilon,1/\epsilon]$ on $D_-$ (resp. $D_+$), and $h^*_-$ (resp. $h^*_+$) be a smooth, complete, bounded conformal metric of curvature varying in $[-1/\epsilon,0]$ on $D_-$ (resp. $D_+$). Then
\begin{enumerate}[(i)]
  \item There exists a convex subset $\Omega\subset \HH^3$ with $\partial_{\infty}\Omega$ a closed disk and a continuous map $\phi:\CP^1\to \partial\Omega\cup\partial_{\infty}\Omega$ sending $D_+$ to $\partial_\infty\Omega$, which is conformal on $D_+$ and such that $\phi^*I=h_-$ on $D_-$.
  \item There exists a convex subset $\Omega\subset\HH^3$ with $\partial_{\infty}\Omega$ a closed disk and a continuous map $\phi:\CP^1\to \partial\Omega\cup\partial_{\infty}\Omega$ sending $D_+$ to $\partial_\infty\Omega$, which is conformal on $D_+$ and such that $\phi^*\III=h^*_-$ on $D_-$.
  \item There exists a convex subset $\Omega\subset \HH^3$ with $\partial_{\infty}\Omega$ a Jordan curve, and a continuous map $\phi:\CP^1\to \partial\Omega\cup\partial_{\infty}\Omega$, such that $\phi^*I=h_-$ on $D_-$ and $\phi^*I=h_+$ on $D_+$.
  \item There exists a convex subset $\Omega\subset \HH^3$ with $\partial_{\infty}\Omega$ a Jordan curve, and a continuous map $\phi:\CP^1\to \partial\Omega\cup\partial_{\infty}\Omega$, such that $\phi^*\III=h^*_-$ on $D_-$ and $\phi^*\III=h^*_+$ on $D_+$.
\item There exists a convex subset $\Omega\subset \HH^3$ with $\partial_{\infty}\Omega$ a Jordan curve, and a continuous map $\phi:\CP^1\to \partial\Omega\cup\partial_{\infty}\Omega$, such that $\phi^*I=h_-$ on $D_-$ and $\phi^*\III=h_+^*$ on $D_+$.
\end{enumerate}
\end{theorem}

The hypothesis on the metrics $h^*_\pm$ are probably not optimal -- one could expect similar results only assuming that those metrics have curvature in $[-1/\epsilon, 1-\epsilon]$, and that their closed geodesics have length $L>2\pi$. However the arguments used in Section \ref{sc:convexH} only function (at this point) with the stronger assumption considered here. The hypothesis that those metrics are ``bounded'' in the sense considered here is also necessary for ``technical'' reasons in the proofs.

Theorem \ref{tm:main-iv} extends the main result of \cite{convexhull},
which dealt only with constant curvature metrics, and has a non-empty intersection with the main result of \cite{weylgen}, which only deals with the induced metric. However it should be pointed out that the argument used here is quite different, since the approximation was done in \cite{weylgen} by metrics on the sphere, while here we use a quasifuchsian approximation argument. One could also find additional motivations in \cite{weylsurvey}.

\subsection{Convex domains in quasifuchsian manifolds}

The proof of Theorem \ref{tm:main-iv} is based on an approximation argument, where complete metrics on a disk together with quasi-symmetric homeomorphisms from $\RP^1$ to itself are approximated, in a suitable sense, by metrics and quasi-symmetric homeomorphisms which are invariant (resp. equivariant) with respect to the action of surface groups.
Those invariant metrics (and equivariant gluing maps) correspond to convex surfaces embedded in quasifuchsian hyperbolic manifolds.

In this quasifuchsian setting, and more generally in the convex co-compact setting discussed below, we will obtain an existence but also a uniqueness statement. We first discuss the quasifuchsian setting here for simplicity, and for reference below.

Let $S$ be a closed, oriented connected surface of genus at least $2$ and let $\cT_S$ denote the Teichm\"uller space of $S$, which is the space of hyperbolic metrics (or conformal structures) on $S$ (up to isotopies). The following statement is a consequence of Theorem \ref{tm:main-cc} below.

\begin{corollary} \label{cr:qf-I}
Let $h$ be a Riemannian metric of curvature greater than $-1$ on $S$, and let $c\in \cT_S$ be a conformal structure on $S$. Then there is a unique (up to isotopy) hyperbolic manifold homeomorphic to $S\times [0,+\infty)$ which has a convex boundary at $S\times\{0\}$ with induced metric isotopic to $h$, and has a conformal structure at infinity on $S\times\{+\infty\}$ isotopic to $c$.
\end{corollary}

We now consider a ``dual'' statement to Corollary \ref{cr:qf-I}, where we prescribe the third fundamental form rather than the induced metric.

\begin{corollary} \label{cr:qf-III}
  Let $h^*$ be a Riemannian metric of curvature less than $1$ on $S$ with closed, contractible geodesics of length greater than $2\pi$, and let $c\in \cT_S$ be a conformal structure on $S$. Then there is a unique (up to isotopy) hyperbolic manifold homeomorphic to $S\times [0,+\infty)$ which has a convex boundary at $S\times\{0\}$ with the third fundamental form isotopic to $h^*$, and has a conformal structure at infinity on $S\times\{+\infty\}$ isotopic to $c$.
\end{corollary}

Note that the above hyperbolic manifold can be isometrically embedded in a unique (up to isometries) quasifuchsian manifold.

A consequence of Corollary \ref{cr:qf-I} concerns quasifuchsian manifolds with induced metric prescribed on one side of the convex core, while the conformal structure at infinity is prescribed on the other side. We only claim existence in that statement, since proving uniqueness would request obtaining a local rigidity statement that appears difficult. The statement is presented here as a theorem since it is not a special case of Theorem \ref{tm:main-cc}.

\begin{theorem} \label{tm:cc-I}
Let $h\in\cT_S$ be a hyperbolic metric on $S$, and let $c\in \cT_S$ be a conformal structure on $S$.
Then there is a quasifuchsian hyperbolic manifold with the conformal structure on its upper boundary component at infinity isotopic to $c$, and with the induced metric on the lower boundary component of its convex core isotopic to $h$.
\end{theorem}

\subsection{Convex domains in convex co-compact manifolds}

We finally state a general result on the existence and uniqueness of convex co-compact hyperbolic manifold containing a geodesically convex subset such that the induced metric, third fundamental form, or conformal metric at infinity is prescribed for each end. The boundary components of the convex co-compact hyperbolic manifolds throughout the paper are necessarily closed oriented surfaces of genus at least two.

\begin{theorem} \label{tm:main-cc}
Let $M$ be a compact 3-manifold with non-empty, incompressible boundary which admits a convex co-compact hyperbolic structure on its interior ${\rm int}(M)$ and let $\partial_lM$ ( $l=1,\ldots,n$, $n\geq 2$ ) denote its boundary components. Let $\cI\sqcup \cJ\sqcup \cK=\{1,\ldots,n\}$ be a partition, where $\cI$, $\cJ$, or $\cK$ might be an empty set. Let $c_i$ ($i\in \cI$) be a conformal structure on $\partial_iM$, $h_j$ ($j\in \cJ$) be a Riemannian metric of curvature greater than $-1$ on $\partial_jM$ and $h_k^*$ ($k\in \cK$) be a Riemannian metric of curvature less than $1$ on $\partial_kM$, with closed, contractible geodesics of length greater than $2\pi$. Then there exists a unique (up to isotopy) convex co-compact hyperbolic structure on ${\rm int}(M)$ which is complete on $\partial_i M$ (if $\cI\not=\emptyset$) with conformal structure isotopic to $c_i$, and with induced metric isotopic to $h_j$ on a surface isotopic to $\partial_j M$ (if $\cJ\not=\emptyset$), and with third fundamental form isotopic to $h_k^*$ on a surface isotopic to $\partial_kM$ (if $\cK\not=\emptyset$).
\end{theorem}

Corollaries \ref{cr:qf-I} and \ref{cr:qf-III} clearly follow from this theorem.

\subsection{Recent related results}

The statements presented here are related to a number of recent results.

Roman Prosanov extended \cite{prosanov:fuchsian} the existence and uniqueness result of Pogorelov for convex domains in $\HH^3$, without any regularity hypothesis beyond convexity, to Fuchsian hyperbolic manifolds. He also proved \cite{prosanov:dual} that hyperbolic manifolds with convex, polyhedral boundary are uniquely determined by the {\em dual metric} on the boundary -- the dual metric is the polyhedral analog of the third fundamental form, already appearing in \cite{HR}. In \cite{prosanov:polyhedral}, he defined a class of hyperbolic metrics with convex, ``bent'' boundary, and proved that any hyperbolic metric with cone singularities of positive curvature can be realized on the boundary. Moreover he proved a uniqueness result for a subclass of those metrics.

In \cite{mesbah:induced}, Abderrahim Mesbah extended Theorem \ref{tm:main-iv}, for quasifuchsian hyperbolic manifolds, to include the prescription of the measured bending lamination on one component (which then has to be a convex pleated surface) and the induced metric or third fundamental form on the other component of the boundary of a convex subset in a quasifuchsian hyperbolic manifold.

Another relevant recent work is \cite{choudhury:measured}, where Diptaishik Choudhury establishes that in the neighborhood of the Fuchsian locus, one can prescribe the {\em measured foliation at infinity}, which is an analog at infinity of the measured bending lamination on the boundary of the convex core, perhaps as the conformal structure at infinity is analog to the hyperbolic metric induced on the boundary of the convex core.

\subsection*{Acknowledgements} We would like to thank the anonymous referee who checked the submitted version of this paper with considerable care, and made a large number of very significant remarks, which lead to important improvements in the exposition and to several corrections.

%% file: short-weyl2.tex
\section{Background material}

\subsection{The hyperbolic and de Sitter spaces}\label{subsec:projective models}
Recall that the hyperboloid models of the 3-dimensional hyperbolic space and de Sitter (dS) space are respectively
$$ H^3:=\{x\in\R^{3,1}| \langle x, x\rangle_{3,1}=-1, x_0>0\},$$
and
 $$ dS^3:=\{x\in\R^{3,1}| \langle x, x\rangle_{3,1}=1\},$$
 where $\langle x, y\rangle_{3,1}=-x_0y_0+x_1y_1+x_2y_2+x_3y_3$. We denote by $\HH^3$ (resp. $\bdS^3$) the projection of $H^3$ (resp. $dS^3$) to $\RP^3$ (with the metric induced from the projection), which is the \emph{projective model} of the hyperbolic (resp. dS) 3-space (we use this model throughout the paper, unless otherwise stated).

Let $\bHS^3$ denote the space $\RP^3$ equipped with the geometric structure induced by the bilinear form $\langle\cdot,\cdot\rangle_{3,1}$, called the 3-dimensional \emph{HS space of signature $(3,0)$}. By construction its geometry is invariant under $PO_0(3, 1)$. It has a component - corresponding to the time-like vectors with $x_0>0$ in $\R^{3,1}$ isometric to $\HH^3$, while the complement of the closure of this component is isometric to $\bdS^3$. The signature of the metric induced on $\bHS^3$ is in fact $(3,0)$ on $\HH^3$ and $(2,1)$ on $\bdS^3$.

Let $\phi:\HH^3\cup \bdS^3 \rightarrow \R^{3,1}$ be a local section of the canonical projection $\pi:\R^{3,1}\rightarrow \RP^3$ with $\phi(\HH^3)=H^3$ and $\phi(\bdS^3)\subset dS^3$. For each point $x\in \HH^3\cup \bdS^3$, we define the \emph{HS scalar product} on $T_x\bHS^3$ as
$$\langle v,w\rangle_{HS}:=\langle\phi_*(v),\phi_*(w)\rangle_{3,1}$$
for all $v,w\in T_x\bHS^3$. The restriction of $\langle\cdot,\cdot\rangle_{HS}$ to $T\HH^3$ (resp. $T\bdS^3$) exactly induces the hyperbolic (resp. dS) metric.

The boundary in $\RP^3$ of $\HH^3$ or $\bdS^3$ coincides with the projectivization in $\RP^3$ of the set of light-like vectors in $\R^{3,1}$, which is conformally equivalent to $\CP^1$ in a canonical way. Let $\overline{\HH^3}=\HH^3\cup\CP^1$ denote the natural (conformal) compactification of $\HH^3$ by $\CP^1$. Note that the action of isometries of $\HH^3$ can extend to $\CP^1$ as projective automorphisms (or M\"obius transformations) and conversely, M\"obius transformations on $\CP^1$ can extend to isometries inside $\HH^3$ (using the Poincar\'{e} upper half-space model) \cite{marden:hyperbolic}. For simplicity, we denote by ${\rm PSL}_2(\C)$ the group of orientation-preserving isometries of $\HH^3$ as well as the group of M\"obius transformations on $\CP^1$. Let $sl_2(\C)$ denote the Lie algebra of ${\rm PSL_2(\C)}$, which can be interpreted as the space of infinitesimal automorphisms of $\overline{\HH^3}$. These are vector fields on $\overline{\HH^3}$ whose flows are isometries (resp. M\"obius transformations) on $\HH^3$ (resp. $\partial\overline{\HH^3}\cong\CP^1$), called \emph{Killing} (resp. $\emph{projective}$) vector fields (see e.g. \cite{HK,bridgeman-bromberg:bound}).

\subsection{The $sl_2(\C)$-bundle over $\overline{\HH^3}$}
\label{subsec:hyperbolic bundles}

Let $E$ denote the trivial $sl_2(\C)$-bundle over $\overline{\HH^3}$. Let $\kappa\in\Gamma(E)$ be a section of $E$. If $x\in\HH^3$, the value of $\kappa$ at $x$, denoted by $\kappa_x$, is a Killing vector field on $\HH^3$, which has the following decomposition (see e.g. \cite{HK,hmcb}):
            $$\kappa_x:=u_x+v_x,$$
         where $u_x$ is the \emph{infinitesimal pure translation at $x$} (namely, $u_x(x)=\kappa_x(x)$ and $\nabla_Xu_x=0$ for all $X\in T_x \HH^3$), while $v_x$ is the \emph{infinitesimal pure rotation at $x$} (namely, $v_x(x)=0$ and there exists a vector say $\sigma(x)\in T_x\HH^3$ such that $\nabla_X v_x=\sigma(x)\times X$ for all $X\in T_x\HH^3$), where $\nabla$ is the Levi-Civita connection of $\HH^3$. As a consequence, we can identify $\kappa_x$ with a pair of tangent vectors at $x$, say $(\tau(x),\sigma(x))\in T_x \HH^3\times T_x \HH^3$, which satisfy that
         $$\tau(x)=\kappa_x(x), \quad\quad \nabla_X\kappa_x=\sigma(x)\times X,$$
         for all $X\in T_x\HH^3$. The \emph{curl} of a vector field $X$ on $\HH^3$ is defined as
         \begin{equation}\label{eq:curl}
         {\rm curl}X:=\sum_{i\in \Z/3\Z}(\langle\nabla_{e_i}X,e_{i+1}\rangle-\langle\nabla_{e_{i+1}}X,e_i\rangle)e_{i+2},
         \end{equation}
         where $(e_1, e_2, e_3)$ is an orthonormal frame of $T\HH^3$ (see e.g. \cite[Section 11.1]{thurston-notes}). A direct computation shows that the above $\sigma(x)$ is one half the curl of $\kappa_x$ at $x$, similarly for the euclidean case in Section \ref{subsec:euclidean bundles}.

         In this way, the restriction to $\HH^3$ of a section $\kappa\in\Gamma(E)$ is identified with a pair of vector fields $(\tau,\sigma)\in\Gamma(T\HH^3)\times\Gamma(T\HH^3)$. For convenience, we call $\tau$ (resp. $\sigma$) the \emph{translation} (resp. \emph{rotation}) \emph{component} of $\kappa$. Under this identification and by a direct computation of the flat connection $D$ for $sl_2(\C)$-bundle over $\HH^3$, we conclude that: for any $x\in\HH^3$ and any $X\in T_x\HH^3$, the one-form $\omega=d^D\kappa$ can be written as (see \cite{HK} or \cite[Section 3]{hmcb}):
         \begin{equation}
           \label{eq:omega}
        \omega(X)=d^{D}_X(\tau,\sigma)=\big(\nabla_X \tau-\sigma(x)\times X, \nabla_X \sigma+\tau(x)\times X \big)=:\big(\omega_{\tau}(X),\omega_{\sigma}(X)\big).
        \end{equation}
        We call $\omega_{\tau}$ (resp. $\omega_{\sigma}$) the \emph{translation} (resp. \emph{rotation}) \emph{part} of $\omega$.
        Note that $\kappa$ is a constant section if and only if $d^D\kappa=0$, also, if and only if $\tau$ and $\sigma$ are hyperbolic Killing fields.

        If $x\in\CP^1$, $\kappa_x$ is a projective vector field on $\CP^1$. Let $\big(\phi_t(z)\big)_{t\in[0, \varepsilon]}$ be a one-parameter family of M\"obius transformations which generate $\kappa_x$. By computing the derivative at $t=0$ of $\phi_t(z)$, $\kappa_x$ has the following form in an affine chart $z$ near $x$ (see e.g. \cite[Section 3.2]{bromberg:rigidity} and \cite[Remark 4.4]{qfmp}):
        \begin{equation*}
        \kappa_x(z)=P(z)\partial_z,
        \end{equation*}
        where $P(z)$ is a complex polynomial of degree at most 2 (called \emph{quadratic polynomials} for simplicity). Moreover, there is a unique Killing field on $\HH^3$, determined by the above $(\phi_t(z))_{t\in[0,\varepsilon]}$, that extends continuously to $\kappa_x$ (see the proof of Lemma  \ref{lm:norm-Killing} for instance).

        \subsection{The convex co-compact structures and their deformations}\label{subsec:deformation}
        
Let $M$ be a compact (irreducible) 3-manifold with non-empty, incompressible boundary. Assume that each boundary component of $M$ is a closed oriented surface of genus at least two. By a result of Waldhausen \cite{waldhausen}, homotopic homeomorphisms of $M$ are also isotopic.
A \emph{$({\rm PSL}_2(\C),\overline{\HH^3})$-structure} on $M$ is a maximal atlas of charts $\{(U_{\alpha},\varphi_{\alpha})\}$ to $\overline{\HH^3}$ with transition maps the restrictions of elements of ${\rm PSL}_2(\C)$, which is a \emph{hyperbolic structure} when restricted to the interior of $M$, and a \emph{complex projective structure} when restricted to $\partial M$ (see e.g. \cite{bridgeman-bromberg:bound}). Under this condition, the interior ${\rm int}(M)$ of $M$ is equipped with a complete, convex co-compact hyperbolic metric, also called a \emph{convex co-compact structure}. (To see this, let $(D_i)_{i=1,\dots, n}$ be a cover of $\partial M$ by embedded open disks, for the complex projective structure at infinity. Each of the $D_i$ bounds a half-space $H_i$ embedded in $M$. The complement $\Omega$ in $M$ of the union of the $H_i$ is then by construction a non-empty, geodesically convex subset of $M$, in the sense that any geodesic segment in $M$ with endpoints in $\Omega$ is contained in $\Omega$. This proves that the hyperbolic structure on 
${\rm int}(M)$ is convex co-compact.)

A convex co-compact hyperbolic structure $g$ on ${\rm int}(M)$ (up to isotopy) determines a representation (unique up to conjugacy) say $\rho_g:\pi_1(M)\rightarrow {\rm PSL}_2(\C)$ which is discrete, faithful and such that $\HH^3/\rho_g(\pi_1(M))$ contains a non-empty compact geodesically convex subset. Such a representation $\rho_g$ 
is said to be \emph{convex co-compact}. If we relax the above compactness condition for the geodesically convex subset to be of finite volume, the corresponding hyperbolic structure (resp. representation 
) is said to be \emph{geometrically finite}. Conversely, a conjugacy class of a convex co-compact representation from $\pi_1(M)$ to $\PSL_2(\C)$ determines an isotopy class of a convex co-compact hyperbolic structure on 
${\rm int}(M)$.

Let $\cCC(M)$ denote the space of convex co-compact structures on ${\rm int}(M)$ (up to isotopy). By the Quasiconformal Stability Theorem of Marden \cite{marden:geometry}, $\cCC(M)$ 
can be identified with an open subset of the space of representations of $\pi_1(M)$ into ${\rm PSL}_2(\C)$ (up to conjugacy) with a topology determined by the images of finitely many generators (note that $\pi_1(M)$ is finitely generated and this topology is independent of the choice of generators).
It deserves to mention that any $g\in\cCC(M)$ (with its 
representation $\rho_g$) is obtained as a \emph{quasiconformal deformation} of an arbitrary fixed element say $g_0\in\cCC(M)$ (with its 
representation $\rho_{g_0}$), in the sense that there is a quasiconformal homeomorphism $f:\CP^1\rightarrow \CP^1$ such that $\rho_g(\gamma)=f\circ\rho_{g_0}(\gamma)\circ f^{-1}$, for all $\gamma\in\pi_1(M)$.

\subsubsection{The Ahlfors-Bers Theorem}

We will use repeatedly below the fact that, when $M$ has incompressible boundary, the space $\cCC(M)$ of convex co-compact hyperbolic metrics on 
${\rm int}(M)$ is parameterized by the Teichm\"uller space of its boundary, while a suitably adapted statement also applies to manifolds with compressible boundary. To state it, notice that $\cCC(M)$ inherits a complex structure from that of ${\rm PSL}_2(\C)$, through the description of elements of $\cCC(M)$ through the image of a finite set of generators of $\pi_1(M)$ in ${\rm PSL}_2(\C)$, while $\cT_{\partial M}$ is equipped with its Weil-Petersson complex structure. The following statement is stated, in the more general case of geometrically finite manifolds, as \cite[Theorem 5.1.3.]{marden:hyperbolic}.

\begin{theorem} \label{tm:ab}
  If the interior of $M$ carries a convex co-compact hyperbolic metric, then the map sending $g\in \cCC(M)$ to the conformal structure on its boundary at infinity defines a biholomorphism between $\cCC(M)$ and $\cT_{\partial M}/{\rm Mod}_0(\partial M)$, where ${\rm Mod}_0(\partial M)$ is the subgroup of the mapping-class group of $\partial M$ generated by Dehn twists along compressible simple closed curves.
\end{theorem}


In particular, if 
${\rm int}(M)$ is homeomorphic to the product of a surface $S$ (of genus at least $2$) by an interval,
the manifold 
${\rm int}(M)$ equipped with a convex co-compact structure is also called a  \emph{quasifuchsian} hyperbolic manifold. In this case, the two boundary components of $M$ are homeomorphic to $S$, the boundary of $M$ is incompressible, and $\cCC(M)$ is biholomorphic to $\cT_S\times \cT_S$.

 \subsubsection{The deformation vector field}\label{subsubsec: deformation vector field}

Let $(U_{\alpha},\varphi_{\alpha}^t)_{t\in [0,\epsilon]}$ be a smooth one-parameter family of convex co-compact structures on 
${\rm int}(M)$. Let 
$dev_t$ and $hol_t, t\in[0,\epsilon]$, be the corresponding developing maps and holonomy representations, respectively, 
where $dev_t:\tilde{M}\rightarrow \overline{\HH^3}$ and $hol_t:\pi_1(M)\rightarrow {\rm PSL}_2(\C)$ with $\tilde{M}$ the universal cover of $M$. This defines a \emph{deformation vector field} say $u$, on $\tilde{M}$, which associates to each $x\in\tilde{M}$ a tangent vector at $t=0$, say $u(x)$, of the curve $t\mapsto dev_0^{-1}(dev_t(x))$.  The equivariance property $dev_t(\gamma\cdot x)=hol_t(\gamma)\circ dev_t(x)$ for all $\gamma\in\pi_1(M)$ implies that the vector field $\gamma_*u-u$ is \textit{Killing} (resp. \emph{projective}) on $\tilde{M}\setminus\partial\tilde{M}$ (resp. $\partial\tilde{M}$):
 \begin{equation}\label{eq:automorphic}
 (\gamma_*)^{-1}u(\gamma\cdot x)=u(x)+dev_0^*\big(\dot{hol}_0(\gamma)(dev_0(x))\big),
 \end{equation}
 where $\dot{hol}_0(\gamma)\in sl_2(\C)$ is the derivative of $hol_t(\gamma)$ at $t=0$, which is a Killing (resp. projective) vector field on $\HH^3$ (resp. $\CP^1$). We call such $u$ \textit{automorphic}. Moreover, we say $u$ is \emph{equivariant} (with respect to the action of $\pi_1(M)$) if in particular $\dot{hol}_0(\gamma)=0$ (in this case, $dev_t$ is obtained by pre-composing $dev_0$ with the lift of diffeomorphisms $f_t$ of $M$ isotopic to the identity), see e.g. \cite{HK,bromberg:rigidity,qfmp}

 Note that if the above first-order deformation preserves the conformal structure at infinity on $\partial_lM$ ($1\leq l\leq n$), then $u$ is a conformal vector field when restricted to $\partial_l\tilde{M}$.

\subsubsection{The deformation section}

Let $\Lambda\subset\CP^1$ be the limit set of the Kleinian group $hol_0(\pi_1(M))$. For simplicity, we identify the universal cover $\tilde{M}$ with its image $\overline{\HH^3}$ under the developing map. Now we introduce a \emph{deformation section} of $E$ over $\overline{\HH^3}\setminus\Lambda$ (associated to a deformation vector field $u$) as follows (see \cite{bromberg:rigidity,qfmp}).

For $x\in\HH^3$, we define a Killing vector field
$s_x\in sl_2(\C)$ that ``best approximates" the vector field $u$ at $x$:
\begin{equation}\label{eq:canonical lift}
s_x(x)=u(x), \quad \quad {\rm{curl}}(s_x)(x)={\rm curl}(u)(x).
\end{equation}
For $z_0\in \CP^1\setminus\Lambda$, let $u(z)=f(z)\partial_z$ be the local expression of $u$ in an affine chart near $z_0$, we define a projective vector field $s_{z_0}$ whose two-jets agree with $f$ at $z_0$. More precisely,
\begin{equation}\label{eq:approx Mobius}
s_{z_0}(z)=\big(f(z_0)+f_z(z_0)(z-z_0)+f_{zz}(z_0)\frac{(z-z_0)^2}{2}\big)\partial_{z}.
\end{equation}
Note that the above $s_x$ (resp. $s_{z_0}$) is uniquely determined by the equations \eqref{eq:canonical lift}  (resp. the equation \eqref{eq:approx Mobius}). We call such a section the \textit{canonical lift} of $u$ and denote it by $s_u$. By definition, the canonical lift of a vector field in $sl_2(\C)$ is a constant section. As a consequence, the canonical lift $s_u$ is \textit{automorphic}, that is, for each $\gamma\in\pi_1(M)$,  $(s_u)_{\gamma\cdot x}-\gamma_*(s_u)_{x}$ is the same element in $sl_2(\C)$ for all $x\in\overline{\HH^3}\setminus\Lambda$. In particular, if $u$ is equivariant, then $s_u$ is \textit{equivariant}, in the sense that $(s_u)_{\gamma\cdot x}=\gamma_*(s_u)_{x}$ for all $\gamma\in\pi_1(M)$ and $x\in\overline{\HH^3}\setminus\Lambda$.

\subsubsection{The deformation one-form}\label{subsubsec:deformation one-forms}

Recall that $E$ has a natural flat connection $D$. Let $d^D$ be the de Rham differential twisted by $D$. We now introduce a \emph{deformation one-form}: $\tilde{\omega}=d^Ds_u$ on $\overline{\HH^3}\setminus\Lambda$, see e.g. \cite{HK,bromberg:rigidity,qfmp}. One can check that $\tilde{\omega}$ is an \emph{equivariant} $E$-valued one-form, which measures how far the deformation vector field $u$ differs from a Killing (resp. projective) field on $\HH^3$ (resp. $\CP^1$). Indeed, $$\tilde{\omega}-\gamma_*\tilde{\omega}=d^D s_u-\gamma_*d^Ds_u=d^D(s_u-\gamma_*s_u)=0.$$
So $\tilde{\omega}$ descends to an $E_M$-valued closed one-form $\omega\in\Omega^1(M,E_M)$ on $M$ (since it is locally exact), where $E_M$ is the quotient of $(\overline{\HH^3}\setminus\Lambda)\times sl_2(\C)$ by $hol_0(\pi_1(M))$, which acts on $\overline{\HH^3}\setminus\Lambda$ by deck transformations and on $sl_2(\C)$ via adjoint representation.

A first-order deformation of a convex co-compact structure $g\in\cCC(M)$ is \emph{trivial} if and only if the corresponding deformation vector field $u$ is the sum of a Killing vector field and an equivariant vector field, or equivalently, the deformation section $s_u$ is the sum of a constant section and an equivariant section, or equivalently, the deformation one-form $\tilde{\omega}=d^Ds_u$ is exact.  As a consequence, there is a one-to-one correspondence between the space of infinitesimal deformations of $M$ modulo trivial deformations and the first-cohomology group $H^1(M,E_M)$ (see \cite{HK,bromberg:rigidity,qfmp}).

\begin{definition}\label{def:norm}
Let $\kappa$ be a Killing vector field on $\HH^3$ (which can be viewed as a constant section of $E$ over $\HH^3$). We define the norm of $\kappa$ at $x$ as
$$\|\kappa\|_{x}:=\|\tau(x)\|_{\HH^3}+\|\sigma(x)\|_{\HH^3}~, $$
where $\tau$ and $\sigma$ are respectively the translation and rotation components of $\kappa$. Let 
$\omega\in \Omega^1(M,E_M)$. The norm of $\omega$ at $p\in M$ is defined as
$$\|\omega\|_p: =\sup_{X\in T_pM:\, \|X\|_{T_pM}=1}\|\omega(X)\|_p~,$$
where $\omega(X)\in E_M$ is a local Killing field on a neighbourhood in $M$ of $p$. We say $\omega$ is bounded on a subset $N\subset M$ if $\|\omega\|_p$ is uniformly bounded for all $p\in N$.
\end{definition}

\subsection{The bundle of Euclidean Killing fields over $\R^3$}

        \label{subsec:euclidean bundles}
        Let $\bar{E}$ denote the trivial bundle of Euclidean Killing fields over $\R^3$. The identification between $\bar{E}$ and $T\R^3\times T\R^3$ can be done along similar lines as hyperbolic case, see \cite[Section 3]{hmcb}. Let $\bar{\kappa}\in\Gamma(\bar{E})$ and $\bar{\kappa}_{\bar{x}}$ be its value at $\bar{x}\in\R^3$.
        We can write $$\bar{\kappa}_{\bar{x}}:=\bar{u}_{\bar{x}}+\bar{v}_{\bar{x}},$$ where $\bar{u}_{\bar {x}}$ is the  \emph{infinitesimal pure translation at $\bar{x}$} (that is, $\bar{u}_{\bar{x}}(\bar{x})=\bar{\kappa}_{\bar{x}}(\bar{x})$ and $\bar{\nabla}_{\bar{X}}\bar{u}_{\bar{x}}=0$ for all $\bar{X}\in T_{\bar{x}}\R^3$), while $\bar{v}_{\bar{x}}$ is the \emph{infinitesimal pure rotation at $\bar{x}$} (that is, $\bar{u}_{\bar{x}}(\bar{x})=0$ and there exists a vector say $\bar{\sigma}(\bar{x})\in T_{\bar{x}}\R^3$ such that $\bar{\nabla}_{\bar{X}} \bar{v}_{\bar{x}}=\bar{\sigma}(\bar{x})\times \bar{X}$ for all $\bar{X}\in T_{\bar{x}}\R^3$), here $\bar{\nabla}$ is the Levi-Civita connection of $\R^3$. Therefore, we can identify a Euclidean Killing field $\bar{\kappa}_{\bar{x}}$ with a pair of tangent vectors at $\bar{x}$, say $(\bar{\tau}(\bar{x}),\bar{\sigma}(\bar{x}))\in T_{\bar{x}}\R^3\times T_{\bar{x}}\R^3$, satisfying that $$\bar{\tau}(\bar{x}):=\bar{\kappa}_{\bar{x}}(\bar{x}), \quad\quad \bar{\nabla}_{\bar{X}} \bar{\kappa}_{\bar{x}}=\bar{\sigma}(\bar{x})\times \bar{X},$$
      for all $\bar{X}\in T_{\bar{x}}\R^3$.

      In this way, a section $\bar{\kappa}\in\Gamma(\bar{E})$ is identified with a pair of vector fields $(\bar{\tau},\bar{\sigma})\in\Gamma(T\R^3)\times\Gamma(T\R^3)$.
       we call $\bar{\tau}$ (resp. $\bar{\sigma}$) the \emph{translation} (resp. \emph{rotation}) \emph{component} of $\bar{\kappa}$. Under this identification and by a direct computation of the flat connection $\bar{D}$ for the bundle $\bar{E}$, it follows that for any $\bar{x}\in\R^3$ and any $\bar{X}\in T_{\bar{x}}\R^3$, the one-form $\bar{\omega}=d^{\bar{D}}\bar{\kappa}$ can be written as (see \cite[Section 3]{hmcb}).
        \begin{equation}\label{eq:euclidean-one-form}
        \bar{\omega}(\bar{X})
        =d^{\bar{D}}_{\bar{X}}(\bar{\tau},\bar{\sigma})
        =\big(
        \bar{\nabla}_{\bar{X}} \bar{\tau}-\bar{\sigma}(\bar{x})\times \bar{X}, \bar{\nabla}_{\bar{X}} \bar{\sigma}
        \big)
        =:\big(
        \bar{\omega}_{\bar{\tau}}(\bar{X}),
        \bar{\omega}_{\bar{\sigma}}(\bar{X})
        \big).
        \end{equation}
        We call $\bar{\omega}_{\bar{\tau}}$ (resp. $\bar{\omega}_{\bar{\sigma}}$) the \emph{translation} (resp. \emph{rotation}) \emph{part} of $\bar{\omega}$.
        Note that $\bar{\kappa}$ is a constant section if and only if $d^{\bar{D}}\bar{\kappa}=0$, also, if and only if $\bar{\tau}$ and $\bar{\sigma}$ are Euclidean Killing fields.

\subsection{The infinitesimal Pogorelov map} \label{subsec:Pogorelov}
We first recall the definition of the infinitesimal Pogorelov map and its key properties (see e.g. \cite{Po,shu} and in particular \cite[Definitions 1.9 and 1.12]{hmcb}). 

Choose an affine chart $x_0 = 1$ and denote by $O^*$ the projective plane in $\bdS^3$ dual to the center $O:=[1,0,0,0]\in\HH^3$ in $\bHS^3$. Let $U=\bHS^3\setminus(O^*\cup\partial\overline{\HH^3})$. Then $U$ is the intersection of $\bHS^3\setminus\partial\overline{\HH^3}$ with the aforementioned affine chart of $\RP^3$.  Let $\R^3$ denote the 3-dimensional Euclidean space and let $\iota:U\rightarrow \R^3$ be an inclusion that sends $O$ to the origin $\bar{O}$ of $\R^3$, with $\iota(\HH^3)$ the unit open ball $\bD^3$. It can be checked directly from the hyperbolic metric on $\HH^3$ that $\iota$ is an isometry on the tangent space at $O$.

For any $x\in U$ and any vector $v\in T_xU$, we can write $v=v^r+v^{\perp}$, where $v^r$ is tangent to the geodesic in $\bHS^3$ going through $O$ and $x$ (called the \emph{radial} component of $v$) and $v^{\perp}$ is orthogonal to $v^r$ (called the \emph{lateral} component of $v$). Let $\Upsilon: TU\rightarrow T\R^3$ be the bundle map over the inclusion $\iota$, which is defined in the following way:
$$\Upsilon(v)=d\iota(v)$$
for all $v\in T_OU$, and
$$\Upsilon(v)=\sqrt{\frac{\langle \hat{x},\hat{x}\rangle_{HS}}{\langle d\iota(\hat{x}),d\iota(\hat{x})\rangle_{\R^3}}}d\iota(v^r)+d\iota(v^{\perp}),$$
for all $v\in T_xU$ with $x\in U$, where 
$\hat{x}$ is an arbitrary non-zero radial vector at $x$ (note that the term under the square root is homogenous in $\hat{x}$), and $\langle\cdot,\cdot\rangle_{HS}$ is the HS scalar product. We call this bundle map an \emph{infinitesimal Pogorelov map}.

Note that $\Upsilon$ sends radial (resp. lateral) directions to radial (resp. lateral) directions. In particular, the radial component of $\Upsilon(v)$ has the same norm as the radial component of $v$ whenever $x\in \HH^3$ or $\bdS^3$ (with respect to corresponding metrics), while $\Upsilon$ acts on the lateral component as the differential of the projective model map: the norm of the lateral component of $\Upsilon(v)$ is $1/{\cosh(\rho)}$ (resp. $1/\sinh(\rho)$) times that of the lateral component of $v$ for $x\in \HH^3$ (resp. $x\in\bdS^3$), where $\rho$ is the hyperbolic distance from $x$ to $O$ (resp. the dS distance from $x$ to $O^*$), see e.g. \cite[Propositions 1.2 and 1.3]{hmcb}.

 We say that a vector field on a surface $S\subset\HH^3$ (resp. a spacelike surface $S'\subset\bdS^3$) is \emph{isometric} if it preserves the induced metric on $S\subset\HH^3$ (resp. $S'\subset\bdS^3$) at first order. Similarly for that on a surface $\bar{S}\subset\R^3$. We stress the following property of the bundle map $\Upsilon$, see \cite[Lemma 1.10]{hmcb}, which is crucial for the local rigidity problem here.

\begin{lemma}\label{lm:Pogorelov isometric}
The following statements are true:
\begin{enumerate}
\item Let $Z$ be a vector field on $U$. Then $Z$ is a Killing vector field (for the HS metric) if and only if $\Upsilon(Z)$ is a Killing vector field (for the Euclidean metric).
\item  Let $v$ be a vector field on an embedded surface $S\subset U$. Then $v$ is an isometric vector field on $S$ if and only if $\Upsilon(v)$ is an isometric vector field on $\iota(S)$.
\end{enumerate}
\end{lemma}

As a consequence of Lemma \ref{lm:Pogorelov isometric}, the infinitesimal Pogorelov map $\Upsilon$ induces the following map between sections of $E$ (restricted to $\HH^3$) and $\bar{E}$ (see e.g. \cite[Definition 3.5]{hmcb}):
\begin{equation*}
\begin{split}
\Psi: E &\rightarrow \bar{E}\\
\kappa&\mapsto \bar{\kappa},
\end{split}
\end{equation*}
which takes a hyperbolic Killing field $\kappa_x$ (associated to $x$) to a Euclidean Killing field $\bar{\kappa}_{\bar{x}}$ (associated to $\bar{x}=\iota(x)$), with 
$\bar{\kappa}_{\bar{x}}(\bar{y}):=\Upsilon(\kappa_x(y))$ for all $y\in \HH^3$ and $\bar{y}=\iota(y)$. For convenience, we also denote by $\Psi:\Gamma(E)\rightarrow\Gamma(\bar{E})$ the restriction of $\Psi$ to $\Gamma(E)$. In particular, $\Psi$ sends flat sections of $E$ for $D$ to flat sections of $\bar{E}$ for $\bar{D}$ and therefore, it respects the connections $D$ and $\bar{D}$ \cite[Proposition 3.6]{hmcb}. We have
\begin{equation}\label{eq:psi(w)}
\Psi(d^D_{X}\kappa):=d^{\bar{D}}_{d\iota(X)}(\Psi(\kappa)),
\end{equation}
for any section $\kappa\in\Gamma(E)$ and $X\in T\HH^3$. For simplicity, we denote $\Psi_*(d^D\kappa):=d^{\bar{D}}(\Psi(\kappa))$.

\subsection{The duality between $\HH^3$ and $\bdS^3$}\label{ssc:dual}
The (polar) duality between a point (resp. a totally geodesic plane) in $\HH^3$ and a totally geodesic plane (resp. a point) in $\bdS^3$ is defined in the following way.

Consider the hyperboloid models $H^3$ and $dS^3$ in $\R^{3,1}$. Given a point $x\in H^3$ (which is the intersection with $H^3$ of the line $l_x\subset\R^{3,1}$ in the direction of $x$), we define the \emph{dual} of $x$ as the intersection $(l_x)^{\perp}\cap dS^3$, where $(l_x)^{\perp}$ is the orthogonal complement in $\R^{3,1}$ of $l_x$ with respect to the Minkowski metric $\langle \cdot,\cdot\rangle_{3,1}$. Given a totally geodesic plane $P\subset H^3$ (which is obtained from the intersection with $H^3$ of a hyperspace $H_P\subset\R^{3,1}$), we define the \emph{dual} of $P$ as the intersection $(H_P)^{\perp}\cap dS^3$. Similarly, we can define the dual of a point $x'\in dS^3$ and the dual of a totally geodesic spacelike plane $P'\subset dS^3$ (note that the spacelike condition ensures that the intersection $(H_{P'})^{\perp}\cap H^3$ is non-empty and the dual of $P'$ is thus well-defined). Through the projective maps, the duality is naturally defined in the projective models $\HH^3$ and $\bdS^3$.

Using this duality, one can define the \emph{dual surface}, denoted by $S^*$, of a convex embedded surface  $S\subset\HH^3$ as the set of points which are dual to the tangent planes $P_x$ to $S$ at all $x\in S$. Similarly, one can define the \emph{dual surface}, denoted by $(S')^*$, of a convex embedded spacelike surface  $S'\subset\bdS^3$. In particular, we denote by $x^*$ the point in $S^*$ which is dual to the tangent plane to $S$ at $x\in S$.

We list the following relations and properties between dual surfaces, see e.g. \cite[Section 3]{shu}.

\begin{proposition}\label{prop:dual}
Let $S\subset\HH^3$ be a strictly convex embedded surface. Then
\begin{enumerate}
  \item The dual surface $S^*$ of $S$ is a strictly convex space-like surface in $\bdS^3$.
  \item The dual surface $(S^*)^*$ of $S^*$ is exactly $S$.
  \item The pull-back of the induced metric on $S^*$ through the duality map is the third fundamental form of $S$ and vice versa.
  \item $S$ has curvature $K\in(-1,+\infty)$ at a point $x\in S$ if and only if $S^*$ has curvature $K^*=K/(K+1)\in(-\infty,1)$ at the dual point $x^*\in S^*$.
\end{enumerate}
\end{proposition}

%% file: short-weyl3.tex
\section{Local rigidity}
\label{sec:local rigidity}

\subsection{General overview}

Let $g\in\cCC(M)$ be a convex co-compact hyperbolic structure on $M$ with conformal structure at infinity $c_i$ on $\partial_i M$ (for each $i\in\cI$), with induced metric $h_j$ on a smooth, strictly convex surface $S_j\subset M$ isotopic to $\partial_jM$ (for each $j\in\cJ$), and with third fundamental form $h^*_k$ on a smooth, strictly convex surface $S_k\subset M$ isotopic to $\partial_kM$ (for each $k\in\cK$). In this section, we aim to show the infinitesimal rigidity of $g$ with respect to $c_i$, $h_j$ and $h^*_k$ over $i\in\cI, j\in\cJ, k\in\cK$ (as stated below).

\begin{proposition}\label{lm:local rigidity-mfld}
For any first-order deformation $\dot{g}\in T_g\cCC(M)$, if $\dot{g}$ preserves the conformal structure at infinity  $c_i$ on $\partial_iM$ (for all $i\in\cI$), the induced metric $h_j$ on $S_j$ (for all $j\in\cJ$) and the third fundamental form $h^*_k$ on $S_k$ (for all $k\in\cK$) at first order, then $\dot{g}$ is trivial.
\end{proposition}

\subsubsection{From $M$ to $\HH^3$}\label{subsec:from M to H3}

For simplicity, we denote by $M$ the hyperbolic manifold with convex co-compact structure $g\in\cCC(M)$. Let $N\subset M$ be the geodesically convex subset with the conformal boundary at infinity $\partial_iM$ and the smooth, strictly convex boundary surfaces $S_j$ and $S_k$ over $i\in\cI$, $j\in\cJ$ and $k\in\cK$. We denote $\partial_iN:=\partial_iM$, $\partial_jN:=S_j$ and $\partial_kN:=S_k$.

Let $\Omega$ denote the the lift of $N$ in the universal cover of $M$ (which is identified with $\overline{\HH^3}$). Up to isometries, we always assume that the center $O$ of $\HH^3$ is contained in the interior of the convex hull of the limit set $\Lambda$ of the Kleinian group of $M$ (which descends to the convex core $C_M$ of $M$), and is thus contained in the interior of $\Omega$. Moreover, we assume that the family $(dev_t)_{t\in [0,\epsilon]}$ of developing maps associated to the first order deformation $\dot{g}$ (see Section \ref{subsubsec: deformation vector field}) is normalized such that the deformation section vanishes at $O$ (namely, $\tau(O)=\sigma(O)=0$). This will help to simplify some estimates we need later (see Lemma \ref{lm:pogorelov_radial and lateral} for instance).

We denote by $\partial_i\Omega\subset\partial\overline{\HH^3}$ ($i\in\cI$) the union of all the lifts in $\partial\overline{\HH^3}$ of the conformal boundary component $\partial_iM$, by $\partial_j\Omega\subset\HH^3$ ($j\in\cJ$) the union of all the lifts in $\HH^3$ of $S_j$, and by $\partial_k\Omega\subset\HH^3$ ($k\in\cK$) the union of all the lifts in $\HH^3$ of $S_k$. Let $(\partial_k\Omega)^*$ denote the dual surface in $\bdS^3$ of the strictly convex surface $\partial_k\Omega\subset\HH^3$ and let $\Omega^*$ denote the convex domain in $\RP^3$, sharing the same boundary with $\Omega$, except with $\partial_k\Omega$ replaced by $(\partial_k\Omega)^*$.

\subsubsection{From $\HH^3$ to $\R^3$}

Recall that $\iota$ is an inclusion in $\R^3$ that sends the projective model $\HH^3$ to the unit open ball $\bD^3$ (see Section \ref{subsec:Pogorelov}).
We denote $\bar{\Omega}:=\iota(\Omega)$,  $\partial_l\bar{\Omega}=\iota(\partial_l\Omega)$ ($1\leq l\leq n$)
and $\bar{\Lambda}=\iota(\Lambda)$. Note that the center $\bar{O}$ of $\bD^3$ is contained in the interior of $\bar{\Omega}$ (by the above assumption) and
$$\partial\bar{\Omega}
=(\cup_{i\in\cI}\partial_i\bar{\Omega})\cup(\cup_{j\in\cJ}\partial_j\bar{\Omega})\cup(\cup_{k\in\cK}\partial_k\bar{\Omega})\cup\bar{\Lambda}~.$$

 Let $\bar{\Omega}^*=\iota(\Omega^*)$ and denote $\partial_k\bar{\Omega}^*:=\iota((\partial_k\Omega)^*)$, $k\in\cK$. Then
$$\partial\bar{\Omega}^*
=(\cup_{i\in\cI}\partial_i\bar{\Omega})\cup(\cup_{j\in\cJ}\partial_j\bar{\Omega})\cup(\cup_{k\in\cK}\partial_k\bar{\Omega}^*)\cup\bar{\Lambda}~.$$

In the case that $M$ is quasifuchsian, the limit set $\bar{\Lambda}$ is a Jordan curve and $\partial\bar{\Omega}\setminus\bar{\Lambda}$ has two connected components. In the case that 
$M$ has $n>2$ boundary components, $\partial\bar{\Omega}\setminus\bar{\Lambda}$ is a countably infinite union of disjoint open disks, and for each $i\in\cI$ (resp. $j\in\cJ$, $k\in\cK$), $\partial_i\bar{\Omega}$ (resp. $\partial_j\bar{\Omega}$, $\partial_k\bar{\Omega}$ and $\partial_k\bar{\Omega}^*$) has countably infinitely many connected components.

\subsubsection{The boundary behavior of the convex domains $\bar{\Omega}$ and $\bar{\Omega}^*$.}

The following describes the behavior of the principal curvatures of the image in $\DD^3\subset\R^3$ (resp. $\R^3\setminus\overline{\DD^3}$) under $\iota$ of a strictly convex surface in $\HH^3$ (resp. a strictly convex spacelike surface in $\bdS^3$), based on the fact that the second fundamental form of the surface changes conformally (see \cite[Lemmas 1.4, 2.5 and 2.8]{hmcb}).


\begin{lemma}\label{lm:euc pcurv bound}
Let $S$ be either $\partial_j\Omega$, for some $j\in \cJ$, or $\partial_k\Omega$ (resp. $(\partial_k\Omega)^*$), for some $k\in \cK$. If $S$ has principal curvatures in $[k_{min},k_{max}]$ for some $k_{max}>k_{min}>0$, then $\iota(S)$ is a convex surface in $\R^3$, with principal curvatures in $[k'_{min},k'_{max}]$ for some $k'_{max}>k'_{min}>0$.
\end{lemma}

\begin{remark}\label{rk:Hausdorff dim}
 It is shown by Sullivan \cite{sullivan:entropy} that $\bar{\Lambda}$ has Hausdorff dimension in $[1,2)$. Moreover, Bishop has proved that the Hausdorff dimension and Minkowski dimension are the same for the limit set of analytically finite Kleinian groups \cite{bishop-minkowski}.
\end{remark}

\begin{remark}\label{rk:regularity-boundary}
 The surface $\partial\bar{\Omega}\setminus\bar{\Lambda}$ (resp.
 $\partial\bar{\Omega}^*\setminus\bar{\Lambda}$) is smooth, while $\partial\bar{\Omega}$ (resp. $\partial\bar{\Omega}^*$) is $C^{1,1}$ (see e.g. \cite[Corollaries 2.6 and 2.9]{hmcb}). In particular, each connected component of $\partial_i\bar{\Omega}$ (resp. $\partial_j\bar{\Omega}$, $\partial_k\bar{\Omega}$, $\partial_k\bar{\Omega}^*$) is contained in $\partial\DD^3$ (resp. tangent to $\partial\DD^3$ along its boundary).
 \end{remark}

Note that we do not claim that $\partial \bar{\Omega}$ (resp. $\partial \bar{\Omega}^*$) is $C^2$, however the principal curvatures are well-defined almost everywhere and we do claim that they are bounded:

  \begin{lemma}\label{prop:boundary regularity}
  $\partial\bar{\Omega}$ (resp. $\partial\bar{\Omega}^*$) is in the class $D_{2,\infty}$, namely, it satisfies the following:
      \begin{enumerate}
      \item The principal curvatures on $\partial\bar{\Omega}$ (resp. $\partial\bar{\Omega}^*$) has positive lower and upper bounds almost everywhere.
      \item $\partial\bar{\Omega}$ (resp. $\partial\bar{\Omega}^*$) can be locally written as the graph of a function with bounded measurable second derivatives.
      \end{enumerate}
   \end{lemma}
   \begin{proof}
      Statement (1) follows from Lemma \ref{lm:euc pcurv bound}, Remarks \ref{rk:Hausdorff dim} and \ref{rk:regularity-boundary}, and the fact that $\partial_i\Omega$ (resp. $\partial_j\Omega$, $\partial_k\Omega$ and $(\partial_k\Omega)^*$) descends to a closed, strictly convex surface in $M$, with principal curvatures bounded between two positive constants. Statement (2) follows from  Statement (1), the $C^{1,1}$-smoothness (see Remark \ref{rk:regularity-boundary}) and strict convexity of $\partial\bar{\Omega}$ (resp. $\partial\bar{\Omega}^*$).
   \end{proof}

For each 
$\bar{x}\in\partial\bar{\Omega}$ (resp. $\bar{x}\in\partial\bar{\Omega}^*$),
let $\theta_{\bar{x}}$ denote the angle between the radial direction along $\bar{O}\bar{x}$ at $\bar{x}$ and the outward-pointing normal direction to 
$\partial\bar{\Omega}$ (resp. $\partial\bar{\Omega}^*$) at $\bar{x}$ and let $\delta_{\partial\bar{\Omega}}(\bar{x},\bar{\Lambda})$ (resp. $\delta_{\partial\bar{\Omega}^*}(\bar{x},\bar{\Lambda})$ ) denote the distance from the point $\bar{x}$ to the limit set $\bar{\Lambda}$ along $\partial\bar{\Omega}$ (resp. $\partial\bar{\Omega}^*$).

Since $\partial_i\bar{\Omega}\subset\partial\DD^3$, it is clear that $\theta_{\bar{x}}=0$  for all $\bar{x}\in \partial_i\bar{\Omega}$ ($i\in\cI$). For the other
$\bar{x}\in\partial\bar{\Omega}\cap\DD^3$ or $\bar{x}\in\partial\bar{\Omega}^*\setminus\overline{\DD^3}$, the estimate of the angle $\theta_{\bar{x}}$ and the distances $\delta_{\partial\bar{\Omega}}(\bar{x},\bar{\Lambda})$, $\delta_{\partial\bar{\Omega}^*}(\bar{x},\bar{\Lambda})$ in \cite[Lemmas 2.4 and 2.7]{hmcb} still work and we conclude the following:

\begin{lemma}\label{lm: angle estimate}
There exists a constant $C>1$, such that
\begin{enumerate}
 \item
 For any $\bar{x}\in\partial\bar{\Omega}\cap\DD^3$
 with $\bar{x}$ close enough to $\bar{\Lambda}$, we have $\theta_{\bar{x}}\leq C\sqrt{1-\|\bar{x}\|}$ and moreover,
     $$\frac{\sqrt{1-\|\bar{x}\|}}{C}\leq
     \delta_{\partial\bar{\Omega}}(\bar{x},\bar{\Lambda})\leq C\sqrt{1-\|\bar{x}\|}~.$$
  \item
 For any $\bar{x}\in\partial\bar{\Omega}^*\setminus\overline{\DD^3}$ with $\bar{x}$ close enough to $\bar{\Lambda}$, we have
 $\theta_{\bar{x}}\leq C\sqrt{\|\bar{x}\|-1}$ and moreover,
 $$\frac{\sqrt{\|\bar{x}\|-1}}{C}\leq \delta_{\partial\bar{\Omega}^*}(\bar{x},\bar{\Lambda})\leq C\sqrt{\|\bar{x}\|-1}~.$$
\end{enumerate}
\end{lemma}

\begin{remark}\label{rk:equidistant-angle}
Let $\tilde{S}^i_t$ be the lift in $\Omega$ of the equidistant surface $S^i_t$ (which is a closed strictly convex surface) at distance $t$ from $\partial_iC_M$ (see Subsection \ref{subsec:foliation}) and let $\bar{S}^i_t:=\iota(\tilde{S}^i_t)$. Note that for each $t\geq 1$, the principal curvatures of $\bar{S}^i_t$ have a uniform positive lower and upper bound for all $i\in\cI$ (by Lemma \ref{lm:euc pcurv bound} and Remark \ref{rk:comparsion}). Statement (1) in Lemma \ref{lm: angle estimate} still holds for any $\bar{x}\in \cup_{i\in\cI}\bar{S}^i_t$ (with $\partial\bar{\Omega}$ replaced by $\cup_{i\in\cI}\bar{S}^i_t$) close enough to $\bar{\Lambda}$.
\end{remark}

\subsubsection{The Vekua lemma}

Let $\bar{I}$ (resp. $\bar{\II}$, $\bar{\III}$) and $\bar{B}$ denote the first (resp. second, third) fundamental form and the shape operator of an embedded strictly convex surface $\bar{S}$ in $\R^3$. We define
$$\bar{B}(\bar{X}):=-\bar{\nabla}_{\bar{X}}\bar{n}~,\quad\quad
 \bar{\II}(\bar{X},\bar{Y})=\bar{I}(\bar{B}(\bar{X}),\bar{Y})~,\quad\quad \bar{\III}(\bar{X},\bar{Y})=\bar{I}(\bar{B}(\bar{X}),\bar{B}(\bar{Y}))~,$$
 for $\bar{X},\bar{Y}\in T\bar{S}$, where $\bar{n}$ is the unit normal vector field on $\bar{S}$ oriented such that the eigenvalues of $\bar{B}$ are positive, and $\bar{\nabla}$ is the Levi-Civita connection of $\R^3$.

Let $\bar{J}$ be the complex structure associated to $\bar{\II}$ and let $\nabla^{\bar{\III}}$ denote the Levi-Civita connection for $\bar{\III}$. By computation (see \cite[Proposition 3.1]{hmcb-v4}), $$\nabla^{\bar{\III}}_{\bar{X}}\bar{Y}=\bar{B}^{-1}\nabla^{\bar{I}}_{\bar{X}}(\bar{B}\bar{Y})~,$$
where $\bar{X},\bar{Y}\in T\bar{S}$ and $\nabla^{\bar{I}}$ is the Levi-Civita connection for $\bar{I}$.

Let $\bar{\partial}_{\bar{\III}}$ be the operator that takes a vector field $\bar{W}$ over $\bar{S}$ to a section $\bar{\partial}_{\bar{\III}}\bar{W}$ of the bundle of one-forms with values in $T\bar{S}$, defined by
\begin{equation*}
(\bar{\partial}_{\bar{\III}}\bar{W})(\bar{X}):={\nabla}^{\bar{\III}}_{\bar{X}}\bar{W}
+\bar{J}{\nabla}^{\bar{\III}}_{\bar{J}\bar{X}}\bar{W}~,
\end{equation*}
for all $\bar{X}\in T\bar{S}$. It follows from a direct computation that
$$ (\bar{\partial}_{\bar{\III}}\bar{W})(\bar{J}\bar{X})=-\bar{J}(\bar{\partial}_{\bar{\III}}\bar{W})(\bar{X})~. $$
Let $\bar{v}$ be a vector field on $\bar{S}$. The following equation will play a central role, so we give it a specific name.
\begin{equation}
  \tag{E}
  \label{eq:E}
  \bar\partial_{\bar{\III}}(\bar{B}^{-1}\bar{v})=0~.
\end{equation}
It ensures that the Lie derivative of the induced metric $\bar{I}$ on $\bar{S}\subset\R^3$ with respect to $\bar v$ is parallel to $\bar{\II}$ (see the proof of Lemma \ref{lm:equation} below for more details).

The following states a crucial relation between an isometric deformation vector field on an embedded surface $\bar{S}\subset\R^3$ and equation \eqref{eq:E}.

\begin{lemma}\label{lm:equation}
Let $\bar{U}$ be an isometric deformation vector field on an embedded convex surface $\bar{S}\subset\R^3$ (with the shape operator $\bar{B}$) and let $\bar{V}$ be its component tangent to $\bar{S}$. Then $\bar{V}$ satisfies equation $(E)$. Conversely, let $\bar{V}$ be a solution of equation \eqref{eq:E}, there exists a unique isometric deformation vector field of $\bar{S}$ whose component tangent to $\bar{S}$ is $\bar{V}$.
\end{lemma}

\begin{proof}
The first statement was shown in \cite[Proposition 3.8]{hmcb-v4} -- note that we refer here to the arxiv version (v4) rather than to the published version.

It suffices to show the second statement. Assume that $\bar{V}$ is a solution of equation \eqref{eq:E}, then by a direct computation (see also \cite[Proposition 3.8]{hmcb-v4}), we have
\begin{equation}\label{eq:II}
\begin{split}
\bar{\II}(\partial_{\bar{\III}}(\bar{B}^{-1}\bar{V})\bar{X},\bar{X})
=\bar{I}(\nabla^{\bar{I}}_{\bar{X}}\bar{V},\bar{X})-\bar{I}(\nabla^{\bar{I}}_{\bar{J}\bar{X}}\bar{V},\bar{J}\bar{X})=0~,\\
\bar{\II}(\partial_{\bar{\III}}(\bar{B}^{-1}\bar{V})\bar{X},\bar{J}\bar{X})
=\bar{I}(\nabla^{\bar{I}}_{\bar{X}}\bar{V},\bar{J}\bar{X})+\bar{I}(\nabla^{\bar{I}}_{\bar{J}\bar{X}}\bar{V},\bar{X})=0~.\\
\end{split}
\end{equation}
Since $\bar{J}$ is the complex structure associated to $\bar{\II}$, then $\bar{\II}(\bar{X},\bar{J}\bar{X})=0$. Note that $$(L_{\bar{V}}\bar{I})(\bar{X},\bar{Y})=\bar{I}(\nabla^{\bar{I}}_{\bar{X}}\bar{V},\bar{Y})
+\bar{I}(\bar{X},\nabla^{\bar{I}}_{\bar{Y}}\bar{V})~.$$
 Assume that $\bar{X}\not=0$. When $\bar{Y}=\bar{J}\bar{X}$, combined with \eqref{eq:II}, then
$$(L_{\bar{V}}\bar{I})(\bar{X},\bar{J}\bar{X})=0
=\bar{\II}(\bar{X},\bar{J}\bar{X})
=2\frac{\bar{I}(\nabla^{\bar{I}}_{\bar{X}}\bar{V},\bar{X})}{\bar{\II}(\bar{X},\bar{X})}\bar{\II}(\bar{X},\bar{J}\bar{X})~.$$
When $\bar{Y}=\bar{X}$, using \eqref{eq:II} again, then
$$(L_{\bar{V}}\bar{I})(\bar{X},\bar{X})
=2\bar{I}(\nabla^{\bar{I}}_{\bar{X}}\bar{V},\bar{X})
=2\frac{\bar{I}(\nabla^{\bar{I}}_{\bar{X}}\bar{V},\bar{X})}{\bar{\II}(\bar{X},\bar{X})}\bar{\II}(\bar{X},\bar{X})~.$$
Define $f(\bar{x}):=\bar{I}(\nabla^{\bar{I}}_{\bar{Z}_{\bar{x}}}\bar{V},\bar{Z}_{\bar{x}})/\bar\II(\bar{Z}_{\bar{x}},\bar{Z}_{\bar{x}})$, where $\bar{Z}_{\bar{x}}\in T_{\bar{x}}\bar{S}$ is an arbitrary non-zero vector. By \eqref{eq:II} and the fact that $\bar\II(\bar{X},\bar{X})=\bar\II(\bar{J}\bar{X},\bar{J}\bar{X})$, it can be directly checked that $f$ depends only on the base point of $\bar{Z}_{\bar{x}}$ and is thus a well-defined function on $\bar{S}$. Therefore,
 \begin{equation}\label{eq:parallel}
(L_{\bar{V}}\bar{I})(\bar{X},\bar{Y})=2f\bar{\II}(\bar{X},\bar{Y})~.
\end{equation}
Let $\bar{U}:=\bar{V}+f\bar{n}$, where $\bar{n}$ is the unit normal vector field on $\bar{S}$ defined above. Then by \eqref{eq:parallel},
$$(L_{\bar{U}}\bar{I})(\bar{X},\bar{Y})=(L_{\bar{V}+f\bar{n}}\bar{I})(\bar{X},\bar{Y})
=(L_{\bar{V}}\bar{I})(\bar{X},\bar{Y})-2f\bar{\II}(\bar{X},\bar{Y})=0~.$$
As a consequence, $\bar{V}$ is the tangential component of the unique isometric vector field $\bar{U}$ on $\bar{S}$.
\end{proof}

\begin{remark}\label{rk:equation}
Note that in Lemma \ref{lm:equation}, the surface $\bar{S}$ can be locally written as the graph of a function say $u=f(x,y)$. If $\bar{U}$ is an isometric deformation vector field on $\bar{S}$, then the tangential component $\bar{V}$ of $\bar{U}$ satisfies the generalized Beltrami equation:
\begin{equation}\label{eq:Beltrami}
w_{\bar{z}}-\mu(z)(w_{z}+\bar{w}_{\bar{z}})=0~,
\end{equation}
where $w(z)$ is the local expression of the tangential component $\bar{V}$ in a local chart $z=x+iy$, and $\mu(z)=(1/2)(f_{\bar{z}\bar{z}}/f_{z\bar{z}})$ with $\|\mu\|_{\infty}<1$ (see e.g. \cite[V.2.2, Eq (2.9) p.397]{vekua}). By Lemma \ref{lm:equation}, the tangential component $\bar{V}$ satisfies equation \eqref{eq:E} if and only if it satisfies equation \eqref{eq:Beltrami}.
\end{remark}


A solution of equation \eqref{eq:E} is said to be \emph{trivial} if it is the tangential component (restricted to $\bar{S}$) of a global Killing field of $\R^3$, which corresponds to a trivial infinitesimal deformation of $\bar{S}$ (see \cite[Section 5]{hmcb-v4} and \cite[V.1, p. 394]{vekua}).

\begin{definition}
  Let $\bar v$ be a vector field defined along a surface $\bar S$. We say that $\bar v$ is in $D_{1,\infty}(T\bar S)$ (resp. $D_{1,p}(T\bar S)$) 
  if the function $\bar{x}\mapsto \bar{v}(\bar x)$ with $\bar x\in\bar{S}$ and $\bar{v}(\bar x)\in T_{\bar x}\bar{S}$, whose derivatives up to order one (which exist in the sense of distribution) are in $L^{\infty}$ (resp. $L^p$).
\end{definition}

We now introduce a result of Vekua \cite[V.2.2, p. 393-402]{vekua}, which is crucial for the proof of Proposition \ref{lm:local rigidity-mfld}.

\begin{lemma}[Vekua]\label{lm:Vekua}
  Let $\bar{\Omega}\subset \R^3$ be a bounded convex subset such that $\partial\bar{\Omega}$ is in $D_{2,\infty}$. Suppose that $\bar v$ is a solution of \eqref{eq:E} which is continuous on $\partial\bar{\Omega}$ and belongs to $D_{1,p}(T\partial\bar{\Omega})$, $p>2$. Then $\bar v$ is trivial. 
\end{lemma}

\begin{remark}
  Vekua assumes in his argument that the deformation vector field $(\xi_*, \eta_*, \zeta_*)$ on the closed surface $S_*$ that he considers is in $D_{1,p}(TS_*)$, $p>2$, see \cite[p. 400]{vekua}. However one of the first steps in Vekua's proof is to apply  projective transformations to the surface $S_*$ and the vector field $(\xi_*, \eta_*, \zeta_*)$, the obtained surface (denoted by $S$) is then a graph $z=z(x,y)$ over the $(xy)$ plane, with the new deformation vector field $(\xi,\eta,\zeta)$. He then considers a complex function $w$ on $S$ 
 (see \cite[Eq (2.27), p. 401]{vekua}), which can be indeed written as $w=u+iv$, where $u$ and $v$ are the scalar products of the deformation vector field $(\xi,\eta,\zeta)$ with the tangent vector fields $(1,0,z_x)$ and $(0,1,z_y)$ to the surface, respectively (see \cite[(2.6) and (2.8), p. 396]{vekua}).

  In Vekua's argument, 
  the surface is assumed to be in the class $D_{2,\infty}$ (see Lemma \ref{prop:boundary regularity} for definition). The tangent vector fields $(1,0,z_x)$ and $(0,1,z_y)$ in the previous paragraph are thus in $D_{1,\infty}(TS)$. Note that the above scalar product expressions of $u$ and $v$ are still valid by replacing the deformation vector field $(\xi,\eta,\zeta)$ with its tangential component. 
  It follows that if the tangential component of the deformation vector field $(\xi,\eta,\zeta)$ is in $D_{1,p}(TS)$, then $u$ and $v$ are both in $D_{1,p}$, and therefore, Vekua's complex function $w$ (appearing in the previous paragraph) is also in $D_{1,p}$.

  The rest of Vekua's argument only uses this vector field $w$. It follows that, if our vector field $\bar v$ (as the tangential component of the deformation vector field $\bar u$, see Section \ref{subsec:deformation field}) is in $D_{1,p}(T\partial\bar\Omega)$, Vekua's argument can be used as it is.
  \end{remark}

We will use this lemma through \cite[Lemma 5.5]{hmcb}, which we recall here in slightly different notations.

\begin{lemma}\label{lm:D(1,p)-condition}
  Let $\bar{\Omega}\subset \R^3$ be a bounded convex subset. Let $\bar{\Lambda}\subset\partial\bar{\Omega}$ be a closed subset of Hausdorff dimension 
  $D_{\bar{\Lambda}}\in[1,2)$.
  Let $p>2$ and let $f:\partial\bar{\Omega}\rightarrow\R$ be a function such that
  \begin{itemize}
  \item $f$ is everywhere $C^{\alpha}$ for some $\alpha\in(D_{\bar{\Lambda}}-1,1)$.
  \item $f$ is smooth outside $\bar{\Lambda}$.
  \item $\int_{\partial\bar{\Omega}\setminus\bar{\Lambda}}\|df\|^pda<\infty$, where $da$ is the area form of the induced metric on $\partial\bar{\Omega}\subset\R^3$.
  \end{itemize}
  Then $df$, considered as an $L^p$ section (defined on $\partial\bar{\Omega}\setminus \bar\Lambda$ as the usual derivative of $f$, and as $0$ on $\bar\Lambda$), is the generalized derivative of $f$ on $\partial \bar{\Omega}$, and $f$ is in 
  $D_{1,p}(\partial\bar{\Omega})$.
\end{lemma}

\subsection{Estimates on $\omega$}

In this subsection, we aim to show the following statement:

\begin{lemma} \label{lm:omega}
  Let $\dot g$ be a first-order deformation of the hyperbolic structure $g\in\cCC(M)$ which leaves invariant the conformal metric at infinity $c_i$ on $\partial_iM=\partial_iN$, $i\in \cI$, the induced metric $h_j$ on $S_j=\partial_jN$, $j\in \cJ$, and the third fundamental form $h_k$ on $S_k=\partial_kN$, $k\in \cK$, at first order. There exists a representative $\omega$ of $\dot g$, which is a closed one-form in $\Omega^1(M,E_M)$ such that:
  \begin{enumerate}[(i)]
  \item $\omega$ vanishes along the geodesic rays starting orthogonally from a support plane to $\partial C_M$, at all points at distance at least $1$ from $C_M$ and contained in the hyperbolic ends corresponding to $\partial_iM$ ($i\in\cI$).
  \item for each $x\in N$, $\|\omega\|_x\leq C\exp(-2d(x,C_M))$ for a constant $C>0$ depending on $\omega$.
  \end{enumerate}
\end{lemma}

\subsubsection{The deformation one-form on $\partial_iM$}\label{subsec:one-form at infty}

Let $\sigma_i$ ($i\in\cI$) denote the complex projective structure on $\partial_iM$. By assumption, the first-order deformation $\dot g$ of $M$ preserves the conformal structure at infinity on $\partial_iM$, so the deformation vector field $u$ restricted to $\partial_i\Omega$, say $u^i$, is holomorphic and has the form $$u^i(z)=f_i(z)\partial_z~,$$ where $f_i(z)$ is a holomorphic function in an affine chart of $\sigma_i$. The (infinitesimal) Schwarzian derivative of $u^i$ has the following expression (see e.g. \cite[Section 2]{bromberg1}):
$$\mathcal{S}(u^i)(z)=f'''_i(z)dz^2~. $$
One can check that $\mathcal{S}(u^i)(\gamma\cdot z)=\mathcal{S}(u^i)(z)$ for all $\gamma\in {\rm PSL}_2(\C)$, so it determines a quadratic differential on $\partial_iM$.
Moreover, $\mathcal{S}(u^i)\equiv 0$ if and only if $f_i(z)$ is a quadratic polynomial.

Note that the canonical lift $s_{u^i}$ of $u^i$ has local expression in the form of \eqref{eq:approx Mobius}. Define $\omega^i:=d^Ds_{u^i}$. Then for each $z_0\in\partial_i\Omega$ and $U=u(z)\partial_z\in T_{z_0}\C$,
\begin{equation}\label{eq:one-form at infty}
\begin{split}
\omega^i(U)(z_0)=(d^D_Us_{u^i})(z_0)&=\big(u(z)\partial_z\big)|_{z_0}\big((s_{u^i})(z)\big)\\
         &=u(z_0)\partial_{z_0}\big(f_i(z_0)+f'_i(z_0)(z-z_0)+\frac{f''_i(z_0)}{2}(z-z_0)^2\big)\partial_{z}\\
         &=\left(\frac{1}{2}u(z_0)f'''_i(z_0)(z-z_0)^2\right)\partial_z~.
\end{split}
\end{equation}
As a result, $(d^D{s_{u^i}})$ takes values in the vector space of homogeneous polynomials of degree 2 (see also \cite[Lemma 4.8]{qfmp}).

\subsubsection{The equidistant foliations of hyperbolic ends}
\label{subsec:foliation}

It is well-known that $M\setminus C_M$ has $n$ connected components, each called a \emph{hyperbolic end} of $M$, denoted by $\cE^l$ ($1\leq l\leq n$). The connected component of the boundary of $C_M$ as a pleated boundary surface of $\cE^l$ is denoted by $\partial_lC_M$. Moreover, each $\cE^l$ is foliated by a family of equidistant surfaces $(S^l_t)_{t>0}$ lying at a distance $t$ from $S^l_0:=\partial_lC_M$. Since $\partial C_M$ is convex, for each point $p\in \cE^l$, there is a unique geodesic ray (say $r_q$) in $\cE^l$ starting from $q\in\partial_lC_M$ and orthogonal to a support plane to $\partial_lC_M$ at $q$ and going through $p$. We define $G: M\setminus C_M\rightarrow \partial_{\infty}M$ to be the hyperbolic Gauss map that takes each point $p\in \cE^l$ to the endpoint at infinity of the geodesic ray $r_q$. In particular, the union $\cup_{q\in\partial_lC_M}r_q$ of these ``vertical" geodesic rays forms another foliation of $\cE^l$.

  Let $\tilde{C}_M$ (resp. $\tilde{\cE}^l$, $\tilde{S}^l_t$) denote the lift in $\HH^3$ of $C_M$ (resp. $\cE^l$, $S^l_t$). Let $\tilde{G}:\HH^3\setminus\tilde{C}_M\rightarrow\CP^1\setminus\Lambda$ be the lifting map of the Gauss map $G$.

\subsubsection{The metric at infinity associated to an equidistant foliation of hyperbolic ends}\label{subsec:metric at infty}
There are several ways to define a conformal metric at infinity on the boundary $\partial_lM$ of a hyperbolic end $\cE^l$, by using the equidistant foliation of $\cE^l$ (see e.g. \cite[Section 3, 4]{qfmp}). Here we consider the metric at infinity on $\partial_iM$ ($i\in\cI$), defined as
\begin{equation}\label{eq:metric at infty}
g_{i,\infty}:=4e^{-2t}G_*I^*_t~,
\end{equation}
where $I^*_t=I_t+2\II_t+\III_t$ is the ``horospherical metric" on an equidistant surface $S^i_{t}$ ($t\geq 1$), with first, second and third fundamental forms on $S^i_{t}$ denoted by $I_t, \II_t, \III_t$.
We stress that the metric $g_{i,\infty}$ is independent of the choice of $t$, and moreover, its conformal structure is the same as that underlies the $\CP^1$ structure on $\partial_iM$ (see \cite{horo} and \cite[Section 3.3]{convexhull} for more details).  (This metric $g_{i,\infty}$ is then the Thurston metric at infinity.)

Let $x_t\in \tilde{S}^i_{t}$ ($t\geq 1$). Let $z_0=\tilde{G}(x_t)$ and $z_1$ be the other endpoint at infinity of the complete geodesic passing through $x_t$ and going towards $z_0$ on one side. We choose an affine chart $z$ near $z_0$ such that $z_0$ is identified with $0$, while $z_1$ is identified with $\infty$ in the Poincar\'{e} upper half-space model ${\mathbb{U}}^3:=\{(z,s):z\in\C, s>0\}$, and moreover,
the metric $\tilde{g}_{i,\infty}$, as the lift of $g_{i,\infty}$ to $\partial_i\Omega$, has the following expression on $T_{z_0}(\partial_i\Omega)\subset T_{z_0}\CP^1$:
\begin{equation}\label{eq:metric at infty-universal}
\tilde{g}_{i,\infty}:=4e^{-2t}\tilde{G}_*\tilde{I}^*_t=|dz|^2~,
\end{equation}
where $\tilde{I}^*_t$ is the lifting metric on $\tilde{S}^i_{t}$ of $I^*_t$.

Note that the deformation one form on $\partial_i\Omega$ takes values in the projective fields given by homogeneous polynomials of degree 2, see \eqref{eq:one-form at infty}. We now consider the extension of such projective fields to Killing fields in $\HH^3$, which will be used later in the proof of Lemma \ref{lm:omega}.

\begin{lemma}\label{lm:norm-Killing}
Let $x_t\in \tilde{S}^i_{t}$ ($t\geq 1$) and let $z$ be an affine chart near $z_0$ as chosen above. Let $P(z)=\lambda (z-z_0)^2\partial_z$ ($\lambda\in\C$ and $\lambda\not=0$) be a projective vector field on $\CP^1\setminus\{z_1\}$. Then the norm of the Killing field extending $P(z)$, at $x_t$, is $4|\lambda|e^{-t}$.
In particular, the norm decreases to zero as $x_t$ tends to $z_0$ along the geodesic ray $r_{x_0}$ orthogonal to $\tilde{S}^i_{0}$ at $x_0$.
\end{lemma}

\begin{proof}
We consider the model ${\mathbb{U}}^3:=\{(z,s):z\in\C,~ s>0\}$, with the metric $(|dz|^2+ds^2)/s^2$. Let $\big(\phi_t(z)\big)_{t\in[0,\epsilon]}$ be a family of M\"obius transformations generating $P(z)$, with $\phi_0(z)=\rm{Id}$ and $$\phi_t(z):=\frac{a(t)z+b(t)}{c(t)z+d(t)}~,$$
where $a(t), b(t), c(t), d(t)\in\C$ and $a(t)d(t)-b(t)c(t)=1$.
A computation shows that
\begin{equation*}
\frac{\partial}{\partial t}|_{t=0}\phi_t(z)=-c'(0)z^2+(a'(0)-d'(0))z+b'(0)~.
\end{equation*}
Therefore,
\begin{equation}\label{eq:proj field}
-c'(0)=\lambda~,\quad a'(0)-d'(0)=b'(0)=0~.
\end{equation}

Moreover, $a(t)d(t)-b(t)c(t)\equiv 1$. Differentiating on both sides of this identity, we obtain $a'(0)+d'(0)=0$. Combined with \eqref{eq:proj field}, we have that $a'(0)=d'(0)=0$.

Now we construct a Killing field on ${\mathbb{U}}^3$ whose extension to $\CP^1\setminus\{z_1\}$ is $P(z)$. Note that if $c(t)\not=0$, there is a canonical way to extend $\phi_t(z)$ to an isometry on $\mathbb{U}^3$ (see \cite[Chapter 1.1]{marden:hyperbolic}) by
\begin{equation}\label{eq:extension 1}
\tilde{\phi}_t(z,s):=\Big(-\frac{\overline{z+\big(d(t)/c(t)\big)}}{c^2(t)\big(|z+(d(t)/c(t))|^2+s^2\big)}+\frac{a(t)}{c(t)}, \,\, \frac{s}{|c(t)|^2\big(|z+\big(d(t)/c(t)\big)|^2+s^2\big)}\Big)~.
\end{equation}
Taking the derivative of \eqref{eq:extension 1} at $t=0$, the desired Killing field, say $\kappa$, on $\mathbb{U}^3$ is
$$\frac{\partial}{\partial t}|_{t=0}\tilde{\phi}_t(z,s)=
\big(\lambda z^2-\overline{\lambda}s^2,\,\, 2s\cdot{\rm Re}(\lambda z)\big)~.$$
This is a parabolic Killing field that vanishes at $0$ and preserves (globally) the horospheres centered at $0$. Using the orthonormal frame $(e_1,e_2,e_3):=(s\partial_x, s\partial_y, s\partial_s)$ of $T\mathbb{U}^3$ and by a direct computation, we have
\begin{equation*}
\begin{split}
\nabla_{e_1}e_1&=e_3~, \quad \nabla_{e_2}e_1=0~, \quad \nabla_{e_3}e_1=0~,\\
\nabla_{e_1}e_2&=0~,\quad   \nabla_{e_2}e_2=e_3~, \quad \nabla_{e_3}e_2=0~,\\
\nabla_{e_1}e_3&=-e_1~,\quad \nabla_{e_2}e_3=-e_2~,\quad \nabla_{e_3}e_3=0~.\\
\end{split}
\end{equation*}
Combined with the formula \eqref{eq:curl}, we obtain
$$\frac{1}{2}{\rm curl}(\kappa)=\big(-i\lambda z^2-i\overline{\lambda}s^2,\,\, 2s\cdot{\rm Im}(\lambda z)\big)~.$$
Thus the norm of $\kappa$ at the point $x_t=(0,s(x_t))$ is $$\|\kappa\|_{x_t}
=\|\tau(x_t)\|_{\mathbb{U}^3}+\|\sigma(x_t)\|_{\mathbb{U}^3}
=\|(-\overline{\lambda}s^2(x_t),0)\|_{\mathbb{U}^3}+\|((-i\overline{\lambda}s^2(x_t),0)\|_{\mathbb{U}^3}
=2|\lambda|s(x_t)~,$$
which decreases to zero as $x_t$ tends to $z_0$ along the geodesic ray orthogonal to $\tilde{S}^i_{0}$ at $x_0$ (note that $s(x_t)$ tends to zero as $t\rightarrow \infty$).

By \eqref{eq:metric at infty-universal}, under the chosen affine chart $z$ near $z_0$,
$$\tilde{G}_*(\tilde{I}^*_t|_{T_{x_t}\tilde{S}^i_{t}})=\frac{1}{4}(e^{2t}|dz|^2)|_{T_{z_0}(\partial_i\Omega)}~.$$
Note that the coordinate $s(x_t)$ is uniquely determined by the above choice of affine chart $z$. Since the horospherical metric $\tilde{I}_t^*$ on $T_{x_t}\tilde{S}^i_{t}$ is equal to the metric on the tangent space of the horosphere $\{s=s(x_t)\}$ at $x_t$, we obtain
$$\frac{1}{4}e^{2t}|dz|^2=\frac 1{s(x_t)^2}|dz|^2~.$$
It follows that $s(x_t)=2e^{-t}$.

 As a consequence, the norm of the Killing field extending $P(z)$, at $x_t$, is equal to $4|\lambda|e^{-t}$. The lemma follows.
\end{proof}

\begin{remark}\label{rk:comparsion}
 Let $G_t: S^i_{t}\rightarrow \partial_iM$ ($t\geq 1$, $i\in\cI$) denote the hyperbolic Gauss map. Since $S^i_1$ is strictly convex, by a direct computation, the principal curvatures $\lambda_t(p)$, $\mu_t(p)$ of the equidistant surface $S^i_{t}$ ($t\geq 1$) are respectively
 $$\lambda_t(p)=\frac{\lambda_1(p)+\tanh t}{1+\lambda_1(p)\tanh t}~, \quad\quad \mu_t(p)=\frac{\mu_1(p)+\tanh t}{1+\mu_1(p)\tanh t}~. $$
 Therefore, the principal curvatures of the closed surface $S^i_{t}$ have a uniform positive upper and lower bound (and thus the second and third fundamental forms $\II_t$ and $\III_t$ are uniformly bounded) for all $t\geq 1$, say $k^i_{\rm max}$ and $k^i_{\rm min}$. Recall that the horospherical metric on $S^i_{t}$ is  $I^*_t=I_t+2\II_t+\III_t$. As a result, on $T_{p}S^i_{t}$,
 $$I_t\leq I_t^*\leq (1+k)^2I_t~,$$
 where $k=\max_{i\in\cI}\{k^i_{\rm max}\}$.
\end{remark}

\subsubsection{Proof of Lemma \ref{lm:omega}.}

We are now ready to give the proof.

\begin{proof}[\textbf{Proof of Lemma \ref{lm:omega}}]
  Let $(g_t)_{t\in[0,\varepsilon]}$ be 
  a one-parameter deformation
  of $g\in\cCC(M)$ with derivative at $t=0$ corresponding to $\dot g$   (up to a trivial deformation). We now construct the desired representative one-form $\omega$ by the following steps:
\begin{enumerate}[\textbf{Step} 1]
\item  Let $u_0$ be a deformation (automorphic) vector field on $\Omega$ associated to 
  $\dot g$,
  whose restriction to $\partial_i\Omega$ is a holomorphic (automorphic) vector field, denoted by $u^i_0$. Let $s_0$ denote the (automorphic) deformation section associated to $u_0$ and in particular, let $s^i_0$ denote the deformation section associated to $u_0^i$ (see Section \ref{subsec:deformation}). For each point $y\in \tilde{\cE}^i$ ($i\in\cI$), there is a unique point $x\in \tilde{S}^i_0$ such that $y$ lies in the geodesic ray $r_x$ orthogonal to a support plane to $\tilde{S}^i_0$ at $x$. Let $z\in\partial_i\Omega$ be the endpoint at infinity of $r_x$ and let $(s^i_0)_z$ be the projective vector field on $\partial_i\Omega$, which is the value in $sl_2(\C)$ at $z$ of $s^i_0$. There is a unique Killing field on $\Omega$, denoted by $\kappa^i_x$, whose extension to $\partial_i\Omega$ is $(s^i_0)_z$. For each $y\in r_x$, we associate to $y$ the same Killing field $\kappa^i_x$. Then we obtain a section $\kappa$ over $\cup_{i\in\cI}\tilde{\cE}^i$, which is constant along $r_x$ for each $x\in \tilde{S}^i_0$. Besides, $\kappa$ is automorphic, since each section $s^i_0$ is automorphic.
\item  Let $\eta(t)$ be a smooth cut-off function which vanishes for $t\leq \delta_0$ (where $\delta_0<1$ is taken small enough) and is equal to 1 for $t\geq 1$.
  Let $\chi_E$ be the characteristic function of $E\subset\Omega$, which is equal to $1$ on $E$ and to $0$ on the complement of $E$. Now we construct a new section, say $s$, over $\Omega$:
  $$s:=\big(1-\eta\circ d(\cdot,C_M)\cdot\chi_{\cup_{i\in\cI}\tilde{\cE}^i}\big)s_0+\big(\eta\circ d(\cdot,C_M)\cdot \chi_{\cup_{i\in\cI}\tilde{\cE}^i}\big)\kappa~.$$
  Then $s$ coincides with $\kappa$ near $\partial_i\Omega$ for all $i\in\cI$ and coincides with $s_0$ near the convex subset $\tilde{C}_M$. It is clear that $s$ is automorphic.
\item Let $\tilde{\omega}=d^{D}s$. Then $\tilde{\omega}$ is a one-form on $\Omega$, which descends to a closed one-form say $\omega$, on $N$, 
  which is cohomologous to the one-form, say $\omega_0$, that descends from $\tilde{\omega}_0:=d^Ds_0$. Indeed, $$\tilde{\omega}-\tilde{\omega}_0=d^D(s-s_0)=\big(-\eta\circ d(\cdot,C_M)\cdot \chi_{\cup_{i\in\cI}\tilde{\cE}^i} \big) d^D(s_0-\kappa)~,$$
 where $d^D(s_0-\kappa)$ is exact on $\cup_{i\in\cI}\tilde{\cE}^i$.
   \end{enumerate}

   By construction, $\omega$ vanishes along each geodesic ray starting from $S^i_1$ orthogonally for all $i\in\cI$, which implies Condition (i). It remains to show $\omega$ also satisfies Condition (ii).

    We first estimate the norm of $\omega$ on the subset $\cup_{t\geq1}S^i_{t}$ of $\cE^i$ ($i\in\cI$). Let $p_t\in S^i_{t}$ and $X\in T_{p_t}M$. By Condition (i), it suffices to consider $\omega(X)$ for $X\in T_{p_t}S^i_{t}$. Denote $z_0:=G(p_t)\in\partial_iM$ and $U:=G_*X\in T_{z_0}(\partial_iM)$. 
    By \eqref{eq:metric at infty},  the horospherical metric $I^*_t$ satisfies
    $$ G_*(I^*_t|_{T_{p_t}S^i_{t}}) :=\frac{1}{4}e^{2t}g_{i,\infty}|_{T_{z_0}(\partial_iM)}~.$$
   We now choose an affine chart $z$ near $z_0$ such that $g_{i,\infty}|_{T_{z_0}(\partial_iM)}=|dz|^2$, as in \eqref{eq:metric at infty-universal}.
   In particular,
 \begin{equation}\label{eq:comparison}
 \|U\|_{g_{i,\infty}}
 = \|G_*X\|_{g_{i,\infty}}=\|X\|_{G^*g_{i,\infty}}
 =2e^{-t}\|X\|_{I^*_t}~.
 \end{equation}
 Assume that $U=u(z)\partial_z$ in the above chosen affine chart $z$ and the holomorphic deformation vector field $u_0^i$ on $\partial_iM$ is expressed as $u_0^i(z)=f_i(z)\partial_z$, where $f_i(z)$ is a holomorphic function. It is clear that $\omega=\omega_0$ on 
 $\cup_{i\in\cI}\partial_iM$. Then by \eqref{eq:one-form at infty},
$$\omega(U)=\omega_0(U)=d^D_Us_0=\frac{1}{2}u(z_0)f'''_i(z_0)(z-z_0)^2\partial_z~.$$
Note that $\omega(X)$ is the Killing field extending the projective field $\omega(U)$. By Lemma \ref{lm:norm-Killing}, we have
\begin{equation*}
\begin{split}
\|\omega(X)\|_{p_t}&= 4e^{-t}\cdot|u(z_0)|\cdot\frac{1}{2}|f'''_i(z_0)|\\
&=2e^{-t}\cdot\|U\|_{g_{i,\infty}}\cdot|f'''_i(z_0)|\\
&=4e^{-2t}\cdot\|X\|_{I^*_t}\cdot|f'''_i(z_0)|\\
&\leq 4C_0 e^{-2t}\cdot\|X\|_{I_t}\cdot|f'''_i(z_0)|\\
&\leq C_1 e^{-2t}\cdot\|X\|_{I_t}~.
\end{split}
\end{equation*}
The second equality follows from the fact that $g_{i,\infty}|_{T_{z_0}(\partial_iM)}=|dz|^2$, the third one follows from \eqref{eq:comparison}, while the first inequality follows from Remark \ref{rk:comparsion}, where $C_0=1+k$. Since $f'''_i(z)dz^2$ is a quadratic differential on $\partial_i\Omega$ and under the above affine chart (see Section \ref{subsec:one-form at infty}), $|f'''_i(z)|$ descends to a well-defined smooth function on the closed surface $\partial_iM$ and is thus bounded. The last inequality follows by taking a uniform bound for all $|f'''_i(z)|$ over all $i\in\cI$ (note that $\cI$ is a finite set).  By Definition \ref{def:norm}, we have $\|\omega\|_{p_t}\leq C_1e^{-2t}$, where $t=d(\cdot,C_M)$.

On the other hand, $\omega$ is uniformly bounded in the convex compact subset $N_0:=N\setminus(\cup_{i\in\cI}\cup_{t>1}S^i_{t})$ of $N$. Note that the function $d(\cdot,C_M)$ has a uniform upper bound on $N_0$. As a consequence, there is a constant $C>0$ depending 
on $\omega$, such that $$\|\omega\|_x\leq C\exp(-2d(x,C_M))~,$$
for all $x\in N$ (compare with the first paragraph in the proof of Theorem 3.5 in \cite{bromberg:rigidity} or the proof of Lemma 3.2, p. 106 of \cite{qfmp}). So $\omega$ also satisfies Condition (ii).
This concludes the proof of the lemma.
\end{proof}

For convenience, henceforth, we denote by $\omega$ the representative one-forms (constructed in Lemma \ref{lm:omega}) on the submanifold $N\subset M$ and its universal cover $\Omega$.

\subsection{Estimates on $\bar\omega$}

The following describes the images of the infinitesimal pure translation and rotation under the map $\Psi$ (see \cite[Lemma 3.7]{hmcb}).

\begin{lemma}\label{lm:pogorelov_translation and rotation}
  Let $x\in(\Omega\cap\HH^3)\setminus\{O\}$ and $\bar{x}:=\iota(x)\in(\bar{\Omega}\cap\DD^3)\setminus\{\bar{O}\}$. Let $R$ (resp. $\bar{R}:=\Upsilon(R)$) denote the unit radial vector in $T_x\HH^3$ (resp. $T_{\bar{x}}\R^3$) pointing away from the center. Let $v$ (resp. $\bar{v}:=(\cosh\rho)\Upsilon(v)$) denote a unit lateral vector in $T_x\HH^3$ (resp. $T_{\bar{x}}\R^3$), where $\rho:=d_{\HH^3}(x,O)$.

Set $r:=d_{\R^3}(\bar{x},\bar{O})$. Then
\begin{equation*}
\begin{split}
\Psi(R,0)&=(\bar{R},0)~,\\
\Psi(0,R)&=(0,\bar{R})~,\\
\Psi(v,0)&=(\sqrt{1-r^2}\,\bar{v},-\frac{r}{\sqrt{1-r^2}}\,\bar{R}\times\bar{v})~,\\
\Psi(0,v)&=(0,\frac{1}{\sqrt{1-r^2}}\,\bar{v})~,
\end{split}
\end{equation*}
where $(R,0)$ represents the infinitesimal pure translation whose value at $x$ is $R$, while $(0,R)$ represents the infinitesimal pure rotation whose curl at $x$ is $2R$. Similarly for the other pairs of tangent vectors.
\end{lemma}

Let $\omega=d^Ds$ be the one-form on $\Omega$ constructed in the proof of Lemma \ref{lm:omega} and let $\bar{\omega}:=\Psi_*(\omega)$. We denote by $\omega^r_\tau$ and $\omega^{\perp}_\tau$ (resp. $\omega^r_\sigma$ and $\omega^{\perp}_\sigma$) the radial and lateral components of the translation (resp. rotation) part $\omega_\tau$ (resp. $\omega_\sigma$) of $\omega$, and by  $\bar{\omega}^r_{\bar{\tau}}$ and $\bar{\omega}^{\perp}_{\bar{\tau}}$ (resp. $\bar{\omega}^r_{\bar{\sigma}}$ and $\bar{\omega}^{\perp}_{\bar{\sigma}}$) the radial and lateral components of the translation (resp. rotation) part $\bar{\omega}_{\bar{\tau}}$ (resp. $\bar{\omega}_{\bar{\sigma}}$) of $\bar{\omega}$. The following is the estimates of their norms in the unit radial and lateral directions. The proof is similar to Lemma 3.8 in \cite{hmcb}. We include it here for completeness. For simplicity, we denote by $\|\cdot\|$ the Euclidean norm henceforth.

\begin{lemma}\label{lm:estimate-one-form}
Let $x\in(\Omega\cap\HH^3)\setminus\{O\}$ and $\bar{x}:=\iota(x)\in(\bar{\Omega}\cap\DD^3)\setminus\{\bar{O}\}$. Let $R$, $\bar{R}$, $v$, $\bar{v}$, $\rho$ and $r$ be the same notations as in Lemma \ref{lm:pogorelov_translation and rotation}. Then there is a constant $C>0$ depending only on $\omega$, such that
\begin{equation*}
\begin{split}
&\|\bar{\omega}^r_{\bar{\tau}}(\bar{R})\|\leq \frac{
C}{1-r}~,\quad\quad
\|\bar{\omega}^r_{\bar{\tau}}(\bar{v})\|\leq \frac{
C}{\sqrt{1-r}}~,\\
&\|\bar{\omega}^r_{\bar{\sigma}}(\bar{R})\|\leq \frac{
C}{1-r}~,\quad\quad
\|\bar{\omega}^r_{\bar{\sigma}}(\bar{v})\|\leq \frac{
C}{\sqrt{1-r}}~,\\
&\|\bar{\omega}^{\perp}_{\bar{\tau}}(\bar{R})\|\leq \frac{
C}{\sqrt{1-r}}~,\quad\quad
\|\bar{\omega}^{\perp}_{\bar{\tau}}(\bar{v})\|\leq
C~,\\
&\|\bar{\omega}^{\perp}_{\bar{\sigma}}(\bar{R})\|\leq \frac{
2C}{(\sqrt{1-r})^3}~,\quad\quad
\|\bar{\omega}^{\perp}_{\bar{\sigma}}(\bar{v})\|\leq \frac{2
C}{1-r}~.
\end{split}
\end{equation*}
\end{lemma}

\begin{proof}
  By the assumption and \eqref{eq:psi(w)}, $\bar{\omega}(d\iota(X))=(\Psi_*\omega)(d\iota(X))=\Psi(\omega(X))$ for all $X\in T\Omega$. In particular,
  $\bar{\omega}^r_{\bar{\tau}}(d\iota(X))$ is equal to the radial component of the translation part, denoted by $\big(\Psi(\omega(X))\big)^r_{\bar{\tau}}$, of $\Psi(\omega(X))$.
  Similarly for $\bar{\omega}^{\perp}_{\bar{\tau}}$, $\bar{\omega}^r_{\bar{\sigma}}$ and $\bar{\omega}^{\perp}_{\bar{\sigma}}$ acting on $d\iota(X)$. Note that
  $$\bar{R}=\Upsilon(R)=\cosh^2\rho \cdot d\iota(R)~, \quad\quad \bar{v}=\cosh\rho\cdot\Upsilon(v)=\cosh\rho\cdot d\iota(v)~, \quad\quad r=\tanh{\rho}~.$$ We divide the discussion into the following four cases:

  \textbf{Case 1}: We consider the case of $\bar{\omega}^r_{\bar{\tau}}$. By the above facts, we obtain
\begin{equation*}
\begin{split}
\|\bar{\omega}^r_{\bar{\tau}}(\bar{R})\|
&=\|\bar{\omega}^r_{\bar{\tau}}\big(\cosh^2\rho \cdot d\iota(R) \,\big)\|
=\cosh^2{\rho}\,\|\big(\Psi({\omega}(R)\big)^r_{\bar{\tau}}\| \\
&=\cosh^2{\rho}\,\|\big(\Psi(\omega^r_{\tau}(R))\big)^r_{\bar{\tau}}\|
=\cosh^2{\rho}\,\|\omega^r_{\tau}(R)\|\\
&\leq \cosh^2{\rho}\,\|\omega_{\tau}\|_x\cdot\| R\|_{\HH^3}
\leq \frac{\|\omega\|_x}{1-r^2} \\
& \leq \frac{C\exp(-2d(x,C_M))}{1-r^2}\leq \frac{C}{1-r}~.
\end{split}
\end{equation*}
Here the first equality on the second line follows from Lemma \ref{lm:pogorelov_translation and rotation}, which shows that the radial component of the translation part of $\Psi(\omega(R))$ only comes from the radial component of the translation part of $\omega(R)$. The second equality on the second line follows from the first equality in Lemma \ref{lm:pogorelov_translation and rotation}.
The last line is a consequence of Lemma \ref{lm:omega} and the fact that $1-r^2\geq 1-r$. The estimate for $\|\bar{\omega}^r_{\bar{\tau}}(\bar{v})\|$ follows in the same way by replacing $\bar{R}$ (resp. $R$ and $\cosh^2\rho$) by $\bar{v}$ (resp. $v$ and $\cosh\rho$).

  \textbf{Case 2}: For the estimates for $\bar{\omega}^r_{\bar{\sigma}}$,
  we have
\begin{equation*}
\begin{split}
\|\bar{\omega}^r_{\bar{\sigma}}(\bar{R})\|
&=\|\bar{\omega}^r_{\bar{\sigma}}\big(\cosh^2\rho \cdot d\iota(R) \,\big)\|
=\cosh^2{\rho}\,\|\big(\Psi({\omega}(R)\big)^r_{\bar{\sigma}}\| \\
&=\cosh^2{\rho}\,\|\big(\Psi(\omega^r_{\sigma}(R))\big)^r_{\bar{\sigma}}\|
=\cosh^2{\rho}\,\|\omega^r_{\sigma}(R)\|\\
&\leq \cosh^2{\rho}\,\|\omega_{\sigma}\|_x\cdot\| R\|_{\HH^3}
\leq \frac{\|\omega\|_x}{1-r^2} \\
& \leq \frac{C\exp(-2d(x,C_M))}{1-r^2}\leq \frac{C}{1-r}~.
\end{split}
\end{equation*}
Here the first equality on the second line follows from Lemma \ref{lm:pogorelov_translation and rotation}, which shows that the radial component of the rotation part of $\Psi(\omega(R))$ only comes from the radial component of the rotation part of $\omega(R)$. The second equality on the second line follows from the second equality in Lemma \ref{lm:pogorelov_translation and rotation}. The estimate for $\|\bar{\omega}^r_{\bar{\sigma}}(\bar{v})\|$ follows in the same way by replacing $\bar{R}$ (resp. $R$ and $\cosh^2\rho$) by $\bar{v}$ (resp. $v$ and $\cosh\rho$).

  \textbf{Case 3}: Concerning the estimates for $\bar{\omega}^{\perp}_{\bar{\tau}}$, we obtain
\begin{equation*}
\begin{split}
\|\bar{\omega}^{\perp}_{\bar{\tau}}(\bar{R})\|
&=\|\bar{\omega}^{\perp}_{\bar{\tau}}\big(\cosh^2\rho \cdot d\iota(R) \,\big)\|
=\cosh^2{\rho}\,\|\big(\Psi({\omega}(R)\big)^{\perp}_{\bar{\tau}}\|\\
&=\cosh^2{\rho}\,\|\big(\Psi(\omega^{\perp}_{\tau}(R))\big)^{\perp}_{\bar{\tau}}\|
=\cosh^2{\rho}\cdot\sqrt{1-r^2} \,\|\omega^{\perp}_{\tau}(R)\|\\
&\leq \frac{1}{\sqrt{1-r^2}}\,\|\omega_{\tau}\|_x\cdot\| R\|_{\HH^3}
\leq \frac{\|\omega\|_x}{\sqrt{1-r^2}} \\
& \leq \frac{C\exp(-2d(x,C_M))}{\sqrt{1-r^2}}\leq \frac{C}{\sqrt{1-r}}~.
\end{split}
\end{equation*}
Here the first equality on the second line follows from Lemma \ref{lm:pogorelov_translation and rotation}, which shows that the lateral component of the translation part of $\Psi(\omega(R))$ only comes from the lateral component of the translation part of $\omega(R)$. The second equality on the second line follows from the third equality in Lemma \ref{lm:pogorelov_translation and rotation}. The estimate for $\|\bar{\omega}^{\perp}_{\bar{\tau}}(\bar{v})\|$ follows in the same way by replacing $\bar{R}$ (resp. $R$ and $\cosh^2\rho$) by $\bar{v}$ (resp. $v$ and $\cosh\rho$).

 \textbf{Case 4}: As for the estimates for $\bar{\omega}^{\perp}_{\bar{\sigma}}$, we have
\begin{equation*}
\begin{split}
\|\bar{\omega}^{\perp}_{\bar{\sigma}}(\bar{R})\|
&=\|\bar{\omega}^{\perp}_{\bar{\sigma}}\big(\cosh^2\rho \cdot d\iota(R) \,\big)\|
=\cosh^2{\rho}\,\|\big(\Psi({\omega}(R)\big)^{\perp}_{\bar{\sigma}}\|\\
&=\cosh^2{\rho}\,\|\big(\Psi(\omega^{\perp}_{\tau}(R)+\omega^{\perp}_{\sigma}(R))\big)^{\perp}_{\bar{\sigma}}\|\\
&\leq\cosh^2{\rho}\cdot\big(\frac{r}{\sqrt{1-r^2}}\|\omega^{\perp}_{\tau}(R)\|
+\frac{1}{\sqrt{1-r^2}}\|\omega^{\perp}_{\sigma}(R)\|\big)\\
&\leq \frac{1}{(\sqrt{1-r^2})^3}\big(\|\omega_{\tau}\|_x\cdot\| R\|_{\HH^3}+ \|\omega_{\sigma}\|_x\cdot\| R\|_{\HH^3})
\leq \frac{2\|\omega\|_x}{(\sqrt{1-r^2})^3} \\
& \leq \frac{2C\exp(-2d(x,C_M))}{(\sqrt{1-r^2})^3}\leq \frac{2C}{(\sqrt{1-r})^3}~.
\end{split}
\end{equation*}
Here the equality on the second line follows from Lemma \ref{lm:pogorelov_translation and rotation}, which shows that the lateral component of the rotation part of $\Psi(\omega(R))$ comes from the lateral components of both the translation and rotation parts of $\omega(R)$. The first inequality follows from the triangle inequality and the last two equalities in Lemma \ref{lm:pogorelov_translation and rotation}. The estimate for $\|\bar{\omega}^{\perp}_{\bar{\sigma}}(\bar{v})\|$ follows in the same way by replacing $\bar{R}$ (resp. $R$ and $\cosh^2\rho$) by $\bar{v}$ (resp. $v$ and $\cosh\rho$).

\end{proof}

Let $s$ be the deformation section defined in the proof of Lemma \ref{lm:omega}. The following is an estimate on the radial and lateral components of the translation and rotation components of the deformation section $\bar{s}:=\Psi(s)$, which is identified with $(\bar{\tau}, \bar{\sigma})\in\Gamma(T\R^3\times T\R^3)$ (see also \cite[Corollary 3.9]{hmcb}).

\begin{lemma}\label{lm:pogorelov_radial and lateral}
Let $x\in(\Omega\cap\HH^3)\setminus\{O\}$ and $\bar{x}:=\iota(x)\in(\bar{\Omega}\cap\DD^3)\setminus\{\bar{O}\}$. Let $\bar{\tau}^r$ (resp. $\bar{\tau}^{\perp}$) denote the radial (resp. lateral) component of the vector field $\bar{\tau}$ on $\bar{\Omega}\cap\DD^3$. Let $\bar{\sigma}^r$ (resp. $\bar{\sigma}^{\perp}$) denote the radial (resp. lateral) component of the vector field $\bar{\sigma}$ on $\bar{\Omega}\cap\DD^3$. Set $r:=d_{\R^3}(\bar{x},\bar{O})$. Then at $\bar{x}$,
\begin{equation*}
\begin{split}
&\|\bar{\tau}^r(\bar{x})\|\leq C|\log(1-r)|~,\quad\quad \|\bar{\tau}^{\perp}(\bar{x})\|\leq C~,\\
&\|\bar{\sigma}^r(\bar{x})\|\leq C|\log(1-r)|~,\quad\quad \|\bar{\sigma}^{\perp}(\bar{x})\|\leq \frac{C}{\sqrt{1-r}}~.\\
\end{split}
\end{equation*}
\end{lemma}

\begin{proof}
By assumption, the infinitesimal deformation $\dot{g}$ of $g$ is normalized such that $\tau(O)=\sigma(O)=0$ (see Section \ref{subsec:from M to H3}).
Let $\gamma:[0,L]\rightarrow \Omega$ be the geodesic parameterized at speed 1, with $\gamma(0)=O$ and $\gamma(L)=x$. Now we consider the image $\bar{s}$ under the Pogerelov map $\Psi$ of the deformation section $s$. Denote $\bar{\gamma}=\iota(\gamma)$, $\bar{O}=\iota(O)=\bar{\gamma}(0)$ and $\bar{x}=\iota(x)=\bar{\gamma}(r)$, where $r=\tanh L$ and $\bar{\gamma}:[0,r]\rightarrow \bar{\Omega}$ is the geodesic parameterized at speed 1. By assumption, $\bar{\tau}(\bar{O})=\tau(O)=0$, $\bar{\sigma}(\bar{O})=\sigma(O)=0$, then combined with the formula \eqref{eq:euclidean-one-form},
\begin{equation*}
\bar{\tau}(\bar{x})
=\int^{r}_0 \bar{\nabla}_{\bar{\gamma}'(s)}\bar{\tau}~ds
=\int^{r}_0\big(\bar{\omega}_{\bar{\tau}}(\bar{\gamma}'(s))
+\bar{\sigma}(\bar{\gamma}(s))\times\bar{\gamma}'(s)\big)ds~,
\quad\quad
\bar{\sigma}(\bar{x})
=\int^{r}_0 \bar{\nabla}_{\bar{\gamma}'(s)}\bar{\sigma}~ds
=\int^{r}_0 \bar{\omega}_{\bar{\sigma}}(\bar{\gamma}'(s)) ds~.
\end{equation*}
In particular, the radial and lateral components of $\bar{\tau}$ and $\bar{\sigma}$ satisfy that
\begin{equation*}
\begin{split}
&\bar{\tau}^r(\bar{x})
=\int^{r}_0 \bar{\omega}^r_{\bar{\tau}}(\bar{\gamma}'(s))ds~,\quad\quad
\bar{\tau}^{\perp}(\bar{x})
=\int^{r}_0\big(\bar{\omega}^{\perp}_{\bar{\tau}}(\bar{\gamma}'(s))+\bar{\sigma}^{\perp}(\bar{\gamma}(s))\times\bar{\gamma}'(s)\big)ds~,\\
&\bar{\sigma}^r(\bar{x})
=\int^{r}_0 \bar{\omega}^r_{\bar{\sigma}}(\bar{\gamma}'(s))ds~,\quad\quad
\bar{\sigma}^{\perp}(\bar{x})
=\int^{r}_0\bar{\omega}^{\perp}_{\bar{\sigma}}(\bar{\gamma}'(s))ds~.\\
\end{split}
\end{equation*}
Combined with Lemma \ref{lm:estimate-one-form}, we obtain
\begin{equation*}
\begin{split}
&\|\bar{\sigma}^r(\bar{x})\|
\leq\int^{r}_0 \|\bar{\omega}^r_{\bar{\sigma}}(\bar{\gamma}'(s))\|ds
\leq \int^{r}_0 \frac{C}{1-s}ds\leq C|\log(1-r)|~,\\
&\|\bar{\sigma}^{\perp}(\bar{x})\|
\leq\int^{r}_0\|\bar{\omega}^{\perp}_{\bar{\sigma}}(\bar{\gamma}'(s))\|ds
\leq \int^{r}_0 \frac{C}{(\sqrt{1-s})^3}ds \leq \frac{C}{\sqrt{1-r}}~,\\
&\|\bar{\tau}^r(\bar{x})\|
\leq\int^{r}_0 \|\bar{\omega}^r_{\bar{\tau}}(\bar{\gamma}'(s))\|ds
\leq \int^{r}_0 \frac{C}{1-s}ds\leq C|\log(1-r)|~,\\
&\|\bar{\tau}^{\perp}(\bar{x})\|
\leq\int^{r}_0\big(\|\bar{\omega}^{\perp}_{\bar{\tau}}(\bar{\gamma}'(s))\|
+\|\bar{\sigma}^{\perp}(\bar{\gamma}(s))\|\big)ds
\leq\int^{r}_0\left(\frac{C}{\sqrt{1-s}}+\frac{C}{\sqrt{1-s}}\right)ds
\leq C~.\\
\end{split}
\end{equation*}
\end{proof}

\subsection{A deformation vector field on $\partial\bar{\Omega}\setminus\bar{\Lambda}$}
\label{subsec:deformation field}

Recall that $s$ is the deformation section constructed in Step 2 of the proof of Lemma \ref{lm:omega}. Let $u$ denote the deformation vector field on $\Omega\setminus\Lambda$ integrated from the representative one-form $\omega=d^Ds$. We denote by $\bar{u}$ the image under the infinitesimal Pogorelov map $\Upsilon$ of $u$ on $\Omega\cap\HH^3$, which is then a deformation vector field on $\bar{\Omega}\cap\DD^3$ and thus defines a deformation vector field on $\partial_j\bar{\Omega}$ and $\partial_k\bar{\Omega}$ for all $j\in\cJ$ and $k\in\cK$ respectively. We now extend the definition of $\bar{u}$ to $\partial_i\bar{\Omega}$ and $\partial_k\bar{\Omega}^*$ for all $i\in\cI$ and $k\in\cK$.

For the extension to $\partial_i\bar{\Omega}$ of $\bar{u}$, note that by construction, the deformation section $s$, which corresponds to the canonical lift of $u$, is constant along each geodesic ray $r_x$ with $x\in \partial_i \tilde{C}_M$. Then $\bar{s}:=\Psi(s)$ is constant along the image $\iota(r_x)$. By definition, $$\bar{u}(\bar{y})=\bar{s}_{\bar{y}}(\bar{y})=\bar{s}_{\bar{x}}(\bar{y})~,$$
 for all $\bar{x}:=\iota(x)$, $\bar{y}:=\iota(y)$ with $y\in r_x$. So $\bar{u}$ has a natural continuous extension to $\partial_i\bar{\Omega}$, defined as $\bar{s}_{\bar{x}}(\bar{z})$, for each $\bar{z}\in\partial_i\bar{\Omega}$.

 For the extension to $\partial_k\bar{\Omega}^*$ of $\bar{u}$, we first extend $u$ to $\partial_k\Omega^*$, through the duality between a couple $(x,P)$ in $\HH^3$ and a couple $(x^*, P^*)$ in $\bdS^3$, where $P$ is a plane containing $x$, while $x^*\in\bdS^3$ is the point dual to $P$, and $P^*\subset\bdS^3$ is the plane dual to $x$. Indeed, for each $x\in\partial_k\Omega$ and the tangent plane $P$ to $\partial_k\Omega$ at $x$, the first-order deformation of the couple $(x,P)$ under the Killing field $s_x$ induces, through the duality (see Section \ref{ssc:dual}), a first-order deformation of $x^*\in\partial_k\bar{\Omega}^*$, so we obtain a deformation vector at $x^*$, denoted by $u(x^*)$, which is equal to $s_x(x^*)$, since $s_x\in sl_2(\C)\cong so(3,1)$ and the isometries in $SO(3,1)$ preserve the duality relation (see e.g. \cite[Remark 1.7]{hmcb}). In this way, $u$ has an extension to $\partial_k\Omega^*$ and under the infinitesimal Pogorelov map $\Upsilon$, $\bar{u}$ has an extension to $\partial_k\bar{\Omega}^*$ for all $k\in\cK$. In particular, there is a natural correspondence between a pair of points $\bar{x}:=\iota(x)\in\partial_k\bar{\Omega}$ and $\bar{x}^*:=\iota(x^*)\in\partial_k\bar{\Omega}^*$. Note that $\partial_j\bar\Omega\setminus\bar\Lambda$ is contained in $\HH^3$ for each $j\in\cJ$. Such a correspondence also holds between $\bar x:=\iota(x)\in \partial_j\bar\Omega$ and $\bar{x}^*:=\iota(x^*)\in\partial_j\bar{\Omega}^*$. For $\bar x\in\partial_i\bar\Omega$ with $i\in\cI$, we define $\bar x^*$ to be $\bar x$ itself.

For simplicity, we still denote by $\bar{u}$ the extended deformation vector field on $(\bar{\Omega}\cup(\cup_{k\in\cK}\partial_k\bar{\Omega}^*))\setminus\bar{\Lambda}$. Recall that $(\bar{\tau},\bar{\sigma})$ denotes the translation and rotation components of the corresponding deformation section $\bar{s}:=\Psi(s)$ on $\bar{\Omega}\cap\DD^3$.

 \begin{lemma}\label{lm:dual-Killing}
             For each $\bar{x}\in\partial_k\bar{\Omega}$ and $\bar{x}^*\in\partial_k\bar{\Omega}^*$, we have
             $$\bar{u}(\bar{x}^*)=\bar{\tau}(\bar{x})+\bar{\sigma}(\bar{x})\times(\bar{x}^*-\bar{x})~.$$
             \end{lemma}

           \begin{proof}
             By definition, $u(x^*)=s_x(x^*)$. Combined with Lemma \ref{lm:Pogorelov isometric} and the definition of $\bar{s}$, under the infinitesimal Pogorelov map $\Upsilon$, we have $$\bar{s}_{\bar{x}}(\bar{x}^*):=\Upsilon(s_x(x^*))=\Upsilon(u(x^*))=\bar{u}(\bar{x}^*)~.$$ By the formulation of infinitesimal euclidean isometries,
             $$\bar{u}(\bar{x}^*)=\bar{s}_{\bar{x}}(\bar{x}^*)=\bar{\tau}(\bar{x})+\bar{\sigma}(\bar{x})\times(\bar{x}^*-\bar{x})~.$$
           \end{proof}

 Let $v$ (resp. $\bar{v}$) denote the orthogonal projection to $(\partial\Omega\cup(\cup_{k\in\cK}\partial_k\Omega^*))\setminus\Lambda$ (resp. $(\partial\bar{\Omega}\cup(\cup_{k\in\cK}\partial_k\bar{\Omega}^*))\setminus\bar{\Lambda}$) of $u$ (resp. $\bar{u}$), called the \emph{tangential component} of $u$ (resp. $\bar{u}$).

\begin{proposition}\label{prop:solution}
 $\bar{v}$ satisfies equation \eqref{eq:E} on $\partial\bar{\Omega}^*\setminus\bar{\Lambda}$. In particular, $\bar{v}$ is holomorphic on $\partial_i\bar{\Omega}$ for all $i\in\cI$.
\end{proposition}

   \begin{proof}
         Note that $\partial\bar{\Omega}^*\setminus\bar{\Lambda}=(\cup_{i\in\cI}\partial_i\bar{\Omega})\cup(\cup_{j\in\cJ}\partial_j\bar{\Omega})\cup(\cup_{k\in\cK}\partial_k\bar{\Omega}^*)$. We first show that $\bar{u}$ satisfies equation \eqref{eq:E} on $(\cup_{j\in\cJ}\partial_j\bar{\Omega})\cup(\cup_{k\in\cK}\partial_k\bar{\Omega}^*)$.
         By assumption, $\dot g$ preserves the induced metric on
         $S_j$ and the third fundamental form on $S_k$ for all $j\in\cJ$ and $k\in\cK$. By Statement (3) of Proposition \ref{prop:dual}, $\dot g$ preserves the induced metric on $\partial_k\Omega^*$. Therefore, the restriction of $u$ to $\partial_j\Omega$ and $\partial_k\Omega^*$ are both isometric for all $j\in\cJ$ and $k\in\cK$. Combined with Statement (2) of Lemma \ref{lm:Pogorelov isometric}, the restriction of $\bar{u}$ to $\partial_j\bar{\Omega}$ and $\partial_k\bar{\Omega}^*$ are also isometric for all $j\in\cJ$ and $k\in\cK$. By Lemma \ref{lm:equation}, $\bar{v}$ satisfies equation \eqref{eq:E} on $(\cup_{j\in\cJ}\partial_j\bar{\Omega})\cup(\cup_{k\in\cK}\partial_k\bar{\Omega}^*)$.

         Since $\partial_i\bar{\Omega}\subset\partial\DD^3$, the shape operator $\bar{B}$ of $\partial_i\bar{\Omega}$ is the identity and thus $\bar{I}=\bar{\II}=\bar{\III}$. Then for each $\bar{X}\in T\partial_i\bar{\Omega}$, $$\partial_{\bar{\III}}(\bar{B}^{-1}\bar{v})(\bar{X})
        =\partial_{\bar{\III}}\bar{v}(\bar{X})
        =
        {\nabla}^{\bar{I}}_{\bar{X}}\bar{v}+\bar{J}{\nabla}^{\bar{I}}_{\bar{J}\bar{X}}\bar{v}~,$$
        which implies that $\bar{v}$ is holomorphic on $\cup_{i\in\cI}\partial_i\bar{\Omega}$ if and only if $\bar{v}$ satisfies equation \eqref{eq:E} on $\cup_{i\in\cI}\partial_i\bar{\Omega}$. We just need to show that $\bar{v}$ is holomorphic on $\cup_{i\in\cI}\partial_i\bar{\Omega}$. Since $\dot g$ preserves the conformal structure at infinity on $\partial_i\Omega$ for all $i\in \cI$, $u$ is a holomorphic vector field on $\cup_{i\in\cI}\partial_i\Omega$.

        By definition, $\bar{u}$ continuously extends to $\partial_i\bar{\Omega}$, 
        and therefore
        the sequence $(\bar{u}^{\perp}(\bar{x}_n))$ of lateral components of $\bar{u}(\bar{x}_n)$ converges to the lateral component  $\bar{u}^{\perp}(\bar{x})$, as $\bar{x}_n$ tends to $\bar{x}\in\partial_i\bar{\Omega}$ along the geodesic ray $\iota(r_x)$, with $r_x$ starting from $\tilde{S}^i_1$ orthogonally. Observe that $\partial_i\bar{\Omega}\subset\partial\DD^3$, the lateral component of $\bar{u}$ coincides with its tangential component $\bar{v}$ when restricted to $\cup_{i\in\cI}\partial_i\bar{\Omega}$, so $\bar{v}(\bar{x})=\bar{u}^{\perp}(\bar{x})$ is the limit of $\bar{u}^{\perp}(\bar{x}_n)$. By the definition of the infinitesimal Pogorelov map $\Upsilon$ (see Section \ref{subsec:Pogorelov}), it sends lateral directions to lateral directions. More precisely, $$\bar{u}^{\perp}(\bar{x}_n)=\Upsilon(u^{\perp}(x_n))=d\iota(u^{\perp}(x_n))~.$$
        Note that $u(x_n)$ tends to $u(x)$ as $x_n$ tends to $x\in\partial_i\Omega$ along the above geodesic ray $r_x$ and $u(x)$ is tangent to $\partial_i\Omega$. Combined with the fact that $\iota: U\supset \HH^3\rightarrow \R^3$ is an inclusion, then for each $\bar{x}\in\partial_i\bar{\Omega}$,
        $$\bar{v}(\bar{x})=\bar{u}^{\perp}(\bar{x})
        =\lim\limits_{n\rightarrow\infty} \bar{u}^{\perp}(\bar{x}_n)=
        \lim\limits_{n\rightarrow\infty} d\iota(u^{\perp}(x_n))=u^{\perp}(x)=u(x)~.$$
        Here the limits are with respect to the euclidean metric on $T\overline{\HH^3}$ or $T\overline{\DD^3}$. The second-to-last equality holds with $\partial_i\Omega$ identified with $\partial_i\bar{\Omega}$ under $\iota$. As a consequence, $\bar{v}$ is also holomorphic on $\cup_{i\in\cI}\partial_i\Omega$. The lemma follows.

        \end{proof}

        \subsection{H\"older estimate on $\bar{v}$}\label{subsec:Holder estimate}

In this subsection, we aim to show the following statement about the H\"older continuity of the tangential component $\bar{v}$ of the deformation vector field $\bar{u}$.

\begin{proposition}\label{prop:Holder estimate}
$\bar{v}$ is $C^{\alpha}$ H\"older continuous on $\partial\bar{\Omega}^*$ for all $\alpha\in (0,1)$.
\end{proposition}

Let $\bar{x}\not=\bar{y}\in\partial\bar{\Omega}\setminus\bar{\Lambda}$, and let $\bar{z}$ be the (unique) point on the geodesic segment $[\bar{x},\bar{y}]$ which is nearest to the center $\bar{O}$. Note that $\bar{z}$ might coincide with $\bar{x}$ (resp. $\bar y$), in which case the angle between $[\bar O,\bar x]$ and $[\bar x,\bar y]$ (resp. between $[\bar O,\bar y]$ and $[\bar y,\bar x]$) is at least $\pi/2$. Otherwise, $\bar z\in (\bar x,\bar y)$ and $[\bar O,\bar z]$ is orthogonal to $[\bar x,\bar y]$. However, one can check that $\bar{z}$ must lie in the interior of the segment $[\bar{x},\bar{y}]$ in some specific cases, for instance, if both $\bar{x}$ and $\bar{y}$ lie on $\partial_i\bar{\Omega}\subset \partial\DD^3$ for some $i\in\cI$.

Before giving the proof of Proposition \ref{prop:Holder estimate}, we first consider the following estimates.
\subsubsection{Distance estimates on $\partial\bar{\Omega}\setminus\bar{\Lambda}$ in $\R^3$}
\label{subsec:distance estimate}

We give some comparisons among the distances $\|\bar{x}-\bar{z}\|$, $\|\bar{y}-\bar{z}\|$, $\|\bar{x}-\bar{y}\|$ and $\sqrt{1-\|\bar{z}\|}$ in the following claim, which will be used in the proof of Lemma \ref{lm:Holder estimate-x close to y} later.

\begin{claim}\label{clm:estimate}
  For $\bar{x}\not=\bar{y}\in\partial\bar{\Omega}\setminus\bar{\Lambda}$, we have
  \begin{equation}\label{ineq:general}
    \|\bar{x}-\bar{y}\| \leq 2\max\{\|\bar{x}-\bar{z}\|,\|\bar{y}-\bar{z}\|\}\leq 2\sqrt{2}\cdot\sqrt{1-\|\bar{z}\|}~.
  \end{equation}
\end{claim}

\begin{proof}
  By the triangle inequality,
  \begin{equation*}
    \|\bar{x}-\bar{y}\|\leq \|\bar{x}-\bar{z}\|+\|\bar{y}-\bar{z}\|\leq  2\max\{\|\bar{x}-\bar{z}\|,\|\bar{y}-\bar{z}\|\}~.
  \end{equation*}
  By definition, $\bar{z}\in \DD^3$ and $\|\bar{z}\|\leq\min\{\|\bar{x}\|,\|\bar{y}\|\}$. Using the cosine formula, we have
  \begin{equation*}
    \|\bar{x}\|^2=\|\bar{z}\|^2+\|\bar{x}-\bar{z}\|^2-2\|\bar{z}\|\cdot\|\bar{x}-\bar{z}\|\cdot\cos\angle\bar{x}\bar{z}\bar{O}~.
  \end{equation*}
  Combined with the facts that $\|\bar{x}\|\leq 1$ and that the angle between $[\bar O,\bar z]$ and $[\bar x,\bar y]$ is at least $\pi/2$, we obtain that
  \begin{equation}\label{eq:cosine}
    \|\bar{x}-\bar{z}\|^2=\|\bar{x}\|^2-\|\bar{z}\|^2+2\|\bar{z}\|\cdot\|\bar{x}-\bar{z}\|\cdot\cos\angle\bar{x}\bar{z}\bar{O}\leq 1-\|\bar{z}\|^2\leq 2(1-\|\bar{z}\|)~.
  \end{equation}
  Similarly, we can show that
  \begin{equation*}
    \|\bar{y}-\bar{z}\|^2\leq 2(1-\|\bar{z}\|)~.
  \end{equation*}
  Therefore,
  \begin{equation*}
    \max\{\|\bar{x}-\bar{z}\|,\|\bar{y}-\bar{z}\|\}\leq \sqrt{2}\sqrt{1-\|\bar{z}\|}~.
  \end{equation*}
  Combined with the above result, this shows the inequality \eqref{ineq:general}.
\end{proof}

\subsubsection{Estimates on $(\bar{\tau}, \bar{\sigma})$}\label{subsec:const section}

Recall that $\bar{\tau}$ and $\bar{\sigma}$ denote the translation and rotation components of the deformation section $\bar{s}:=\Psi(s)$ on $\bar{\Omega}\cap\DD^3$.
We set $$\delta\bar{\tau}=\bar{\tau}-\bar{\tau}_0~,\quad\quad  \delta\bar{\sigma}=\bar{\sigma}-\bar{\sigma}_0~,$$
where $\bar{\tau}_0\in\Gamma(T\bar{\Omega})$ is a Killing vector field and $\bar{\sigma}_0\in\Gamma(T\bar{\Omega})$ is one half of the curl vector field of $\bar{\tau}_0$, with $\bar{\tau}_0(\bar{z})=\bar{\tau}(\bar{z})$ and  $\bar{\sigma}_0(\bar{z})=\bar{\sigma}(\bar{z})$. Also, $(\bar{\tau}_0,\bar{\sigma}_0)$ is identified with a constant section $\bar{\kappa}_0\in\Gamma(\bar{E})$ (see Section \ref{subsec:euclidean bundles}). In particular, $(\bar{\kappa}_0)_{\bar{x}}(\bar{y})=\bar{\tau}_0(\bar{y})$ for all $\bar{x},\bar{y}\in\bar{\Omega}$.

Let $\bar{x}\in\partial\bar{\Omega}\setminus\bar{\Lambda}$. Recall that the deformation one form $\bar{\omega}:=\Psi_*(\omega)$ is defined on $\bar{\Omega}\subset\DD^3$. For $\bar{x}\in(\partial\bar{\Omega}\setminus\bar\Lambda)\cap\DD^3$, by \eqref{eq:euclidean-one-form}, we have $\bar{\omega}_{\bar{\tau}}(\bar{X})=\bar{\nabla}_{\bar{X}} \bar{\tau}-\bar{\sigma}(\bar{x})\times \bar{X}$ and $\bar{\omega}_{\bar\sigma}(\bar X)=\bar{\nabla}_{\bar{X}} \bar{\sigma}$. Observe that $\bar{z}$ (defined at the beginning of Section \ref{subsec:Holder estimate}) can be connected to $\bar{x}$ along first the geodesic segment (with respect to the spherical metric) on the sphere of radius $\|\bar{z}\|$ centered at $\bar{O}$ with endpoint (say $\bar{\xi}$) lying on the line segment $[\bar{O},\bar{x}]$ and then along the radial segment from $\bar{\xi}$ to $\bar{x}$ (see Figure \ref{subfig:integral path-interior}). As shown in \cite[Propositions 4.4 and 4.5]{hmcb}, by considering the integration of $\|\bar{\omega}_{\bar{\tau}}\|$ (resp. $\|\bar{\omega}_{\bar{\sigma}}\|$) along these two paths and using the estimates (see Lemma \ref{lm:estimate-one-form}) of the radial and lateral components of the one-forms $\bar{\omega}_{\bar{\tau}}$ and $\bar{\omega}_{\bar{\sigma}}$, the following estimates on $\|\delta\bar{\tau}(\bar{x})\|$, $\|\delta\bar{\tau}^{\perp}(\bar{x})\|$, $\|\delta\bar \sigma(\bar x)\|$ and $\|\delta\bar\sigma^r(\bar x)\|$ are obtained.

Note that $(\partial\bar\Omega\setminus\bar\Lambda)\cap\DD^3=(\cup_{j\in\cJ}\partial_j\bar{\Omega})\cup(\cup_{k\in\cK}\partial_k\bar{\Omega})$, while $(\partial\bar\Omega\setminus\bar\Lambda)\cap\partial\DD^3=\cup_{i\in\cI}\partial_i\bar{\Omega}$.

\begin{lemma}\label{lm:estimate-hmcb}
For $\bar{x}\in(\cup_{j\in\cJ}\partial_j\bar{\Omega})\cup(\cup_{k\in\cK}\partial_k\bar{\Omega})$, we have the following estimates:
\begin{equation*}
\begin{split}
&\|\delta\bar{\tau}(\bar{x})\|
\leq C \left|\log\left(\frac{1-\|\bar{x}\|}{1-\|\bar{z}\|}\right)\right|~,\\
&\|\delta\bar{\tau}^{\perp}(\bar{x})\|
\leq C \|\bar{x}-\bar{z}\|+C\big(\sqrt{1-\|\bar{z}\|}-\sqrt{1-\|\bar{x}\|}\big)~;\\
&\|\delta\bar\sigma(\bar x)\|\leq \frac{C\|\bar x-\bar z\|}{1-\|\bar z\|}+C(\frac{1}{\sqrt{1-\|\bar x\|}}-\frac{1}{\sqrt{1-\|\bar z\|}})~,\\
&\|\delta\bar\sigma^r(\bar x)\|\leq C\left|\log\left(\frac{1-\|\bar x\|}{1-\|\bar z\|}\right)\right|~.\\
\end{split}
\end{equation*}
In particular, if $\|\bar{x}-\bar{z}\|$ is small enough, then
\begin{equation}\label{eq:x-z-small}
\begin{split}
\sqrt{1-\|\bar{x}\|}\cdot\|\delta\bar{\tau}(\bar{x})\|
&\leq C \sqrt{1-\|\bar{x}\|}\cdot \left|\log\left(\frac{1-\|\bar{x}\|}{1-\|\bar{z}\|}\right)\right|
\leq C \|\bar{x}-\bar{z}\|\cdot \big|\log(\|\bar{x}-\bar{z}\|)\big|~.
 \\
 \|\delta\bar{\tau}^{\perp}(\bar{x})\|
 &\leq C\|\bar{x}-\bar{z}\|+C\big(\sqrt{1-\|\bar{z}\|}-\sqrt{1-\|\bar{x}\|}\big)
 \leq C \|\bar{x}-\bar{z}\|~.
\end{split}
\end{equation}
\end{lemma}

\begin{claim}\label{clm:deltatau}
 Assume that 
 $\bar{x}\in \cup_{i\in\cI}\partial_i\bar{\Omega}$ and $\|\bar{x}-\bar{z}\|$ is small enough. Then
\begin{equation*}
\|\delta\bar{\tau}^{\perp}(\bar{x})\|\leq C\|\bar{x}-\bar{z}\|~.
\end{equation*}
\end{claim}

\begin{proof}
Let $(\bar{x}_t)_{t\in[\varepsilon_0,1]}$ be a family of points on the segment $[\bar{x},\bar{z}]$ with the distance $\|\bar{x}_t-\bar{x}\|=1-t$. By Lemma \ref{lm:omega}, the one-form $\omega$ over $\Omega$ (including those points near the boundary $\partial_i\Omega$) has a uniform bound. Observe that $\bar{z}$ can be connected to $\bar{x}_t$ along first the geodesic segment on the sphere of radius $\|\bar{z}\|$ centered at $\bar{O}$ with endpoint (say $\bar{\xi}_t$) lying on the line segment $[\bar{O},\bar{x}_t]$ and then along the segment from $\bar{\xi}_t$ to $\bar{x}_t$ (see Figure \ref{subfig:integral path-boundary}).  As shown in \cite[Proposition 4.4]{hmcb}), integrating $\|\bar{\omega}_{\bar{\tau}}\|$ (resp. $\|\bar{\omega}_{\bar{\sigma}}\|$) along these two paths and combining the estimates (see Lemma \ref{lm:estimate-one-form}) of the radial and lateral components of the one-forms $\bar{\omega}_{\bar{\tau}}$ (resp. $\bar{\omega}_{\bar{\sigma}}$), the results of Lemma \ref{lm:estimate-hmcb} actually still hold with $\bar{x}$ replaced by $\bar{x}_t$.


\begin{figure}
\begin{subfigure}[b]{0.46\textwidth}
\centering
\begin{tikzpicture}[scale=0.85]
\draw  (0,0) ellipse (3 and 3);
\draw [snake=snake][white][fill=gray,opacity=0.3] (-2.6,1.5) --(-1,1.2)--(0,1.1)-- (1,1.2)--(2.6,1.5);
\draw [snake=snake](-2.6,1.5) --(-1,1.2)--(0,1.1)-- (1,1.2)--(2.6,1.5);
\draw [densely dotted, thick, snake=snake] (-2.6,1.5) --(-2.3,1.52)--(-1.2,1.55)-- (-0.8,1.6)--(0,1.5)-- (0.8,1.6)--(1.2,1.55)--(2.2,1.52)--(2.6,1.5);
\draw [snake=snake] (-2.95,0.48)--(0,-3);
\draw [densely dotted, thick, snake=snake][white] [fill=blue,opacity=0.15] (-2.95,0.48)--(-2.7,0.35)--(-1.5,-0.75)--(-1.0,-1.5)--(-0.4,-2.15)--(0,-3);
\draw [densely dotted, thick, snake=snake] (-2.95,0.48)--(-2.7,0.35)--(-1.5,-0.75)--(-1.0,-1.5)--(-0.4,-2.15)--(0,-3);
\draw[line width=0.25mm][white][fill=gray,opacity=0.3](-2.6,1.5) .. controls (-2.1,2.5) and (2.1,2.5) .. (2.6,1.5);
\draw[line width=0.25mm](-2.6,1.5) .. controls (-2.1,2.5) and (2.1,2.5) .. (2.6,1.5);
\draw[line width=0.25mm][white][fill=blue,opacity=0.15](-2.95,0.48) .. controls (-3.1,-1.1) and (-1.5,-2.8) .. (0,-3);
\draw[line width=0.25mm](-2.95,0.48) .. controls (-3.1,-1.1) and (-1.5,-2.8) .. (0,-3);
\draw [line width=0.2mm, densely dashed, red](0,0) ellipse (2.51 and 2.51);
\draw [line width=0.2mm, densely dashed,red](0,0) .. controls (-2,1.95) and (-2,1.95) .. (-2,1.95);
\draw [fill=black] (0,0) ellipse (0.04 and 0.04);
\node at (0.3,0) {$\bar O$};
\draw [fill=black] (-2,1.95) ellipse (0.04 and 0.04);
\node at (-1.85,2.15) {$\bar x$};
\draw [fill=black] (-2.82,-0.5) ellipse (0.04 and 0.04);
\node at (-2.6,-0.55) {$\bar y$};
\draw[blue,line width=0.2mm](-2.82,-0.5) .. controls (-2,1.95) and (-2,1.95) .. (-2,1.95);
\draw[densely dotted, blue,line width=0.2mm](0,0) .. controls (-2.4,0.8) and (-2.4,0.8) .. (-2.4,0.8);
\draw [fill=black] (-2.38,0.8) ellipse (0.04 and 0.04);
\node at (-2.6,0.8) {$\bar z$};
\draw [line width=0.2mm, densely dashed, red](0,0) ellipse (2.51 and 2.51);
\draw [line width=0.2mm, densely dashed,red](0,0) .. controls (-2,1.95) and (-2,1.95) .. (-2,1.95);
\draw [fill=blue] (-1.8,1.75) ellipse (0.04 and 0.04);
\node at (-1.45,1.8) {$\bar{\xi}$};
\draw [line width=0.35mm, red](-1.8,1.75) .. controls (-2,1.95) and (-2,1.95) .. (-2,1.95);
\draw [line width=0.35mm, red](-2.38,0.8) .. controls (-2.2,1.35) and (-1.8,1.75) .. (-1.8,1.75);
\end{tikzpicture}
\caption{\small{$\bar{x}\in\partial\bar{\Omega}\cap \DD^3$.}}
\label{subfig:integral path-interior}
\end{subfigure}
\begin{subfigure}[b]{0.46\textwidth}
\centering
\begin{tikzpicture}[scale=0.85]
\draw  (0,0) ellipse (3 and 3);
\draw [snake=snake][white][fill=yellow,opacity=0.15] (-2.6,1.5) --(-1,1.2)--(0,1.1)-- (1,1.2)--(2.6,1.5);
\draw [snake=snake](-2.6,1.5) --(-1,1.2)--(0,1.1)-- (1,1.2)--(2.6,1.5);
\draw [densely dotted, thick, snake=snake] (-2.6,1.5) --(-2.3,1.52)--(-1.2,1.55)-- (-0.8,1.6)--(0,1.5)-- (0.8,1.6)--(1.2,1.55)--(2.2,1.52)--(2.6,1.5);
\draw [snake=snake] (-2.95,0.48)--(0,-3);
\draw [densely dotted, thick, snake=snake][white] [fill=blue,opacity=0.15] (-2.95,0.48)--(-2.7,0.35)--(-1.5,-0.75)--(-1.0,-1.5)--(-0.4,-2.15)--(0,-3);
\draw [densely dotted, thick, snake=snake] (-2.95,0.48)--(-2.7,0.35)--(-1.5,-0.75)--(-1.0,-1.5)--(-0.4,-2.15)--(0,-3);
\draw[line width=0.25mm][white][fill=yellow,opacity=0.15](-2.6,1.5) .. controls (-1.5,3.4925) and (1.56,3.46) .. (2.6,1.5);
\draw[line width=0.25mm](-2.6,1.5) .. controls (-1.5,3.4925) and (1.56,3.46) .. (2.6,1.5);
\draw[line width=0.25mm][white][fill=blue,opacity=0.15](-2.95,0.48) .. controls (-3.1,-1.1) and (-1.5,-2.8) .. (0,-3);
\draw[line width=0.25mm](-2.95,0.48) .. controls (-3.1,-1.1) and (-1.5,-2.8) .. (0,-3);
\draw [fill=black] (0,0) ellipse (0.04 and 0.04);
\draw [line width=0.35mm, red](-1.8,1.75) .. controls (-2,1.95) and (-2,1.95) .. (-2,1.95);
\draw [line width=0.35mm, red](-2.38,0.8) .. controls (-2.2,1.35) and (-1.8,1.75) .. (-1.8,1.75);
\draw [line width=0.2mm, red, densely dashed](-2,1.95)--(0,0);
\node at (0.3,0) {$\bar O$};
\draw [fill=black] (-2,1.95) ellipse (0.04 and 0.04);
\node at (-2.8, 2) {$\bar x_t$};
\draw [->] (-2.6,2) .. controls (-2.3,2.1) and (-2.3,2.1) .. (-2.05,2);
\draw [fill=black] (-1.86,2.36) ellipse (0.04 and 0.04);
\node at (-1.8,2.6) {$\bar x$};
\draw [fill=black] (-2.82,-0.5) ellipse (0.04 and 0.04);
\node at (-2.64,-0.58) {$\bar y$};
\draw[blue,line width=0.2mm](-2.82,-0.5) .. controls (-1.85,2.36) and (-1.85,2.36) .. (-1.85,2.36);
\draw[densely dotted, blue,line width=0.2mm](0,0) .. controls (-2.4,0.8) and (-2.4,0.8) .. (-2.4,0.8);
\draw [fill=black] (-2.38,0.8) ellipse (0.04 and 0.04);
\node at (-2.6,0.8) {$\bar z$};
\draw [<-] (-1.7, 1.75) .. controls (-1.4,1.95) and (-1.4,1.95) .. (-1.2,1.95);
\draw [fill=blue] (-1.8,1.75) ellipse (0.04 and 0.04);
\node at (-1,1.93) {$\bar{\xi}_t$};
\draw [line width=0.2mm, densely dashed, red](0,0) ellipse (2.51 and 2.51);
\end{tikzpicture}
\caption{\small{$\bar{x}\in\partial\bar{\Omega}\cap\partial\DD^3$.}}
\label{subfig:integral path-boundary}
\end{subfigure}
\caption{\small{Integral path from $\bar{z}$ to $\bar{x}$ (resp. $\bar{x}_t$), where $\bar{x}_t$ tends to $\bar{x}$ as $t\rightarrow 1$ (the shaded regions are two connected components of $\partial\bar{\Omega}\setminus\bar\Lambda$).}}
\end{figure}
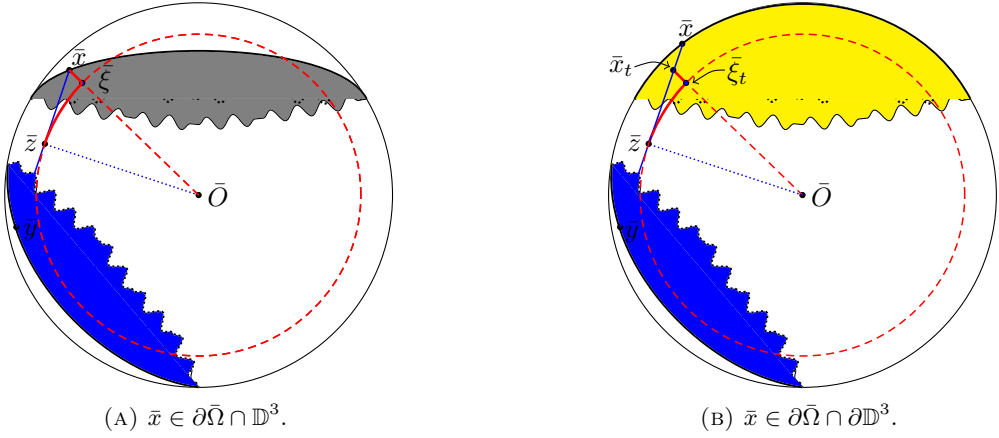

By assumption that $\|\bar{x}-\bar{z}\|$ is small enough, $\|\bar{x}_t-\bar{z}\|$ is small enough. For $t\in(\epsilon_0,1)$, we can apply $\bar{x}_t\in\DD^3\cap\bar{\Omega}$ to the second inequality in Equation \eqref{eq:x-z-small} and obtain that
\begin{equation*}
\|\delta\bar{\tau}^{\perp}(\bar{x}_t)\|\leq C\|\bar{x}_t-\bar{z}\|~.
\end{equation*}
Let $t\rightarrow 1$. Then $\bar{x}_t$ converges to $\bar{x}$. Since the action on the lateral vectors under the infinitesimal Pogorelov map $\Upsilon$ continuously extends to the boundary $\cup_{i\in\cI}\partial_i\bar{\Omega}$ (see e.g. the proof of Proposition \ref{prop:solution}),
\begin{equation*}
\|\delta\bar{\tau}^{\perp}(\bar{x})\|\leq C\|\bar{x}-\bar{z}\|~.
\end{equation*}
The claim follows.
\end{proof}

\subsubsection{Estimates on projection angles to $T\partial\bar{\Omega}$ and $T\partial\bar{\Omega}^*$ of radial and lateral vectors}\label{subsec:estimate-angles}

Let $\bar{x}\in\partial\bar{\Omega}\setminus\bar{\Lambda}$ and let $P_{\bar{x}}$ denote the tangent plane to $\partial\bar{\Omega}$ at $\bar{x}$. Let $\alpha^r_{\bar{x}}$ (resp. $\alpha^{\perp}_{\bar{x}}$) denote the projection angles to the tangent planes $P_{\bar{x}}$ of a radial (resp. lateral) vector at $\bar{z}$. Recall that $\theta_{\bar{x}}$ is the angle between the outwards-pointing normal vector to $\partial\bar{\Omega}$ at $\bar{x}$ and the radial direction $\bar{O}\bar{x}$ at $\bar{x}$.

\begin{claim}\label{clm:proj angle}
  Let $\bar{x}, \bar{y}\in\partial\bar{\Omega}\setminus\bar{\Lambda}$ with $\bar{x}$, $\bar{y}$ close enough to $\bar{\Lambda}$, then
  \begin{equation*}
    \max\{\cos\alpha^r_{\bar{x}},\cos\alpha^r_{\bar{y}}\}
    \leq C\sqrt{1-\|\bar{z}\|}~.
 \end{equation*}
 In particular, the radial component $(\bar{x}-\bar{y})^r$ of $\bar{x}-\bar{y}$ at $\bar{z}$ satisfies that
 \begin{equation}\label{eq:x-y-r}
 \|(\bar{x}-\bar{y})^r\|\leq C\|\bar{x}-\bar{y}\|\cdot\sqrt{1-\|\bar z\|}~.
 \end{equation}
\end{claim}

\begin{proof}
  Let $Q_{\bar{x}}$ denote the plane orthogonal to the segment $[\bar O,\bar x]$ at $\bar{x}$ and let $\zeta^r_{\bar x}$ denote  the projection angle to the plane $Q_{\bar{x}}$ of a radial vector at $\bar{z}$, while the (dihedral) angle between the tangent plane $P_{\bar{x}}$ to $\partial\bar{\Omega}$ at $\bar{x}$ and the plane $Q_{\bar{x}}$ is $\theta_{\bar{x}}$.  As a consequence, the projection angle $\alpha^r_{\bar{x}}$ satisfies
  \begin{equation}\label{eq:proj-angle-interior}
       \zeta^r_{\bar x}-\theta_{\bar{x}}\leq \alpha^r_{\bar{x}}\leq \zeta^r_{\bar x}+\theta_{\bar{x}}~,
  \end{equation}
       where the equalities hold only if the plane $O\bar{x}\bar{y}$ is orthogonal to $P_{\bar{x}}$. In particular, the left equality holds if $\bar z\in(\bar x,\bar y)$ and the inward-pointing normal direction to $P_{\bar{x}}$ at $\bar{x}$ lies on the boundary $\bar{O}\bar{x}$ or inside the triangle $\bar{x}\bar{O}\bar{y}$ (as shown in Figure \ref{fig:proj-angle-interior1}), or $\bar z =\bar x$ (as shown in Figure \ref{fig:proj-angle-z=x}), while the right equality holds if $\bar z\in(\bar x,\bar y)$ and the inward-pointing normal direction to $P_{\bar{x}}$ at $\bar{x}$ lies on the boundary $\bar{O}\bar{x}$ or outside the triangle $\bar{x}\bar{O}\bar{y}$ (as shown in Figure \ref{fig:proj-angle-interior2}). In particular, $\theta_{\bar{x}}=0$ if and only if 
       the inward-pointing normal direction to $P_{\bar{x}}$ at $\bar{x}$ lies on the segment $\bar{O}\bar{x}$.

  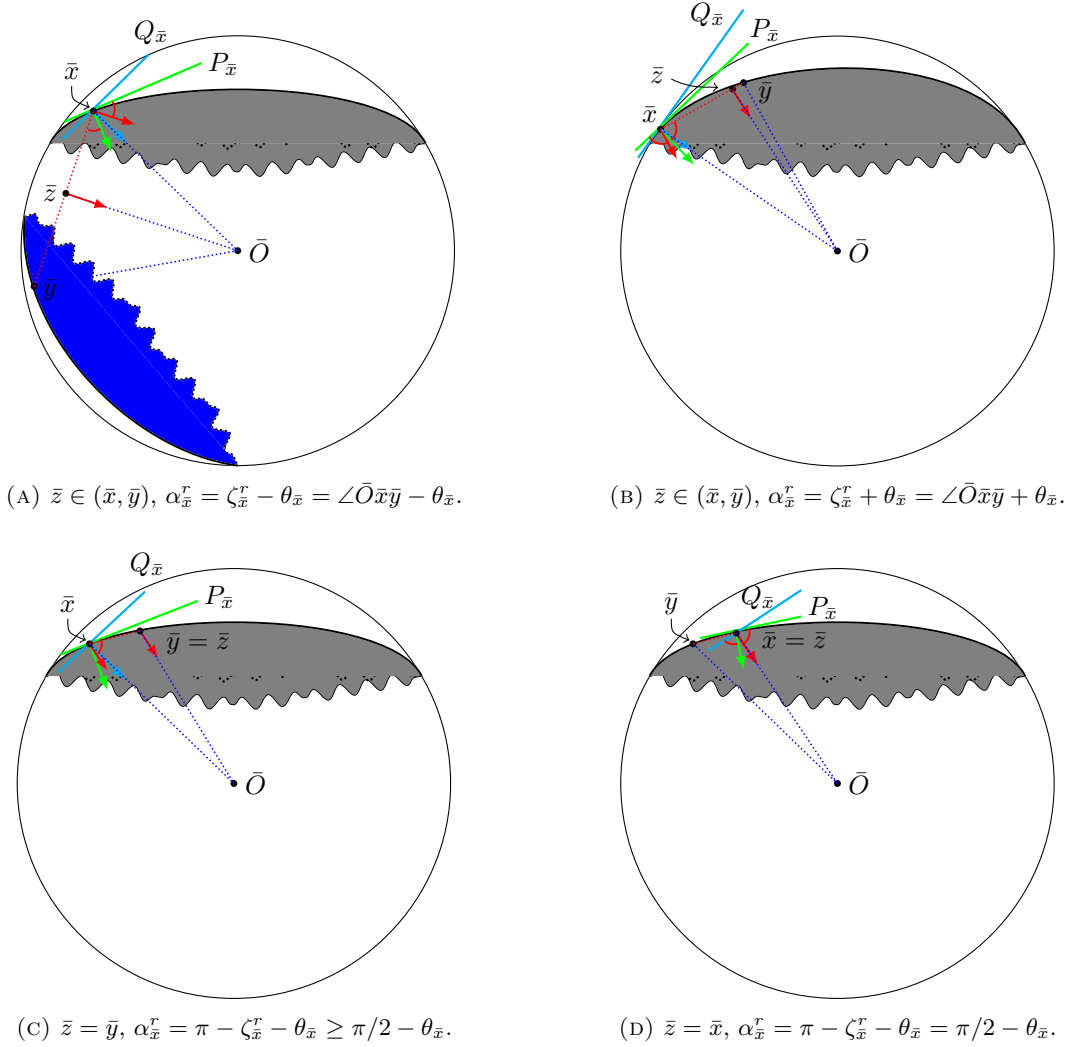
\begin{figure}
\begin{subfigure}[b]{0.46\textwidth}
\centering
\begin{tikzpicture}[scale=0.95]
\draw  (0,0) ellipse (3 and 3);
\draw [snake=snake][white][fill=gray,opacity=0.3] (-2.6,1.5) --(-1,1.2)--(0,1.1)-- (1,1.2)--(2.6,1.5);
\draw [snake=snake](-2.6,1.5) --(-1,1.2)--(0,1.1)-- (1,1.2)--(2.6,1.5);
\draw [densely dotted, thick, snake=snake] (-2.6,1.5) --(-2.3,1.52)--(-1.2,1.55)-- (-0.8,1.6)--(0,1.5)-- (0.8,1.6)--(1.2,1.55)--(2.2,1.52)--(2.6,1.5);
\draw [snake=snake] (-2.95,0.48)--(0,-3);
\draw [densely dotted, thick, snake=snake][white] [fill=blue,opacity=0.15] (-2.95,0.48)--(-2.7,0.35)--(-1.5,-0.75)--(-1.0,-1.5)--(-0.4,-2.15)--(0,-3);
\draw [densely dotted, thick, snake=snake] (-2.95,0.48)--(-2.7,0.35)--(-1.5,-0.75)--(-1.0,-1.5)--(-0.4,-2.15)--(0,-3);
\draw[line width=0.25mm][white][fill=gray,opacity=0.3](-2.6,1.5) .. controls (-2.1,2.5) and (2.1,2.5) .. (2.6,1.5);
\draw[line width=0.25mm](-2.6,1.5) .. controls (-2.1,2.5) and (2.1,2.5) .. (2.6,1.5);
\draw[line width=0.25mm][white][fill=blue,opacity=0.15](-2.95,0.48) .. controls (-3.1,-1.1) and (-1.5,-2.8) .. (0,-3);
\draw[line width=0.25mm](-2.95,0.48) .. controls (-3.1,-1.1) and (-1.5,-2.8) .. (0,-3);
\draw [fill=black] (0,0) ellipse (0.04 and 0.04);
\node at (0.3,0) {$\bar O$};
\draw[green][line width=0.3mm](-2.4,1.8) -- (-0.5,2.62);
\node at (-0.2,2.6) {$P_{\bar x}$};
\draw[-latex][green][line width=0.3mm](-2,1.95) -- (-1.75,1.38);
\draw[cyan][line width=0.3mm](-2.43,1.55) -- (-1.23,2.74);
\draw[-latex][cyan][line width=0.3mm](-2,1.95) -- (-1.53,1.5);
\node at (-1.2,3.05) {$Q_{\bar x}$};
\draw[-latex][red][line width=0.25mm](-2,1.95) -- (-1.42,1.76);
\draw [fill=black] (-2,1.95) ellipse (0.04 and 0.04);
\node at (-2.3,2.5) {$\bar x$};
\draw [->] (-2.3,2.3) .. controls (-2.3,2.1) and (-2.3,2.1) .. (-2.05,2);
\draw[densely dotted,blue,line width=0.2mm](-2,1.95) .. controls (0,0) and (0,0) .. (0,0);
\draw [fill=black] (-2.82,-0.5) ellipse (0.04 and 0.04);
\node at (-2.6,-0.55) {$\bar y$};
\draw[densely dotted,blue,line width=0.2mm](-2.82,-0.5) .. controls (0,0) and (0,0) .. (0,0);
\draw[red,densely dotted,line width=0.2mm](-2.82,-0.5) .. controls (-2,1.95) and (-2,1.95) .. (-2,1.95);
\draw[densely dotted, blue,line width=0.2mm](0,0) .. controls (-2.4,0.8) and (-2.4,0.8) .. (-2.4,0.8);
\draw[-latex][red][line width=0.25mm](-2.38,0.8) -- (-1.8,0.6);
\draw [fill=black] (-2.38,0.8) ellipse (0.04 and 0.04);
\node at (-2.6,0.8) {$\bar z$};
\draw [line width=0.25mm][red] (-1.73,1.87) arc (-30:33:0.2);
\draw [line width=0.25mm][red] (-1.9,1.7) arc (-60:-110:0.2);
\end{tikzpicture}
\caption{\small{$\bar z\in(\bar x,\bar y)$, $\alpha^r_{\bar x}={\zeta}^r_{\bar x}-\theta_{\bar x}=\angle{\bar O\bar x\bar y}-\theta_{\bar x}$.}}
\label{fig:proj-angle-interior1}
\vspace{0.4cm}
\end{subfigure}
\begin{subfigure}[b]{0.46\textwidth}
\centering
\begin{tikzpicture}[scale=0.95]
\draw  (0,0) ellipse (3 and 3);
\draw [snake=snake][white][fill=gray,opacity=0.3] (-2.6,1.5) --(-1,1.2)--(0,1.1)-- (1,1.2)--(2.6,1.5);
\draw [snake=snake](-2.6,1.5) --(-1,1.2)--(0,1.1)-- (1,1.2)--(2.6,1.5);
\draw [densely dotted, thick, snake=snake] (-2.6,1.5) --(-2.3,1.52)--(-1.2,1.55)-- (-0.8,1.6)--(0,1.5)-- (0.8,1.6)--(1.2,1.55)--(2.2,1.52)--(2.6,1.5);
\draw[line width=0.25mm][white][fill=gray,opacity=0.3](-2.6,1.5) .. controls (-1.85,2.8) and (1.85,3) .. (2.6,1.5);
\draw[line width=0.25mm](-2.6,1.5) .. controls (-1.85,2.8) and (1.85,3) .. (2.6,1.5);
\draw [fill=black] (0,0) ellipse (0.04 and 0.04);
\node at (0.3,0) {$\bar O$};
\node at (-1,3.08) {$P_{\bar x}$};
\draw[cyan][line width=0.3mm](-2.76,1.3) -- (-1.3,3.37);
\draw[green][line width=0.3mm](-2.8,1.4) -- (-1.23,2.9);
\node at (-1.8,3.3) {$Q_{\bar x}$};
\draw [fill=black] (-1.3,2.35) ellipse (0.04 and 0.04);
\node at (-1,2.2) {$\bar y$};
\draw [fill=black] (-1.45,2.26) ellipse (0.04 and 0.04);
\node at (-2.5,2.45) {$\bar z$};
\draw[-latex][red][line width=0.25mm](-1.45,2.25) -- (-1.2,1.85);
\draw [->] (-2.3,2.4) .. controls (-2,2.25) and (-1.8,2.3) .. (-1.65,2.28);
\draw [fill=black] (-2.45,1.7) ellipse (0.04 and 0.04);
\node at (-2.6,1.9) {$\bar x$};
\draw[-latex][green][line width=0.25mm](-2.45,1.7) -- (-1.98,1.2);
\draw[-latex][red][line width=0.25mm](-2.45,1.7) -- (-2.22,1.28);
\draw[-latex][cyan][line width=0.25mm](-2.45,1.7) -- (-2,1.4);
\draw[densely dotted,blue,line width=0.2mm](-2.45,1.7) .. controls (0,0) and (0,0) .. (0,0);
\draw[densely dotted,blue,line width=0.2mm](-1.3,2.35) .. controls (0,0) and (0,0) .. (0,0);
\draw[red,densely dotted,line width=0.2mm](-2.45,1.7)--(-1.3,2.35);
\draw[densely dotted,blue,line width=0.2mm](-1.45,2.25)-- (0,0);
\draw [line width=0.25mm][red] (-2.33,1.52) arc (-60:-150:0.2);
\draw [line width=0.25mm][red] (-2.3,1.55) arc (-50:30:0.2);
\end{tikzpicture}
\caption{\small{$\bar z\in(\bar x,\bar y)$, $\alpha^r_{\bar x}={\zeta}^r_{\bar x}+\theta_{\bar x}=\angle{\bar O\bar x\bar y}+\theta_{\bar x}$.}}
\label{fig:proj-angle-interior2}
\vspace{0.4cm}
\end{subfigure}
\quad
\begin{subfigure}[b]{0.46\textwidth}
\centering
\begin{tikzpicture}[scale=0.95]
\draw  (0,0) ellipse (3 and 3);
\draw [snake=snake][white][fill=gray,opacity=0.3] (-2.6,1.5) --(-1,1.2)--(0,1.1)-- (1,1.2)--(2.6,1.5);
\draw [snake=snake](-2.6,1.5) --(-1,1.2)--(0,1.1)-- (1,1.2)--(2.6,1.5);
\draw [densely dotted, thick, snake=snake] (-2.6,1.5) --(-2.3,1.52)--(-1.2,1.55)-- (-0.8,1.6)--(0,1.5)-- (0.8,1.6)--(1.2,1.55)--(2.2,1.52)--(2.6,1.5);
\draw[line width=0.25mm][white][fill=gray,opacity=0.3](-2.6,1.5) .. controls (-2.1,2.5) and (2.1,2.5) .. (2.6,1.5);
\draw[line width=0.25mm](-2.6,1.5) .. controls (-2.1,2.5) and (2.1,2.5) .. (2.6,1.5);
\draw [fill=black] (0,0) ellipse (0.04 and 0.04);
\node at (0.3,0) {$\bar O$};
\draw[green][line width=0.3mm](-2.4,1.8) -- (-0.5,2.55);
\node at (-0.2,2.6) {$P_{\bar x}$};
\draw[-latex][green][line width=0.3mm](-2,1.95) -- (-1.75,1.3);
\draw[cyan][line width=0.3mm](-2.43,1.55) -- (-1.23,2.68);
\draw[cyan][-latex][line width=0.3mm](-2,1.95) -- (-1.5,1.45);
\node at (-1.2,3.05) {$Q_{\bar x}$};
\draw[-latex][red][line width=0.25mm](-1.3,2.13) -- (-1.05,1.75);
\draw[-latex][red][line width=0.25mm](-2,1.94) -- (-1.75,1.58);
\draw [fill=black] (-2,1.95) ellipse (0.04 and 0.04);
\node at (-2.3,2.5) {$\bar x$};
\draw [->] (-2.3,2.3) .. controls (-2.3,2.1) and (-2.3,2.1) .. (-2.05,2);
\draw [fill=black] (-1.3,2.13) ellipse (0.04 and 0.04);
\node at (-0.5,1.95) {$\bar y=\bar z$};
\draw[densely dotted,blue,line width=0.2mm](-1.3,2.13) .. controls (0,0) and (0,0) .. (0,0);
\draw[densely dotted,blue,line width=0.2mm](-2,1.95) .. controls (0,0) and (0,0) .. (0,0);
\draw[red,densely dotted,line width=0.2mm](-2,1.95) .. controls (-1.3,2.13) and (-1.3,2.13) .. (-1.3,2.13);
\draw [line width=0.25mm][red] (-1.9,1.8) arc (-50:20:0.2);
\end{tikzpicture}
\caption{\small{$\bar z=\bar{y}$, $\alpha^r_{\bar x}=\pi-{\zeta}^r_{\bar x}-\theta_{\bar x}\geq\pi/2-\theta_{\bar x}$.}}
\label{fig:proj-angle-z=y}
\end{subfigure}
\begin{subfigure}[b]{0.46\textwidth}
\centering
\begin{tikzpicture}[scale=0.95]
\draw  (0,0) ellipse (3 and 3);
\draw [snake=snake][white][fill=gray,opacity=0.3] (-2.6,1.5) --(-1,1.2)--(0,1.1)-- (1,1.2)--(2.6,1.5);
\draw [snake=snake](-2.6,1.5) --(-1,1.2)--(0,1.1)-- (1,1.2)--(2.6,1.5);
\draw [densely dotted, thick, snake=snake] (-2.6,1.5) --(-2.3,1.52)--(-1.2,1.55)-- (-0.8,1.6)--(0,1.5)-- (0.8,1.6)--(1.2,1.55)--(2.2,1.52)--(2.6,1.5);
\draw[line width=0.25mm][white][fill=gray,opacity=0.3](-2.6,1.5) .. controls (-2.1,2.5) and (2.1,2.5) .. (2.6,1.5);
\draw[line width=0.25mm](-2.6,1.5) .. controls (-2.1,2.5) and (2.1,2.5) .. (2.6,1.5);
\draw [fill=black] (0,0) ellipse (0.04 and 0.04);
\node at (0.3,0) {$\bar O$};
\draw[green][line width=0.3mm](-1.9,2.03) -- (-0.5,2.33);
\node at (-0.2,2.45) {$P_{\bar x}$};
\draw[-latex][green][line width=0.3mm](-1.4,2.1) -- (-1.3,1.58);
\draw[cyan][line width=0.3mm](-1.78,1.85) -- (-0.5,2.7);
\draw[-latex][red][line width=0.3mm](-1.4,2.1) -- (-1.1,1.65);
\node at (-1.1,2.6) {$Q_{\bar x}$};
\draw [fill=black] (-2,1.95) ellipse (0.04 and 0.04);
\node at (-2.3,2.5) {$\bar y$};
\draw [->] (-2.3,2.3) .. controls (-2.3,2.1) and (-2.3,2.1) .. (-2.05,2);
\draw[densely dotted,blue,line width=0.2mm](-2,1.95) .. controls (0,0) and (0,0) .. (0,0);
\draw [fill=black] (-1.4,2.1) ellipse (0.04 and 0.04);
\node at (-0.6,2) {$\bar x=\bar z$};
\draw[densely dotted,blue,line width=0.2mm](-1.4,2.1) .. controls (0,0) and (0,0) .. (0,0);
\draw[densely dotted,red,line width=0.2mm](-2,1.95) .. controls (-1.4,2.1) and (-1.4,2.1) .. (-1.4,2.1);
\draw [line width=0.25mm][red] (-1.3,1.95) arc (-55:15:0.2);
\draw [line width=0.25mm][red] (-1.38,1.95) arc (-75:-130:0.2);
\end{tikzpicture}
\caption{\small{$\bar z=\bar x$, $\alpha^r_{\bar x}=\pi-{\zeta}^r_{\bar x}-\theta_{\bar x}=\pi/2-\theta_{\bar x}$.}}
\label{fig:proj-angle-z=x}
\end{subfigure}
\caption{Examples of the projection angle $\alpha^r_{\bar{x}}$ (with $\bar{x}$, $\bar{y}$ close to $\bar{\Lambda}$) in the case that the plane $\bar O\bar x\bar y$ is orthogonal to the tangent plane $P_{\bar x}$, with the amount of the angle $\alpha^r_{\bar x}$ marked and $\partial\bar\Omega$ partially drawn.}
 \label{fig:proj-angle}
\end{figure}

       Since $\bar{\Omega}$ is strictly convex and the center $\bar{O}$ is contained in the interior of $\bar{\Omega}$, $\min\{\|\bar{x}\|,\|\bar{y}\|\}\geq r_0$ for a constant $r_0>0$.
        If $\bar z\in(\bar x,\bar y)$, then $\zeta^r_{\bar x}=\angle\bar{O}\bar{x}\bar{y}$ and $\cos\zeta^r_{\bar x}=\cos\angle{\bar{O}\bar{x}\bar{y}}=\|\bar{x}-\bar{z}\|/\|\bar{x}\|$.
        Combined with the estimate \eqref{ineq:general} and Lemma \ref{lm: angle estimate}, we obtain
         \begin{equation}\label{eq:proj-angle-cos}
         \cos(\zeta^r_{\bar{x}}-\theta_{\bar{x}})
   =\cos\zeta^r_{\bar{x}}\cdot\cos\theta_{\bar{x}}
   +\sin\zeta^r_{\bar{x}}\cdot\sin\theta_{\bar{x}}
   \leq\|\bar{x}-\bar{z}\|/{r_0}+C\sqrt{1-\|\bar{x}\|}\leq C\sqrt{1-\|\bar{z}\|}~,
         \end{equation}
         for $\bar{x}$ close enough to $\bar{\Lambda}$. Note that the constants $C$ above change accordingly, and we keep the same notation for simplicity (similarly henceforth).

   Note that the projection angle $\zeta^r_{\bar x}$ is always in $[0,{\pi}/{2}]$ and $\theta_{\bar{x}}\geq 0$. If $\bar z\in\partial[\bar x,\bar y]$, then $\alpha^r_{\bar x}=\pi-{\zeta}^r_{\bar x}-\theta_{\bar x}\geq\pi/2-\theta_{\bar x}$ (in particular, $\zeta^r_{\bar x}=\pi/2$ if  $\bar{x}=\bar{z}$), as shown in Figure \ref{fig:proj-angle-z=y} and Figure \ref{fig:proj-angle-z=x}. Using Lemma \ref{lm: angle estimate}, for $\bar{x}$ close enough to $\bar{\Lambda}$,
   $$\cos\alpha^r_{\bar x}\leq \sin \theta_{\bar x}\leq C\sqrt{1-\|\bar x\|}\leq C\sqrt{1-\|\bar z\|}~.$$

   Combined with \eqref{eq:proj-angle-interior} and \eqref{eq:proj-angle-cos}, it follows that for $\bar{x}$ close enough to $\bar{\Lambda}$, $\alpha^r_{\bar{x}}$ always satisfies that $$\cos\alpha^r_{\bar{x}}\leq C\sqrt{1-\|\bar{z}\|}~.$$
   Similarly for $\alpha^r_{\bar{y}}$.

   If $\bar z\in (\bar x,\bar y)$, the vector $\bar{z}$ is orthogonal to $\bar{x}-\bar{y}$. Then $(\bar{x}-\bar{y})^r=0$ and \eqref{eq:x-y-r} naturally holds. If $\bar{z}=\bar{x}$, since $[\bar x,\bar y]\subset \bar\Omega$, the angle between $\bar x-\bar y$ and the plane $Q_{\bar x}$ orthogonal to $[\bar O,\bar x]=[\bar O, \bar z]$ is less than the angle between $Q_{\bar x}$ and the tangent plane $P_{\bar x}$ (which is equal to $\theta_{\bar{x}}$). Combined with Lemma \ref{lm: angle estimate}, for $\bar x$ close enough to $\bar\Lambda$,
   $$\|(\bar x-\bar y)^r\|\leq\|\bar x-\bar y\|\cdot\sin \theta_{\bar x}\leq C\|\bar x-\bar y\|\cdot\sqrt{1-\|\bar z\|}~.$$
   If $\bar{z}=\bar{y}$, we consider the plane $Q_{\bar y}$ orthogonal to $[\bar O,\bar y]=[\bar O, \bar z]$ and the tangent plane $P_{\bar y}$ instead. Similarly, for $\bar y$ close enough to $\bar\Lambda$,
   $$\|(\bar x-\bar y)^r\|\leq\|\bar x-\bar y\|\cdot\sin \theta_{\bar y}\leq C\|\bar x-\bar y\|\cdot\sqrt{1-\|\bar z\|}~.$$
   The claim follows.
   \end{proof}

            Let $\bar{x}\in\partial_k\bar{\Omega}$ and let $\bar{x}^*\in\partial_k\bar{\Omega}^*$ be the dual of the tangent plane to $\partial_k\bar{\Omega}$ at $\bar{x}$. We denote by $P_{\bar{x}}$ (resp. $P_{\bar{x}^*}$) the plane tangent to $\partial_k\bar{\Omega}$ and $\partial_k\bar{\Omega}^*$ at $\bar{x}$ (resp. $\bar{x}^*$). Let $\theta_{\bar{x}^*}$ denote the angle between the radial direction $\bar{O}\bar{x}^*$ at $\bar{x}^*$ and the outward-pointing normal vector to $\partial_k\bar{\Omega}^*$ at $\bar{x}^*$. Let $(\bar{x}^*-\bar{x})^r$ and $(\bar{x}^*-\bar{x})^{\perp}$ denote the radial and lateral components of $(\bar{x}^*-\bar{x})$ at $\bar{x}$. The following are the estimates for these quantities (see \cite[Proposition 4.6]{hmcb}).

\begin{lemma}\label{lm:estimate-dual points}
           There exists a constant $C>0$ such that for $\bar{x}$ close enough to $\bar{\Lambda}$,
          \begin{enumerate}
            \item
            $\|(\bar{x}^*-\bar{x})^r\|\leq C(1-\|\bar{x}\|)~.$
           \item
           $\|(\bar{x}^*-\bar{x})^{\perp}\|\leq C\sqrt{1-\|\bar{x}\|}~.$
            \item
            $\theta_{\bar{x}^*}\leq C\sqrt{1-\|\bar{x}\|}~.$
          \end{enumerate}
          \end{lemma}

          Let $\alpha^r_{\bar{x}^*}$ (resp. $\alpha^{\perp}_{\bar{x}^*}$) denote the projection angles to $P_{\bar{x}^*}$ of a radial (resp. lateral) vector at $\bar{z}$.

       \begin{claim}\label{clm:estimate-angle-dual}
        Let $\bar{x}^*,\bar{y}^*\in\cup_{k\in\cK}\partial_k\bar{\Omega}^*$ with $\bar{x}^*$, $\bar{y}^*$ close enough to $\bar{\Lambda}$, then
         \begin{equation}\label{ineq:max-dual}
   \max\{\cos\alpha^r_{\bar{x}^*},\cos\alpha^r_{\bar{y}^*}\}
   \leq C\sqrt{1-\|\bar{z}\|}~.
   \end{equation}
       \end{claim}

       \begin{proof}
       Let $\angle(P_{\bar{x}},P_{\bar{x}^*})$ denote the (un-oriented dihedral) angle between $P_{\bar{x}}$ and $P_{\bar{x}^*}$, which is equal to the angle between the outward-pointing normal vectors $n_{\bar{x}}$ and $n_{\bar{x}^*}$ to $\partial\bar{\Omega}$ and $\partial\bar{\Omega}^*$ at $\bar{x}$ and $\bar{x}^*$ respectively. By the duality between a totally geodesic plane in $\HH^3$ and a point in $\bdS^3$ (considered in the Klein models, see Section \ref{ssc:dual}), the radial vector $\bar{O}\bar{x}^*$ is orthogonal to the plane $P_{\bar{x}}$ (which is tangent to $\partial\bar{\Omega}$ at $\bar{x}$). Then $\angle(P_{\bar{x}},P_{\bar{x}^*})$ is equal to the angle between $\bar{O}\bar{x}^*$ and $n_{\bar{x}^*}$, that is, $\theta_{\bar{x}^*}$. Combined with the result for $\alpha^r_{\bar{x}}$ and $\alpha^r_{\bar{y}}$ in the proof of Claim \ref{clm:proj angle}, we obtain
        \begin{equation*}
        \begin{split}
        &\zeta^r_{\bar x}-(\theta_{\bar{x}}+\theta_{\bar{x}^*})
        \leq\alpha^r_{\bar{x}}-\angle(P_{\bar{x}},P_{\bar{x}^*})\leq \alpha^r_{\bar{x}^*}\leq \alpha^r_{\bar{x}}+\angle(P_{\bar{x}},P_{\bar{x}^*}) =\zeta^r_{\bar x}+(\theta_{\bar{x}}+\theta_{\bar{x}^*})~,
        \\
        &\zeta^r_{\bar y}-(\theta_{\bar{y}}+\theta_{\bar{y}^*})\leq
        \alpha^r_{\bar{y}}-\angle(P_{\bar{y}},P_{\bar{y}^*})\leq
        \alpha^r_{\bar{y}^*}\leq
        \alpha^r_{\bar{y}}+\angle(P_{\bar{y}},P_{\bar{y}^*})
        =\zeta^r_{\bar y}+(\theta_{\bar{y}}+\theta_{\bar{y}^*})~.
        \end{split}
        \end{equation*}
       Applying similar estimate as in the proof of Claim \ref{clm:proj angle} for $\cos\alpha^r_{\bar{x}}$ and $\cos\alpha^r_{\bar{y}}$ and using Lemma \ref{lm: angle estimate}, the inequality \eqref{ineq:general} and Statement (3) of Lemma \ref{lm:estimate-dual points}, then for $\bar{x}$ close enough to $\bar{\Lambda}$,
        \begin{equation*}
        \begin{split}
   \cos\alpha^r_{\bar{x}^*}\leq\cos(\zeta^r_{\bar x}-(\theta_{\bar{x}}+\theta_{\bar{x}^*}))
   &=\cos\zeta^r_{\bar x}\cdot\cos(\theta_{\bar{x}}+\theta_{\bar{x}^*})
   +\sin\zeta^r_{\bar x}\cdot\sin(\theta_{\bar{x}}+\theta_{\bar{x}^*})\\
   &\leq\|\bar{x}-\bar{z}\|/{r_0}+C\sqrt{1-\|\bar{x}\|}\leq C\sqrt{1-\|\bar{z}\|}~.
   \end{split}
   \end{equation*}
   Similarly for $\alpha^r_{\bar{y}^*}$. The claim follows.
        \end{proof}

Note that for each $\bar x\in\partial\bar\Omega\cap\HH^3$, we can define the dual point of $\bar x$, denoted by $\bar x^*$.

          \begin{claim}\label{clm:Lipschitz}
          There exists a constant $C>0$ such that
          \begin{enumerate}
             \item  for any 
                 $\bar{x},\bar{y}\in \partial\bar{\Omega}\setminus\bar{\Lambda}$, we have
            $$\frac{1}{C}\|\bar{x}^*-\bar{y}^*\|\leq \|\bar{x}-\bar{y}\|\leq C\|\bar{x}^*-\bar{y}^*\|~.$$
            \item for any 
                $\bar{x},\bar{y}\in\partial\bar{\Omega}\setminus\bar\Lambda$ lying in different components of $\partial\bar{\Omega}\setminus\bar\Lambda$, we have
            $$\|\bar{x}-\bar{y}\|\leq 2\|\bar{x}-\bar{y}^*\|~. $$
             In addition,  if $\bar{x}$ and $\bar{y}$ are close enough to $\bar\Lambda$, then
              \begin{equation*}
            \max\{\sqrt{1-\|\bar x\|},\sqrt{1-\|\bar y\|}\}\leq C\|\bar x-\bar y\|~,
            \end{equation*}
            and
             \begin{equation*}
             \|\bar{x}-\bar{y}^*\|\leq C\|\bar{x}-\bar{y}\|~.
            \end{equation*}
            \end{enumerate}
          \end{claim}

          \begin{proof}
          We first show Statement (1). By the definition of the duality and the strict convexity of $\partial\bar{\Omega}$, the map $D:\partial\bar{\Omega}\rightarrow \partial\bar{\Omega}^*$ (resp. $D^*:\partial\bar{\Omega}^*\rightarrow \partial\bar{\Omega}$) that takes $\bar{x}$ to $\bar{x}^*$ (resp. $\bar{x}^*$ to $\bar{x}$) is continuous. Note that $D$ and $D^*$ are exactly the identity when restricted to the intersection with $\partial\DD^3$. Furthermore, $D$ is a $C^1$ homeomorphism on $\partial\bar{\Omega}$. Therefore, there is a uniform constant $C>0$, such that $$\frac{1}{C}\|\bar{x}^*-\bar{y}^*\|\leq \|\bar{x}-\bar{y}\|\leq C\|\bar{x}^*-\bar{y}^*\|~,$$
          for all $\bar{x}, \bar{y}\in\partial\bar{\Omega}$. Statement (1) follows. 

          We now consider Statement (2). Without loss of generality, we assume that $\bar{x}\in(\cup_{i\in\cI}\partial_i\bar{\Omega})\cup(\cup_{j\in\cJ}\partial_j\bar{\Omega})$, $\bar{y}\in\cup_{k\in\cK}\partial_k\bar{\Omega}$. Note that $\partial\bar{\Omega}$ is strictly convex and by the definition of duality between convex surfaces in $\HH^3$ and $\bdS^3$, the point $\bar{y}^*$ lies outside of $\partial\DD^3$ and on the other side (opposite to $\bar{x}$) of the tangent plane to $\cup_{k\in\cK}\partial_k\bar{\Omega}$ at $\bar{y}$. Assume that $\bar{y}$ lies on a connected component, say $\partial^0_k\bar{\Omega}$, of $\partial_k\bar{\Omega}\subset\DD^3$ for some $k\in\cK$, with the boundary of $\partial^0_k\bar{\Omega}$ a Jordan curve say $\cC^0_k$ contained in $\bar{\Lambda}$. Observe that the distance between $\bar{y}^*$ and $\bar{y}$ is less than the distance between $\bar{y}^*$ and any point $\bar{p}\in\partial\bar{\Omega}\setminus(\partial^0_k\bar{\Omega}\cup \cC^0_k)$, as shown in Figure \ref{fig:dist-dual}. Thus,
          $$\|\bar{x}-\bar{y}\|\leq \|\bar{x}-\bar{y}^*\|+\|\bar{y}-\bar{y}^*\|\leq 2\|\bar{x}-\bar{y}^*\|~.$$
 By Statement (1) of Lemma \ref{lm: angle estimate}, for $\bar x$ close enough to $\bar \Lambda$,
 $$\sqrt{1-\|\bar x\|}\leq C\delta_{\partial\bar{\Omega}}(\bar x, \bar \Lambda)~.$$
 Since the principal curvatures of $\partial\bar\Omega\setminus\bar{\Lambda}$ are uniformly bounded between two positive constants (see Lemma \ref{prop:boundary regularity}) and $\bar x$, $\bar y$ lie in different components of $\partial\bar\Omega\setminus\bar\Lambda$, then there exists a constant $C>0$ such that
 $$\delta_{\partial\bar{\Omega}}(\bar x, \bar \Lambda)\leq C\|\bar x-\bar y\|~.$$
So we have $\sqrt{1-\|\bar x\|}\leq C\|\bar x-\bar y\|.$ Similarly, for $\bar y$ close enough to $\bar \Lambda$, we can show that
$\sqrt{1-\|\bar y\|}\leq C\|\bar x-\bar y\|$. This implies that for $\bar x$, $\bar y$ close enough to $\bar\Lambda$,
\begin{equation}\label{eq:diff components}
\max\{\sqrt{1-\|\bar x\|},\sqrt{1-\|\bar y\|}\}\leq C\|\bar x-\bar y\|~.
\end{equation}
We stress that \eqref{eq:diff components} only works for $\bar{x},\bar{y}$ lying in different components of $\partial\bar{\Omega}\setminus\bar\Lambda$. If $\bar x$ and $\bar y$ lie on the same component of $\partial\bar{\Omega}\setminus\bar\Lambda$, \eqref{eq:diff components} does not hold since we can take $\|\bar x-\bar y\|$ arbitrarily small.

Observe that $\|\bar{x}-\bar{y}^*\|\leq\|\bar x-\bar y\|+\|\bar{y}-\bar{y}^*\|$. By Statements (1)-(2) of Lemma \ref{lm:estimate-dual points}, $$\|\bar y-\bar y^*\|\leq C\sqrt{1-\|\bar y\|}~.$$ Using \eqref{eq:diff components}, for $\bar x$, $\bar y$ close enough to $\bar\Lambda$, $$\|\bar{x}-\bar{y}^*\|\leq C\|\bar{x}-\bar{y}\|~.$$
Statement (2) follows.

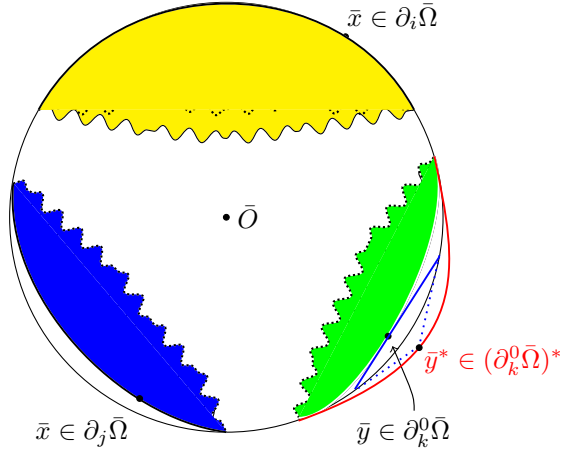
\begin{figure}
\centering
\begin{tikzpicture}[scale=0.95]
\draw  (0,0) ellipse (3 and 3);
\draw [snake=snake][white][fill=yellow,opacity=0.1] (-2.6,1.5) --(-1,1.2)--(0,1.1)-- (1,1.2)--(2.6,1.5);
\draw [snake=snake](-2.6,1.5) --(-1,1.2)--(0,1.1)-- (1,1.2)--(2.6,1.5);
\draw [fill=black] (1.65,2.52) ellipse (0.04 and 0.04);
\node at (2.3,2.8) {$\bar x\in\partial_i\bar\Omega$};
\draw [densely dotted, thick, snake=snake] (-2.6,1.5) --(-2.3,1.52)--(-1.2,1.55)-- (-0.8,1.6)--(0,1.5)-- (0.8,1.6)--(1.2,1.55)--(2.2,1.52)--(2.6,1.5);
\draw [snake=snake] (-2.95,0.48)--(0,-3);
\draw [densely dotted, thick, snake=snake][white] [fill=blue,opacity=0.15] (-2.95,0.48)--(-2.7,0.35)--(-1.5,-0.75)--(-1.0,-1.5)--(-0.4,-2.15)--(0,-3);
\draw [densely dotted, thick, snake=snake] (-2.95,0.48)--(-2.7,0.35)--(-1.5,-0.75)--(-1.0,-1.5)--(-0.4,-2.15)--(0,-3);
\draw[line width=0.25mm][white][fill=yellow,opacity=0.1](-2.6,1.5) .. controls (-1.5,3.4925) and (1.56,3.46) .. (2.6,1.5);
\draw[line width=0.25mm](-2.6,1.5) .. controls (-1.5,3.4925) and (1.56,3.46) .. (2.6,1.5);
\draw[line width=0.25mm][white][fill=blue,opacity=0.15](-2.95,0.48) .. controls (-3.1,-1.1) and (-1.5,-2.8) .. (0,-3);
\draw[line width=0.25mm](-2.95,0.48) .. controls (-3.1,-1.1) and (-1.5,-2.8) .. (0,-3);
\draw [fill=black] (0,0) ellipse (0.04 and 0.04);
\node at (0.3,0) {$\bar O$};
\draw [fill=black] (-1.2,-2.53) ellipse (0.04 and 0.04);
\node at (-2,-3) {$\bar x\in\partial_j\bar\Omega$};
\draw [snake=snake](1,-2.8) --(2.2,-0.5)--(2.88,0.85);
\draw [snake=snake][white][fill=green,opacity=0.15] (1,-2.8) --(1.2,-2)-- (2,-0.3)--(2.88,0.85);
\draw [snake=snake][densely dotted][thick] (1,-2.8) --(1.2,-2)-- (2,-0.3)--(2.88,0.85);
\draw[line width=0.25mm](1,-2.8) .. controls (1.8,-2.8) and (3.25,-0.3) .. (2.88,0.85);
\draw[line width=0.25mm][white][fill=green,opacity=0.15](1,-2.8) .. controls (1.8,-2.8) and (3.25,-0.3) .. (2.88,0.85);
\draw [fill=black] (2.24,-1.66) ellipse (0.04 and 0.04);
\node at (2.45,-3) {$\bar y\in\partial^0_k\bar\Omega$};
\draw[line width=0.15mm][<-](2.3,-1.8) .. controls (2.6,-2.5) and (2.6,-2.5) .. (2.6,-2.7);
\draw[line width=0.25mm][red](1,-2.84) .. controls (3.48,-1.95) and (3.2,-0.7) .. (2.88,0.85);
\draw[line width=0.25mm][blue](1.78,-2.4) .. controls (2.95,-0.55) and (2.95,-0.55) .. (2.95,-0.55);
\draw[line width=0.25mm][blue,dotted](2.95,-0.53) .. controls (2.67,-1.8) and (2.67,-1.8) .. (2.67,-1.8);
\draw[line width=0.25mm][blue,dotted](2.67,-1.8)-- (1.78,-2.42);
\draw [fill=black] (2.67,-1.82) ellipse (0.04 and 0.04);
\node [red] at (3.7,-2) {$\bar y^*\in(\partial^0_k\bar\Omega)^*$};
\end{tikzpicture}
\caption{\small{$\|\bar{y}^*-\bar{y}\|<\|\bar{y}^*-\bar{x}\|$ for all $\bar{x}\in\partial\bar{\Omega}\setminus(\partial^0_k\bar{\Omega}\cup\cC^0_k)$ (where $\partial\bar{\Omega}$ is partially drawn).}}
\label{fig:dist-dual}
\end{figure}
\end{proof}

          We are now ready to show the following key lemma.

\begin{lemma}\label{lm:Holder estimate-x close to y}
          $\bar{v}$ is $C^{\alpha}$ H\"older continuous on $\partial\bar{\Omega}^*\setminus\bar{\Lambda}$ for all $\alpha\in (0,1)$. More precisely, for any $\alpha\in (0,1)$, there exists a constant $C_{\alpha}>0$ (depending on $\alpha$) such that for all $\bar{x},\bar{y}\in \partial\bar{\Omega}^*\setminus\bar{\Lambda}$,
          $$\|\bar{v}(\bar{x})-\bar{v}(\bar{y})\|\leq C_{\alpha}\|\bar{x}-\bar{y}\|^{\alpha}~.$$
\end{lemma}

\begin{proof}




  Since $\bar{v}$ is smooth on $\partial\bar{\Omega}^*\setminus\bar{\Lambda}$ (see Remark \ref{rk:regularity-boundary}), it is a Lipschitz vector field when restricted to a compact subset of $\partial\bar\Omega\setminus \bar\Lambda$. We fix a constant $\epsilon>0$ (to be made precise below) and denote by
  $$ \bar K_1 = \{ \bar x\in \partial\bar\Omega ~|~ d(\bar x,\bar\Lambda)\geq \epsilon\}~. $$

  Note that $\bar K_1$ is a compact subset of $\partial\bar\Omega\setminus\bar\Lambda$. We claim that it suffices to discuss the H\"older continuity of $\bar{v}$ for the set $\bar{K}_2:=\partial\bar{\Omega}^*\setminus ({\rm int}\bar{K}_1\cup\bar{\Lambda})$.
  Indeed, assume that Lemma \ref{lm:Holder estimate-x close to y} holds respectively for $\bar{K}_1$ and $\bar{K}_2$.  For any $\bar{x}\in\bar{K}_1$ and $\bar{y}\in \bar{K_2}$, 
  there exists a point $\bar x'\in\partial\bar K_1$ such that $\max\{\|\bar x-\bar x'\|,\|\bar x'-\bar y\|\}\leq \|\bar{x}-\bar{y}\|$.
  Then
  \begin{equation*}
    \begin{split}
\|\bar{v}(\bar{x})-\bar{v}(\bar{y})\|
&\leq\|\bar{v}(\bar{x})-\bar{v}(\bar{x}')\|+\|\bar{v}(\bar{x}')-\bar{v}(\bar{y})\|\\
&\leq C_{\alpha}(\|\bar{x}-\bar{x}'\|^{\alpha}
      +\|\bar{x}'-\bar{y}\|^{\alpha})\\
&\leq 2C_{\alpha}\|\bar{x}-\bar{y}\|^{\alpha}~.
    \end{split}
  \end{equation*}
  The claim follows.

  We now show the H\"older continuity of $\bar{v}$ restricted to $\bar{K}_2$. Let $\bar{x}, \bar{y}\in\bar{K}_2$. Recall that $[\bar{x},\bar{y}]\subset \DD^3$ denotes the geodesic segment through $\bar{x}$ and $\bar{y}$, while $\bar{z}\in [\bar{x},\bar{y}]$ denotes the point which is nearest to the center $\bar{O}$.

  Recall that in Subsection \ref{subsec:const section},
  $$\delta\bar{\tau}=\bar{\tau}-\bar{\tau}_0~,\quad\quad  \delta\bar{\sigma}=\bar{\sigma}-\bar{\sigma}_0~, $$
  where $(\bar{\tau}_0,\bar{\sigma}_0)\in\Gamma(T\bar{\Omega})\times\Gamma(T\bar{\Omega})$ satisfies that $\bar{\tau}_0(\bar{z})=\bar{\tau}(\bar{z})$ and  $\bar{\sigma}_0(\bar{z})=\bar{\sigma}(\bar{z})$, which determines a flat section $\bar{\kappa}_0\in\Gamma(\bar{E})$: $(\bar{\kappa}_0)_{\bar{x}}=(\bar{\kappa}_0)_{\bar{y}}$ for all $\bar{x},\bar{y}\in\bar{\Omega}$. It follows from 
  the definition of $\bar\sigma$ and $\bar\tau$ as the rotation and translation components of an infinitesimal isometry that $$(\bar{\kappa}_0)_{\bar{y}}(\bar{x})=\bar{\tau}_0(\bar{x})=\bar{\sigma}(\bar{z})\times(\bar{x}-\bar{z})+\bar{\tau}(\bar{z})~,$$
     for all $\bar{x}, \bar{y}\in\bar{\Omega}$. Note that $\bar{v}=\Pi\bar{u}$, where $\bar{u}$ is the extended deformation vector field given in Section \ref{subsec:deformation field}, while $\Pi$ is the orthogonal projection to $\partial\bar{\Omega}^*$, which behaves at $\bar{x}\in\partial\bar{\Omega}^*$, as a morphism $\Pi_{\bar{x}}:T_{\bar{x}}\R^3\rightarrow T_{\bar{x}}\partial\bar{\Omega}^*$. Note that $\bar{u}(\bar{x})=\bar{\tau}(\bar{x})$. 

     Observe that it suffices to discuss the case that $\|\bar x-\bar y\|\leq 1/2$. Otherwise, we can always take the constant $C_{\alpha}$ big enough for each $\alpha\in (0,1)$.

     If $\bar{x},\bar{y}\in(\cup_{i\in\cI}\partial_i\bar{\Omega})\cup(\cup_{j\in\cJ}\partial_j\bar{\Omega})$, we have
  \begin{equation*}
          \begin{split}
          \|\bar{v}(\bar{x})-\bar{v}(\bar{y})\|
          &=\|\Pi_{\bar{x}}(\bar{\tau}(\bar{x}))-
          \Pi_{\bar{y}}(\bar{\tau}(\bar{y}))\|\\
          &=\|\Pi_{\bar{x}}(\delta\bar{\tau}+\bar{\tau}_0)(\bar{x})-
          \Pi_{\bar{y}}(\delta\bar{\tau}+\bar{\tau}_0)(\bar{y})\|\\
          &\leq
          \|\Pi_{\bar{x}}(\bar{\tau}_0(\bar{x}))-\Pi_{\bar{y}} (\bar{\tau}_0(\bar{y}))\|+\|\Pi_{\bar{x}}(\delta\bar{\tau}(\bar{x}))- \Pi_{\bar{y}}(\delta\bar{\tau}(\bar{y}))\|\\
          & \leq \|\Pi_{\bar{x}}\big(\bar{\tau}_0(\bar{x})- \bar{\tau}_0(\bar{y})\big)\|
          +\|\big(\Pi_{\bar{x}}-\Pi_{\bar{y}}\big) \bar{\tau}_0(\bar{y})\|+
          \|\Pi_{\bar{x}}\big(\delta\bar{\tau}(\bar{x})\big)\|
          +\|\Pi_{\bar{y}}\big(\delta\bar{\tau}(\bar{y})\big)\|~.
          \end{split}
          \end{equation*}

 If $\bar{x}^*,\bar{y}^*\in\cup_{k\in\cK}\partial_k\bar{\Omega}^*$, by Lemma \ref{lm:dual-Killing}, we have
        \begin{equation*}\label{ineq:case2}
        \begin{split}
        \|\bar{v}(\bar{x}^*)-\bar{v}(\bar{y}^*)\|&=\|\Pi_{\bar{x}^*}\big(\bar{u}(\bar{x}^*)\big)-\Pi_{\bar{y}^*}\big(\bar{u}(\bar{y}^*)\big)\|\\
        &= \|\Pi_{\bar{x}^*}\big(\bar{\tau}(\bar{x})+\bar{\sigma}(\bar{x})\times(\bar{x}^*-\bar{x})\big)
        -\Pi_{\bar{y}^*}\big(\bar{\tau}(\bar{y})+\bar{\sigma}(\bar{y})\times(\bar{y}^*-\bar{y})\big)\|\\
        &= \|\Pi_{\bar{x}^*}\big((\delta\bar{\tau}+\bar{\tau}_0)(\bar{x})
        +\bar\sigma(\bar{x})\times(\bar{x}^*-\bar{x})\big)
        -\Pi_{\bar{y}^*}\big((\delta\bar{\tau}+\bar{\tau}_0)(\bar{y})
        +\bar\sigma(\bar{y})\times(\bar{y}^*-\bar{y})\big)\|\\
        &\leq \|\Pi_{\bar{x}^*}\big(\bar{\tau}_0(\bar{x})-\bar{\tau}_0(\bar{y})\big)\|+\|(\Pi_{\bar{x}^*}-\Pi_{\bar{y}^*})\bar{\tau}_0(\bar{y})\|
        +\|\Pi_{\bar{x}^*}\big(\delta\bar{\tau}(\bar{x})\big)\|+\|\Pi_{\bar{y}^*}\big(\delta\bar{\tau}(\bar{y})\big)\|\\
        &\quad+
        \|\Pi_{\bar{x}^*}\big(\bar{\sigma}(\bar{x})\times(\bar{x}^*-\bar{x})\big)-\Pi_{\bar{y}^*}\big(\bar{\sigma}(\bar{y})\times(\bar{y}^*-\bar{y})\big)\|~.\\
                \end{split}
        \end{equation*}

If $\bar{x}\in(\cup_{i\in\cI}\partial_i\bar{\Omega})\cup(\cup_{j\in\cJ}\partial_j\bar{\Omega})$, $\bar{y}^*\in\cup_{k\in\cK}\partial_k\bar{\Omega}^*$, by Lemma \ref{lm:dual-Killing}, we have
        \begin{equation*}
        \begin{split}
        \|\bar{v}(\bar{x})-\bar{v}(\bar{y}^*)\|&=\|\Pi_{\bar{x}}\big(\bar{u}(\bar{x})\big)-\Pi_{\bar{y}^*}\big(\bar{u}(\bar{y}^*)\big)\|\\
        &= \|\Pi_{\bar{x}}\big(\bar{\tau}(\bar{x})\big)
        -\Pi_{\bar{y}^*}\big(\bar{\tau}(\bar{y})+\bar{\sigma}(\bar{y})\times(\bar{y}^*-\bar{y})\big)\|\\
        &= \|\Pi_{\bar{x}}\big((\delta\bar{\tau}+\bar{\tau}_0)(\bar{x})\big)
        -\Pi_{\bar{y}^*}\big((\delta\bar{\tau}+\bar{\tau}_0)(\bar{y})+\bar{\sigma}(\bar{y})\times(\bar{y}^*-\bar{y})\big)\|\\
        &\leq \|\Pi_{\bar{x}}\big(\bar{\tau}_0(\bar{x})-\bar{\tau}_0(\bar{y})\big)\|+\|(\Pi_{\bar{x}}-\Pi_{\bar{y}^*})\big(\bar{\tau}_0(\bar{y})\big)\|
        +\|\Pi_{\bar{x}}\big(\delta\bar{\tau}(\bar{x})\big)\|+\|\Pi_{\bar{y}^*}\big(\delta\bar{\tau}(\bar{y})\big)\|\\
        &\quad +\|\Pi_{\bar{y}^*}\big(\bar{\sigma}(\bar{y})\times(\bar{y}^*-\bar{y})\big)\|~.
        \end{split}
        \end{equation*}

  To conclude, it suffices to estimate the following terms:

  \textbf{(1)} The term of the form $\|\Pi_{\bar{x}}\big(\bar{\tau}_0(\bar{x})- \bar{\tau}_0(\bar{y})\big)\|$. 
  Observe that the radial and lateral components of $\bar{\tau}_0(\bar{x})-\bar{\tau}_0(\bar{y})$, at $\bar{z}$, are respectively:
  \begin{equation*}
    \begin{split}
      \big(\bar{\tau}_0(\bar{x})-\bar{\tau}_0(\bar{y})\big)^r&=\bar{\sigma}^{\perp}(\bar{z})\times (\bar{x}-\bar{y})^{\perp}~,\\
      \big(\bar{\tau}_0(\bar{x})-\bar{\tau}_0(\bar{y})\big)^{\perp}
      &=\bar{\sigma}^r(\bar{z})\times (\bar{x}-\bar{y})^{\perp}+
      \bar{\sigma}^{\perp}(\bar{z})\times (\bar{x}-\bar{y})^r~,\\
    \end{split}
  \end{equation*}
  where $\bar{\sigma}^r(\bar{z})$ and $\bar{\sigma}^{\perp}(\bar{z})$ denote the radial and lateral components of $\bar{\sigma}(\bar{z})$ at $\bar{z}$, while $(\bar{x}-\bar{y})^r$ and $(\bar{x}-\bar{y})^{\perp}$ denote the radial and lateral components of $\bar{x}-\bar{y}$ at $\bar{z}$. Using the projection angles $\alpha^r_{\bar{x}}$ and $\alpha^{\perp}_{\bar{x}}$ introduced in Subsection \ref{subsec:estimate-angles}, we have
  \begin{equation}\label{eq:Case1}
    \begin{split}
      \|\Pi_{\bar{x}}\big(\bar{\tau}_0(\bar{x})-\bar{\tau}_0(\bar{y})\big)\|
      &=\|\Pi_{\bar{x}}\big(\big(\bar{\tau}_0(\bar{x})-\bar{\tau}_0(\bar{y})\big)^r
      +\big(\bar{\tau}_0(\bar{x})-\bar{\tau}_0(\bar{y})\big)^{\perp}\big)\|\\
      &\leq\|\Pi_{\bar{x}}\big(\bar{\sigma}^{\perp}(\bar{z})\times (\bar{x}-\bar{y})^{\perp}\big)\|+ \|\Pi_{\bar{x}}\big(\bar{\sigma}^r(\bar{z})\times (\bar{x}-\bar{y})^{\perp}+\bar{\sigma}^{\perp}(\bar{z})\times (\bar{x}-\bar{y})^r\big)\|\\
      &\leq \|\bar{\sigma}^{\perp}(\bar{z})\times (\bar{x}-\bar{y})^{\perp}\|\cdot\cos\alpha^r_{\bar{x}}
      +\big(\|\bar{\sigma}^r(\bar{z})\times (\bar{x}-\bar{y})^{\perp}\|+\|\bar{\sigma}^{\perp}(\bar{z})\times (\bar{x}-\bar{y})^r\|\big)\cdot \cos \alpha^{\perp}_{\bar{x}}\\
      &\leq \frac{C}{\sqrt{1-\|\bar{z}\|}}\cdot\|\bar{x}-\bar{y}\|\cdot C\sqrt{1-\|\bar{z}\|}+ \big(C\big|\log(1-\|\bar{z}\|)\big|\cdot\|\bar{x}-\bar{y}\|
      +\frac{C\|\bar x-\bar y\|\sqrt{1-\|\bar z\|}}{\sqrt{1-\|\bar z\|}}\big)\\
      &\leq C\|\bar{x}-\bar{y}\|+C \big|\log(\|\bar{x}-\bar{y}\|)\big| \cdot \|\bar{x}-\bar{y}\|+C\|\bar{x}-\bar{y}\|\\
      &\leq C\|\bar{x}-\bar{y}\|+ C_{\alpha} \|\bar{x}-\bar{y}\|^{\alpha}\\
      &\leq C_{\alpha}\|\bar{x}-\bar{y}\|^{\alpha}~,\\
    \end{split}
  \end{equation}
  for all $\alpha\in (0,1)$, where $C_{\alpha}$ is a constant depending on $\alpha$. The third inequality follows from Lemma \ref{lm:pogorelov_radial and lateral} and Claim \ref{clm:proj angle}.
  The fourth inequality follows from the inequality \eqref{ineq:general}. The last line follows from the hypothesis that $\|\bar{x}-\bar{y}\|\leq 1/2<1$. By Claim \ref{clm:estimate-angle-dual},
  the above estimate \eqref{eq:Case1} also works when $\Pi_{\bar{x}}$ is replaced with $\Pi_{\bar{x}^*}$. Combined with Claim \ref{clm:Lipschitz}, the corresponding H\"older estimates also hold.

  \textbf{(2)} The term of the form $\|\big(\Pi_{\bar{x}}-\Pi_{\bar{y}}\big) \bar{\tau}_0(\bar{y})\|$. 
  Observe that the radial and lateral components of $\bar{\tau}_0(\bar{y})$ at $\bar{z}$ are respectively:
    \begin{equation*}
    \begin{split}
    \bar{\tau}_0^r(\bar{y})&=\bar{\sigma}^{\perp}(\bar{z})\times(\bar{y}-\bar{z})^{\perp}+\bar{\tau}^r(\bar{z})~,\\
    \bar{\tau}_0^{\perp}(\bar{y})&
    =\bar{\sigma}^r(\bar{z})\times(\bar{y}-\bar{z})^{\perp}
    +\bar{\sigma}^{\perp}(\bar{z})\times(\bar{y}-\bar{z})^r
    +\bar{\tau}^{\perp}(\bar{z})~.\\
    \end{split}
  \end{equation*}
  It follows using Lemma \ref{lm:pogorelov_radial and lateral}, the inequality \eqref{ineq:general} and Claim \ref{clm:proj angle} that
  \begin{eqnarray*}
    \|\bar{\tau}_0^r(\bar{y})\|
    & \leq & \|\bar{\sigma}^{\perp}(\bar{z})\times(\bar{y}-\bar{z})^{\perp}\|
    +\|\bar{\tau}^r(\bar{z}) \|\\
    & \leq & \frac{C}{\sqrt{1-\|\bar z\|}}\cdot\|\bar y-\bar x\|+C|\log(1-\|\bar z\|)| \\
    & \leq & C +C |\log(\|\bar y-\bar x\|)| \\
    & \leq & C |\log(\|\bar y-\bar x\|)|~, \\
    \|\bar{\tau}_0^{\perp}(\bar{y})\|
    & \leq & \|\bar{\sigma}^r(\bar{z})\|\cdot\|(\bar{y}-\bar{z})^{\perp}\|
    +\|\bar{\sigma}^{\perp}(\bar{z})\|\cdot\|(\bar{y}-\bar{z})^r\|
    +\|\bar{\tau}^{\perp}(\bar{z})\| \\
    & \leq & C|\log(1-\|\bar z\|)|\cdot\|\bar{y}-\bar{z}\|+\frac{C}{\sqrt{1-\|\bar z\|}}\cdot \|\bar y-\bar z\|\cdot C\sqrt{1-\|\bar z\|} +C \\
    & \leq & C|\log(\|\bar y-\bar z\|)|\cdot\|\bar{y}-\bar{z}\|+C \\
    & \leq & C~.
  \end{eqnarray*}
  Putting both terms together, we obtain that
  $$ \|\bar{\tau}_0(\bar{y})\|\leq C |\log(\|\bar y-\bar x\|)|~. $$

  We can compute the operator norm of $\Pi_{\bar x}-\Pi_{\bar y}$ as follows. To simplify notations we set $\theta=\angle(\nx,\ny)$. Let $e_1$ be a unit vector orthogonal to both $n_{\bar x}$ and $n_{\bar y}$ (where $n_{\bar x}$ and $n_{\bar y}$ are the outward-pointing unit normal vectors to $\partial\bar{\Omega}$ at $\bar x$ and $\bar y$), and let
  $$ e_2=\frac{\nx-\ny}{2\sin(\theta/2)}~,\quad\quad e_3=\frac{\nx+\ny}{2\cos(\theta/2)}~. $$
  so that $(e_1, e_2, e_3)$ is an orthonormal basis.
  A direct computation shows that the matrix of $\Pi_{\bar x}-\Pi_{\bar y}$ in this basis is
  $$ \left(
    \begin{array}[h]{ccc}
      0 & 0 & 0 \\
      0 & 0 &\sin\theta \\
      0 & \sin\theta & 0
    \end{array}
  \right)~, $$
  and the norm of $\Pi_{\bar x}-\Pi_{\bar y}$ is $\sin\theta$. However it follows from the uniform upper bound on the principal curvatures of $\partial\bar\Omega\setminus\bar{\Lambda}$ (see Lemma \ref{prop:boundary regularity}) that, if $\|\bar x-\bar y\|$ is small enough, then
  $$ \theta=\angle(n_{\bar x},n_{\bar y})\leq C \|\bar x-\bar y\|~, $$
  for some constant $C>0$, and it follows that
  \begin{equation}\label{eq:proj-diff}
  \|\Pi_{\bar x}-\Pi_{\bar y}\|\leq C \|\bar x-\bar y\|~.
  \end{equation}
  As a consequence,
  $$ \|(\Pi_{\bar x}-\Pi_{\bar y})(\bar{\tau}_0(\bar{y}))\| \leq \|\Pi_{\bar x}-\Pi_{\bar y}\|\cdot\|\bar{\tau}_0(\bar{y})\|\leq C \|\bar x-\bar y\|\cdot|\log(\|\bar x-\bar y\|)|~, $$
  so for all $\alpha\in (0,1)$, there exists $C_\alpha>0$ such that
  $$ \|(\Pi_{\bar x}-\Pi_{\bar y})(\bar{\tau}_0(\bar{y}))\| \leq C_\alpha \|\bar x-\bar y\|^\alpha~. $$

  Note that the above inequality \eqref{eq:proj-diff} still holds when $\bar{x}$, $\bar{y}$ are replaced by  $\bar x^*$, $\bar y^*$  (resp. $\bar x$, $\bar y^*$) for $\|\bar x^*-\bar y^*\|$ (resp. $\|\bar x-\bar y^*\|$)
  small enough, since the principal curvatures of $\partial\bar\Omega^*\setminus\bar{\Lambda}$ are uniformly bounded between two positive constants (see Lemma \ref{prop:boundary regularity}). Combined with Claim \ref{clm:Lipschitz}, then for each $\alpha\in (0,1)$, there exists $C_\alpha>0$ such that
    \begin{equation*}
  \|(\Pi_{\bar x^*}-\Pi_{\bar y^*})(\bar{\tau}_0(\bar{y}))\|
  \leq C\|\bar x^*-\bar y^*\|\cdot|\log(\|\bar x-\bar y\|)|
  \leq C\|\bar x^*-\bar y^*\|\cdot|\log(\|\bar x^*-\bar y^*\|)|
  \leq C_{\alpha}\|\bar x^*-\bar y^*\|^{\alpha}~,
  \end{equation*}
   and for $\bar x$ and $\bar y$ lying in different components of $\partial\bar\Omega\setminus \bar\Lambda$,
\begin{equation*}
  \|(\Pi_{\bar x}-\Pi_{\bar y^*})(\bar{\tau}_0(\bar{y}))\|
  \leq C\|\bar x-\bar y^*\|\cdot|\log(\|\bar x-\bar y\|)|
  \leq C\|\bar x-\bar y^*\|\cdot|\log(\|\bar x-\bar y^*\|)|
  \leq C_{\alpha}\|\bar x-\bar y^*\|^{\alpha}~.
\end{equation*}

  \textbf{(3)} The term of the form $\|\Pi_{\bar{x}}\big(\delta\bar{\tau}(\bar{x})\big)\|$.   Let $\beta^r_{\bar{x}}$ and $\beta^{\perp}_{\bar{x}}$ denote the projection angle to the tangent plane $P_{\bar{x}}$ of a radial and lateral vector at $\bar{x}$ and let $\delta\bar{\tau}^{r}(\bar{x})$ and $\delta\bar{\tau}^{\perp}(\bar{x})$ denote the radial and lateral components of $\delta\bar{\tau}(\bar{x})$ at $\bar x$. Then $\beta^r_{\bar{x}}=\pi/2-\theta_{\bar{x}}$ and
          $\beta^{\perp}_{\bar{x}}\in [0, \theta_{\bar{x}}]$. In particular, $\theta_{\bar{x}}=0$ if $\bar{x}\in\partial_i\bar{\Omega}$.

          \begin{equation}\label{eq:Case3}
          \begin{split}
          \|\Pi_{\bar{x}}(\delta\bar{\tau}(\bar{x}))\|&\leq \|\Pi_{\bar{x}}(\delta\bar{\tau}^{\perp}(\bar{x}))\|+
          \|\Pi_{\bar{x}}(\delta\bar{\tau}^{r}(\bar{x}))\|\\
          &=\|\delta\bar{\tau}^{\perp}(\bar{x})\|\cdot \cos\beta^{\perp}_{\bar{x}}+
          \|\delta\bar{\tau}^{r}(\bar{x})\|\cdot\cos\beta^r_{\bar{x}}\\
          &\leq \|\delta\bar{\tau}^{\perp}(\bar{x})\|\cdot 1 + \|\delta\bar{\tau}^r(\bar{x})\|\cdot \sin\theta_{\bar{x}}\\
          &\leq \|\delta\bar{\tau}^{\perp}(\bar{x})\| + \|\delta\bar{\tau}(\bar{x})\|\cdot C\sqrt{1-\|\bar{x}\|} \\
          &\leq C\|\bar x-\bar z\|+C\|\bar{x}-\bar{z}\|\cdot \big|\log\|\bar{x}-\bar{z}\|\big|\\
          &\leq C_\alpha\|\bar{x}-\bar{y}\|^{\alpha},
          \end{split}
          \end{equation}
          for all $\alpha\in(0,1)$, where $C_{\alpha}$ is a constant depending on $\alpha$. The third inequality follows from Lemma \ref{lm: angle estimate}, while the fourth inequality results from \eqref{eq:x-z-small} and Claim \ref{clm:deltatau}. Note that the estimate for the term $\|\Pi_{\bar{y}}\big(\delta\bar{\tau}(\bar{y})\big)\|$ follows in the same way as above.

          Let $\beta^r_{\bar{x}^*}$ and $\beta^{\perp}_{\bar{x}^*}$ denote the projection angles to the tangent plane $P_{\bar{x}^*}$ of a radial and lateral vector at $\bar{x}$. By the fact that $\angle(P_{\bar{x}},P_{\bar{x}^*})=\theta_{\bar x^*}$ (see the proof of Claim \ref{clm:estimate-angle-dual}), we have
        \begin{equation*}
       \pi/2-(\theta_{\bar x}+\theta_{\bar x^*}) =\beta^r_{\bar{x}}-\angle(P_{\bar{x}},P_{\bar{x}^*})
       \leq \beta^r_{\bar{x}^*}
       \leq \beta^r_{\bar{x}}+\angle(P_{\bar{x}},P_{\bar{x}^*})
       = \pi/2-(\theta_{\bar x}-\theta_{\bar x^*})~.
        \end{equation*}
        Hence, combined with Lemma \ref{lm: angle estimate} and Lemma \ref{lm:estimate-dual points},
        \begin{equation}\label{eq:beta}
        \cos\beta^r_{\bar x^*}\leq \sin(\theta_{\bar x}+\theta_{\bar x^*})\leq \sin\theta_{\bar x}+\sin\theta_{\bar x^*}\leq C\sqrt{1-\|\bar x\|}~.
        \end{equation}
        Similarly for $\beta^r_{\bar{y}^*}$. Thanks to Statement (3) of Lemma \ref{lm:estimate-dual points}, the above estimate \eqref{eq:Case3} also works when $\Pi_{\bar{x}}$ (or $\Pi_{\bar{y}}$) is replaced with $\Pi_{\bar{x}^*}$ (or $\Pi_{\bar{y}^*}$). Combined with Claim \ref{clm:Lipschitz}, the corresponding H\"older estimates also hold.

   \textbf{(4)} The term $\|\Pi_{\bar{x}^*}\big(\bar{\sigma}(\bar{x})\times(\bar{x}^*-\bar{x})\big)-\Pi_{\bar{y}^*}\big(\bar{\sigma}(\bar{y})\times(\bar{y}^*-\bar{y})\big)\|$
   for $\bar{x}^*,\bar{y}^*\in\cup_{k\in\cK}\partial_k\bar{\Omega}^*$.
   Recall that $\bar\sigma(\bar x)=\delta\bar\sigma(\bar x)+\bar\sigma_0(\bar x)$.  Since $(\bar\tau_0,\bar\sigma_0)$ represents a constant section $\kappa_0\in\Gamma(\bar E)$, whose value at each point of $\bar\Omega$ is the same Euclidean Killing field. Note that the curl of an Euclidean Killing field is constant, so $\bar\sigma_0$ is constant, namely $\bar\sigma_0(\bar x)=\bar\sigma_0(\bar z)=\bar\sigma(\bar z)$ for all $\bar x\in\bar{\Omega}$.
  \begin{equation}\label{eq:both dual}
  \begin{split}
  &\quad\|\Pi_{\bar{x}^*}\big(\bar{\sigma}(\bar{x})\times(\bar{x}^*-\bar{x})\big)-\Pi_{\bar{y}^*}\big(\bar{\sigma}(\bar{y})\times(\bar{y}^*-\bar{y})\big)\|\\
  &\leq\|\Pi_{\bar{x}^*}\big(\delta\bar{\sigma}(\bar{x})\times(\bar{x}^*-\bar{x})\big)\|
 +\|\Pi_{\bar{y}^*}\big(\delta\bar{\sigma}(\bar{y})\times(\bar{y}^*-\bar{y})\big)\|
 +\|\Pi_{\bar{x}^*}\big(\bar{\sigma}(\bar z)\times (\bar x^*-\bar x)\big)-\Pi_{\bar y^*}\big(\bar{\sigma}(\bar z)\times (\bar y^*-\bar y)\big)\|~.
 \end{split}
 \end{equation}

Observe that the first two terms on the right-hand side of \eqref{eq:both dual} are of the same form, we just need to consider the first term.   Note that the radial and lateral components, at $\bar{x}$, of $\delta\bar{\sigma}(\bar{x})\times(\bar{x}^*-\bar{x})$ are respectively
 \begin{equation*}
              \begin{split}
              \big(\delta\bar{\sigma}(\bar{x})\times(\bar{x}^*-\bar{x})\big)^r
              &=\delta\bar{\sigma}^{\perp}(\bar{x})\times(\bar{x}^*-\bar{x})^{\perp},\\
              \big(\delta\bar{\sigma}(\bar{x})\times(\bar{x}^*-\bar{x})\big)^{\perp}
              &=\delta\bar{\sigma}^r(\bar{x})\times(\bar{x}^*-\bar{x})^{\perp}+
              \delta\bar{\sigma}^{\perp}(\bar{x})\times(\bar{x}^*-\bar{x})^r~.\\
              \end{split}
              \end{equation*}
 By Lemma \ref{lm:estimate-hmcb}, Lemma \ref{lm:estimate-dual points}, Claim \ref{clm:Lipschitz}, the above estimate for $\beta^r_{\bar x^*}$ and the fact that $\|\bar z\|\leq \|\bar x\|$, we have
\begin{equation*}
\begin{split}
\|\Pi_{\bar{x}^*}\big(\delta\bar{\sigma}(\bar{x})\times(\bar{x}^*-\bar{x})\big)^r\|
&\leq\|\delta\bar{\sigma}^{\perp}(\bar{x})\times(\bar{x}^*-\bar{x})^{\perp}\|\cdot \cos \beta^r_{\bar{x}^*}\\
&\leq C\left(\frac{\|\bar x-\bar z\|}{1-\|\bar z\|}+(\frac{1}{\sqrt{1-\|\bar x\|}}-\frac{1}{\sqrt{1-\|\bar z\|}})\right)\cdot C\sqrt{1-\|\bar x\|}\cdot C\sqrt{1-\|\bar z\|}\\
&\leq C\|\bar x-\bar z\|+C(\sqrt{1-\|\bar z\|}-\sqrt{1-\|\bar x\|})\\
&\leq C\|\bar x-\bar y\|\leq C\|\bar x^*-\bar y^*\|~,\\
\|\Pi_{\bar{x}^*}\big(\delta\bar{\sigma}(\bar{x})\times(\bar{x}^*-\bar{x})\big)^{\perp}\|
&\leq\|\delta\bar{\sigma}^r(\bar{x})\times(\bar{x}^*-\bar{x})^{\perp}\|+\|\delta\bar{\sigma}^{\perp}(\bar{x})\times(\bar{x}^*-\bar{x})^r\|\\
&\leq C\left|\log\left(\frac{1-\|\bar x\|}{1-\|\bar z\|}\right)\right|\cdot C\sqrt{1-\|\bar x\|} \\
& + C\left(\frac{\|\bar x-\bar z\|}{1-\|\bar z\|}+(\frac{1}{\sqrt{1-\|\bar x\|}}-\frac{1}{\sqrt{1-\|\bar z\|}})\right)\cdot C(1-\|\bar x\|)\\
&\leq C \|\bar{x}-\bar{z}\|\cdot \big|\log(\|\bar{x}-\bar{z}\|)\big|+C\|\bar x-\bar z\|+C(\sqrt{1-\|\bar z\|}-\sqrt{1-\|\bar x\|})\\
&\leq C_{\alpha}\|\bar x-\bar y\|^{\alpha}\leq C_{\alpha}\|\bar x^*-\bar y^*\|^{\alpha}~,
\end{split}
\end{equation*}
for all $\alpha\in(0,1)$, where $C_{\alpha}>0$ is a constant depending on $\alpha$.

Now we consider the last term on the right-hand side of \eqref{eq:both dual}.
Observe that
\begin{equation}\label{eq:killing-proj-diff}
\begin{split}
&\quad\|\Pi_{\bar{x}^*}\big(\bar{\sigma}(\bar z)\times (\bar x^*-\bar x)\big)-\Pi_{\bar y^*}\big(\bar{\sigma}(\bar z)\times (\bar y^*-\bar y)\big)\|\\
&\leq \|\Pi_{\bar{x}^*}\big(\bar{\sigma}(\bar z)\times ((\bar x^*-\bar x)-(\bar y^*-\bar y))\big)\|
+\|(\Pi_{\bar{x}^*}-\Pi_{\bar{y}^*})\big(\bar{\sigma}(\bar z)\times (\bar y^*-\bar y)\big)\|~.\\
\end{split}
\end{equation}

For the first term on the right-hand side of \eqref{eq:killing-proj-diff}, we consider the radial and lateral components, at $\bar x$, of $\bar{\sigma}(\bar z)\times ((\bar x^*-\bar x)-(\bar y^*-\bar y))$ respectively. By Lemma \ref{lm:pogorelov_radial and lateral}, Claim \ref{clm:Lipschitz} and the above estimate for $\beta^r_{\bar x^*}$, we obtain
\begin{equation*}
\begin{split}
\|\Pi_{\bar{x}^*}\big(\bar{\sigma}(\bar z)\times ((\bar x^*-\bar x)-(\bar y^*-\bar y))\big)^r\|
&\leq \|\bar\sigma^{\perp}(\bar z)\|\cdot\|((\bar x^*-\bar y^*)-(\bar x-\bar y)\big)^{\perp}\|\cdot\cos\beta^r_{\bar x^*}\\
&\leq \frac{C}{\sqrt{1-\|\bar z\|}}\cdot C\|\bar x-\bar y\|\cdot C\sqrt{1-\|\bar x\|}\\
&\leq C\|\bar x-\bar y\|\leq C\|\bar x^*-\bar y^*\|~,\\
\|\Pi_{\bar{x}^*}\big(\bar{\sigma}(\bar z)\times((\bar x^*-\bar x)-(\bar y^*-\bar y))\big)^{\perp}\|
&\leq\|\bar{\sigma}^r(\bar z)\|\cdot\|\big((\bar x^*-\bar y^*)-(\bar x-\bar y)\big)^{\perp}\|+\|\bar{\sigma}^{\perp}(\bar z)\|\cdot\|\big((\bar x^*-\bar y^*)-(\bar x-\bar y)\big)^r\|\\
&\leq C\|\log(1-\|\bar z\|)\|\cdot C\|\bar x-\bar y\|+\frac{C}{\sqrt{1-\|\bar z\|}}\cdot C\|\bar x-\bar y\|\cdot C\sqrt{1-\|\bar x\|}\\
&\leq C_{\alpha}\|\bar x-\bar y\|^{\alpha}\leq C_{\alpha}\|\bar x^*-\bar y^*\|^{\alpha}~,\\
\end{split}
\end{equation*}
for all $\alpha\in (0,1)$, where $C_{\alpha}$ depends on $\alpha$.

   For the last term on the right-hand side of  \eqref{eq:killing-proj-diff}, applying the estimate for $\|\Pi_{\bar x^*}-\Pi_{\bar y^*}\|$ in term (2), Lemma \ref{lm:pogorelov_radial and lateral} and Lemma \ref{lm:estimate-dual points} yield that
   \begin{equation*}
   \begin{split}
   \quad\|(\Pi_{\bar{x}^*}-\Pi_{\bar{y}^*})\big(\bar{\sigma}(\bar z)\times (\bar y^*-\bar y)\big)\|
   &\leq\|\Pi_{\bar{x}^*}-\Pi_{\bar{y}^*}\|\cdot\|\bar\sigma(\bar z)\|\cdot \|\bar y^*-\bar y\|\\
   &\leq C\|\bar x^*-\bar y^*\|\cdot \max\{\frac{C}{\sqrt{1-\|\bar z\|}}, C\|\log(1-\|\bar z\|)\|\} \cdot \sqrt{1-\|\bar y\|}\\
   &\leq C\|\bar x^*-\bar y^*\|~.
   \end{split}
   \end{equation*}

   Combining the above estimates, for each $\alpha\in(0,1)$, there is a constant $C_{\alpha}>0$ (depending on $\alpha$) such that $$\|\Pi_{\bar{x}^*}\big(\bar{\sigma}(\bar{x})\times(\bar{x}^*-\bar{x})\big)-\Pi_{\bar{y}^*}\big(\bar{\sigma}(\bar{y})\times(\bar{y}^*-\bar{y})\big)\|\leq
   C_{\alpha}\|\bar x^*-\bar y^*\|^{\alpha}.$$

  \textbf{(5)} The term $\|\Pi_{\bar{y}^*}\big(\bar{\sigma}(\bar{y})\times(\bar{y}^*-\bar{y})\big)\|$ for $\bar{x}\in(\cup_{i\in\cI}\partial_i\bar{\Omega})\cup(\cup_{j\in\cJ}\partial_j\bar{\Omega})$, $\bar{y}^*\in\cup_{k\in\cK}\partial_k\bar{\Omega}^*$.
  We consider the the radial and lateral components, at $\bar{y}$, of $\bar{\sigma}(\bar{y})\times(\bar{y}^*-\bar{y})$ respectively. Note that $\|\bar{z}\|\leq\|\bar{y}\|$, $\|\bar{x}-\bar{y}\|\leq 1/2<1$, and $\bar x$, $\bar y$ lie in different components of $\partial\bar\Omega\setminus\bar \Lambda$. Applying the estimate \eqref{eq:beta} for $\beta^r_{\bar x^*}$ (with $\bar x^*$ replaced by $\bar y^*$), Lemma \ref{lm:pogorelov_radial and lateral}, Lemma \ref{lm:estimate-dual points} and Claim \ref{clm:Lipschitz}, we obtain
              \begin{equation*}
              \begin{split}
              \|\Pi_{\bar{y}^*}\big(\bar{\sigma}(\bar{y})\times(\bar{y}^*-\bar{y})\big)^r\|
              &=\|\bar{\sigma}^{\perp}(\bar{y})\times(\bar{y}^*-\bar{y})^{\perp}\|\cdot \cos \beta^r_{\bar{y}^*}\\
              &\leq \frac{C}{\sqrt{1-\|y\|}}\cdot C\sqrt{1-\|\bar y\|}\cdot C\sqrt{1-\|\bar y\|}\\
              &\leq C\sqrt{1-\|\bar{y}\|}\leq C\|\bar{x}-\bar{y}\|
              \leq C\|\bar{x}-\bar{y}^*\|~,\\
              \|\Pi_{\bar{y}^*}\big(\bar{\sigma}(\bar{y})\times(\bar{y}^*-\bar{y})\big)^{\perp}\|
              &\leq\|\bar{\sigma}^r(\bar{y})\times(\bar{y}^*-\bar{y})^{\perp}\|+
             \|\bar{\sigma}^{\perp}(\bar{y})\times(\bar{y}^*-\bar{y})^r\|\\
             &\leq C\big|\log\sqrt{1-\|\bar{y}\|}\big|\cdot\sqrt{1-\|\bar{y}\|}
             +\frac{C}{\sqrt{1-\|\bar{y}\|}}\cdot(1-\|\bar{y}\|)\\
              &\leq C_{\alpha}\left(\sqrt{1-\|\bar{y}\|}\right)^{\alpha}\leq C_{\alpha}\|\bar{x}-\bar{y}\|^{\alpha} \leq C_{\alpha}\|\bar{x}-\bar{y}^*\|^{\alpha}~,\\
              \end{split}
              \end{equation*}
              for all $\alpha\in(0,1)$, where $C_{\alpha}$ is a constant depending on $\alpha$. Therefore, for each $\alpha\in(0,1)$, there exists a constant $C_{\alpha}>0$ (depending on $\alpha$) such that
              \begin{equation*}
              \|\Pi_{\bar{y}^*}\big(\bar{\sigma}(\bar{y})\times(\bar{y}^*-\bar{y})\big)\|\leq C_{\alpha}\|\bar{x}-\bar{y}^*\|^{\alpha}~.
              \end{equation*}

              As a consequence, the lemma follows.
\end{proof}

\begin{proof}[\textbf{Proof of Proposition \ref{prop:Holder estimate}}]
   By Lemma \ref{lm:Holder estimate-x close to y}, $\bar{v}$ is $C^{\alpha}$ H\"older continuous on $\partial\bar{\Omega}^*\setminus\bar{\Lambda}$ for all $\alpha\in (0,1)$. In particular, $\bar{v}$ is uniformly continuous on $\partial\bar{\Omega}^*\setminus\bar{\Lambda}$. Therefore, $\bar{v}$ can be continuously extended to the limit set $\bar{\Lambda}$ in a unique way and is then continuous on the whole $\partial\bar{\Omega}^*$. As a consequence, $\bar{v}$ is $C^{\alpha}$ H\"older continuous on $\partial\bar{\Omega}^*$ for all $\alpha\in (0,1)$.
\end{proof}

\subsection{Analytical properties of $\bar{v}$}
     Recall that $\tilde{S}^i_t$ is the lift in $\Omega$ of the equidistant surface $S^i_t$ at distance $t$ from $\partial_iC_M$ and $\bar{S}^i_t=\iota(\tilde{S}^i_t)$. We take $t_0\geq 1$ large enough such that the distance from the center $\bar{O}$ to $\cup_{i\in\cI}\bar{S}^i_{t_0}$ is at least $r_0$ for a constant $r_0\in (0,1)$.

     Let $\tilde{G}_{t_0}:\cup_{i\in\cI}\tilde{S}^i_{t_0}\rightarrow\cup_{i\in\cI}\partial_i\Omega$ denote the hyperbolic Gauss map (see Subsection \ref{subsec:foliation}). The following is an adapted version of Proposition 5.3 in \cite{hmcb}, which can be proved in the same manner.

    \begin{lemma}\label{lm:Minkowski-Lp}
   Let $\delta$ be the distance to $\bar{\Lambda}$ on the surface $\bar{\Sigma}:=(\cup_{i\in\cI}\bar{S}^i_{t_0})\cup(\cup_{j\in\cJ}\partial_j\bar{\Omega})\cup(\cup_{k\in\cK}\partial_k\bar{\Omega})\cup\bar{\Lambda}\subset \R^3$. For any $p>0$, $\log\delta$ is in $L^p$ for the area form of the induced metric on $\bar{\Sigma}\subset\R^3$.
   \end{lemma}

\begin{proposition}\label{prop:Lp bound partially}
$\bar{\nabla}\bar{v}\in L^{p}\big(T(\partial\bar{\Omega}^*)\big)$ for all $p>2$.
\end{proposition}

\begin{proof}
It is known that $\bar{\nabla}\bar{v}\in L^{p}(T(\cup_{j\in\cJ}\partial_j\bar{\Omega}))$ for all $p>2$, see \cite[Corollary 5.4]{hmcb}. Also, one can show that $\bar{\nabla}\bar{v}\in L^{p}\big(T(\cup_{k\in\cK}\partial_k\bar{\Omega}^*)\big)$ for all $p>2$, following the method in \cite[Lemma 5.8]{hmcb}. It suffices to show that $\bar{\nabla}\bar{v}\in L^{p}(T(\cup_{i\in\cI}\partial_i\bar{\Omega}))$ for all $p>2$.

We define $\bar{G}=\iota \circ \tilde{G}_{t_0}\circ\iota^{-1}:\cup_{i\in\cI}\bar{S}^i_{t_0}\rightarrow \cup_{i\in\cI}\partial_i\bar{\Omega}$. For $\bar{x}\in\cup_{i\in\cI}\partial_i\bar{\Omega}$, the lateral component of $\bar{u}(\bar{x})$ is exactly its tangential component $\bar{v}(\bar{x})$ and $\bar{x}$ uniquely determines the point $\bar{\eta}_{\bar{x}}:=\bar{G}^{-1}(\bar{x})\in \bar{S}^i_{t_0}$. Then $\bar{v}(\bar{x})=\Pi_{\bar{x}}(\bar{u}(\bar{x}))$ and by the extension to $\cup_{i\in\cI}\partial_i\bar{\Omega}$ of $\bar{u}$ defined in Section \ref{subsec:deformation field},
\begin{equation*}
                 \bar{u}(\bar{x})
                =\bar{\tau}(\bar{x})=\bar{\kappa}_{\bar{x}}(\bar{x})
                =\bar{\kappa}_{\bar{\eta}_{\bar{x}}}(\bar{x})
                =\bar{\sigma}(\bar{\eta}_{\bar{x}})\times(\bar{x}-\bar{\eta}_{\bar{x}})
                +\bar{\tau}(\bar{\eta}_{\bar{x}})~,
                \end{equation*}
                where $\bar{\kappa}\in\Gamma(\bar{E})$ is the constant section along the geodesic $[\bar{x}, \bar{\eta}_{\bar{x}}]$ with $\bar{\kappa}_{\bar{\eta}_{\bar{x}}}(\bar{\eta}_{\bar{x}})=\bar{\tau}(\bar{\eta}_{\bar{x}})$.

                We claim that for any $\bar{X}\in T_{\bar{x}}(\cup_{i\in\cI}\partial_i\bar{\Omega})$ with $\|\bar{X}\|\leq 1$, we have that $\|\bar{\nabla}_{\bar{X}}\bar{v}\|\leq C\big|\log\sqrt{1-\|\bar{\eta}_{\bar{x}}\|}\big|$ for 
                $\bar{x}$ close to $\bar{\Lambda}$. Indeed,
        \begin{equation*}
        \begin{split}
        \|\bar{\nabla}_{\bar{X}}\bar{v}\|
        &=\|\bar{\nabla}_{\bar{X}}(\Pi\bar{\tau})\|
        =\|(\bar{\nabla}_{\bar{X}}\Pi)\bar{\tau}+\Pi_{\bar{x}}\big(\bar{\nabla}_{\bar{X}}\bar{\tau}\big)\|\\
        &\leq \|\bar{\nabla}_{\bar{X}}\Pi\|\cdot\|\bar{\sigma}(\bar{\eta}_{\bar{x}})\times(\bar{x}-\bar{\eta}_{\bar{x}})
                +\bar{\tau}(\bar{\eta}_{\bar{x}})\|
        +\|\Pi_{\bar{x}}\big(\bar{\nabla}_{\bar{X}}(\bar{\sigma}(\bar{\eta}_{\bar{x}})\times(\bar{x}-\bar{\eta}_{\bar{x}})
                +\bar{\tau}(\bar{\eta}_{\bar{x}}))\big)\|~. \\
        \end{split}
        \end{equation*}
   Similarly as discussed in Lemma \ref{lm:Holder estimate-x close to y}, it suffices to consider the behavior of $\|\bar{\nabla}_{\bar{X}}\bar{v}\|$ near the limit set $\bar{\Lambda}$.
   \begin{Step-nabla}\label{step:no.1}
	We first estimate the first term on the right-hand side of the above inequality.
    \end{Step-nabla}
        Since $\cup_{i\in\cI}\partial_i\bar{\Omega}$ is contained in $\partial \DD^3$, the principal curvatures are both $1$ at each point of $\cup_{i\in\cI}\partial_i\bar{\Omega}$. This implies that $\|\bar{\nabla}_{\bar{X}}\Pi\|\leq C$ for a constant $C>0$.
        Note that $\cup_{i\in\cI}\bar{S}^i_{t_0}$ is tangent to $\partial\DD^3$ along $\bar{\Lambda}$ and lies outside of the convex hull of $\bar{\Lambda}$ (which contains $\bar{O}$ in its interior). Besides, the segment $[\bar{\eta}_{\bar{x}},\bar{x}]$ is the image under $\iota$ of the geodesic in $\HH^3$ orthogonal to the equidistant surface $\cup_{i\in\cI}\tilde{S}^i_{t_0}$. We consider the triangle $\bar{O}\bar{x}\bar{\eta}_{\bar{x}}$. Observe that the angle $\angle\bar{O}\bar{\eta}_{\bar{x}}\bar{x}\in [\pi/2,\pi]$ (as shown in Figure \ref{fig:Poincare-Klein} for instance).
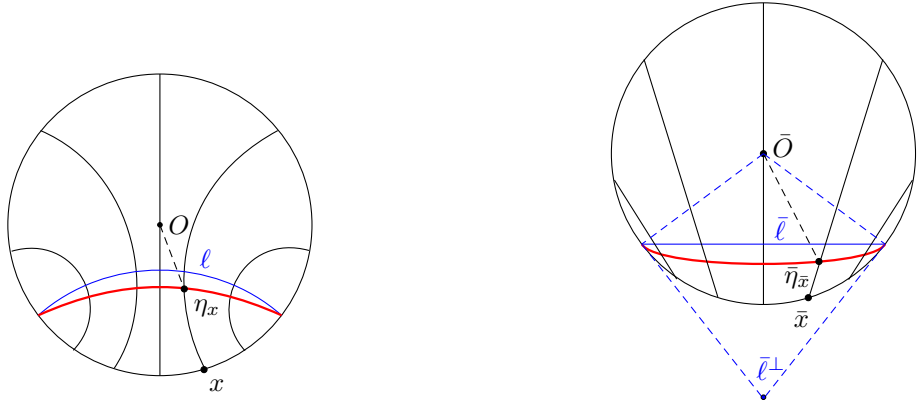
\begin{figure}
\begin{subfigure}[b]{0.46\textwidth}
\centering
\begin{tikzpicture}[scale=1]
\draw  (0,0) ellipse (2 and 2);
\draw [fill=black] (0,0) ellipse (0.03 and 0.03);
\node at (0.25,0) {$O$};
\draw (0,2) .. controls (0,-2) and (0,-2) .. (0,-2);
\draw (-1.57,1.25) .. controls (-0.3,0.7) and (0,-1) .. (-0.6,-1.9);
\draw (1.55,1.25) .. controls (0,0.4) and (0.2,-1.2) .. (0.6,-1.92);
\draw (-1.95,-0.35) .. controls (-1.1,-0.1) and (-0.6,-1.1) .. (-1.14,-1.664);
\draw (1.96,-0.34) .. controls (1.1,-0.1) and (0.6,-1.1) .. (1.12,-1.67);
\draw[blue] (-1.6,-1.2) .. controls (-0.7,-0.4) and (0.7,-0.4) .. (1.6,-1.2);
\node [blue] at (0.6,-0.45) {$\ell$};
\draw [line width=0.3mm] [red] (-1.6,-1.2) .. controls (-0.5,-0.7) and (0.5,-0.7) .. (1.6,-1.2);
\draw [fill=black] (0.32,-0.85) ellipse (0.04 and 0.04);
\node at (0.62,-1.1) {$\eta_x$};
\draw [fill=black] (0.58,-1.92) ellipse (0.04 and 0.04);
\node at (0.75,-2.13) {$x$};
\draw [densely dashed](0,0) .. controls (0.32,-0.85) and (0.32,-0.85) .. (0.32,-0.85);
\end{tikzpicture}
\end{subfigure}
\begin{subfigure}[b]{0.46\textwidth}
\centering
\begin{tikzpicture}
\draw  (0,0) ellipse (2 and 2);
\draw [fill=black] (0,0) ellipse (0.04 and 0.04);
\node at (0.26,0.1) {$\bar{O}$};
\draw (0,2) .. controls (0,-2) and (0,-2) .. (0,-2);
\draw[blue] (-1.6,-1.2) .. controls (1.6,-1.2) and (1.6,-1.2) .. (1.6,-1.2);
\node [blue] at (0.22,-1) {$\bar{\ell}$};
\draw [line width=0.3mm] [red](-1.6,-1.2) .. controls (-1.4,-1.55) and (1.4,-1.55) .. (1.6,-1.2);
\draw (-1.57,1.25) .. controls (-0.6,-1.9) and (-0.6,-1.9) .. (-0.6,-1.9);
\draw (-1.96,-0.35) .. controls (-1.14,-1.664) and (-1.14,-1.664) .. (-1.14,-1.664);
\draw (1.96,-0.35) .. controls (1.12,-1.67) and (1.12,-1.67) .. (1.12,-1.67);
\draw (1.57,1.25) .. controls (0.6,-1.9) and (0.6,-1.9) .. (0.6,-1.9);
\draw [fill=black] (0.73,-1.43) ellipse (0.04 and 0.04);
\node at (0.45,-1.65) {$\bar{\eta}_{\bar x}$};
\draw [fill=black] (0.59,-1.91) ellipse (0.04 and 0.04);
\node at (0.5,-2.2) {$\bar{x}$};
\draw [densely dashed](0,0) .. controls (0.73,-1.43) and (0.73,-1.43) .. (0.73,-1.43);
\draw [densely dashed,blue](0,0) .. controls (-1.6,-1.2) and (-1.6,-1.2) .. (-1.6,-1.2);
\draw [densely dashed,blue](0,0) .. controls (1.6,-1.2) and (1.6,-1.2) .. (1.6,-1.2);
\draw [densely dashed,blue](1.6,-1.2) .. controls (0,-3.25) and (0,-3.25) .. (0,-3.25);
\draw [densely dashed,blue](-1.6,-1.2) .. controls (0,-3.25) and (0,-3.25) .. (0,-3.25);
\draw [fill=blue] (0,-3.23) ellipse (0.03 and 0.03);
\node [blue] at (0.1,-2.85) {$\bar{\ell}^{\perp}$};
\end{tikzpicture}
\end{subfigure}
\caption{\small{The geodesics and angles in Poincar\'{e} model (left) and Klein model (right) of $\HH^2$, where the bold line represents an equidistant curve from a geodesic along the orthogonal direction towards $\partial_{\infty}\HH^2$ and the point $\bar{\ell}^{\perp}$ is the dual of the geodesic $\bar{\ell}$.}}
\label{fig:Poincare-Klein}
\end{figure}

 Then by the cosine formula $$\|\bar{x}-\bar{\eta}_{\bar{x}}\|^2+\|\bar{\eta}_{\bar{x}}\|^2-2\|\bar{x}
        -\bar{\eta}_{\bar{x}}\|\cdot\|\bar{\eta}_{\bar{x}}\|\cos \angle\bar{O}\bar{\eta}_{\bar{x}}\bar{x} =1~,$$
  it follows that
  $$\|\bar{x}-\bar{\eta}_{\bar{x}}\|\leq \sqrt{1-\|\bar{\eta}_{\bar{x}}\|^2}
        =\sqrt{(1+\|\bar{\eta}_{\bar{x}}\|)(1-\|\bar{\eta}_{\bar{x}}\|)}
        \leq 2\sqrt{1-\|\bar{\eta}_{\bar{x}}\|}~.$$
        So we obtain
        \begin{equation}\label{ineq:x-eta}
        \|\bar{x}-\bar{\eta}_{\bar{x}}\|\leq 2\sqrt{1-\|\bar{\eta}_{\bar{x}}\|}~.
        \end{equation}
        Combined with Lemma \ref{lm:pogorelov_radial and lateral},
        we have
        \begin{equation*}
        \begin{split}
        \|\bar{\nabla}_{\bar{X}}\Pi\|\cdot\|\bar{\sigma}(\bar{\eta}_{\bar{x}})\times(\bar{x}-\bar{\eta}_{\bar{x}})
                +\bar{\tau}(\bar{\eta}_{\bar{x}})\|
                &\leq C\big(\|\bar{\sigma}(\bar{\eta}_{\bar{x}})\|\cdot\|\bar{x}-\bar{\eta}_{\bar{x}}\|+\|\bar{\tau}(\bar{\eta}_{\bar{x}})\|\big)\\
                &\leq C\max\left\{|\log (1-\|\bar{\eta}_{\bar{x}}\|)|~, \frac{1}{\sqrt{1-\|\bar{\eta}_{\bar{x}}\|}}\right\}\cdot\sqrt{1-\|\bar{\eta}_{\bar{x}}\|}+ C |\log(1-\|\bar{\eta}_{\bar{x}}\|)|\\
                &\leq C |\log\sqrt{1-\|\bar{\eta}_{\bar{x}}\|}|~.
        \end{split}
        \end{equation*}
    \begin{Step-nabla}\label{step:no.1b}
	We now consider the second term.
    \end{Step-nabla}
   Note that $\bar{\eta}_{\bar{x}}=\bar{G}^{-1}(\bar{x})$. Combined with the formula \eqref{eq:euclidean-one-form}, we have
    \begin{equation}\label{eq:nabla}
    \begin{split}
    \bar{\nabla}_{\bar{X}} \big(\bar{\tau}(\bar{\eta}_{\bar{x}})\big)&=\bar{\nabla}_{\bar{X}} \big(\bar{\tau}(\bar{G}^{-1}(\bar{x}))\big)
    =\bar{\nabla}_{\bar{G}^*\bar{X}} \big(\bar{\tau}(\bar{\eta}_{\bar{x}})\big)=
    \bar{\sigma}(\bar{\eta}_{\bar{x}})\times \bar{G}^*\bar{X}+\bar{\omega}_{\bar{\tau}}(\bar{G}^*\bar{X})~,\\
   \bar{\nabla}_{\bar{X}} \big(\bar{\sigma}(\bar{\eta}_{\bar{x}})\big)&=\bar{\nabla}_{\bar{X}} \big(\bar{\sigma}(\bar{G}^{-1}(\bar{x}))\big)=\bar{\nabla}_{\bar{G}^*\bar{X}} \big(\bar{\sigma}(\bar{\eta}_{\bar{x}})\big)=\bar{\omega}_{\bar{\sigma}}(\bar{G}^*\bar{X})~.
    \end{split}
    \end{equation}
    We now view $\bar{\nabla}_{\bar{X}}\big(\bar{\sigma}(\bar{\eta}_{\bar{x}})\times(\bar{x}-\bar{\eta}_{\bar{x}})\big)$ and $\bar{\nabla}_{\bar{X}}(\bar{\tau}(\bar{\eta}_{\bar{x}}))$ as tangent vectors in $T_{\bar{\eta}_{\bar{x}}}\bar{\Omega}$. Combined with \eqref{eq:nabla}, the lateral and radial components of $\bar{\nabla}_{\bar{X}}\big(\bar{\sigma}(\bar{\eta}_{\bar{x}})\times(\bar{x}-\bar{\eta}_{\bar{x}})\big)$, at $\bar{\eta}_{\bar{x}}$, are respectively:
    \begin{equation}\label{eq:part1}
    \begin{split}
   &\Big(\bar{\nabla}_{\bar{X}}\big(\bar{\sigma}(\bar{\eta}_{\bar{x}})\times(\bar{x}-\bar{\eta}_{\bar{x}})\big)\Big)^{\perp}
    =\big(\bar{\omega}_{\bar{\sigma}}(\bar{G}^*\bar{X})\times(\bar{x}-\bar{\eta}_{\bar{x}})
    +\bar{\sigma}(\bar{\eta}_{\bar{x}})\times(\bar{X}-\bar{G}^*\bar{X})\big)^{\perp}\\
    &=\bar{\omega}^r_{\bar{\sigma}}(\bar{G}^*\bar{X})\times(\bar{x}-\bar{\eta}_{\bar{x}})^{\perp}
    +\bar{\omega}^{\perp}_{\bar{\sigma}}(\bar{G}^*\bar{X})\times(\bar{x}-\bar{\eta}_{\bar{x}})^r
    +\bar{\sigma}^r(\bar{\eta}_{\bar{x}})\times(\bar{X}-\bar{G}^*\bar{X})^{\perp}
     +\bar{\sigma}^{\perp}(\bar{\eta}_{\bar{x}})\times(\bar{X}-\bar{G}^*\bar{X})^r~,\\
   &\Big(\bar{\nabla}_{\bar{X}}\big(\bar{\sigma}(\bar{\eta}_{\bar{x}})\times(\bar{x}-\bar{\eta}_{\bar{x}})\big)\Big)^r
  =\big(\bar{\omega}_{\bar{\sigma}}(\bar{G}^*\bar{X})\times(\bar{x}-\bar{\eta}_{\bar{x}})
    +\bar{\sigma}(\bar{\eta}_{\bar{x}})\times(\bar{X}-\bar{G}^*\bar{X})\big)^r\\
    & =\bar{\omega}^{\perp}_{\bar{\sigma}}(\bar{G}^*\bar{X})\times(\bar{x}-\bar{\eta}_{\bar{x}})^{\perp}+
    \bar{\sigma}^{\perp}(\bar{\eta}_{\bar{x}})\times(\bar{X}-\bar{G}^*\bar{X})^{\perp}~.\\
    \end{split}
    \end{equation}
     The lateral and radial components of $\bar{\nabla}_{\bar{X}}\big(\bar{\tau}(\bar{\eta}_{\bar{x}})\big)$, at $\bar{\eta}_{\bar{x}}$, are respectively:
    \begin{equation}\label{eq:part2}
    \begin{split}
    \Big(\bar{\nabla}_{\bar{X}}\big(\bar{\tau}(\bar{\eta}_{\bar{x}})\big)\Big)^{\perp}
    &=\bar{\sigma}^r(\bar{\eta}_{\bar{x}})\times(\bar{G}^*\bar{X})^{\perp}+\bar{\sigma}^{\perp}(\bar{\eta}_{\bar{x}})\times(\bar{G}^*\bar{X})^r
    +\bar{\omega}^{\perp}_{\bar{\tau}}(\bar{G}^*\bar{X})~,\\
     \Big(\bar{\nabla}_{\bar{X}}\big(\bar{\tau}(\bar{\eta}_{\bar{x}})\big)\Big)^r
    &=\bar{\sigma}^{\perp}(\bar{\eta}_{\bar{x}})\times(\bar{G}^*\bar{X})^{\perp}+\bar{\omega}^r_{\bar{\tau}}(\bar{G}^*\bar{X})~.\\
    \end{split}
    \end{equation}
    Here $\bar{G}^*(\bar{X})$ and $\bar{X}-\bar{G}^*(\bar{X})$ are viewed as  tangent vectors in $T_{\bar{\eta}_{\bar{x}}}\bar{\Omega}$, and $\bar{G}^*\bar{X}=(\bar{G}^*\bar{X})^{r}+(\bar{G}^*\bar{X})^{\perp}$.


    We will need the fact that 
    \begin{equation}
      \label{eq:gauss}
      \|\bar{G}^*\bar{X}\|\leq C \|\bar X\|~,
    \end{equation}
    for all $\bar{X}\in T_{\bar{x}}(\cup_{i\in\cI}\partial_i\bar{\Omega})$ and $\bar x\in\cup_{i\in\cI}\partial_i\bar{\Omega}$.

    Since $\bar{G}$ is the hyperbolic Gauss map, but the norm considered is Euclidean, so we outline with some care the argument to prove this inequality. To simplify the argument, we assume that the tangent plane to $\cup_{i\in\cI}\bar{S}^i_{t_0}$ at $\bar{\eta}_{\bar x}=\bar{G}^{-1}(\bar{x})$ is the horizontal plane $P_r$ of equation $x_3=r$ for some $r\in (0,1)$, see Figure \ref{fig:gauss} (note that the coordinates here are chosen such that the projective model $\HH^3$ is the open unit ball $\DD^3:=\{(x_1,x_2,x_3):x_1^2+x_2^2+x_3^2<1\}$). We first notice that, since the principal curvatures of $\iota^{-1}(\cup_{i\in\cI}\bar{S}^i_{t_0})=\cup_{i\in\cI}\tilde{S}^i_{t_0}$ are uniformly bounded (see Remark \ref{rk:comparsion}), the norm of the pull-back by $\bar G$ of a vector on $\partial\DD^3$ is within a bounded multiplicative constant of the norm of the pull-back of this vector by the Gauss map $\bar G_r$ of $P_r$, which is defined as $\bar{G}_r=\iota \circ G_{r}\circ\iota^{-1}:P_r\rightarrow \partial\DD^3$, where $G_r:\iota^{-1}(P_r)\rightarrow\partial_{\infty}\HH^3$ denotes the hyperbolic Gauss map. So it is sufficient to prove the inequality with $\bar G$ replaced by $\bar G_r$. So we now focus on estimating this Gauss map $\bar{G}_r$ of $P_r$.

    The image in the projective model of the point
    $P^*_r$ dual to $P_r$ has equation $(0,0, 1/r)$. The Gauss map $\bar G_r$ sends any point $\bar y\in P_r$ to the intersection with $\partial \DD^3$ of the line segment from $\bar y$ to $P^*_r$. We denote by $\alpha$ the angle between $\bar{G}_r^*\bar{X}$
     and $[\bar{\eta}_{\bar x}, P_r^*]$ and by $\beta$ the angle between $\bar X$ and $[\bar x, P_r^*]$. Finally let $x_3$ be the vertical coordinate of 
     $\bar{x}$. An elementary geometric argument then shows that
    $$ \| \bar{G}_r^*\bar{X}\| =\frac{\sin\beta}{\sin\alpha} \frac{1-r^2}{1-rx_3}\| \bar{X}\|~.$$
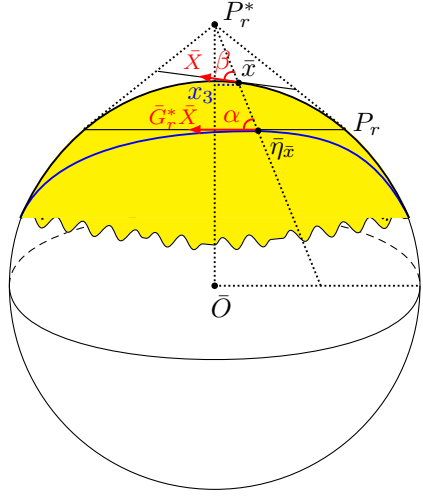
\begin{figure}
\begin{tikzpicture}[scale=0.9]
\draw  (0,0) ellipse (3 and 3);
\draw[densely dashed] (-3,0) .. controls (-3,1.45) and (3,1.45) .. (3,0);
\draw (-3,0) .. controls (-3,-1.45) and (3,-1.45) .. (3,0);
\draw [snake=snake][white][fill=yellow,opacity=0.15] (-2.83,1.0) --(-1,0.7)--(0,0.6)-- (1,0.7)--(2.83,1.0);
\draw [snake=snake](-2.83,1.0) --(-1,0.7)--(0,0.6)-- (1,0.7)--(2.83,1.0);
\draw [densely dotted, thick, snake=snake] (-2.84,1.0) --(-2.3,1.15)--(-1.2,1.18)-- (-0.8,1.15)--(0,1.15)-- (0.8,1.18)--(1.2,1.15)--(2.2,1.13)--(2.84,1.0);
\draw[line width=0.25mm][white][fill=yellow,opacity=0.15](-2.83,1.0) .. controls (-1.88,3.68) and (1.88,3.68) .. (2.83,1.0);
\draw [fill=black] (0,0) ellipse (0.04 and 0.04);
\node at (0.1,-0.3) {$\bar O$};
\draw[line width=0.25mm, densely dotted](0,0)--(0,3.85);
\draw[line width=0.25mm][blue](-2.84,1.0) .. controls (-2.25,2.39) and (0.4,2.26) .. (0.65,2.28)..controls (1.55,2.235) and (2.3,2.1) .. (2.84,1.0);
\draw[line width=0.25mm](-2.84,1.0) .. controls (-1.68,3.685) and (1.68,3.685) .. (2.84,1.0);
\node at (2.23,2.38) {$P_r$};
\draw[line width=0.15mm](-0.83,3.16)--(1.2,2.89);
\node at (0.5,3.2) {$\bar{x}$};
\draw[line width=0.27mm][blue][densely dotted](0.35,2.96)--(0,2.96);
\node at (-0.2,2.81) {\color{blue}{$x_3$}};
\node at (0.12,3.29) {\color{red}$\beta$};
\node at (-0.31,3.33) {\footnotesize{\color{red}$\bar X$}};
\draw [line width=0.25mm][red](0.3,3.14) arc (65:178:0.11);
\draw[-latex][red][line width=0.25mm](0.35,3) -- (-0.25,3.09);
\draw[line width=0.15mm][black](-1.9,2.3)--(1.9,2.3);
\node at (1,2.02) {$\bar{\eta}_{\bar x}$};
\draw [line width=0.25mm][red](0.58,2.45) arc (80:185:0.13);
\node at (0.25,2.5) {\color{red}$\alpha$};
\node at (-0.6,2.48) {\footnotesize{\color{red}$\bar{G}^*_r\bar X$}};
\draw[-latex][red][line width=0.25mm](0.65,2.3) -- (-0.4,2.3);
\draw[line width=0.25mm][densely dotted](-1.91,2.31)--(0,3.85);
\draw[line width=0.25mm][densely dotted](1.91,2.31)--(0,3.85);
\draw[line width=0.25mm][densely dotted](0,0)--(3,0);
\draw[line width=0.25mm][densely dotted](1.55,0)--(0,3.85);
\draw [fill=black] (0,3.85) ellipse (0.04 and 0.04);
\draw [fill=black](0.35,2.99) ellipse (0.04 and 0.04);
\draw [fill=black] (0.635,2.285) ellipse (0.04 and 0.04);
\node at (0.36,4) {$P^*_r$};
\end{tikzpicture}
\caption{Estimate on the Gauss map $\bar{G}_r$ of $P_r$ (the shaded region is a connected component of $\cup_{i\in\cI}\partial_i\bar{\Omega}$, and $P_r$ is the tangent plane to $\cup_{i\in\cI}\bar{S}^i_{t_0}$ at $\bar{\eta}_{\bar x}$).}
\label{fig:gauss}
\end{figure}
   However it is clear that
   $$ 1\leq \frac{1-r^2}{1-rx_3}\leq 1+r \leq 2~,$$
   while $\sin\beta\leq \sin\alpha$. This concludes the proof of \eqref{eq:gauss}.

    Combining this estimate with Remark \ref{rk:equidistant-angle}, we have
    \begin{equation}\label{eq:GX}
    \begin{split}
    \|(\bar{G}^*\bar{X})^{\perp}\|&\leq\|\bar{G}^*\bar{X}\|\leq C~,\\
    \|(\bar{G}^*\bar{X})^{r}\|&=\|\bar{G}^*\bar{X}\|\sin \theta_{\bar{\eta}_{\bar{x}}}\leq  C\sqrt{1-\|\bar{\eta}_{\bar{x}}\|}~,
    \end{split}
    \end{equation}
     for all $\|\bar{X}\|\leq 1$, where $\theta_{\bar{\eta}_{\bar{x}}}$ is the angle at $\bar{\eta}_{\bar{x}}$ between the outward-pointing normal vector to $\cup_{i\in\cI}\bar{S}^i_{t_0}$ at $\bar{\eta}_{\bar{x}}$ and the radial direction along $\bar{O}\bar{\eta}_{\bar{x}}$ at $\bar{\eta}_{\bar{x}}$. Let $\alpha_{\bar{\eta}_{\bar{x}}}:=\angle{\bar{x}\bar{O}\bar{\eta}_{\bar{x}}}$. 
    Note that $\|\bar{\eta}_{\bar{x}}\|\in [r_0, 1)$ by assumption.
    \begin{equation*}
    \sin\alpha_{\bar{\eta}_{\bar{x}}}
    =\sqrt{1-\cos^2\alpha_{\bar{\eta}_{\bar{x}}}}
    =\sqrt{1-\big(\frac{1+\|\bar{\eta}_{\bar{x}}\|^2-\|\bar{x}-\bar{\eta}_{\bar{x}}\|^2}{2\|\bar{\eta}_{\bar{x}}\|}\big)^2}
    \leq \frac{\|\bar{x}-\bar{\eta}_{\bar{x}}\|\sqrt{2(1+\|\bar{\eta}_{\bar{x}}\|^2)}}{2\|\bar{\eta}_{\bar{x}}\|}
    \leq C\|\bar{x}-\bar{\eta}_{\bar{x}}\|
    \leq C\sqrt{1-\|\bar{\eta}_{\bar{x}}\|}~.
    \end{equation*}
    The last inequality follows from \eqref{ineq:x-eta}. Combined with the assumption that $\|\bar{X}\|\leq 1$ and \eqref{eq:GX}, we have
    \begin{equation*}
    \begin{split}
    \|(\bar{X}-\bar{G}^*\bar{X})^{\perp}\|&\leq\|\bar{X}^{\perp}\|+\|(\bar{G}^*\bar{X})^{\perp}\|\leq C~,\\
    \|(\bar{X}-\bar{G}^*\bar{X})^{r}\|
    &\leq\|\bar{X}^r\|+\|(\bar{G}^*\bar{X})^r\|\\
    &
    \leq \|\bar{X}\|\cdot\sin\alpha_{\bar{\eta}_{\bar{x}}}+\|(\bar{G}^*\bar{X})^r\|\\
    &\leq  C\sqrt{1-\|\bar{\eta}_{\bar{x}}\|}~.\\
    \end{split}
    \end{equation*}
   By \eqref{ineq:x-eta}, we have
    \begin{equation*}
    \|(\bar{x}-\bar{\eta}_{\bar{x}})^{\perp}\|
    \leq\|\bar{x}-\bar{\eta}_{\bar{x}}\|\leq 2\sqrt{1-\|\bar{\eta}_{\bar{x}}\|}~.
    \end{equation*}
    On the other hand,
    \begin{equation*}
    \begin{split}
     \|(\bar{x}-\bar{\eta}_{\bar{x}})^r\|
     &=\|\bar{x}-\bar{\eta}_{\bar{x}}\|\cdot\cos(\pi-\angle{\bar{O}\bar{\eta}_{\bar{x}}\bar{x}})\\
     &=\|\bar{x}-\bar{\eta}_{\bar{x}}\|\frac{1-\|\bar{x}-\bar{\eta}_{\bar{x}}\|^2-{\|\bar{\eta}_{\bar{x}}\|}^2}{2\|\bar{x}-\bar{\eta}_{\bar{x}}\|\cdot\|\bar{\eta}_{\bar{x}}\|}\\
     &\leq 1-\|\bar{\eta}_{\bar{x}}\|~.
    \end{split}
    \end{equation*}
     The last inequality follows from applying the triangle inequality $\|\bar{x}-\bar{\eta}_{\bar{x}}\|\geq \|\bar{x}\|-\|\bar{\eta}_{\bar{x}}\|= 1-\|\bar{\eta}_{\bar{x}}\|$ into the above numerator.

     Let $P_{\bar{x}}$ denote the plane tangent to $\cup_{i\in\cI}\partial_i\bar{\Omega}$ at $\bar{x}$. The projection angle, say $\beta^r_{\bar{\eta}_{\bar{x}}}$, to $P_{\bar{x}}$ of a radial line through $\bar{\eta}_{\bar{x}}$ is $\pi/2-\alpha_{\bar{\eta}_{\bar{x}}}$, while the projection angle, say $\beta^{\perp}_{\bar{\eta}_{\bar{x}}}$, to $P_{\bar{x}}$ of a lateral vector at $\bar{\eta}_{\bar{x}}$ is in $[0,\alpha_{\bar{\eta}_{\bar{x}}}]$.

    Substituting the decomposition $\bar{G}^*\bar{X}=(\bar{G}^*\bar{X})^{r}+(\bar{G}^*\bar{X})^{\perp}$ into the $\bar{\omega}^r_{\bar{\sigma}}(\cdot)$, $\bar{\omega}^{\perp}_{\bar{\sigma}}(\cdot)$, $\bar{\omega}^r_{\bar{\tau}}(\cdot)$ and $\bar{\omega}^{\perp}_{\bar{\tau}}(\cdot)$ in \eqref{eq:part1} and \eqref{eq:part2}, and using Lemma \ref{lm:estimate-one-form}, Lemma \ref{lm:pogorelov_radial and lateral}, \eqref{ineq:x-eta} and the above estimates, we have
    \begin{equation*}
    \begin{split}
    \|\Pi_{\bar{x}}\big(\big(\bar{\nabla}_{\bar{X}}\big(\bar{\sigma}(\bar{\eta}_{\bar{x}})\times(\bar{x}-\bar{\eta}_{\bar{x}})\big)^{\perp}\big)\|
    &\leq\|\big(\bar{\nabla}_{\bar{X}}\big(\bar{\sigma}(\bar{\eta}_{\bar{x}})\times(\bar{x}-\bar{\eta}_{\bar{x}})\big)^{\perp}\|
    \leq C\big|\log\sqrt{1-\|\bar{\eta}_{\bar{x}}\|}\big|~, \\
    \|\Pi_{\bar{x}}\big(\big(\bar{\nabla}_{\bar{X}}\big(\bar{\sigma}(\bar{\eta}_{\bar{x}})\times(\bar{x}-\bar{\eta}_{\bar{x}})\big)^r\big)\|
    &=\|\big(\bar{\nabla}_{\bar{X}}\big(\bar{\sigma}(\bar{\eta}_{\bar{x}})\times(\bar{x}-\bar{\eta}_{\bar{x}})\big)^r\|\cdot \cos\beta^r_{\bar{\eta}_{\bar{x}}}\\
    &=\|\big(\bar{\nabla}_{\bar{X}}\big(\bar{\sigma}(\bar{\eta}_{\bar{x}})\times(\bar{x}-\bar{\eta}_{\bar{x}})\big)^r\|\cdot \sin\alpha_{\bar{\eta}_{\bar{x}}}\leq C~.\\
    \|\Pi_{\bar{x}}\big(\big(\bar{\nabla}_{\bar{X}}(\bar{\tau}(\bar{\eta}_{\bar{x}}))\big)^{\perp}\big)\|
    &\leq \|\big(\bar{\nabla}_{\bar{X}}(\bar{\tau}(\bar{\eta}_{\bar{x}}))\big)^{\perp}\|\leq C\big|\log\sqrt{1-\|\bar{\eta}_{\bar{x}}\|}\big|~, \\
    \|\Pi_{\bar{x}}\big(\big(\bar{\nabla}_{\bar{X}}(\bar{\tau}(\bar{\eta}_{\bar{x}}))\big)^r\big)\|
    &= \|\big(\bar{\nabla}_{\bar{X}}(\bar{\tau}(\bar{\eta}_{\bar{x}}))\big)^r\|\cdot \cos \beta^r_{\bar{\eta}_{\bar{x}}}= \|\big(\bar{\nabla}_{\bar{X}}(\bar{\tau}(\bar{\eta}_{\bar{x}}))\big)^r\|\cdot \sin\alpha_{\bar{\eta}_{\bar{x}}}\leq C~.\\
     \end{split}
     \end{equation*}
     Combining Steps 1-2, we obtain that
     \begin{equation*}
     \|\bar{\nabla}_{\bar{X}}\bar{v}\|\leq C\big|\log\sqrt{1-\|\bar{\eta}_{\bar{x}}\|}\big|~.
     \end{equation*}
     The claim follows. Let $\delta_{\cup_{i\in\cI}\bar{S}^i_{t_0}}(\bar{x},\bar{\Lambda})$ denote the distance from the point $\bar{x}\in\bar{\Omega}$ to the limit set $\bar{\Lambda}$ along the surface $\cup_{i\in\cI}\bar{S}^i_{t_0}$. By Remark \ref{rk:equidistant-angle},
     \begin{equation*}
     \|\bar{\nabla}_{\bar{X}}\bar{v}\|\leq C\big|\log\delta_{\cup_{i\in\cI}\bar{S}^i_{t_0}}(\bar{\eta}_{\bar{x}},\bar{\Lambda})\big|~, \end{equation*}
     for all $\bar{X}\in T_{\bar{x}}(\cup_{i\in\cI}\partial_i\bar{\Omega})$ with $\|\bar{X}\|\leq 1$. Combined with Lemma \ref{lm:Minkowski-Lp},  we have $\bar{\nabla}\bar{v}\in L^{p}(T(\cup_{i\in\cI}\partial_i\bar{\Omega}))$.
      \end{proof}

\begin{proof}[\textbf{Proof of Proposition \ref{lm:local rigidity-mfld}}]
  We consider the representative one-form $\omega$ associated to the infinitesimal deformation $\dot g$ (see Lemma \ref{lm:omega}) and the corresponding deformation vector field $\bar{u}$ on $\partial\bar{\Omega}^*\setminus\bar{\Lambda}$ under the infinitesimal Pogorelov map. By Lemma \ref{prop:boundary regularity}, $\partial\bar{\Omega}^*$ is in $D_{2,\infty}$. By Proposition \ref{prop:solution}, the tangential component $\bar{v}$ of $\bar{u}$ is a solution of Equation \eqref{eq:E}.

  It follows from Proposition \ref{prop:Holder estimate} that $\bar{v}$ is $C^{\alpha}$ on $\partial\bar{\Omega}^*$ for all $\alpha\in(0,1)$. Combined with Remarks \ref{rk:Hausdorff dim} and \ref{rk:regularity-boundary}, Proposition \ref{prop:Lp bound partially} and Lemma \ref{lm:D(1,p)-condition}, $\bar{v}$ is in the class of $D_{1,p}$. It then follows from Lemma \ref{lm:Vekua} that $\bar{v}$ is a trivial solution, which means that $\bar{u}$ is a trivial infinitesimal deformation. 
  After adding a global Killing vector field, we can therefore suppose that 
  $\bar u=0$ on $\partial\bar{\Omega}^*$.

  Since $\bar u$ is the image of $u$ by the infinitesimal Pogorelov map, $u=0$ on $\partial\Omega^*$ (see Lemma \ref{lm:Pogorelov isometric}). $u$ is therefore zero on each of the $\partial_i\Omega, i\in \cI$, on all the $\partial_j\Omega, j\in \cJ$, and by duality, also on each of the $\partial_k\Omega, k\in \cK$.
  Since the deformation vector field $u$ is zero on each boundary component, it implies that $\dot g$ does not change the holonomy representation of $(M,g)$ (at first order), and moreover the deformation vector field is zero on each boundary surface of $N\subset M$.
\end{proof}

%% file: short-weyl4.tex
\section{Convex co-compact manifolds}
\label{sec:convex-co compact}

\subsection{Metrics of constant curvatures on the boundary}\label{subsec:metrics}

Before considering the proof of Theorem \ref{tm:main-cc} in the general case, we provide here a proof in the special case where the boundary metrics $h_j$ and $h^*_k$ have constant curvature. This special case will serve as an illustration of the ideas used in the proof, but it will also be used in the proof of the general case.

Recall that each boundary component of $M$ is incompressible and $\cCC(M)$ is equivalently the 
quasiconformal deformation space of a convex co-compact 
representation of $\pi_1(M)$ (see Section \ref{subsec:deformation}). As a special case of Theorem \ref{tm:ab},
$\cCC(M)$ is a complex analytic manifold biholomorphic to the Teichm\"uller space $\cT_{\partial M}$ of $\partial M$, of (real) dimension $\sum^n_{l=1}(6\mathfrak{g}_l-6)$, where $\mathfrak{g}_l$ is the genus of the boundary component $\partial_lM$.

\begin{lemma} \label{lm:special}
  Theorem \ref{tm:main-cc} holds when $h_j$ has constant curvature $K_j\in(-1,0)$ and  $h^*_k$ has constant curvature $K^*_k\in (-\infty, 0)$, where $j\in \cJ$, $k\in \cK$, and $\cI\sqcup \cJ\sqcup \cK=\{1,\dots,n\}$.
\end{lemma}

\begin{proof}
For each $k\in\cK$, let $K_k=K^*_k/(1-K^*_k)$. For the prescribed constants $K_j$ and $K_k$ (with $j\in\cJ, k\in\cK$), we define a map 
$$\phi_{\cJ,\cK}:\cCC(M)\rightarrow\Pi_{l\in\cI\sqcup\cJ\sqcup\cK}\cT_{\partial_lM}~,$$
which sends a metric $g\in\cCC(M)$ to the conformal structures at infinity on $\partial_iM$ for $i\in \cI$, the hyperbolic metrics homothetic to the induced metrics on the $K_j$-surfaces (i.e. the surface of constant curvature $K_j$) isotopic to $\partial_jM$ for $j\in\cJ$, and the hyperbolic metrics homothetic to the the third fundamental forms on the $K_k$-surfaces isotopic to $\partial_kM$ for $k\in\cK$.

Theorem \ref{tm:main-cc} will immediately follow if we can show that $\phi_{\cJ,\cK}$ is a homeomorphism, by considering 
$c_i\in\cT_{\partial_iM}$ for each $i\in\cI$, 
$|K_j|h_j\in\cT_{\partial_jM}$ for each $j\in\cJ$, and 
$|K^*_k|h^*_k\in \cT_{\partial_kM}$ for each $k\in\cK$ in the target domain.

It is not hard to see that $\phi_{\cJ,\cK}$ is well-defined, by the Ahlfors Finiteness Theorem \cite[Chap 3.1]{marden:hyperbolic}, Statement (4) of Proposition \ref{prop:dual} and the fact that each end of a convex co-compact hyperbolic manifold admits a unique foliation by surfaces of constant (Gauss) curvature (also called $K$-foliation for simplicity) \cite[Theorem 2]{Lab91}.
Moreover, $\phi_{\cJ,\cK}$ is smooth, see e.g. \cite{Ber70}.

By the Ahlfors-Bers Theorem, $\cCC(M)$ has the same dimension as the space 
$\Pi_{l\in\cI\sqcup\cJ\sqcup\cK}\cT_{\partial_lM}$. Combined with Proposition \ref{lm:local rigidity-mfld}, $\phi_{\cJ,\cK}$ is a local diffeomorphism. Note that the principal curvatures on the surfaces homotopic to $\partial_jM$ and $\partial_kM$ (which have constant curvature) are uniformly bounded, so $\phi_{\cJ,\cK}$ is proper by Lemma \ref{lm:compact} below. Combining the above results, $\phi_{\cJ,\cK}$ is a homeomorphism and the lemma follows.
\end{proof}

\subsection{Proof of Theorem \ref{tm:main-cc}}

We prove here Theorem \ref{tm:main-cc}.

We consider a compact manifold $M$ with non-empty boundary, and whose interior admits a convex co-compact hyperbolic metric. Let 
$\cE^l$ ($1\leq l\leq n$) be  an end of $M$, which is topologically the product of a closed surface $S_l$ of genus at least $2$ by $\R_{>0}$.

\begin{definition}\label{def:submflds}
We introduce the following notations:
  \begin{itemize}
  \item For each $i\in \cI$ and each 
      $c_i\in \cT_{\partial_iM}$, let $\cU_i(c_i)\subset \cCC(M)$ be the subset consisting of convex co-compact metrics on ${\rm int}(M)$ such that the conformal structure at infinity on $\partial_iM$ is $c_i$.
  \item For each $j\in \cJ$ and each smooth metric $h_j$ of curvature $K>-1$, let $\cV_j(h_j)\subset \cCC(M)$ be the space of convex co-compact metrics 
      on ${\rm int}(M)$ such that $\cE^j$ contains a closed surface $S_j$ isotopic to $\partial_jM$ with induced metric isotopic to $h_j$.
  \item For each $k\in \cK$ and each smooth metric $h^*_k$ of curvature $K<1$, with closed, contractible geodesics of length $L>2\pi$, let $\cW_k(h^*_k)\subset \cCC(M)$ be the space of convex co-compact metrics on ${\rm int}(M)$ such that $\cE^k$ contains a closed surface $S_k$ isotopic to $\partial_kM$ with third fundamental form isotopic to $h^*_k$.
  \end{itemize}
\end{definition}

Before proving the following lemma, we introduce another perspective of $\cCC(M)$. Let $\cCP(\partial_lM)$ denote the space of complex projective structures on $\partial_lM$ $(l\in\cI\sqcup\cJ\sqcup\cK)$ and let $\cCP(\partial M)$ denote the space of complex projective structures on $\partial M$, which is the product of the spaces $\cCP(\partial_lM)$. 
For each $l$, $\cCP(\partial_lM)$ is a complex manifold, with a smooth structure and a complex structure that can be defined in terms of the images in ${\rm PSL}_2(\C)$ of a finite generating set, see \cite{dumas-survey} for more details. The space $\cCP(\partial M)$ is a complex symplectic manifold, and $\cCC(M)$ is a Lagrangian submanifold of $\cCP(\partial M)$ of (real) dimension $\sum_{l=1}^n(6\mathfrak{g}_l-6)$, see \cite{McMullen}.

\begin{lemma}\label{lm:codim-submanifold}
  The subsets $\cU_i(c_i)$ (resp. $\cV_j(h_j)$, $\cW_k(h^*_k)$) are smooth submanifolds of $\cCC(M)$ of codimension $6\mathfrak{g}_i-6$ (resp. $6\mathfrak{g}_j-6$, $6\mathfrak{g}_k-6$), where $\mathfrak{g}_l$ is the genus of the boundary component $\partial_lM$ ($1\leq l\leq n$). 
\end{lemma}

\begin{proof}
  By the Ahlfors-Bers Theorem, $\cU_i(c_i)$ is a smooth submanifold of codimension $6\mathfrak{g}_i-6$ for each $i\in\cI$.

  Recall that $\cCP(\partial M)$ is the space of complex projective structures on $\partial M$. If $\sigma=(\sigma_1,\cdots, \sigma_n)\in\cCP(\partial M)$, then each $\sigma_l$ defines a hyperbolic end $\cE^l$, which admits a unique foliation by surfaces $S_{l,K}$ of constant curvature $K\in (-1,0)$, see \cite[Theorem 2]{Lab91}.

  For each $j\in \cJ$, and for each metric $h_j$ on $\partial_jM$, we denote by $\bar \cV_j(h_j)$ the subset of $\cCP(\partial M)$ of elements such that the hyperbolic end $\cE^j$ contains a surface $S_j$ homotopic to $\partial_jM$ with induced metric $h_j$. We will show that $\bar \cV_j(h_j)$ is a smooth submanifold of $\cCP(\partial M)$. Elements of $\bar \cV_j(h_j)$ are determined by a complex projective structure $\sigma_l\in \cCP(\partial_lM)$ for each $l\neq j$, such that $(h_j, B_j)$ satisfies the Gauss and Codazzi equations, where $B_j:TS_j\rightarrow TS_j$ is a bundle morphism. Those equations form an elliptic system in $B_j$ of index $6\mathfrak{g}_j-6$, see \cite[Lemma 3.1]{L4}, which we will call $\cS_j$.

  Since $\cS_j$ has index $6\mathfrak{g}_j-6$, the space of solutions of the linearization of $\cS_j$ (which correspond to the first-order deformation $\dot{B}_j$ of $B_j$ with fixed induced metric $h_j$ on $S_j$) has dimension at least $6\mathfrak{g}_j-6$. We claim that this dimension is always equal to $6\mathfrak{g}_j-6$. Indeed, suppose that it is strictly higher, that is, there is in $T_{\sigma_j}\cCP(\partial_jM)$ a space $W_j$ of dimension at least $6\mathfrak{g}_j-5$ of deformations $\dot B_j$ with fixed induced metric $h_j$ on $S_j$. Let $K\in (-1,0)$, consider for all $l\neq j$ the space $W_l\subset T_{\sigma_l}\cCP(\partial_lM)$ of first-order deformations for which the induced metric $h_{l,K}$ on $S_{l,K}$ does not vary. Similarly as discussed above for $\dot B_j$, the space $W_l$ has dimension at least $6\mathfrak{g}_l-6$.
  As a consequence,
  \begin{equation*}
  \dim\left(\Pi_{l=1}^n W_l\right)\geq \sum_{l=1}^n(6\mathfrak{g}_l-6)+1~,
  \end{equation*}
  and
   \begin{equation*}
  \codim \left(\Pi_{l=1}^n W_l\right) \leq \sum_{l=1}^n(6\mathfrak{g}_l-6)-1 < \dim(\cCC(M))~.
    \end{equation*}
  This implies that there exists a non-zero first-order deformation in $\Pi_{l=1}^n W_l$ which is tangent to $\cCC(M)$, which translates as a non-zero first-order deformation that preserves at first order the induced metric on $S_j$ and on all the $S_{l,K}$ $(l\neq j)$. This contradicts the infinitesimal rigidity result \cite[Lemma 5.1]{hmcb}. This shows that $\codim(T_{\sigma}\bar\cV_j(h_j))=6\mathfrak{g}_j-6$ at each point, so that $\bar\cV_j(h_j)$ is a smooth submanifold of $\cCP(\partial M)$ of codimension $6\mathfrak{g}_j-6$.

  Note that $\cV_j(h_j)=\bar \cV_j(h_j)\cap \cCC(M)$.
  For each $\sigma\in \cV_j(h_j)$, we have 
  \begin{equation*}
  \begin{split}
  \dim((T_{\sigma}\bar\cV_j(h_j))\cap(T_\sigma\cCC(M)))
  &\geq \dim(T_{\sigma}\bar \cV_j(h_j))-\codim(\cCC(M))~,
   \end{split}
 \end{equation*}
 where $\codim(\cCC(M))$ denotes the codimension of $\cCC(M)$ in $\cCP(\partial M)$. Therefore,
  \begin{equation*}
  \begin{split}
  \dim((T_{\sigma}\bar\cV_j(h_j))\cap(T_\sigma\cCC(M)))
  &\geq \sum_{l=1}^n(12\mathfrak{g}_l-12)-(6\mathfrak{g}_j-6)-\sum_{l=1}^n(6\mathfrak{g}_l-6)\\
  &\geq \sum_{l\not= j}(6\mathfrak{g}_l-6)~.
  \end{split}
  \end{equation*}
  We claim that
  $$\dim((T_{\sigma}\bar\cV_j(h_j))\cap(T_\sigma\cCC(M)))=\sum_{l\not= j}(6\mathfrak{g}_l-6)~.$$
  Otherwise, $$\dim((T_{\sigma}\bar\cV_j(h_j))\cap(T_\sigma\cCC(M)))\geq \sum_{l\not= j}(6\mathfrak{g}_l-6)+1~.$$
  Let $W_{\sigma}$ denote the space of the first-order deformations of $\sigma$ which preserve the induced metric $h_j$ on $S_j$ and the induced metric $h_{l,K}$ on $S_{l,K}$ $(l\not=j)$. 
  By the main result in \cite[Theorem 0.1]{hmcb}, for the prescribed metric $h_j$ on $\partial_jM$, and each hyperbolic metric $h'_l\in\cT_{\partial_lM}$ ($l\not=j$), there is a unique convex co-compact hyperbolic metric, say $g$ (more precisely, $g\in\cV_j(h_j)$), on ${\rm int}(M)$ with the induced metric isotopic to $h_j$ on a surface isotopic to $\partial_jM$ and with the induced metric isotopic to $(1/|K|)h'_l$ on a surface isotopic to $\partial_lM$ ($l\not=j$). Also, it is shown in \cite[Lemma 5.1]{hmcb} that there is no non-trivial first-order deformation of $g$ in $\cCC(M)$ that preserves the induced metric $h_j$ and the induced metric $(1/|K|)h'_l$ for all $l\not=j$. Therefore,
  $$\dim W_{\sigma}=\dim((T_{\sigma}\bar\cV_j(h_j))\cap(T_\sigma\cCC(M)))-\sum_{l\not=j}\dim(\cT_{\partial_lM})
  \geq \sum_{l\not= j}(6\mathfrak{g}_l-6)+1-\sum_{l\not=j}(6\mathfrak{g}_l-6)\geq 1~.$$
  This contradicts again the infinitesimal rigidity result \cite[Lemma 5.1]{hmcb}. So the claim follows and $\cV_j(h_j)$ is a smooth submanifold of $\cCC(M)$ of codimension $6\mathfrak{g}_j-6$.

  The same argument can be used, for all $k\in \cK$, to show that the deformation space in $\cCP(\partial_kM)$ which leaves invariant at first order the third fundamental form $h_k^*$ on a surface 
  $S_k\subset \cE^k$ has codimension $6\mathfrak{g}_k-6$ --- one only needs to consider, for $l\neq k$, the deformations for which the $K$-surfaces $S_{l,K}$ have constant third fundamental form, instead of constant induced metric, and use the rigidity result \cite[Lemma 5.7]{hmcb} relative to the third fundamental form on the boundary. This leads to the fact that $\cW_k(h^*_k)$ is a smooth submanifold of $\cCC(M)$ of codimension $6\mathfrak{g}_k-6$ for all $k\in \cK$.
\end{proof}

We now introduce basic definitions that will be needed for transversality arguments below. 

\begin{definition}\label{def:intersection}
Let $V_1, V_2, \ldots, V_k$ ($k\geq 2$) be linear subspaces of $\R^n$. We say that they have transversal intersection (sometimes also called normal intersection, see \cite[Chapter 1, \S 5]{GP}) if
$$V_{i_1}+(V_{i_2}\cap \cdots\cap V_{i_k})=\R^n~, $$
whenever $\{i_1,i_2,\ldots, i_k\}=\{1,2,\ldots,k\}$, or equivalently, $$\codim(V_1)+\cdots+\codim(V_k)=\codim(V_1\cap\cdots\cap V_k)~.$$
\end{definition}

Here the sum of the codimensions of the submanifolds  $\cU_i(c_i), i\in \cI$, $\cV_j(h_j), j\in \cJ$ and $\cW_k(h^*_k), k\in \cK$ is equal to the dimension of $\cCC(M)$. Combined with Definition \ref{def:intersection}, 
these submanifolds intersect tranversally at a point $g\in \cCC(M)$ if and only if  $$ \left(\cap_{i\in \cI} T_g\cU_i(c_i)\right)\cap \left(\cap_{j\in \cJ} T_g\cV_j(h_j)\right)\cap \left(\cap_{k\in \cK} T_g\cW_k(h^*_k)\right)=\{ 0\}~, $$
see \cite[Chapter 1, \S 5, p. 32]{GP}.

\begin{lemma} \label{lm:transversal}
 Assume that $g\in\cCC(M)$ is an intersection point of the submanifolds $\cU_i(c_i)$, $\cV_j(h_j)$ and $\cW_k(h^*_k)$ over all $i\in \cI, j\in \cJ, k\in \cK$. Then these submanifolds intersect at $g$ transversally. In particular, $g$ is an isolated intersection point.
\end{lemma}

\begin{proof}
  Since the sum of the codimensions is equal to $\dim(\cCC(M))$, to see these submanifolds intersect transversally at $g$, it suffices to show at $g\in \cCC(M)$,
  $$ \left(\cap_{i\in \cI} T_g\cU_i(c_i)\right)\cap \left(\cap_{j\in \cJ} T_g\cV_j(h_j)\right)\cap \left(\cap_{k\in \cK} T_g\cW_k(h^*_k)\right)=\{ 0\}~. $$

  Suppose by contradiction that the above intersection is not zero, there is a non-trivial tangent vector $\dot{g}\in T_g\cCC(M)$, which preserves the conformal structure $c_i$ (on $\partial_iM$), the induced metric $h_j$ (on a surface isotopic to $\partial_jM$), and the third fundamental form $h^*_k$ (on a surface isotopic to $\partial_kM$) at first order for all $i\in\cI$, $j\in\cJ$ and $k\in\cK$. This contradicts Proposition \ref{lm:local rigidity-mfld}. In particular, $g$ is an isolated intersection point. The lemma follows.
\end{proof}

We will also need a compactness lemma. From here on, a ``boundary data" for $M$ is an $n$-tuple composed of conformal structures $c_i$ on $\partial_iM$, $i\in \cI$, metrics $h_j$ of curvatures $K> -1$ on $\partial_jM$, $j\in \cJ$, and metrics $h^*_k$ of curvature $K<1$ with closed, contractible geodesics of length $L>2\pi$ on $\partial_kM$, $k\in \cK$.

We denote the set of all these boundary data (up to isotopy for each element of the $n$-tuples) by $\cX$, which is equipped with the product topology induced by the $C^{\infty}$-topology for each component space (in particular, for each $i\in\cI$, the component space is the Teichm\"uller space $\cT_{\partial_iM}$, which is equivalently the space of hyperbolic metrics on $\partial_iM$, up to isotopy).

For each $l\in\cI\sqcup\cJ\sqcup\cK$, a sequence $(h_{l,n})_{n\in\N}$ of Riemannian metrics on $\partial_lM$ in the $l$-th component space is said to be {\em bounded} if for each homotopy class $[\gamma]\in\pi_1(\partial_lM)$, the sequence of the lengths of the shortest curves (which is unique for $l\in\cI$) in $[\gamma]$ with respect to $h_{l,n}$ are bounded between two positive constants. A sequence of boundary data in $\cX$ is said to be {\em bounded} if each component sequence is bounded.

We first recall a compactness result of Labourie \cite[Th\'{e}or\`{e}me D]{L1} about the convergence of isometric immersions of a surface into $\HH^3$, which will be used in the sequel.

\begin{theorem}[Labourie]\label{thm:Labourie}
  Let $\iota_n:S\rightarrow \HH^3$ be a sequence of immersions of a surface $S$ such that the pull-back $\iota^*_n(h)$ of the hyperbolic metric $h$ of $\HH^3$ converges smoothly to a metric $g_0$. Assume that there is a point $p\in S$ such that the sequence of the images $\iota_n(p)$ are contained in a compact subset of $\HH^3$. If the integral of the mean curvature is uniformly bounded, then there is a subsequence of $\iota_n$ which converges smoothly to an isometric immersion $\iota_{\infty}$ such that $\iota^*_{\infty}(h)=g_0$.
\end{theorem}

It deserves mentioning that Theorem \ref{thm:Labourie} is local, in the sense that no assumption of compactness of $S$ or completeness of $\iota^*_n(h)$ is required. This theorem has an analog in the de Sitter space $\bdS^3$ (see \cite[Theorem 5.5]{these}), as stated below.

\begin{theorem}\label{thm:schlenker}
 Let $\iota_n: S \rightarrow \bdS^3$ be a sequence of  immersions of a surface $S$ such that the pull-back $\iota^*_n(h^*)$ of the dS metric $h^*$ of $\bdS^3$ converges smoothly to a metric $g_0$. Assume that there is a point $p\in S$ such that the sequence of 
 the 1-jets $j^1\iota_n(p)$ are contained in a compact subset of 
 $T\bdS^3$. 
 If the integral of the mean curvature is uniformly bounded, then there is a subsequence of $\iota_n$ which converges smoothly to an isometric immersion $\iota_{\infty}$ such that $\iota^*_{\infty}(h^*)=g_0$.
\end{theorem}

We now introduce the following estimate on the principal curvatures of complete, embedded convex surfaces in $\HH^3$ of variable curvature (see e.g. \cite[Lemma 4.1]{weylgen}).

\begin{lemma}\label{lm:principal-curv}
Let $\epsilon\in(0,1)$. There exists a constant $0<k_0<1$ (depending only on $\epsilon$) such that if $\Sigma\subset\HH^3$ is a complete, embedded convex surface with induced metric of curvature varying in $[-1+\epsilon,1/\epsilon]$ and injectivity radius bounded from below by $\epsilon$, then the principal curvatures of $\Sigma$ are in $[k_0,1/k_0]$. 
\end{lemma}

This lemma has a dual version in the following, which gives an estimate on the principal curvatures of complete, embedded spacelike convex surfaces in $\bdS^3$ of variable curvature.

\begin{lemma}\label{lm:principal-curv-dual}
Let $\epsilon\in(0,1)$. There exists a constant $0<k_0<1$ (depending only on $\epsilon$) such that if $\Sigma\subset\bdS^3$ is a complete, embedded spacelike convex surface with induced metric of curvature varying in $[-1/\epsilon,1-\epsilon]$ and with each closed, contractible geodesic having length $L\geq 2\pi+\epsilon$, then the principal curvatures of $\Sigma$ are in $[k_0,1/k_0]$.
\end{lemma}

\begin{proof}
  Let $\Sigma^*$ be the surface dual to $\Sigma$ in $\HH^3$. It is a complete, embedded convex surface which has curvature  $K^*=K/(1-K)\in [-1+\epsilon/(1+\epsilon), (1-\epsilon)/\epsilon]$, by Proposition \ref{prop:dual} and the curvature condition of $\Sigma$. We claim that the injectivity radius of the surface $\Sigma^*$ has a positive lower bound. Otherwise, suppose that there is a sequence  $(\Sigma_n)_{n\in\N}$ of complete, embedded spacelike convex surfaces in $\bdS^3$ (satisfying the conditions of curvature and injectivity radius given in the assumption) 
  for which the dual surface $\Sigma_n^*$ has an injectivity radius going to zero, there would be a sequence of closed geodesics on $\Sigma_n^*$ whose lengths tend to zero, and this implies that there would be a sequence of closed, contractible geodesics of $\Sigma_n$ with lengths tending to $2\pi$. This contradicts the assumption.

  One can thus apply Lemma \ref{lm:principal-curv} to $\Sigma^*$, which therefore has principal curvatures in $[k_0, 1/k_0]$ for some $k_0\in(0,1)$. The principal curvatures of $\Sigma$ are the inverse of the corresponding principal curvatures of $\Sigma^*$, and the lemma follows.
\end{proof}

We will need the following statement for the proof of the next lemma.

\begin{lemma} \label{lm:closed}
  Let $(g_n)_{n\in \N}$ be a sequence of convex co-compact metrics on 
  ${\rm int}(M)$, converging to a limit $g_\infty$ in $\cCC(M)$. Let 
  $\epsilon\in(0,1)$.
  \begin{enumerate}
  \item Let $j\in \cJ$, assume that for each $n\in \N$, $({\rm int}(M), g_n)$ contains in the end 
      $\cE^j$ an embedded, locally convex surface with 
      induced metric 
      $h_{j,n}$  of curvature $K\in [-1+\epsilon, 1/\epsilon]$ and injectivity radius bounded from below by $\epsilon$. Assume further that 
      $h_{j,n}\to h_{j,\infty}$, where $h_{j,\infty}$ is a smooth metric on $\partial_jM$. Then $({\rm int}(M), g_\infty)$ contains in the end 
      $\cE^j$ an embedded, locally convex surface with induced metric $h_{j,\infty}$.
  \item Let $k\in \cK$, assume that for each $n\in \N$, $({\rm int}(M), g_n)$ contains in the end 
      $\cE^k$ an embedded, locally convex surface with third fundamental form 
      $h^{*}_{k,n}$ of curvature $K\in [-1/\epsilon, 1-\epsilon]$ with closed, contractible geodesics of length $L\geq 2\pi+\epsilon$. Assume further that 
      $h^*_{k,n}\to h_{k,\infty}^*$. Then $({\rm int}(M), g_\infty)$ contains in the end 
      $\cE^k$ an embedded, locally convex surface with third fundamental form $h_{k,\infty}^*$.
  \end{enumerate}
\end{lemma}

\begin{proof}
  For point (1), for each $n\in \N$, let $\rho_n:\pi_1(\partial_jM)\to {\rm PSL}_2(\C)$ be the restriction to $\pi_1(\partial_jM)$ of the holonomy representation of $g_n$, and let $\rho_\infty:\pi_1(\partial_jM)\to {\rm PSL}_2(\C)$ be be the restriction to $\pi_1(\partial_jM)$ of the holonomy representation of $g_\infty$. By construction, there exists for each $n\in \N$ an isometric embedding 
  $\phi_n:(\widetilde{\partial_jM}, \tilde{h}_{j,n})\to \HH^3$, where we denote by 
  $\tilde{h}_{j,n}$ the lift of 
  $h_{j,n}$ to $\widetilde{\partial_jM}$, which is equivariant with respect to $\rho_n$.

  The $\phi_n$, $n\in \N$ are well-defined up to left composition by an element of ${\rm PSL}_2(\C)$. Let $x_0\in \widetilde{\partial_jM}$ and let $y_0\in \HH^3$. Choosing the normalization by conjugation in a suitable way, we can suppose that $\phi_n(x_0)=y_0$ for each $n\in \N$. Since we already know that the images of the $\phi_n$ have principal curvature uniformly bounded (see Lemma \ref{lm:principal-curv}) and $h_{j,n}\to h_{j,\infty}$ by assumption, it then follows from Theorem \ref{thm:Labourie} that, after extracting a subsequence, $(\phi_n)_{n\in \N}$ converges to an isometric embedding 
  $\phi_\infty: (\widetilde{\partial_jM}, \tilde{h}_{j,\infty})\to \HH^3$. Note that the assumption that $(g_n)_{n\in \N}$ converges to $g_{\infty}$ implies that $(\rho_n)_{n\in \N}$ converges to $\rho_{\infty}$. Since $\phi_n$ is equivariant under $\rho_n$ for all $n$, it follows that $\phi_\infty$ is equivariant under $\rho_\infty$, and point (1) follows.

  Point (2) is proved in the same manner, except that one needs to consider the dual surfaces in the de Sitter space and 
  apply Theorem \ref{thm:schlenker} and Lemma \ref{lm:principal-curv-dual} instead.
\end{proof}

\begin{lemma} \label{lm:compact}
  Let $(g_n)_{n\in \N}$ be a sequence in $\cCC(M)$ 
  such that $({\rm int}(M),g_n)$ contains a geodesically convex subset with boundary data $X_n\in\cX$, $n\in \N$. Suppose that the sequence $(X_n)_{n\in\N}$ remain bounded in $\cX$. Then $(g_n)_{n\in \N}$ has a convergent subsequence.
  Moreover, if $(X_n)_{n\in\N}$ converges to $X_{\infty}$, then $(g_n)_{n\in \N}$ has a subsequence converging to a limit $g_{\infty}$ such that $({\rm int}(M), g_{\infty})$ contains a geodesically convex subset with boundary data $X_{\infty}$.
  \end{lemma}

\begin{proof}
  Let $X_n=(c_{i,n}, h_{j,n}, h^*_{k,n})_{i\in\cI, j\in\cJ, k\in\cK}$ and let $\cE^n_l$ denote the hyperbolic end of $({\rm int}(M),g_n)$, $l\in \cI\sqcup\cJ\sqcup\cK$. Let $S^n_i$ ($i\in\cI$) denote the boundary at infinity of $({\rm int}(M),g_n)$ whose conformal structure is $c_{i,n}$ and let $S^n_j$ (resp. $S^n_k$) denote the smooth convex surface in $\cE^n_j$ (resp. $\cE^n_k$) whose induced metric (resp. third fundamental form) is $h_{j,n}$ (resp. $h^*_{k,n}$), with $j\in\cJ$ (resp. $k\in\cK$).

  By assumption, the sequence of the induced metrics $h_{j,n}$ (resp. third fundamental forms $h^*_{k,n}$) on $S^n_j$ (resp. $S^n_k$) remains bounded in the space of $C^{\infty}$ metrics on $\partial_jM$ (resp. $\partial_kM$) with curvature $K>-1$ (resp. $K<1$ and with closed, contractible geodesics of length $L>2\pi$). Combined with Proposition \ref{prop:dual}, the sequence of the induced metrics (say $h_{k,n}$) on $S^n_k$ is also bounded. Note that $S^n_j$ (resp. $S^n_k$) is a smooth convex closed surface of the same topological type for all $n$. Then the injectivity radius of $S^n_j$ (resp. $S^n_k$) is bounded from below by some constant $\epsilon>0$ for all $n$. Combining the above curvature conditions for $h_{j,n}$ and $h^*_{k,n}$ (and thus for $h_{k,n}$) and applying Lemma \ref{lm:principal-curv} to a lift of $S^n_j$ (resp. $S^n_k$) to $\HH^3$, the principal curvatures on $S^n_j$ (resp. $S^n_k$) are uniformly bounded between two positive constants.

  This allows us to apply a similar analysis as Proposition 3.11 of \cite{convexhull} to show that the nearest point retraction from the boundary at infinity of $\cE^n_j$ (resp. $\cE^n_k$) with the conformal hyperbolic metric to the smooth convex surface $S^n_j$ (resp. $S^n_k$) with the induced metric is uniformly bi-Lipschitz for all $n$. By 
  the relation between the first and third fundamental forms and the uniform boundedness of principal curvatures on $S^n_k$ (as shown above), 
  the duality map from $S^n_k$ to its dual surface $(S^n_k)^*$ (both with respect to induced metrics) is uniformly bi-Lipschitz for all $n$. So there is a uniformly bi-Lipschitz map from the boundary at infinity of $\cE^n_k$ (with the conformal hyperbolic metric) to $S^n_k$ (with the third fundamental form).

  As a consequence, the conformal structures at infinity say $c_{l,n}$ on the boundary at infinity of $\cE^l$ ($l\in\cJ\sqcup\cK$) are also bounded. Combined with the hypothesis that the conformal structure $c_{i,n}$ on $S^n_i$ are bounded, it follows from the Ahlfors-Bers theorem (see Theorem \ref{tm:ab}) that the sequence $(g_n)_{n\in \N}$ of convex co-compact metrics realizing the boundary data $X_n$ is also bounded, and thus has a convergent sub-sequence.

  If $(X_n)_{n\in\N}$ converges to $X_{\infty}\in\cX$, using the above result, then, up to extracting a subsequence, $(g_n)_{n\in \N}$ converges to a limit, say $g_{\infty}\in\cCC(M)$.  By assumption, $({\rm int}(M),g_n)$ contains a geodesically convex subset with boundary data $X_n\in\cX$. Using the Ahlfors-Bers Theorem, the conformal structure on the boundary component $\partial_iM$ at infinity of $({\rm int}(M),g_{\infty})$ is the $i$-th component prescribed in $X_{\infty}$. The existence of locally convex surfaces with induced metric or third fundamental form prescribed in $\cX$ in the other ends of $({\rm int}(M),g_{\infty})$ follows from Lemma \ref{lm:closed}.
\end{proof}

\begin{proof}[\textbf{Proof of Theorem \ref{tm:main-cc}}]
  Let $X^0=(c_{i,0}, h_{j,0}, h^*_{k,0})_{i\in\cI, j\in\cJ, k\in\cK}$ be a boundary data for $M$. We assume that $X^0$ is realized by a certain metric $g_0\in \cCC(M)$. For simplicity, we denote $c_{i,0}=c_i$, $h_{j,0}=h_j$ and $h^*_{k,0}=h^*_k$.

  Similarly, let $X^1=(c_{i,1}, h_{j,1}, h^*_{k,1})_{i\in\cI, j\in\cJ, k\in\cK}$ be another boundary data of the same type. We then choose a smooth 1-parameter family $(X^t)_{t\in [0,1]}$ in $\cX$ connecting $X^0$ to $X^1$. We denote by $U\subset [0,1]$ the maximal subset of $[0,1]$ containing $0$ such that $X^t$ is realized as the boundary data of a metric $g_t\in \cCC(M)$ for all $t\in U$.

  Note that under the smooth variation of the boundary data $c_i$, $h_j$ and $h^*_k$ determined by $X^t$, the submanifolds $\cU_i(c_i)$, $\cV_j(h_j)$, $\cW_k(h^*_k)$ vary smoothly in $\cCC(M)$ (this can be seen by the smoothness of the parametrization map in the proof of Ahlfors-Bers Theorem \cite{Ber70} and the main results in \cite{hmcb}). By Lemma \ref{lm:transversal}, the submanifolds $\cU_i(c_i)$, $\cV_j(h_j)$ and $\cW_k(h^*_k)$ over all $i\in \cI, j\in \cJ, k\in \cK$ intersect transversally at $g_0$ (with $g_0$ isolated). So the submanifolds $\cU^t_i(c_{i,t})$, $\cV^t_j(h_{j,t})$, $\cW^t_k(h^*_{k,t})$ corresponding to the smooth variation $c_{i,t}, h_{j,t}, h^*_{k,t}$ determined by $X^t$ keep intersecting each other for $t$ sufficiently small, with each intersection point isolated and varying smoothly in $\cCC(M)$. Therefore, $U$ is open. By Lemma \ref{lm:compact}, $U$ is also closed, so $U=[0,1]$.

  It follows from this argument that any metric realizing $X^0$ can be deformed in $\cCC(M)$ uniquely (once the path $(X^t)_{t\in [0,1]}$ is chosen) to a metric realizing $X^1$. Note that $\cCC(M)$ is path-connected. So each boundary data can be realized, and moreover the number of realizations is the same for all boundary data, since each intersection point (which realizes the boundary data $X^0$) of $\cU_i(c_i)$, $\cV_j(h_j)$ and $\cW_k(h^*_k)$ over all $i\in \cI, j\in \cJ, k\in \cK$ varies smoothly and keeps isolated under the smooth variation $(X^t)_{t\in [0,1]}$.

  However we already know from Lemma \ref{lm:special} that for constant curvature metrics this number is $1$, and this concludes the proof.
\end{proof}

%% file: short-weyl5.tex
\section{Convex domains in $\HH^3$}
\label{sc:convexH}

We turn in this section to convex domains in $\HH^3$. We will be using the results obtained in the previous section and obtain results on convex domains in $\HH^3$ by approximating them, in a suitable sense, by convex domains in quasifuchsian manifolds.

\subsection{The classical conformal welding and generalized gluing maps}

\subsubsection{The classical conformal welding of a Jordan curve in $\CP^1$}

Let $C\subset\CP^1$ be an oriented Jordan curve. We denote by $U^+$ and $U^-$ the connected components of $\CP^1\setminus C$ on the left and right sides of $C$, which corresponds to the positive and negative sides (or the ``upper" and ``lower" sides) of $C$. By the Riemann mapping theorem, there is a conformal (or bi-holomorphic) map, say $\varphi^+_C$ (resp. $\varphi^-_C$), that maps the upper (resp. lower) half-plane of $\C$ to $U^+$ (resp. $U^-$). Let $\partial\varphi^+_C$ (resp. $\partial\varphi^-_C$) denote the boundary correspondence (this is a homeomorphism by Carath\'eodory's theorem). The classical \emph{conformal welding} is defined as $(\partial\varphi^-_C)^{-1}\circ\partial\varphi^+_C$, which is well-defined up to pre- and post-composition by elements of $\PSL_2(\R)$. For simplicity, we consider the \emph{normalized} representative (that is, it preserves $0,1,\infty\in\partial_{\infty}\HH^2\cong \RP^1$) and denote it by $\Phi_{C}$ (see e.g. \cite{convexhull}).

We first recall the notion of quasi-circles and quasisymmetric maps.

\begin{definition} A Jordan curve $C\subset\CP^1$ is said to be a $k$-quasi-circle ($k\geq 1$) if it is the image of $\RP^1$ under a $k$-quasiconformal homeomorphism from $\CP^1$ to itself. We say $C$ is normalized if it goes through the three points $0,1,\infty$.
\end{definition}

\begin{definition}
An orientation-preserving homeomorphism $f:\RP^1\rightarrow \RP^1$ is called $k$-quasisymmetric ($k\geq 1$) if it is the boundary extension of a $k$-quasiconformal homeomorphism from $\HH^2$ to itself.

We say a quasisymmetric homeomorphism $f$ is quasifuchsian if there is a surface group $\pi_1(S)$ and two Fuchsian representations $\rho_i:\pi_1(S)\rightarrow \PSL_2(\R)$ ($i=1,2$) such that $f$ is $(\rho_1,\rho_2)$-equivariant. More precisely, $$f(\rho_1(\gamma)x)=\rho_2(\gamma)f(x),$$ for all $\gamma\in\pi_1(S)$ and $x\in\RP^1$.
\end{definition}

Let $\cQC$ denote the space of normalized quasi-circles in $\CP^1$ (with the Hausdorff topology) and let $\cT$ denote the universal Teichm\"uller space, which is the space of normalized quasisymmetric homeomorphisms of $\RP^1$. The following is a classical result of Ahlfors and Bers \cite{ahlfors-reflections,ahlfors:riemann}.

\begin{theorem}\label{thm:Ahlfors-Bers}
Let $C\in\cQC$ be a normalized $k$-quasi-circle. Then the conformal welding $\Phi_C$ of $C$ is a normalized 
$k$-quasisymmetric homeomorphism in $\cT$. Conversely, given a normalized $k$-quasisymmetric homeomorphism $f\in\cT$, there is a unique normalized 
$k$-quasi-circle in $\cQC$ whose conformal welding is $f$.
\end{theorem}

Recall that the concepts of complex structure and conformal structure are equivalent on two-dimensional oriented Riemannian manifolds (see e.g. \cite[Theorem 1.7]{imayoshi-taniguchi}).

\begin{lemma}\label{lm:conf-bihol}
Let $X_1$, $X_2$ be the Riemann surfaces induced by oriented Riemannian surfaces $(S_1, h_1)$ and $(S_2, h_2)$, respectively. Then $f:(S_1, h_1)\rightarrow (S_2,h_2)$ is an orientation-preserving conformal map if and only if $f:X_1\rightarrow X_2$ is a bi-holomorphic map.
\end{lemma}

\subsubsection{The gluing maps for unbounded convex domains in $\HH^3$ of the first type}

We now consider the first type of unbounded convex domain (see Section \ref{subsec:unbounded convex domain}). Let $\Omega$ be a convex subset in $\HH^3$ with its convex boundary $\partial\Omega\subset\HH^3$ an open complete disk and its ideal boundary $\partial_{\infty}\Omega\subset(\partial_{\infty}\HH^3\cong\CP^1)$ a closed disk. We denote by $\partial(\partial_\infty\Omega)$ the boundary of $\partial_{\infty}\Omega$ in $\partial_{\infty}\HH^3$. Then $\partial(\partial_{\infty}\Omega)=\partial_{\infty}(\partial\Omega)$ is a Jordan curve in $\CP^1$.

Let $I$ (resp. $\III$) denote the induced metric (resp. third fundamental form) on $\partial\Omega$ and let $\HH^{2+}$ (resp. $\HH^{2-}$) denote the upper (resp. lower) half-plane in $\C$. Let 
$\varphi^c_{\Omega}: \HH^{2+}\rightarrow {\rm int}(\partial_{\infty}\Omega)$ be a bi-holomorphic map. We denote by $h^{\pm}_0$ the hyperbolic metric on $\HH^{2\pm}$.  Let $\varphi^I_{\Omega}: \HH^{2-}\rightarrow \partial\Omega$ be a map such that $(\varphi^I_{\Omega})^*I$ is conformal to $h^{-}_0$ and let $\varphi^{\III}_{\Omega}: \HH^{2-}\rightarrow \partial\Omega$ be a map such that $(\varphi^{\III}_{\Omega})^*\III$ is conformal to $h^-_0$. By the Carath\'eodory's theorem and Lemma \ref{lm:conf-bihol}, each of the maps $\varphi^c_{\Omega}$, $\varphi^I_{\Omega}$ and $\varphi^{\III}_{\Omega}$ extends to a homeomorphism on the boundary of $\HH^2$. We define the ``conformal welding" between the conformal structure on 
${\rm int}(\partial_{\infty}\Omega)$ and the conformal structure on $\partial\Omega$ induced by $I$ (resp. $\III$), called the \emph{$I$-gluing map} (resp. \emph{$\III$-gluing map}) of $\Omega$ in the following. 

\begin{definition}\label{def:I-III gluing maps of first type}
  The $I$-gluing map of $\Omega$ is defined as $$\Phi^I_{\Omega}:=(\partial \varphi^I_{\Omega})^{-1}\circ\partial \varphi^c_{\Omega}:\RP^1\rightarrow \RP^1.$$ The $\III$-gluing map of $\Omega$ is defined as $$\Phi^{\III}_{\Omega}:=(\partial \varphi^{\III}_{\Omega})^{-1}\circ\partial \varphi^c_{\Omega}:\RP^1\rightarrow \RP^1~. $$
  Here the maps $\varphi^c_{\Omega}$, $\varphi^I_{\Omega}$ and $\varphi^{\III}_{\Omega}$ are always chosen such that the above gluing maps are \emph{normalized}, that is, preserving $0,1,\infty$.
\end{definition}

 We stress that under the normalization condition, the above pair $(\varphi^c_{\Omega},\varphi^I_{\Omega})$ (resp. $(\varphi^c_{\Omega},\varphi^{\III}_{\Omega})$) is unique, up to left composition with a common element of $\PSL_2(\C)$. In particular, the $I$- (resp. $\III$-) gluing map $\Phi^I_{\Omega}$ (resp. $\Phi^{\III}_{\Omega}$) of $\Omega$ is independent of the choice of the representative pair $(\varphi^c_{\Omega},\varphi^I_{\Omega})$ (resp. $(\varphi^c_{\Omega},\varphi^{\III}_{\Omega})$). The above definition is well-defined.

\begin{example} \label{ex:gluing map I}
  If $\partial_{\infty}(\partial\Omega)=\partial(\partial_{\infty}\Omega)$ is a round circle in $\CP^1$, and the induced metric (or third fundamental form) on $\partial\Omega$ has constant curvature,  then the gluing maps $\Phi^{I}_{\Omega}$ and $\Phi^{\III}_{\Omega}$ are both the identity.

If $\partial_{\infty}(\partial\Omega)=\partial(\partial_{\infty}\Omega)$ is a Jordan curve invariant under the group $\Gamma_{\rho}$ given by the quasifuchsian representation  $\rho:\pi_1(S)\rightarrow\PSL_2(\C)$, then $\Omega/\Gamma_{\rho}$ is a convex subset in the quasifuchsian hyperbolic manifold $\HH^3/\Gamma_{\rho}$ with one conformal boundary surface $(\partial_{\infty}\Omega)/\Gamma_{\rho}$ and one convex boundary surface $(\partial\Omega)/\Gamma_{\rho}$.

The map $(\varphi^c_{\Omega})^{-1}: {\rm int}(\partial_{\infty}\Omega)\rightarrow \HH^{2+}$ conjugates the action of $\rho$ on $\partial_{\infty}\Omega$ to a properly discontinuous action by isometries on $\HH^2$, which is the Fuchsian representation $\rho_c$ of the complex structure $c$ on $({\rm int}(\partial_{\infty}\Omega))/\Gamma_{\rho}:$
$$(\varphi^c_{\Omega})^{-1} \circ\rho\circ \varphi^c_{\Omega}=\rho_c.$$

The map $(\varphi^I_{\Omega})^{-1}: \partial\Omega\rightarrow \HH^{2-}$ (resp. $(\varphi^{\III}_{\Omega})^{-1}: \partial\Omega\rightarrow \HH^{2-}$) conjugates the action of $\rho$ on $\partial\Omega$ to a properly discontinuous action by isometries on $\HH^2$, which is the Fuchsian representation $\rho_I$ (resp. $\rho_{\III}$) of the hyperbolic metric in the conformal class of the induced metric (resp. third fundamental form) on $(\partial\Omega)/\Gamma_{\rho}$ (via the Riemannian uniformization theorem):
      $$(\varphi^I_{\Omega})^{-1}\circ\rho\circ \varphi^I_{\Omega}=\rho_I \quad \quad (resp. \,\, (\varphi^{\III}_{\Omega})^{-1}\circ\rho\circ\varphi^{\III}_{\Omega} =\rho_{\III} \,\, ).$$
      Combined with Definition \ref{def:I-III gluing maps of first type}, we obtain that
      $$\Phi^I_{\Omega}\circ\rho_c=\rho_I\circ \Phi^I_{\Omega}\quad \quad (resp. \,\, \Phi^{\III}_{\Omega}\circ\rho_c=\rho_{\III}\circ \Phi^{\III}_{\Omega}\,\, ).$$
      Therefore, $\Phi^{I}_{\Omega}$ (resp. $\Phi^{\III}_{\Omega}$) is a $(\rho_c,\rho_I)$-equivariant (resp. $(\rho_c,\rho_{\III})$-equivariant) quasisymmetric homeomorphism.
\end{example}


\subsubsection{The gluing maps for unbounded convex domains in $\HH^3$ of the second type}
For the second type of unbounded convex domain (see Section \ref{subsec:unbounded convex domain}), let $\Omega\subset \HH^3$ be a convex domain such that $\partial\Omega$ is the disjoint union of two complete disks, denoted here by $\partial_-\Omega$ and $\partial_+\Omega$, which ``meet'' at infinity along a Jordan curve, oriented such that $\partial_+\Omega$ and $\partial_-\Omega$ lie on its left and right-hand sides respectively (see e.g. \cite{convexhull}).

In this case, $\partial_{\infty}(\partial_-\Omega)=\partial_{\infty}(\partial_+\Omega)=\partial_{\infty}\Omega$ is a Jordan curve. Let $I_{\pm}$ and $\III_{\pm}$ denote the induced metric and third fundamental form on $\partial_{\pm}\Omega$. Let $\varphi^{I_{\pm}}_{\Omega}: \HH^{2\pm}\rightarrow \partial_{\pm}\Omega$ be a map such that $(\varphi^{I_{\pm}}_{\Omega})^*(I_{\pm})$ is conformal to $h^{\pm}_0$ and let $\varphi^{\III_{\pm}}_{\Omega}: \HH^{2\pm}\rightarrow \partial_{\pm}\Omega$ be a map such that $(\varphi^{\III_{\pm}}_{\Omega})^*(\III_{\pm})$ is conformal to $h^{\pm}_0$. Note that each of the maps $\varphi^{I_{\pm}}_{\Omega}$ and $\varphi^{\III_{\pm}}_{\Omega}$ extends to a homeomorphism on the boundary of $\HH^2$, by Lemma \ref{lm:conf-bihol} and Caratheodory's theorem.

In the following, we define the ``conformal welding" between the conformal structure on $\partial_{+}\Omega$ induced by $I_+$ (resp. $\III_+$) and the conformal structure on $\partial_{-}\Omega$ induced by $I_-$ (resp. $\III_-$), called the \emph{$I$-gluing map} (resp. \emph{$\III$-gluing map}) of $\Omega$. Besides, we can also define a ``conformal welding" between the conformal structure on $\partial_{+}\Omega$ induced by $\III_+$ and the conformal structure on $\partial_{-}\Omega$ induced by $I_-$, called the \emph{$(\III,I)$-mixed gluing map} of $\Omega$ (see more precisely below). 

\begin{definition}\label{def:I-III gluing maps of second type}
  The $I$-gluing map of $\Omega$ is defined as $$\Phi^{I}_{\Omega}:=(\partial \varphi^{I_-}_{\Omega})^{-1}\circ\partial \varphi^{I_+}_{\Omega}:\RP^1\rightarrow \RP^1.$$ The $\III$-gluing map of $\Omega$ is defined as  $$\Phi^{\III}_{\Omega}:=(\partial \varphi^{\III_{-}}_{\Omega})^{-1}\circ\partial \varphi^{\III_{+}}_{\Omega}:\RP^1\rightarrow \RP^1.$$ The $(\III,I)$-mixed gluing map of $\Omega$ is defined as $$\Phi^{(\III,I)}_{\Omega}:=(\partial \varphi^{I_{-}}_{\Omega})^{-1}\circ\partial \varphi^{\III_{+}}_{\Omega}:\RP^1\rightarrow \RP^1~.$$
  Here the maps $\varphi^{I_{\pm}}_{\Omega}$ and $\varphi^{\III_{\pm}}_{\Omega}$ are also chosen such that the above gluing maps are \emph{normalized}.
\end{definition}

We stress that under the normalization condition, the above pair $(\varphi^{I_+}_{\Omega},\varphi^{I_-}_{\Omega})$ (resp. $(\varphi^{\III_+}_{\Omega},\varphi^{\III_-}_{\Omega})$, $(\varphi^{\III_+},\varphi^{I_-})$) is unique, up to left composition with a common element of $\PSL_2(\C)$. As a consequence, the $I$- (resp. $\III$-, $(\III,I)$-) gluing map $\Phi^I_{\Omega}$ (resp. $\Phi^{\III}_{\Omega}$, $\Phi^{(\III,I)}_{\Omega}$) of $\Omega$ is independent of the choice of the representative pair $(\varphi^{I_+}_{\Omega},\varphi^{I_-}_{\Omega})$ (resp. $(\varphi^{\III_+}_{\Omega},\varphi^{\III_-}_{\Omega})$, $(\varphi^{\III_+},\varphi^{I_-})$). The above definition is well-defined. Similarly we can also define the $(I,\III)$-mixed gluing map of $\Omega$. Without loss of generality, we only consider the  $(\III,I)$-mixed gluing map here.

\begin{example} \label{ex:gluing map II}
  If $\partial_{\infty}\Omega$ is a round circle in $\CP^1$, 
  and the induced metric (or third fundamental form) on $\partial\Omega$ has constant curvature,  the gluing maps $\Phi^{I}_{\Omega}$, $\Phi^{\III}_{\Omega}$ and $\Phi^{(\III,I)}_{\Omega}$ are the identity. 
  If $\partial_{\infty}\Omega$ is a Jordan curve invariant under the group $\Gamma_{\rho}$ given by a quasifuchsian representation $\rho:\pi_1(S)\rightarrow\PSL_2(\C)$, then $\Omega/\Gamma_{\rho}$ is a convex subset in the quasifuchsian hyperbolic manifold $\HH^3/\Gamma_{\rho}$ with two convex boundary surfaces $(\partial_{\pm}\Omega)/\Gamma_{\rho}$.

Similarly as Example \ref{ex:gluing map II}, the map $(\varphi^{I_{\pm}}_{\Omega})^{-1}: \partial_{\pm}\Omega\rightarrow \HH^{2\pm}$ (resp. $(\varphi^{\III_{\pm}}_{\Omega})^{-1}: \partial_{\pm}\Omega\rightarrow \HH^{2\pm}$) conjugates the action of $\rho$ on $\partial_{\pm}\Omega$ to the Fuchsian representation $\rho_{I^{\pm}}$ (resp. $\rho_{\III^{\pm}}$) of the hyperbolic metric in the conformal class of the induced metric (resp. third fundamental form) on $(\partial_{\pm}\Omega)/\Gamma_{\rho}$. As a consequence, the gluing map $\Phi^{I}_{\Omega}$ (resp. $\Phi^{\III}_{\Omega}$) is a $(\rho_{I^+},\rho_{I^-})$-equivariant (resp. $(\rho_{\III^+},\rho_{\III^-})$-equivariant) quasisymmetric homeomorphism. The mixed gluing map $\Phi^{(\III,I)}_{\Omega}$ 
is a $(\rho_{\III^+},\rho_{I^-})$-equivariant 
quasisymmetric homeomorphism.
\end{example}
The above gluing maps can be viewed as generalized versions of the classical conformal welding.

\subsubsection{Approximation of quasisymmetric maps}
It is shown in \cite[Lemma 3.4]{BS17} and \cite[Proposition 9.1]{convexhull} that a quasisymmetric homeomorphism can be approximated by a sequence of quasifuchsian quasisymmetric homeomorphisms:

\begin{proposition}\label{prop:qcirc approx}
Any quasisymmetric homeomorphism $f:\RP^1\rightarrow \RP^1$ is the $C^0$-limit of a sequence of uniformly quasisymmetric quasifuchsian homeomorphisms $f_n$.
\end{proposition}

\subsection{Fuchsian approximations of non-positively curved metrics}

For the convenience of construction, in this subsection we use the Poincar\'{e} unit disk model of $\HH^2$, which is the unit open disk $\DD^2=\{z\in\C:|z|<1\}$ with the Poincar\'{e} metric $$h_0=\frac{4}{(1-|z|^2)^2}|dz|^2.$$

\begin{definition}
Let $h$ be a smooth, complete conformal metric on an open disk $D\subset\CP^1$ and let $\rho:\pi_1(S)\rightarrow \PSL_2(\R)$ be a Fuchsian representation. We say that $h$ is $\rho$-invariant if there is a conformal map $\varphi:\HH^2\rightarrow D$ with $\varphi^*h=e^{2u}h_0$, where $u:\HH^2\rightarrow \R$ is a smooth function, such that
 $$e^{2u(\rho(\gamma)x)}h_0(\rho(\gamma)x)=e^{2u(x)}h_0(x),$$
for all $\gamma\in\pi_1(S)$ and $x\in\HH^2$.
\end{definition}

\begin{proposition}\label{prop:approx III}
Let $h^*$ be  a smooth, complete, bounded
conformal metric of curvature $K_{h^*}\in[-1/\epsilon,0]$ on $\DD^2$ and let $(\rho^*_n)_{n\in\N^+}$ be a sequence of Fuchsian representations of $\Gamma_n:=\pi_1(S_n)$ whose fundamental domain contains a closed hyperbolic disk $B_H(0,\tau_n)$ centered at $0$ with radius $\tau_n$, where $\tau_n$ strictly increases to $\infty$. Then there is a sequence of $\rho^*_n$-invariant conformal metrics, say $h^*_n$, on $\DD^2$, with curvature $K_{h^*_n}\in[-1/\epsilon',0]$ (for some $\epsilon'>0$), which uniformly converges to $h^*$ smoothly on compact subsets of $\DD^2$.
\end{proposition}

\begin{proof}
Let $P^*_n$ denote the fundamental domain of $\rho^*_n$ that contains the closed hyperbolic disk $B_H(0,\tau_n)$. Let $u^*:\DD^2\rightarrow \R$ be a smooth function such that $h^*=e^{2u^*}h_0$. 
Since $h^*$ is a complete, bounded (see Definition \ref{def:bounded conf}) conformal metric with curvature at least $-1/\epsilon$, combined with \cite[Theorem 2.8]{mateljevic}, there exist constants $c^*_1>0>c^*_0=\log\sqrt{\epsilon}$ such that $c^*_0\leq u^*\leq c^*_1$ .

We are going to define a real-valued function on $\DD^2$, called $v^*_n$, which is used to construct the desired metric $h^*_n$.
If $u^*$ is constant on $\DD^2$, we set $v^*_n:=u^*$. If $u^*$ is not constant on $\DD^2$, without loss of generality, we assume that $u^*$ is not constant on $B_H(0,\tau_1)$. We construct the desired $h^*_n$ using the following steps:

\textbf{Step 1}: For each $n\in \N$, we construct a smooth $v^*_n:\DD^2\rightarrow\R$ which satisfies the following conditions:
\begin{enumerate}[(i)]
 \item $c^*_0\leq v^*_n<c^*_1$ on $\mathring{B}(0,t_n)$, where $\mathring{B}(0,t_n)$ denote the open euclidean disk of radius $t_n$ and $t_n=\tanh(\tau_n/2)$.
  \item $v^*_n$ is equal to $c^*_1$ outside $\mathring{B}(0,t_n)$, and equal to $c^*_0$ in $B(0,l_n)$, for a sequence $(l_n)_{n\in \N}$ (to be determined) such that $l_n<t_n$ and $l_n\to 1$ as $n\to \infty$.
  \item  $v^*_n$ is rotationally symmetric (i.e. $v^*_n$ is invariant under the rotation centered at $0$) on $\DD^2$.
  \item There is a constant $C>1$ such that on $\mathring{B}_H(0,t_n)\setminus B_H(0,l_n)$, the (Gaussian) curvature of $e^{2v^*_n}h_0$ satisfies
    $$ -C/\epsilon\leq e^{-2v^*_n}(\Delta_{h_0}v^*_n-1)
    \leq 0~,$$
    for $n$ sufficiently large, where $\Delta_{h_0}v^*_n:=-tr(D^{h_0} dv^*_n)=-((1-|z|^2)^2/4)(\partial_{xx}+\partial_{yy})v^*_n$, and $D^{h_0}$ is the Levi-Civita connection of $h_0$.
\end{enumerate}
We will construct $v_n^*$ as a function on $|z|:=r\in[0,1)$ -- which is then extended to a rotationally invariant function on $\DD^2$ -- depending on 
the parameter $t_n\in (0,1)$ and will show that, 
if $n$ is sufficiently large, then $v_n^*$ satisfies conditions (i)--(iv). Note that $v_n^*$ satisfies condition (i) and thanks to the expression of the Laplace operator for radial functions, 
to show condition (iv), it suffices to show that there exists a constant $C>1$ such that
\begin{equation}
  \label{eq:polar}
  C-1=\frac {C}{\epsilon} e^{2c_0^*}-1\geq \frac {(1-r^2)^2}{4}((v_n^*)''(r) + \frac 1r(v_n^*)'(r)) \geq -1~,
\end{equation}
for $n$ sufficiently large (note that $e^{2c^*_0}=\epsilon$).

Given $s\in(0,1)$, we define $l,t\in [0,1]$ such that $l<s<t$ and that
\begin{equation}
  \label{eq:mean}
  \frac 2{1-s} = \frac 1{1-l}+\frac 1{1-t}~,
\end{equation}
\begin{equation}
  \label{eq:log}
  \frac{(1-l)(1-t)}{(1-s)^2} = \exp(2(c^*_1-c^*_0))~.
\end{equation}
Note that there is indeed a unique possible of $l,t$ satisfying those conditions.
This is more easily seen by setting $l'=1/(1-l), s'=1/(1-s), t'=1/(1-t)$, since the conditions then become
$$ 2s'=l'+t'~, ~~ l't'=s'^2\exp(-2(c^*_1-c^*_0))~, $$
and if we further set
$$ d=s'-l'=t'-s' $$
then $l't'=s'^2-d^2$ and the second equation can be written as
$$ s'^2-d^2= s'^2\exp(-2(c^*_1-c^*_0))~, $$
which uniquely determines $d$, and therefore $l',t'$ and $l,t$. More precisely,
\begin{equation*}
d=s'\sqrt{1-e^{-2(c^*_1-c^*_0)}},\quad l'=s'(1-\sqrt{1-e^{-2(c^*_1-c^*_0)}}), \quad t'=s'(1+\sqrt{1-e^{-2(c^*_1-c^*_0)}})~.
\end{equation*}

Moreover it is important to note that $l,t\to 1$ as $s\to 1$, since as $s\to 1$, 
$s'\to \infty$, $d\to \infty$ proportionally to $s'$, and therefore 
$l', t'\to \infty$ proportionally to $s'$, too.


We now define $v_n^*$ as follows. For each $t_n=\tanh(\tau_n/2)$, set $$s_n:=1-(1-t_n)(1+\sqrt{1-e^{-2(c^*_1-c^*_0)}})~,$$ $$l_n:=1-(1-s_n)(1-\sqrt{1-e^{-2(c^*_1-c^*_0)}})^{-1}
=1-e^{2(c^*_1-c^*_0)}(1-t_n)(1+\sqrt{1-e^{-2(c^*_1-c^*_0)}})^2~.$$
Then $l_n, s_n, t_n$ satisfies $0<l_n<s_n<t_n<1$, \eqref{eq:mean} and \eqref{eq:log}.
\begin{itemize}
\item For $r\in [0,l_n]$, $v_n^*(r)=c_0^*$~,
\item for $r\in (l_n,s_n)$, $v_n^*(r)=f_n(r)=(1/2)(\log((1-l_n)/(1-r)) -(r-l_n)/(1-l_n))+c^*_0$~,
\item for $r\in [s_n,t_n)$~, $v_n^*(r)=g_n(r)=(1/2)(\log((1-r)/(1-t_n))+(r-t_n)/(1-t_n))+c^*_1$~,
\item for $r\in [t_n,1)$, $v_n^*(r)=c^*_1$~.
\end{itemize}
It is easy to check that $f_n(l_n)=c^*_0$ and that $g_n(t_n)=c^*_1$. Moreover,
\begin{equation*}
\begin{split}
  g_n(s_n)-f_n(s_n)
  &= \frac{1}{2}(\log\frac{1-s_n}{1-t_n}+\frac{s_n-t_n}{1-t_n})+c^*_1 - \frac{1}{2}(\log\frac{1-l_n}{1-s_n}-\frac{s_n-l_n}{1-l_n})-c^*_0 \\
  &= \frac{1}{2}\left(\log\frac{(1-s_n)^2}{(1-t_n)(1-l_n)} + 1 - \frac{1-s_n}{1-t_n} + 1-\frac{1-s_n}{1-l_n}\right) + c^*_1-c^*_0 \\
  & = \frac{1}{2}\log\frac{(1-s_n)^2}{(1-t_n)(1-l_n)}+ c^*_1-c^*_0~,
\end{split}
\end{equation*}
where the last step used \eqref{eq:mean}. However \eqref{eq:log} then shows that the right-hand side is zero, so that $f_n(s_n)=g_n(s_n)$.

One can also check quite directly that
$$ f'_n(r) = \frac{1}{2}\left( \frac 1{1-r} - \frac 1{1-l_n}\right)~, $$
$$ g'_n(r) =  \frac{1}{2}\left( -\frac 1{1-r} + \frac 1{1-t_n}\right)~, $$
so that $f'_n(l_n)=g'_n(t_n)=0$.
Moreover, $f'_n(s_n)=g'_n(s_n)$ by \eqref{eq:mean}.

The second derivatives of $f_n$ and $g_n$ can also be computed quite directly.
$$ f''_n(r)=\frac 1{2(1-r)^2}~,~~ g''_n(r) = -\frac 1{2(1-r)^2}~. $$
It follows that
\begin{eqnarray*}
   \frac{(1-r^2)^2}4 \left(\frac{f'_n(r)}r+f''_n(r)\right) & = & \frac{(1+r)^2}8 \left(1+ \frac{1-r}r (1- \frac{1-r}{1-l_n})\right)~,
\end{eqnarray*}
for $r\in[l_n,s_n]$. Note that on $[l_n,s_n]$ the right-hand side is greater than $0$ and less than $(1-l_n)/l_n+1$ (which is uniformly bounded for sufficiently large $n$, since $l_n\rightarrow 1$). So the restriction of $v^*_n$ to $[l_n,s_n]$ satisfies \eqref{eq:polar}.

Similarly,
\begin{eqnarray*}
 \frac{(1-r^2)^2}4 \left(\frac{g'_n(r)}r+g''_n(r)\right) & = & \frac{(1+r)^2}8\left(- 1+\frac{(1-r)^2}r(\frac{1}{1-t_n}-\frac{1}{1-r}) \right)~,
\end{eqnarray*}
for $r\in [s_n,t_n]$.  Note that on $[s_n, t_n]$ the right-hand side is greater than $-1/2$ and less than the following:
$$\frac{1-s_n}{1-t_n}\cdot \frac{t_n-s_n}{s_n}~,$$
which is uniformly bounded for sufficiently large $n$, since by definition 
$1-t_n$ is proportional to $1-s_n$, and also $0<s_n<t_n<1$ with $s_n\rightarrow 1$. So the restriction of $v^*_n$ to $[s_n,t_n]$ satisfies \eqref{eq:polar}. Moreover, the restriction of $v^*_n$ to $[0,l_n]$ and $[t_n,1)$ clearly satisfies \eqref{eq:polar}.

With an additional smoothing by convolution, $v^*_n$ can be replaced by a smooth function while changing its first and second derivatives by at most a small constant. Since the convolution of two functions has the derivative property that $(f\ast g)^{(n)}=f^{(n)}\ast g$ for all $n\in\N$, one can choose an appropriate convolution such that the new obtained function (still denote by $v^*_n$) satisfies (i)-(iv). This concludes the proof of Step 1.

\textbf{Step 2}: Let $w^*_n:=\max\{u^*,v^*_n\}$. Then $w_n$ is equal to $u^*$ on $B(0,l_n)$, and equal to $c^*_1$ outside $B(0,t_n)$. So the metric $e^{2w^*_n}h_0$ has constant curvature $-1/e^{2c^*_1}$ outside $B(0,t_n)$. At each point $z\in B(0,t_n)$ where $u^*$ and $v^*_n$ take the same value, $w^*_n(z)<c^*_1$ and might not be differentiable. However since $w^*_n$ is defined as the maximum of two functions, the restriction of $w^*_n$ to any curve 
through a non-differentiable point (say $z$) of $w^*_n$ in the locus where $u^*=v^*_n$
has a non-negative jump in its first derivative 
at $z$. This implies that $(\partial_{xx}+\partial_{yy})w^*_n$ is bounded from below by a non-negative constant (in a distributional sense). We can therefore smooth $w^*_n$ (e.g. by convolution) in a neighborhood $U_z$ (which is contained in $B(0,t_n)$) of the locus where $u^*=v^*_n$, to obtain a new function (denoted by $\hat{w}^*_n$) such that $\hat{w}^*_n$ is smooth, strictly convex and
satisfies the following in $U_z$ (with value in $[c^*_0,c^*_1]$):
$$ C-1\geq \frac {(1-r^2)^2}{4}(\partial_{xx}+\partial_{yy})\hat{w}_n^* \geq -1~,$$
for a constant $C>1$.

Therefore, we obtain a smooth metric $e^{2\hat{w}^*_n}h_0$ on $\DD^2$ with curvature $-C/\epsilon\leq K_{e^{2\hat{w}^*_n}h_0}\leq 0$, which coincides with $h^*$ on $\mathring{B}(0,l_n)$ and has constant curvature $-1/e^{2c^*_1}$ outside $B(0,t_n)$.


\textbf{Step 3}: Now we construct $h^*_n$ in the following way: for any $\gamma_n\in\Gamma_n$, we push forward the metric $e^{2\hat{w}^*_n}h_{0}$ on $P_n$ to $\rho^*_n(\gamma_n)P_n$ by $\rho^*_n(\gamma_n)$. Note that the metric on each $\rho^*_n(\gamma_n)P_n$ extends to the boundary $\partial(\rho^*_n(\gamma_n)P_n)$ in the $C^{\infty}$-sense, 
since the metric $e^{2\hat{w}^*_n} h_0$ on $P_n$ is a scaling of $h_0$ by a constant at the points sufficiently close to $\partial P_n$ and we are translating the metric on $P_n$ to $\rho^*_n(\gamma_n)P_n$ by isometries of $h_0$.

By construction, $h^*_n$ is $\rho^*_n$-invariant. Moreover, for each compact subset $K'\subset\DD^2$, there is $N'=N'(K')>0$ sufficiently large, such that $K'\subset \mathring{B}(0,l_{N'})$, and thus $h^*_n=h^*$ on $K'$ for all $n\geq N'$. Therefore, $h^*_n$ uniformly converges to $h^*$ smoothly on compact subsets of $\DD^2$.
\end{proof}

\subsection{Fuchsian approximations of metrics with $K>-1$}
\begin{proposition}\label{prop:approx I}
Let $h$ be a smooth, complete, bounded
conformal metric of curvature  $K_h\in[-1+\epsilon,1/\epsilon]$ on $\DD^2$ and let $(\rho_n)_{n\in\N^+}$ be a sequence of Fuchsian representations of surface groups $\Gamma_n:=\pi_1(S_n)$ whose fundamental domain contains a closed hyperbolic disk $B_H(0,\tau_n)$ centered at $0$ with radius $\tau_n$, where $\tau_n$ strictly increases to $\infty$. Then there is a sequence of $\rho_n$-invariant conformal metrics, say $h_n$, on $\DD^2$, with curvature $K_{h_n}\in [-1+\epsilon',1/\epsilon']$ for some $\epsilon'>0$, which uniformly converges to $h$ smoothly on compact subsets of $\DD^2$.
\end{proposition}

\begin{proof}
Let $P_n$ denote the fundamental domain of $\rho_n$ that contains the closed hyperbolic disk $B_H(0,\tau_n)$. Let $u:\DD^2\rightarrow \R$ be a smooth function such that $h=e^{2u}h_0$. Note that $h$ is a complete, bounded (see Definition \ref{def:bounded conf}) conformal metric with curvature at least $-1+\epsilon$. Then there exist constants $c_1>c_0:=-\log\sqrt{1-\epsilon}>0$ such that $c_0\leq u\leq c_1$ (see \cite[Theorem 2.8]{mateljevic}).

We are going to define a real-valued function on $\DD^2$, called $v_n$, which is used to construct the desired metric $h_n$.
If $u$ is constant on $\DD^2$, we set $v_n:=u$. If $u$ is not constant on $\DD^2$, without loss of generality, we assume that $u$ is not constant on $B_H(0,\tau_1)$. We construct the desired $h_n$ using the following steps:

\textbf{Step 1}:
For each $n$, we construct a smooth function $v_n:\DD^2\rightarrow\R$ satisfying the following conditions:
\begin{enumerate}[(i)]
  \item $c_0<v_n\leq c_1$ on $\mathring{B}(0,t_n)$, where $\mathring{B}(0,t_n)$ denote the open euclidean disk of radius $t_n$ and $t_n=\tanh(\tau_n/2)$.
  \item $v_n$ is equal to $c_0$ outside $\mathring{B}(0,t_n)$, and equal to $c_1$ in $B(0,l_n)$, for a sequence $(l_n)_{n\in \N}$ (to be determined) such that $l_n<t_n$ and $l_n\to 1$ as $n\to \infty$.
  \item  $v_n$ is rotationally symmetric (i.e. $v_n$ is invariant under the rotation centered at $0$).
  \item On $\mathring{B}(0,t_n)\setminus B_H(0,l_n)$, there is a constant $C>1$ such that the (Gaussian) curvature of $e^{2v_n}h_0$ satisfies
      $$-1+(\epsilon/C)\leq K_{e^{2v_n}h_0}=e^{-2v_n}(\Delta_{h_0}v_n-1)\leq C/\epsilon,$$
      for $n$ sufficiently large. 
\end{enumerate}

Similarly as the construction in Proposition \ref{prop:approx III}, we will construct $v_n$ as a function on $|z|:=r\in[0,1)$ -- which is then extended to a rotationally invariant function on $\DD^2$ satisfying conditions (i)--(iv). Note that $v_n$ satisfies condition (i) with $v_n\geq c_0>0$ and thanks to the expression of the Laplace operator for radial functions, to show condition (iv), it suffices to show that there exists a constant $C>1$ (to be determined later) such that
\begin{equation}
  \label{eq:polar-hyp}
   e^{2c_0}(1-\epsilon/C)-1\geq \frac {(1-r^2)^2}{4}(v_n''(r) + \frac 1r v'_n(r)) \geq -(C/\epsilon)e^{2c_0}-1~,
\end{equation}
for $n$ sufficiently large.
Denote $a:=\epsilon/(1-\epsilon)$. Replacing $c^*_0, c^*_1$ by 
$c_0/a, c_1/a$
in \eqref{eq:log} and using the corresponding $l_n, s_n, t_n$ as defined in the proof of Proposition \ref{prop:approx III} (with $c^*_0, c^*_1$ replaced by $c_0/a, c_1/a$), we construct $v_n$ below:
\begin{itemize}
\item For $r\in [0,l_n]$, $v_n(r)=c_1$~,
\item for $r\in (l_n,s_n)$, 
    $v_n(r)=f_n(r)=(a/2)(\log((1-r)/(1-l_n))+(r-l_n)/(1-l_n))+c_1$~,
\item for $r\in [s_n,t_n)$~, 
    $v_n(r)=g_n(r)=(a/2)(\log((1-t_n)/(1-r)) -(r-t_n)/(1-t_n))+c_0$~,
\item for $r\in [t_n,1)$, $v_n(r)=c_0$~.
\end{itemize}
One can check by a direct computation that $f_n(l_n)=c_1$, $g_n(t_n)=c_0$ and $f_n(s_n)=g_n(s_n)$. Furthermore,
$$ 
f'_n(r) = \frac{a}{2}\left( -\frac 1{1-r} + \frac {1}{1-l_n}\right)~,\quad g'_n(r) = \frac{a}{2}\left(\frac {1}{1-r}-\frac 1{1-t_n}\right)~. $$
Then $f'_n(l_n)=g'_n(t_n)=0$  and $f'_n(s_n)=g'_n(s_n)$. Moreover,
$$ 
f''_n(r)=-\frac a{2(1-r)^2}~,\quad g''_n(r) = \frac a{2(1-r)^2}~. $$
Recall that $a=\epsilon/(1-\epsilon)$. It follows that
\begin{eqnarray*}
\frac{(1-r^2)^2}4 \left(\frac{f'_n(r)}r+f''_n(r)\right) & = &  
\frac{\epsilon(1+r)^2}{8(1-\epsilon)}\left(-1+\frac{(1-r)^2}r\left(\frac{1}{1-l_n}-\frac{1}{1-r}\right) \right)~,
\end{eqnarray*}
for $r\in[l_n,s_n]$. Note that on $[l_n,s_n]$ the right-hand side is less than $0$ and greater than 
the following:
$$
-\frac{\epsilon}{1-\epsilon}(\frac{1}{2}+\frac{1-l_n}{1-s_n}\cdot\frac{s_n-l_n}{2l_n})~,$$
which is uniformly bounded for sufficiently large $n$, using the facts that $1-l_n$ is proportional to $1-s_n$ (by definition), $0<l_n<s_n<1$ and $l_n\rightarrow 1$. In particular, there exists a constant $C_0>0$, such that $$-\frac{\epsilon}{1-\epsilon}(\frac{1}{2}+\frac{1-l_n}{1-s_n}\cdot\frac{s_n-l_n}{2l_n})
\geq -\frac{\epsilon}{1-\epsilon}(\frac{1}{2}+C_0)~.$$
On the other hand,
\begin{eqnarray*}
   \frac{(1-r^2)^2}4 \left(\frac{g'_n(r)}r+g''_n(r)\right) & = & \frac{\epsilon(1+r)^2}{8(1-\epsilon)} \left(1+ \frac{1-r}r \left(1- \frac{1-r}{1-t_n}\right)\right)~,
\end{eqnarray*}
for $r\in[s_n,t_n]$.
Note that on $[s_n,t_n]$ the right-hand side is less than 
$\frac{\epsilon}{2(1-\epsilon)}$ and greater than the following:
$$-\frac{\epsilon}{1-\epsilon}(\frac{1-s_n}{1-t_n}\cdot \frac{t_n-s_n}{2s_n}-\frac{1}{8})~,$$
which is uniformly bounded for sufficiently large $n$, using the facts that $1-t_n$ is proportional to $1-s_n$ (by definition), $0<s_n<t_n<1$ and $s_n\rightarrow 1$. In particular, there exists a constant $C_1>0$, such that
$$-\frac{\epsilon}{1-\epsilon}(\frac{1-s_n}{1-t_n}\cdot \frac{t_n-s_n}{2s_n}-\frac{1}{8})\geq-\frac{\epsilon}{1-\epsilon}(C_1-\frac{1}{8}) ~.$$
Note that $e^{2c_0}=1/(1-\epsilon)$ and $\epsilon\in(0,1)$. Take $C=\max\{C_0+2,C_1+2\}$, we have
\begin{equation*}
-\frac{\epsilon}{1-\epsilon}\max\{\frac{1}{2}+C_0, C_1-\frac{1}{8}\}>\frac{-(C+\epsilon-\epsilon^2)}{\epsilon(1-\epsilon)}=-(C/\epsilon)e^{2c_0}-1~,
\end{equation*}
\begin{equation*}
0<\frac{\epsilon}{2(1-\epsilon)}<\frac{\epsilon}{1-\epsilon}\cdot\frac{C-1}{C}=e^{2c_0}(1-\epsilon/C)-1~.
\end{equation*}
So the restriction of $v_n$ to $[l_n,t_n]$ satisfies \eqref{eq:polar-hyp}.

As discussed in Proposition \ref{prop:approx III}, one can apply to $v_n$ an appropriate convolution such that the new function (still denote by $v_n$) satisfies (i)-(iv). This concludes the proof of Step 1.

\textbf{Step 2}: Let $w_n:=\min\{u,v_n\}$. Then $w_n$ is equal to $u$ on $B(0,l_n)$, and equal to $c_0$ outside $B(0,t_n)$. So the metric $e^{2w_n}h_0$ has constant curvature $-1+\epsilon$ outside $B(0,t_n)$. At each point $z\in B(0,t_n)$ where $u$ and $v_n$ take the same value, $w_n(z)>c_0$ but might not be differentiable there. However since $w_n$ is defined as the minimum of two functions, the restriction of $w_n$ to any curve through a non-differentiable point (say $z$) of $w_n$ in the locus where $u=v_n$ has a non-positive jump in its first derivative at $z$. This implies that $(\partial_{xx}+\partial_{yy})w_n$ is bounded from above by a non-positive constant (in a distributional sense). We can therefore smooth $w_n$ (e.g. by convolution) in 
a neighborhood $U_z$ (which is contained in $B(0,t_n)$) of the locus where $u=v_n$, to obtain a new obtained function denoted by $\hat{w}_n$) such that $\hat{w}_n$ is smooth, strictly concave and satisfies the following in $U_z$ (with value in $[c_0,c_1]$):
$$ e^{2c_0}(1-\epsilon/C)-1\geq \frac {(1-r^2)^2}{4}(\partial_{xx}+\partial_{yy})\hat{w}_n \geq -(C/\epsilon)e^{2c_0}-1~,$$
for a constant $C>1$.

Therefore, we obtain a smooth metric $e^{2\hat{w}_n}h_0$ on $\DD^2$ with curvature $-1+(\epsilon/C)\leq K_{e^{2\hat{w}_n}h_0}\leq  C/\epsilon$, which coincides with $h$ on $\mathring{B}(0,l_n)$ and has constant curvature $-1+\epsilon$ outside $B(0,t_n)$.

\textbf{Step 3}: Now we construct $h_n$ in the following way: for any $\gamma_n\in\Gamma_n$, we push forward the metric $e^{2\hat{w}_n}h_{0}$ on $P_n$ to $\rho_n(\gamma_n)P_n$ by $\rho_n(\gamma_n)$. Note that the metric on each $\rho_n(\gamma_n)P_n$ extends to the boundary $\partial(\rho_n(\gamma_n)P_n)$ in the $C^{\infty}$-sense, since the metric $e^{2\hat{w}_n} h_0$ on $P_n$ is a scaling of $h_0$ by a constant at the points sufficiently close to $\partial P_n$ and we are translating the metric on $P_n$ to $\rho_n(\gamma_n)P_n$ by isometries of $h_0$.

By construction, $h_n$ is $\rho_n$-invariant. Moreover, for each compact subset $K\subset\DD^2$, there is $N=N(K)>0$ sufficiently large, such that $ K\subset\mathring{B}(0,l_N)$, and thus $h_n=h$ on $K$ for all $n\geq N$. Therefore, $h_n$ uniformly converges to $h$ smoothly on compact subsets of $\DD^2$.
\end{proof}

We remark here that the sequence of metrics $(h^*_n)_{n\in\N^+}$ (resp. $(h_n)_{n\in\N^+}$) constructed in Proposition \ref{prop:approx III} (resp. Proposition \ref{prop:approx I}) to approximate the prescribed bounded conformal metric $h^*$ (resp. $h$) are {\em uniformly bounded} conformal metrics, in the sense that $h^*_n\leq e^{2c^*_1}h_0$ (resp. $h_n\leq e^{2c_1}h_0$) for all $n\in\N^+$, where $c^*_1>0$ (resp. $c_1>0$) depends only on $h^*$ (resp. $h$).

\subsection{Proof of Theorem \ref{tm:main-iv}}\label{subsec: proof-Thm 1.2}

Let $D_-, D_+\subset\CP^1$ be disjoint quasi-disks with $\partial D_-=\partial D_+$, denoted by $C$ for simplicity, which is oriented such that $D_+$ and $D_-$ lie on the left and right sides of $C$. By definition, $C$ is a quasi-circle. Let $\epsilon\in(0,1)$ and let $h_-$ (resp. $h_+$) be a smooth, complete, bounded conformal metric of curvature varying in $[-1+\epsilon,1/\epsilon]$ on $D_-$ (resp. $D_+$), and $h^*_-$ (resp. $h^*_+$) be a smooth, complete, bounded conformal metric of curvature varying in $[-1/\epsilon,0]$ on $D_-$ (resp. $D_+$).

\subsubsection{Equivalent statements of Theorem \ref{tm:main-iv}}

\begin{proposition}\label{prop:equiv}
Let $\Phi_C$ be the classical conformal welding of $C$. The statements (i)-(v) in Theorem \ref{tm:main-iv} are respectively equivalent to the following:
\begin{enumerate}[(I)]
  \item There is a convex domain $\Omega\subset\HH^3$ of first type with the induced metric on $\partial\Omega$ isometric to $h_-$, and the $I$-gluing map $\phi^I_{\Omega}$ being $\Phi_C$.

  \item There is a convex domain $\Omega\subset\HH^3$ of first type with the third fundamental form on $\partial\Omega$ isometric to $h^*_-$,  and the $\III$-gluing map $\phi^{\III}_{\Omega}$ being $\Phi_C$.

  \item There is a convex domain $\Omega\subset\HH^3$ of second type with the induced metric on $\partial_{\pm}\Omega$ isometric to $h_{\pm}$, and the $I$-gluing map $\phi^{I}_{\Omega}$ being $\Phi_C$.

\item There is a convex domain $\Omega\subset\HH^3$ of second type with the third fundamental form on $\partial_{\pm}\Omega$ isometric to $h^*_{\pm}$, and the $\III$-gluing map $\phi^{\III}_{\Omega}$ being $\Phi_C$.
 \item There is a convex domain $\Omega\subset\HH^3$ of second type with the induced metric (resp. third fundamental form) on $\partial_{-}\Omega$ (resp. $\partial_{+}\Omega$) isometric to $h_-$ (resp. $h^*_+$), and the $(\III,I)$-mixed gluing map $\phi^{(\III,I)}_{\Omega}$ being $\Phi_C$.

\end{enumerate}
\end{proposition}

\begin{proof}
We discuss the following two cases, according to the types of convex domains.

\textbf{Case 1}: For the case of convex domains of first type (including the equivalence for Statements (I)-(II)), we just need to show the equivalence between Statements (i) and (I), since the equivalence between Statements (ii) and (II) will follow in the same way.

We first show that Statement (i) implies Statement (I). Let $\Omega$ be the convex subset appearing in Statement (i). Since $\phi^*I=h_-$ on $D_-$, the induced metric on $\partial\Omega$ is isometric to $h_-$. Recall that $\Phi_C=(\partial\varphi^-_C)^{-1}\circ (\partial\varphi^+_C)$, where $\varphi^{\pm}_C:\HH^{2\pm}\rightarrow D_{\pm}$ is the biholomorphic map (that takes $0,1,\infty$ to the same three points of $C$). Note that $\varphi^{\pm}_C$ is well-defined up to left composition with a common M\"obius transformation.

Set $\varphi^+:=\phi\circ \varphi^+_C:\HH^{2+}\rightarrow \partial_{\infty}\Omega$ and $\varphi^-:=\phi\circ\varphi^-_C:\HH^{2-}\rightarrow \partial\Omega$. By assumption, $\phi$ is bi-holomorphic from $D_+$ to $\partial_{\infty}\Omega$, then $\varphi^+$ is bi-holomorphic. Since $\phi^*I=h_-$ on $D_-$, we have
$$(\varphi^-)^*(I)=(\phi\circ \varphi^-_C)^*(I)=(\varphi^-_C)^*(\phi^*I)=(\varphi^-_C)^*(h_-),$$
which is conformal to the hyperbolic metric $h^-_0$ (this follows from Lemma \ref{lm:conf-bihol} and the fact that $\varphi^-_C: \HH^{2-}\rightarrow D_-$ is bi-holomorphic and $h_-$ is a conformal metric on $D_-$). Note that
$$(\partial\varphi^-)^{-1}\circ\partial\varphi^+
=(\phi\circ\partial\varphi^-_C)^{-1}\circ(\phi\circ \partial\varphi^+_C)=(\partial\varphi^-_C)^{-1}\circ (\partial\varphi^+_C)=\Phi_C,$$
which is normalized. So $(\varphi^+,\varphi^-)$ is a representative pair for the $I$-gluing map $\Phi^I_{\Omega}$ of $\Omega$. Then there is an element $f_1\in\PSL_2(\C)$ such that $\varphi^c_{\Omega}=f_1\circ\varphi^+$ and $\varphi^I_{\Omega}=f_1\circ\varphi^-$.
Hence,
$$\Phi^I_{\Omega}
=(\partial\varphi^I_{\Omega})^{-1}\circ\partial\varphi^c_{\Omega}
=(f_1\circ\partial\varphi^-)^{-1}\circ(f_1\circ\partial\varphi^+)
=\Phi_{C}.$$
Therefore, $\Omega$ is the convex domain in Statement (I).

We now show that Statement (I) implies Statement (i).  Let $\Omega$ be the convex subset given in Statement (I). Recall that $\Phi^I_{\Omega}=(\partial\varphi^I_{\Omega})^{-1}\circ\partial\varphi^c_{\Omega}$ and $\Phi_C=(\partial\varphi^-_C)^{-1}\circ (\partial\varphi^+_C)$.
By Statement (I), $\Phi^I_{\Omega}=\Phi_C$, then
\begin{equation}\label{eq:condition-(I)}
(\partial\varphi^I_\Omega)\circ(\partial\varphi^-_C)^{-1}=\partial\varphi^c_{\Omega}\circ(\partial\varphi^+_C)^{-1}.
\end{equation}
We claim that $\Omega$ is a desired convex domain for Statement (i). Set $\phi:=\varphi^c_{\Omega}\circ(\varphi_C^+)^{-1}$ on $D_+$ and $\phi:=\varphi^I_{\Omega}\circ(\varphi^-_C)^{-1}$ on $D_-$. By \eqref{eq:condition-(I)}, $\phi$ is continuous on $\CP^1=D_+\cup C\cup D_-$. Moreover, $\phi$ is bi-holomorphic from $D_+$ to $\partial_{\infty}\Omega$. It remains to show that $\phi^*I=h_-$ on $D_-$. Note that
\begin{equation}\label{eq:pull-back}
\phi^*I=(\varphi^I_{\Omega}\circ(\varphi^-_C)^{-1})^*I
=((\varphi^-_C)^{-1})^*((\varphi^I_{\Omega})^*I).
\end{equation}
By the definition of $\varphi^I_{\Omega}: \HH^{2-}\rightarrow \partial\Omega$, it follows that $(\varphi^I_{\Omega})^*I$ is conformal to $h^-_0$, namely, $$(\varphi^I_{\Omega})^*I=e^{2u_I}h^-_0,$$
where $u_I:\HH^{2-}\rightarrow \R$ is a smooth function. By assumption, $h_-$ is a conformal metric on $D_-$, and $\varphi^-_C:\HH^{2-}\rightarrow D_-$ is bi-holomorphic. Then by Lemma \ref{lm:conf-bihol},
$$(\varphi^-_C)^*h_-=e^{2u_{h_-}}h^-_0,$$
where $u_{h_-}:\HH^{2-}\rightarrow \R$ is a smooth function. Since $I$ and $h_-$ are isometric by Statement (I), it then follows from Lemma \ref{lm:conf-bihol} that the isometry between $(\HH^{2-}, e^{2u_I}h^-_0)$ and $(\HH^{2-}, e^{2u_{h_-}}h^-_0)$ is a bi-holomorphic map on $\HH^{2-}$ (also a hyperbolic isometry). So this isometry is an identity and therefore $u_I=u_{h_-}$.

Combined with \eqref{eq:pull-back}, we obtain that
$$\phi^*I
=((\varphi^-_C)^{-1})^*(e^{2u_I}h^-_0)
=((\varphi^-_C)^{-1})^*(e^{2u_{h_-}}h^-_0)
=((\varphi^-_C)^{-1})^*((\varphi^-_C)^*h_-)=h_-.$$


\textbf{Case 2}:
For the case of convex domains of second type (including the equivalence for Statements (III)-(V)), we just need to show one of them, since the others will follow in the same philosophy.
We show the equivalence between Statements (iii) and (III) for instance.

We first show that Statement (iii) implies Statement (III). We can apply analogous analysis as Case 1. Let $\Omega$ be the convex subset appearing in Statement (iii). Set $\varphi^+:=\phi\circ \varphi^+_C:\HH^{2+}\rightarrow \partial_+\Omega$ and $\varphi^-:=\phi\circ\varphi^-_C:\HH^{2-}\rightarrow \partial_-\Omega$.  Let $I_{\pm}$ denote the induced metric on $\partial_{\pm}\Omega$. By Statement (iii), $\phi^*I_{\pm}=h_{\pm}$, so the induced metric on $\partial_{\pm}\Omega$ is isometric to $h_{\pm}$. Then $$(\varphi^{\pm})^*(I_{\pm})=(\phi\circ \varphi^{\pm}_C)^*(I_{\pm})=(\varphi^{\pm}_C)^*(\phi^*I_{\pm})=(\varphi^{\pm}_C)^*(h_{\pm}),$$
which is conformal to the hyperbolic metric $h^{\pm}_0$ (using Lemma \ref{lm:conf-bihol} again). Note that
$$(\partial\varphi^-)^{-1}\circ\partial\varphi^+
=(\phi\circ\partial\varphi^-_C)^{-1}\circ(\phi\circ \partial\varphi^+_C)=(\partial\varphi^-_C)^{-1}\circ (\partial\varphi^+_C)=\Phi_C,$$
which is normalized. So $(\varphi^+,\varphi^-)$ is a representative pair for the $I$-gluing map $\Phi^I_{\Omega}$ of $\Omega$. Then there is an element $f_3\in\PSL_2(\C)$ such that $\varphi^{I_+}_{\Omega}=f_3\circ\varphi^+$ and $\varphi^{I_-}_{\Omega}=f_3\circ\varphi^-$. Hence,
$$\Phi^I_{\Omega}
=(\partial\varphi^{I_-}_{\Omega})^{-1}\circ\partial\varphi^{I_+}_{\Omega}
=(f_3\circ\partial\varphi^-)^{-1}\circ(f_3\circ\partial\varphi^+)
=\Phi_{C}.$$
Therefore, $\Omega$ is the desired convex domain in Statement (III).

We now show that Statement (III) implies Statement (iii).  Let $\Omega$ be the convex subset given in Statement (III). Recall that $\Phi^I_{\Omega}=(\partial\varphi^{I_-}_{\Omega})^{-1}\circ\partial\varphi^{I_+}_{\Omega}$ and $\Phi_C=(\partial\varphi^-_C)^{-1}\circ (\partial\varphi^+_C)$.
By Statement (III), $\Phi^I_{\Omega}=\Phi_C$,
then
\begin{equation}\label{eq:condition-(III)}
(\partial\varphi^{I_-}_\Omega)\circ(\partial\varphi^-_C)^{-1}
=\partial\varphi^{I_+}_{\Omega}\circ(\partial\varphi^+_C)^{-1}.
\end{equation}
We claim that $\Omega$ is a desired convex domain for Statement (iii). Set $\phi:=\varphi^{I_+}_{\Omega}\circ(\varphi_C^+)^{-1}$ on $D_+$ and $\phi:=\varphi^{I_-}_{\Omega}\circ(\varphi^-_C)^{-1}$ on $D_-$. By \eqref{eq:condition-(III)}, $\phi$ is continuous on $\CP^1=D_+\cup C\cup D_-$. It remains to show that $\phi^*I_{\pm}=h_{\pm}$. Note that
\begin{equation}\label{eq:pull-back-III}
(\varphi^{I_{\pm}}_{\Omega}\circ(\varphi^{\pm}_C)^{-1})^*I_{\pm}
=((\varphi^{\pm}_C)^{-1})^*((\varphi^{I_{\pm}}_{\Omega})^*I_{\pm}).
\end{equation}
By the definition of $\varphi^{I_{\pm}}_{\Omega}$, it follows that $(\varphi^{I_{\pm}}_{\Omega})^*I_{\pm}$ is conformal to $h^{\pm}_0$, namely, $$(\varphi^{I_{\pm}}_{\Omega})^*I_{\pm}=e^{2u_{I_{\pm}}}h^{\pm}_0,$$
where $u_{I_{\pm}}:\HH^{2\pm}\rightarrow \R$ is a smooth function. By assumption, $h_{\pm}$ is a conformal metric on $D_{\pm}$, and $\varphi^{\pm}_C:\HH^{2\pm}\rightarrow D_{\pm}$ is bi-holomorphic. Then by Lemma \ref{lm:conf-bihol},
$$(\varphi^{\pm}_C)^*h_{\pm}=e^{2u_{h_{\pm}}}h^{\pm}_0,$$
where $u_{h_{\pm}}:\HH^{2\pm}\rightarrow \R$ is a smooth function. Since $I_{\pm}$ and $h_{\pm}$ are isometric by Statement (III), applying Lemma \ref{lm:conf-bihol} again as above, it follows that $u_{I_{\pm}}=u_{h_{\pm}}$. Combining \eqref{eq:pull-back-III}, we have
$$\phi^*I_{\pm}
=((\varphi^{\pm}_C)^{-1})^*(e^{2u_{I_{\pm}}}h^{\pm}_0)
=((\varphi^{\pm}_C)^{-1})^*(e^{2u_{h_{\pm}}}h^{\pm}_0)
=((\varphi^{\pm}_C)^{-1})^*((\varphi^{\pm}_C)^*h_{\pm})=h_{_{\pm}}.$$

\end{proof}

\subsubsection{Convex domains in quasifuchsian case}

\begin{proposition}\label{prop:quasifuchsian realization}
Let $\rho_-, \rho_+:\pi_1(S)\rightarrow \PSL_2(\R)$ be two Fuchsian groups and let $f\in\cT$ be a quasifuchsian  quasisymmetric homeomorphism which is $(\rho_+,\rho_-)$-equivariant. Let $h_{\pm}$ and $h^*_{\pm}$ be the metrics on $D_{\pm}$ given above. Assume that $h_-, h^*_-$ (resp. $h_+, h^*_+$ ) are $\rho_-$-invariant (resp. $\rho_+$-invariant). Then
\begin{enumerate}
  \item There exists a convex domain $\Omega\subset\HH^3$ of first type, invariant under a quasifuchsian representation from $\pi_1(S)$ to $\PSL_2(\C)$, with the induced metric on $\partial\Omega$ isometric to $h_-$ and the $I$-gluing map of $\Omega$ being $f$.
  \item There exists a convex domain $\Omega\subset\HH^3$ of first type, invariant under a quasifuchsian representation from $\pi_1(S)$ to $\PSL_2(\C)$, with the third fundamental form on $\partial\Omega$ isometric to $h^*_-$ and the $\III$-gluing map of $\Omega$ being $f$.
  \item There exists a convex domain $\Omega\subset\HH^3$ of second type, invariant under a quasifuchsian representation from $\pi_1(S)$ to $\PSL_2(\C)$, with the induced metric on $\partial_{\pm}\Omega$ isometric to $h_{\pm}$ and the $I$-gluing map of $\Omega$ being $f$.
  \item There exists a convex domain $\Omega\subset\HH^3$ of second type, invariant under a quasifuchsian representation from $\pi_1(S)$ to $\PSL_2(\C)$, with the third fundamental form on $\partial_{\pm}\Omega$ isometric to $h^*_{\pm}$ and the $\III$-gluing map of $\Omega$ being $f$.
   \item There exists a convex domain $\Omega\subset\HH^3$ of second type, invariant under a quasifuchsian representation from $\pi_1(S)$ to $\PSL_2(\C)$, with the induced metric (resp. third fundamental form) on $\partial_-\Omega$ (resp. $\partial_+\Omega$) isometric to $h_{-}$ (resp. $h^*_+$) and the $(\III,I)$-mixed gluing map of $\Omega$ being $f$.
\end{enumerate}

\end{proposition}

\begin{proof}
We show the proposition in two cases, according to the two types of convex domains:

\textbf{Case 1}: For the case of convex domains of first type (namely Statements (1)-(2)), it suffices to show Statement (1), since Statement (2) will follow analogously.

Let $c_+$ denote the conformal structure on $S$ invariant under the action of $\rho_+$. Since $h_-$ is invariant under the action of $\rho_-$, then $h_-$ descends to a smooth, complete metric on $S$, say $\bar{h}_-$. By Corollary \ref{cr:qf-I}, there is a unique (up to isometries) quasifuchsian hyperbolic manifold $M$ which contains a convex subset $N$ with induced metric on the boundary $\partial N$ isometric to $\bar{h}_-$ and the conformal structure on the boundary at infinity $\partial_{\infty}N$ conformal to $c_+$. We consider the lift of $N$ in the universal cover (which is identified with $\HH^3$) of $M$ and denote it by $\Omega$. Then the induced metric on $\partial\Omega$ is isometric to $h_-$.

By Example \ref{ex:gluing map I}, the $I$-gluing map $\Phi^I_{\Omega}:\RP^1\rightarrow\RP^1$ is a $(\rho_+,\rho_-)$-equivariant quasisymmetric homeomorphism. Note that $f:\RP^1\rightarrow\RP^1$ is also a $(\rho_+,\rho_-)$-equivariant quasisymmetric homeomorphism. With the normalized condition of gluing maps $\Phi^I_{\Omega}$ and the element $f$ in the universal Teichm\"uller space $\cT$, we have that $f=\Phi^I_{\Omega}$.

\textbf{Case 2}: For the case of convex domains of second type (namely Statements (3)-(5)), it suffices to show one of them, since the other statements will follow in a similar way.
We show Statement (4) for instance.

Note that $h^*_{\pm}$ is invariant under the action of $\rho_{\pm}$, $h^*_{\pm}$ descends to a smooth, complete metric on $S$, say $\bar{h}^*_{\pm}$. Applying Theorem \ref{tm:main-cc} to the case that $n=2$ with $M$ homeomorphic to the product of $S$ (homeomorphic to $\partial_1M$ and $\partial_2M$) by $[0,1]$, $\cI=\cJ=\emptyset$, $h^*_1=\bar{h}^*_{+}$ and $h^*_2=\bar{h}^*_{-}$,  there is a unique (up to isometries) quasifuchsian hyperbolic manifold $M$ which contains a convex subset $N$ with the third fundamental form on the boundary $\partial_{\pm}N$ isometric to $\bar{h}^*_{\pm}$. We consider the lift of $N$ in the universal cover (which is identified with $\HH^3$) of $M$ and denote it by $\Omega$. Then the third fundamental form on $\partial_{\pm}\Omega$ is isometric to $h^*_{\pm}$.

By Example \ref{ex:gluing map II}, the $\III$-gluing map $\Phi^{\III}_{\Omega}:\RP^1\rightarrow\RP^1$ is a $(\rho_+,\rho_-)$-equivariant (normalized) quasisymmetric homeomorphism. Since $f:\RP^1\rightarrow\RP^1$ is also a $(\rho_+,\rho_-)$-equivariant (normalized) quasisymmetric homeomorphism, we have that $f=\Phi^{\III}_{\Omega}$.
\end{proof}

\subsubsection{Approximation of convex domains}

Let $(f_n)_{n\in\N^+}$ be a sequence of $(\rho^+_n,\rho^-_n)$-equivariant quasisymmetric homeomorphisms (where $\rho^{\pm}_n$ are Fuchsian representations of a surface group $\Gamma_n$).
Let $h_{\pm}$ be a smooth, complete, bounded conformal metrics of curvature varying in $[-1+\epsilon,1/\epsilon]$ on $\DD^2$ and let $(h^{\pm}_n)_{n\in\N^+}$ be a sequence of $\rho^+_n$ (resp. $\rho^-_n$)-invariant 
uniformly bounded conformal metrics (i.e. $h^{\pm}_n\leq e^{2c_1}h_0$ for a constant $c_1>0$) with curvature varying in $[-1+\epsilon/C,C/\epsilon]$ on $\DD^2$ for a constant $C>1$. We first give the following statements concerning the convergence of convex domains of first type.

\begin{proposition}\label{prop:convergence-I}
  Let $(\iota_n: (\DD^2,h^-_n)\rightarrow \HH^3)_{n\in\N^+}$ be a sequence of isometric immersions whose image together with a part of $\partial_{\infty}\HH^3$ bounds a convex domain of first type, which is invariant under a quasifuchsian representation of $\Gamma_n$ with $I$-gluing map $f_n$.
  Assume that
 \begin{enumerate}[(i)]
 \item $(\iota_n(0))_{n\in\N^+}$ is contained in a compact subset of $\HH^3$,
 \item the maps $f_n$ are uniformly quasi-symmetric,
 \item $(f_n)_{n\in\N^+}$ converges to $f\in\cT$ in the $C^0$-sense, and
 \item $(h^-_n)_{n\in\N^+}$ converges $C^\infty$ to $h_-$ on compact subsets of $\DD^2$.
 \end{enumerate}
Then  $(\iota_n)_{n\in\N^+}$ has a subsequence which converges smoothly on compact subsets to an isometric immersion $\iota:(\DD^2,h_-)\rightarrow\HH^3$ whose image together with a part of $\partial_{\infty}\HH^3$ bounds a convex domain of first type, with $I$-gluing map $f$.
\end{proposition}

The proof of this proposition will use the following two lemmas.

\begin{lemma} \label{lm:injectivity}
  Let $h$ be a complete conformal metric on $\DD^2$ such that
  \begin{itemize}
  \item $h$ is bounded, in the sense that there exists a constant $c>1$ such that $h\leq c h_{0}$, where $h_0$ is the hyperbolic metric on $\DD^2$,
  \item $h$ has curvature $K_h\in [-1, 1/\epsilon]$ for some $\epsilon>0$.
  \end{itemize}
  Then there exists $l_0>0$, depending only on $c$ and $\epsilon$, such that the injectivity radius of $h$ is at least $l_0$.
\end{lemma}

\begin{proof}
   We first consider the case where $h$ admits a simple closed geodesic on $\DD^2$. Let $\gamma$ be a simple closed geodesic in $(\DD^2, h)$ of length $l$. Since $\DD^2$ is simply connected, $\gamma$ bounds a disk $\Omega\subset \DD^2$.

 Note that $h$ is complete and has curvature at least $-1$. Therefore $h\geq h_0$ (see \cite[Theorem 2.8]{mateljevic}). So the length of $\gamma$ for $h_{0}$ is bounded from above: $L_{h_{0}}(\gamma)\leq L_h(\gamma)=l$.

  The hyperbolic isoperimetric inequality then implies that the area of $\Omega$ for $h_{0}$ is also bounded from above by
  $$ A_{h_{0}}(\Omega) \leq L_{h_{0}}(\gamma) \leq l~. $$
  Since $h$ is bounded, it then follows that
  $$ A_h(\Omega) \leq c A_{h_{0}}(\Omega) \leq cl~. $$
  It then follows that the total curvature of $\Omega$ for $h$ is at most $cl/\epsilon$.

  However the Gauss-Bonnet theorem implies that the total curvature of $\Omega$ for $h$ must be equal to $2\pi$, since $\gamma$ is geodesic for $h$. Thus we finally obtain that
  $$ 2\pi \leq \frac{cl}\epsilon~, $$
  so that $l\geq 2\pi \epsilon/c$, and the injectivity radius of $(\DD^2,h)$ is at least $\pi \epsilon/c$. If $h$ admits no simple closed geodesics on $\DD^2$, since $h$ is complete on $\DD^2$,  then its injectivity radius is infinite and is naturally at least $\pi \epsilon/c$ (which is a constant provided in the above case). So we can take $l_0=\pi \epsilon/c$ as the lower bound of injectivity radius of $h$. This concludes the lemma.
\end{proof}

\begin{lemma} \label{lm:distance}
  Let $M$ be a quasifuchsian hyperbolic manifold, and let $\Sigma\subset M$ be a closed, locally convex surface embedded in $M$ such that the induced metric $I$ on $\Sigma$
  \begin{itemize}
  \item is bounded, in the sense that $I\leq ch_{-1}$, where $c>0$ and $h_{-1}$ is the hyperbolic metric conformal to $I$,
  \item it has curvature $K_{I}\in [-1+\epsilon,1/\epsilon]$, where $\epsilon>0$.
  \end{itemize}
  Then there exists $r>0$ (depending only on $c$ and $\epsilon$) such that all points of $\Sigma$ are at distance at most $r$ from the convex core of $M$.
\end{lemma}

\begin{proof}
  Let $\partial^+_{\infty}M$ and $\partial^-_{\infty}M$ denote the upper and lower boundary components at infinity of $M$ and let $\partial^+C_M$ and $\partial^-C_M$ denote the upper and lower boundary components of the convex core $C_M$ of $M$. Note that the surface $\Sigma$ is a locally convex surface embedded in $M$ with curvature $K_{I}\in [-1+\epsilon,1/\epsilon]$. Without loss of generality, we assume that $\Sigma$ lies in the hyperbolic end with the boundary at infinity $\partial^+_{\infty}M$. Then $\Sigma$ lies strictly above $\partial^+C_M$.

  Let $I^*$ be the metric defined in $\Sigma$ as
  $$ I^* = \frac 12(I+2\II+\III)~, $$
  where $\II$ and $\III$ are the second and third fundamental forms on $\Sigma$, respectively. The metric $I^*$ is conformal to the pull-back on $\Sigma$ of the conformal metric at infinity by the hyperbolic Gauss map $G: \Sigma\rightarrow\partial^+_{\infty}M$, see \cite[Theorem 5.8]{volume}. We also denote by $h^*_{-1}$ the hyperbolic metric on $\Sigma$ conformal to $I^*$. The push-forward metric $G_*h^*_{-1}$ is therefore the hyperbolic metric in the conformal class on $\partial^{+}_{\infty}M$ .

  We first prove that there is a constant $c_0>1$ such that
  \begin{equation}
    \label{eq:comp1}
    I^*\leq c_0h^*_{-1}.
  \end{equation}
  This follows from the following points.
  \begin{itemize}
  \item It follows from Lemma \ref{lm:injectivity} that the injectivity radius of $I$ is bounded from below by a fixed positive constant. It therefore follows from  Lemma \ref{lm:principal-curv} that there exists a constant that bounds from above the principal curvatures of $\Sigma$.
  \item As a consequence, there exists a constant $c_1>1$ such that $(1/c_1)I\leq I^*\leq c_1I$ on $\Sigma$.
  \item This implies that $I^*$ is uniformly quasi-conformal to $I$, and it follows that there exists a constant $c_2>1$ such that $(1/c_2) h_{-1}\leq h^*_{-1}\leq c_2h_{-1}$.
  \item 
  By the assumption that $I$ is complete with curvature $K_I\geq -1$ and is bounded, we have $h_{-1}\leq I\leq ch_{-1}$.
  \end{itemize}
  Putting those inequalities together yields the desired inequality \eqref{eq:comp1}.

  Now $\Sigma$ is equal to the Epstein surface (see \cite{epstein:envelopes}) associated to $G_*I^*$ -- this is the point of view developed in \cite{horo}. Let $h_{Th}$ denote the Thurston metric on $\partial^+_{\infty}M$. The Epstein surface of $h_{Th}$ is exactly the upper boundary component $\partial^+C_M$ of the convex core $C_M$, and the canonical hyperbolic metric in the conformal class of $h_{Th}$ is equal to $G_*h^*_{-1}$. 
By the fact that $G_*h^*_{-1}\leq h_{Th}$ (see \cite[Chapter 1, p. 1]{anderson1998}) and Equation \eqref{eq:comp1}, we obtain
$$ G_*I^*\leq c_0G_*h^*_{-1}\leq c_0h_{Th}~. $$
Using the fact that the equidistant surface at distance $t$ (where the positive direction is the normal direction towards $\partial^+_{\infty}M$) from the Epstein surface of a conformal metric say $h$ on $\partial^+_{\infty}M$ is the Epstein surface of the metric $e^{2t}h$ (see \cite[Theorem 2.1]{epstein:envelopes}), the Epstein surface of $c_0h_{Th}$ is the $(\log c_0)/2$-equidistant surface from $\partial^+C_M$ (as the Epstein surface of $h_{Th}$). We recall another fact here that for any two conformal metrics say $h_1$, $h_2$ on $\partial^+_{\infty}M$ with $h_1\leq h_2$ and their Epstein surfaces immersed in $M$, the Epstein surface of $h_1$ lies below the Epstein surface of $h_2$, see \cite[Lemma 3.12]{compare}.

Combining the above facts, the locally convex surface $\Sigma$ (as the Epstein surface of $G_*I^*$) lies between $\partial^+C_M$ and the $(\log c_0)/2$-equidistant surface from $\partial^+C_M$. The result follows.
\end{proof}

\begin{proof}[Proof of Proposition \ref{prop:convergence-I}]
  By assumption, the sequence of metrics $h^-_n$ are uniformly bounded conformal metrics (i.e. $h^-_n\leq e^{2c_1}h_0$ for a constant $c_1>0$) and have curvature varying in $[-1+\epsilon/C, C/\epsilon]$ for a constant $C>1$. Combined with Lemma \ref{lm:injectivity}, the injectivity radius of $h^-_n$ are uniformly bounded from below by a constant depending only on $c_1$, $\epsilon$ and $C$.

  Since $h^-_n$ is isometric to the induced metric on $\iota_n(\DD^2)$, it follows from Lemma \ref{lm:principal-curv} that the principal curvatures of $\iota_n(\DD^2)$ are uniformly bounded between two positive constants (depending only on $c_1$, $\epsilon$ and $C$). Combined with Theorem \ref{thm:Labourie} and assumption (i), $\iota_n$ converges smoothly on compact subsets to an isometric immersion $\iota: (\DD^2, h_-)\rightarrow \HH^3$. We denote by $\Omega_n$ (resp. $\Omega$) the convex connected component of $\HH^3\setminus\iota_n(\DD^2)$ (resp. $\HH^3\setminus\iota(\DD^2)$), and denote by $I_n$ (resp. $I$) the induced metric on $\partial\Omega_n$ (resp. $\partial\Omega$), which is isometric to $h^-_n$ (resp. $h_-$).

  We claim that $\partial_{\infty}\Omega$ is a closed quasi-disk in $\CP^1$ (which implies that $\Omega$ is a convex domain of first type). To see this, note that the principal curvatures of $\partial\Omega_n=\iota_n(\DD^2)$ are uniformly bounded between two positive constants (as shown above). This implies that the hyperbolic Gauss map say $G_n$ from $\partial\Omega_n$ to $\CP^1\setminus \partial_{\infty}\Omega_n$ is uniformly quasi-conformal, by comparing the induced metric $I_n$ on $\partial\Omega_n$ (which is isometric to $h^-_n$) with the metric $I_n+2\II_n+\III_n$ (whose push-forward metric by $G_n$ is a conformal metric on $\CP^1\setminus \partial_{\infty}\Omega_n$), where $\II_n$, $\III_n$ are the second and third fundamental forms of $\partial\Omega_n$.
  Note that $G_n$ is a homeomorphism and its boundary map $\partial G_n$ is the identity from $\partial_{\infty}(\partial\Omega_n)$ to itself.
  Recall that the $I$-gluing map of $\Omega_n$ is $\Phi^I_{\Omega_n}=f_n=(\partial\varphi^I_{\Omega_n})^{-1}\circ\partial\varphi^c_{\Omega_n}$, where $\varphi^c_{\Omega_n}: \HH^{2+}\rightarrow \partial_{\infty}\Omega_n\subset\CP^1$ is a bi-holomorphic map, while $\varphi^I_{\Omega_n}:\HH^{2-}\rightarrow \partial\Omega_n$ is a map such that $(\varphi^I_{\Omega_n})^*(I_n)$ is conformal to $h^-_0$.  Let $\varphi^{c-}_{\Omega_n}: \HH^{2-}\rightarrow \CP^1\setminus \partial_{\infty}\Omega_n$ be a bi-holomorphic map such that $(\partial\varphi^{c-}_{\Omega_n})^{-1}\circ\partial\varphi^c_{\Omega_n}$ is the conformal welding of the Jordan curve  $\partial_{\infty}(\partial\Omega_n)$. In particular, we have
  \begin{equation*}
  \begin{split}
  f_n=\Phi^I_{\Omega_n}&=(\partial\varphi^I_{\Omega_n})^{-1}\circ \partial\varphi^{c-}_{\Omega_n} \circ\big((\partial\varphi^{c-}_{\Omega_n})^{-1}\circ\partial\varphi^c_{\Omega_n}\big)\\
 &=(\partial\varphi^I_{\Omega_n})^{-1}\circ \partial G^{-1}_n\circ\partial\varphi^{c-}_{\Omega_n} \circ\big((\partial\varphi^{c-}_{\Omega_n})^{-1}\circ\partial\varphi^c_{\Omega_n}\big)\\
 &=\partial\big((\varphi^I_{\Omega_n})^{-1}\circ G^{-1}_n\circ\varphi^{c-}_{\Omega_n}\big) \circ\big((\partial\varphi^{c-}_{\Omega_n})^{-1}\circ\partial\varphi^c_{\Omega_n}\big)
  \end{split}
  \end{equation*}
 By the above discussion, $(\varphi^I_{\Omega_n})^{-1}\circ G^{-1}_n\circ\varphi^{c-}_{\Omega_n}$ is a uniformly quasiconformal map (see e.g. \cite[Proposition 3.14]{convexhull}) and its boundary map is thus a uniformly quasisymmetric homeomorphism. It then follows from hypothesis (ii) that the conformal welding of $\partial_{\infty}(\partial\Omega_n)$ is uniformly quasisymmetric, so by Theorem \ref{thm:Ahlfors-Bers}, the $\partial_\infty \iota_n(\DD^2)=\partial_{\infty}(\partial\Omega_n)$ are quasi-circles with a quasi-conformal constant which is bounded from above by a fixed constant.

  The sequence of quasi-circles $\partial_\infty\iota_n(\DD^2)$, $n\in\N^+$, which are uniform quasi-circles, has a subsequence converging either to a quasi-circle or to a point. However, it cannot converge to a point, since otherwise the distance between $\iota_n(\DD^2)$ and the boundary of the convex hull of the corresponding quasifuchsian manifolds would go to $\infty$, contradicting Lemma \ref{lm:distance}.

  Since  $\partial_{\infty}\Omega\subset\CP^1$ has more than one point, by Carath\'{e}odory's theorem, $(\varphi^c_{\Omega_n})_{n\in\N^+}$ has a subsequence $(\varphi^c_{\Omega_{n_k}})_{k\in\N^+}$ converging to a bi-holomorphic map, say $\varphi^c_{\Omega}:\HH^{2+}\rightarrow \partial_{\infty}\Omega$. Combined with Lemma \ref{lm:conf-bihol} and assumption (iv), the corresponding subsequence $(\varphi^I_{\Omega_{n_k}})_{k\in\N^+}$ converges to a map, say $\varphi^I_{\Omega}:\HH^{2-}\rightarrow \partial\Omega$, with $(\varphi^I_{\Omega})^*I$ conformal to $h^-_0$. Since the gluing map $\Phi^I_{\Omega_n}=f_n=(\partial\varphi^I_{\Omega_n})^{-1}\circ\partial\varphi^c_{\Omega_n}$ is normalized, the map $(\partial\varphi^I_{\Omega})^{-1}\circ\partial\varphi^c_{\Omega}$, as the limit of $f_{n_k}$, is also normalized and is thus the $I$-gluing map $\Phi^I_{\Omega}$ of $\Omega$ (combined with the above properties about $\varphi^I_{\Omega}$ and  $\varphi^c_{\Omega}$). By assumption (ii), $f_n=\Phi^I_{\Omega_n}$ converges to $f$ up to subsequence.  As a consequence, $\Phi^I_{\Omega}=f$. The proposition follows.
\end{proof}

We now consider the corresponding convergence problem for convex domains of second type. Notice that we have the compactness result about the convergence of isometric immersions of a surface into $\bdS^3$ (see Theorem \ref{thm:schlenker}) and the estimate on principal curvatures of a complete, embedded spacelike convex surface in $\bdS^3$ of variable curvature (see Lemma \ref{lm:principal-curv-dual}).
We list one case (for the $I$-gluing map) in the following for instance, since the cases for $\III$-gluing map and $(\III,I)$-gluing map can be stated and shown in the same manner, by considering the isometric immersions of the disk $(\DD^2, h^{*\pm}_n)$ or $(\DD^2, h^*_\pm)$ in $\bdS^3$ instead.

\begin{proposition}\label{prop:convergence-II}
Let $(\iota^{\pm}_n: (\DD^2,h^{\pm}_n)\rightarrow \HH^3)_{n\in\N^+}$ be a sequence of a pairs of isometric immersions whose images bound a convex domain of second type, which is invariant under a quasifuchsian representation of $\Gamma_n$ with its $I$-gluing map $f_n$. Assume that
 \begin{enumerate}[(i)]
   \item $(\iota^{\pm}_n(0))_{n\in\N^+}$ is contained in a compact subset $K^{\pm}$ of $\HH^3$,
   \item the maps $f_n$ are uniformly quasi-symmetric,
   \item $f_n$ converges to $f\in\cT$ in the $C^0$-sense, and
   \item $(h^{\pm}_n)_{n\in\N^+}$ 
  converges $C^{\infty}$  to $h_{\pm}$ on compact subsets of $\DD^2$.
 \end{enumerate}
Then $(\iota^{\pm}_n)_{n\in\N^+}$ has a subsequence which converges smoothly on compact subsets to a couple of isometric immersions $\iota^{\pm}:(\DD^2,h_{\pm})\rightarrow\HH^3$ whose images bound a convex domain of second type with its $I$-gluing map $f$.
\end{proposition}
\begin{proof}  The proof is similar to that of Proposition \ref{prop:convergence-I}. We include it here for completeness.

   By assumption, the sequence of metrics $h^{\pm}_n$ are uniformly bounded conformal metrics (i.e. $h^{\pm}_n\leq e^{2c_1}h_0$ for some constant $c_1>0$) and have curvature varying in $[-1+\epsilon/C, C/\epsilon]$ for a constant $C>1$. Combined with Lemma \ref{lm:injectivity}, the injectivity radius of $h^{\pm}_n$ are uniformly bounded from below by a constant depending only on $c_1$, $\epsilon$ and $C$.

   Note also that $h^{\pm}_n$ is isometric to the induced metric on $\iota^{\pm}_n(\DD^2)$, it follows from Lemma \ref{lm:principal-curv} that the principal curvatures of $\iota^{\pm}_n(\DD^2)$ are uniformly bounded between two positive constants (depending only on $c_1$, $\epsilon$ and $C$). Combined with Theorem \ref{thm:Labourie} and assumption (i), $\iota^{\pm}_n$ converges smoothly on compact subsets to an isometric immersion $\iota^{\pm}: (\DD^2, h_{\pm})\rightarrow \HH^3$. We denote by $\Omega_n$ (resp. $\Omega$) the convex domain bounded by $\iota^{\pm}_n(\DD^2)$ (resp. $\iota^{\pm}(\DD^2)$), and denote by $I^{\pm}_n$ (resp. $I^{\pm}$) the induced metric on $\partial_{\pm}\Omega_n$ (resp. $\partial_{\pm}\Omega$), which is isometric to $h^{\pm}_n$ (resp. $h_{\pm}$).

 We claim that $\partial_{\infty}\Omega$ is a quasi-circle in $\CP^1$ (which implies that $\Omega$ is a convex domain of second type). To see this, since the principal curvatures of $\partial_{\pm}\Omega_n=\iota^{\pm}_n(\DD^2)$ are uniformly bounded between two positive constants (as shown above). This implies that the hyperbolic Gauss map say $G^{\pm}_n$ from 
  $\partial_{\pm}\Omega_n$ to $D^{\pm}_n$ (where $D^+_n\cup D^+_n=\CP^1\setminus\partial_{\infty}\Omega_n$) is uniformly quasi-conformal, by comparing the induced metric $I^{\pm}_n$ on $\partial_{\pm}\Omega_n$ (which is isometric to $h^{\pm}_n$) with the metric $I^{\pm}_n+2\II^{\pm}_n+\III^{\pm}_n$ (whose push-forward metric by $G^{\pm}_n$ is a conformal metric on $D^{\pm}_n$), where $\II^{\pm}_n$, $\III^{\pm}_n$ are the second and third fundamental forms of $\partial_{\pm}\Omega_n$. Note that $G^{\pm}_n$ is a homeomorphism and its boundary map $\partial G^{\pm}_n$ is the identity from $\partial_{\infty}(\partial_{\pm}\Omega_n)=\partial_{\infty}\Omega_n$ to itself. Recall that the $I$-gluing map of $\Omega_n$ is $\Phi^I_{\Omega_n}=(\partial\varphi^{I_-}_{\Omega_n})^{-1}\circ\partial\varphi^{I_+}_{\Omega_n}$, where $\varphi^{I_{\pm}}_{\Omega_n}: \HH^{2\pm}\rightarrow \partial_{\pm}\Omega_n$ is a map such that $(\varphi^{I_{\pm}}_{\Omega_n})^*(I^{\pm}_n)$ is conformal to $h^{\pm}_0$.  Let $\varphi^{c\pm}_{\Omega_n}: \HH^{2\pm}\rightarrow D^{\pm}_n$ be a bi-holomorphic map such that $(\partial\varphi^{c-}_{\Omega_n})^{-1}\circ\partial\varphi^{c+}_{\Omega_n}$ is the conformal welding of the 
  Jordan curve $\partial_{\infty}\Omega_n$. In particular, we have
  \begin{equation*}
  \begin{split}
  f_n=\Phi^I_{\Omega_n}&=(\partial\varphi^{I-}_{\Omega_n})^{-1}\circ \partial\varphi^{c-}_{\Omega_n}\circ \big( (\partial\varphi^{c-}_{\Omega_n})^{-1} \circ \partial\varphi^{c+}_{\Omega_n}\big) \circ\big((\partial\varphi^{c+}_{\Omega_n})^{-1}\circ\partial\varphi^{I+}_{\Omega_n}\big)\\
 &=\big((\partial\varphi^{I-}_{\Omega_n})^{-1}\circ
 \partial (G_n^{-})^{-1}\circ \partial\varphi^{c-}_{\Omega_n}\big)\circ \big( (\partial\varphi^{c-}_{\Omega_n})^{-1}\circ \partial\varphi^{c+}_{\Omega_n}\big) \circ\big((\partial\varphi^{c+}_{\Omega_n})^{-1}\circ\partial G^{+}_n\circ\partial\varphi^{I+}_{\Omega_n}\big)\\
 &=\partial\big(\big((\varphi^{c-}_{\Omega_n})^{-1}\circ G^{-}_n\circ\varphi^{I-}_{\Omega_n}\big)^{-1}\big) \circ\big((\partial\varphi^{c-}_{\Omega_n})^{-1}\circ\partial\varphi^{c+}_{\Omega_n}\big)
 \circ\partial\big((\varphi^{c+}_{\Omega_n})^{-1}\circ G^{+}_n\circ\varphi^{I+}_{\Omega_n}\big)
  \end{split}
  \end{equation*}
 By the above discussion, $(\varphi^{c\pm}_{\Omega_n})^{-1}\circ G^{\pm}_n\circ(\varphi^{I\pm}_{\Omega_n})$ is a uniformly quasiconformal map (see e.g. \cite[Proposition 3.14]{convexhull}) and its boundary map is thus a uniformly quasisymmetric homeomorphism. It then follows from hypothesis (ii) that the conformal welding of $\partial_{\infty}\Omega_n$ is uniformly quasisymmetric, so by Theorem \ref{thm:Ahlfors-Bers}, the $\partial_\infty \iota_n^{\pm}(\DD^2)=\partial_{\infty}\Omega_n$ are quasi-circles with a quasi-conformal constant which is bounded from above by a fixed constant.

  Since  $\partial_{\infty}(\partial_{\pm}\Omega)=\partial_{\infty}\Omega$ has more than one point, by assumption (iv), Lemma \ref{lm:conf-bihol} and Carath\'{e}odory's theorem, $(\varphi^{I_{\pm}}_{\Omega_n})_{n\in\N^+}$ has a subsequence $(\varphi^{I_{\pm}}_{\Omega_{n_k}})_{k\in\N^+}$ converging to a map, say $\varphi^{\pm}_{\Omega}:\HH^{2\pm}\rightarrow \partial_{\pm}\Omega$, with $(\varphi^{\pm}_{\Omega})^*I_{\pm}$ conformal to $h^{\pm}_0$. Since the gluing map $\Phi^{I}_{\Omega_n}=f_n=(\partial\varphi^{I_-}_{\Omega_n})^{-1}\circ\partial\varphi^{I_+}_{\Omega_n}$ is normalized, the map $(\partial\varphi^{I_-}_{\Omega})^{-1}\circ\partial\varphi^{I_+}_{\Omega}$, as the limit of $f_{n_k}$, is also normalized and is thus the $I$-gluing map $\Phi^I_{\Omega}$ of $\Omega$ (combined with the above properties about $\varphi^{I_{\pm}}_{\Omega}$). By assumption (iii), $f_n=\Phi^I_{\Omega_n}$ converges to $f$ up to subsequence. As a consequence, $\Phi^I_{\Omega}=f$. The lemma follows.
\end{proof}

\begin{proof}[\textbf{Proof of Theorem \ref{tm:main-iv}}]
Let $D_-, D_+\subset\CP^1$ be disjoint quasi-disks with $\partial D_-=\partial D_+$, denoted by $C$ for simplicity, which is oriented such that $D_+$ and $D_-$ lie on the left-hand and right-hand sides of $C$. Using Theorem \ref{thm:Ahlfors-Bers}, the conformal welding $\Phi_C$ is quasisymmetric. By Proposition \ref{prop:qcirc approx}, $\Phi_C$ can be approximated by a sequence of uniformly quasisymmetric quasifuchsian homeomorphisms, say $f_n:\RP^1\to \RP^1$. We assume that $f_n$ is $(\rho^+_n,\rho^-_n)$-equivariant, with $\rho^+_n$ and $\rho^-_n$ two Fuchsian representations of the same surface group $\Gamma_n$.

Note that since $\Gamma_n$ is a surface group, $\Gamma_n$ is residually finite -- this follows from Selberg's Lemma, or from \cite{hempel}.
Therefore, there exists a finite-index subgroup of $\Gamma_n$ (which does not contain 
finitely many non-trivial elements of $\Gamma_n$) such that its representation under $\rho^{\pm}_n$ has a fundamental domain $P^{\pm}_n$ in $\DD^2$ containing a closed hyperbolic disk $B_H(0,\tau_n)$, where $\tau_n$ strictly increases to $\infty$.

For the given metrics $h_{\pm}$ on $D_{\pm}$, we consider the conformal metrics $(\varphi^{\pm}_0)^*h_{\pm}$ on $\DD^2$, where $\varphi_0^{\pm}: (\DD^2, h_0)\rightarrow (D_{\pm},h_{\pm})$ are a pair of (fixed) conformal maps. 
Applying Proposition \ref{prop:approx I} for the conformal metrics $(\varphi^{\pm}_0)^*h_{\pm}$ and the sequence $(\rho^{\pm}_n)_{n\in\N^+}$ of pairs of Fuchsian representations, there is a sequence of pairs of $\rho^{\pm}_n$-invariant conformal metrics, say $h^{\pm}_n$, on $\DD^2$, which uniformly converges to $(\varphi^{\pm}_0)^*h_{\pm}$ smoothly on compact subsets of $\DD^2$.

For the given metrics $h^*_{\pm}$ on $D_{\pm}$, we consider the conformal metrics $(\psi^{\pm}_0)^*h^*_{\pm}$ on $\DD^2$, where $\psi^{\pm}_0: (\DD^2, h_0)\rightarrow (D_{\pm},h^*_{\pm})$ is a conformal map. Applying Proposition \ref{prop:approx III} for the conformal metric $(\psi^{\pm}_0)^*(h^*_{\pm})$ and the sequence $(\rho^{\pm}_n)_{n\in\N^+}$ of pairs of Fuchsian representations, there is a sequence of pairs of $\rho^{\pm}_n$-invariant conformal metrics, say $h^{*\pm}_n$, on $\DD^2$, which uniformly converges to $(\psi^{\pm}_0)^*(h^*_{\pm})$ smoothly on compact subsets of $\DD^2$.

We divide the discussion into two cases below, according to the types of convex domains we want to realize.

\textbf{Case 1}: For Statements (i)-(ii), it suffices to show Statements (i), since Statements (ii) will follow in a similar way, by replacing $h_-$ with $h^*_-$.

Applying Statement (1) of Proposition \ref{prop:quasifuchsian realization} for the $(\rho^+_n,\rho^-_n)$-equivariant quasisymmetric homeomorphism $f_n$ and $\rho^-_n$-invariant metric $h^-_n$, there exists a convex domain $\Omega_n\subset\HH^3$ of first type, invariant under a quasifuchsian group, with the induced metric on $\partial\Omega_n$ isometric to $h^-_n$ and the $I$-gluing map of $\Omega_n$ being $f_n$.

Up to isometries, we can assume that the $\Omega_n$ are located in good positions (namely, their boundary are realized by isometric immersions of $(\DD^2, h^{-}_n)$ with the images of $0$ contained in a compact subset of $\HH^3$).

By Proposition \ref{prop:convergence-I}, after passing to a subsequence, $(\Omega_n)_{n\in\N^+}$ converges to a convex domain $\Omega\subset\HH^3$ of first type with the induced metric on $\partial\Omega$ isometric to $h_-$, and the $I$-gluing map of ${\Omega}$ being $\Phi_C$. Combined with Proposition \ref{prop:equiv}, Statement (i) follows.

\textbf{Case 2}: For Statements (iii)-(v), we just need to show one of them, since the others will follow in the same manner, by replacing $h_{\pm}$ with $h^*_{\pm}$. We show Statement (iii) for instance.

Applying Statement (3) of Proposition \ref{prop:quasifuchsian realization} for the $(\rho^+_n,\rho^-_n)$-equivariant quasisymmetric homeomorphism $f_n$ and $\rho^{\pm}_n$-invariant metric $h^{\pm}_n$, there exists a convex domain $\Omega_n\subset\HH^3$ of second type, which is invariant under a quasifuchsian group, with the induced metric on $\partial_{\pm}\Omega_n$ isometric to $h^{\pm}_n$ and the $I$-gluing map of $\Omega_n$ being $f_n$.

We set $\Lambda_n=\partial_\infty \Omega_n$, which is the asymptotic boundary of $\partial_{\pm}\Omega_n$. Since the hyperbolic Gauss map is uniformly quasiconformal, similarly as shown in Proposition \ref{prop:convergence-II}, each $\Lambda_n$ is a $K$-quasicircle, with $K$ independent of $n$. After normalizing by a sequence of hyperbolic isometries, we can assume that $(\Lambda_n)_{n\in \N}$ converges in the Hausdorff topology to a $K$-quasicircle $\Lambda\subset \partial_\infty \HH^3$.

Let $\iota^{\pm}_n: (\DD^2, h^{\pm}_n)\rightarrow \HH^3$ be a pair of isometric embeddings with $\iota^{\pm}_n(\DD^2)=\partial_{\pm}\Omega_n$ and let $x_n^{\pm}$ be the images of 0 on $\partial_{\pm}\Omega_n$ under $\iota^{\pm}_n$.
Denote by $G_n:\partial\Omega_n\to \partial_\infty\HH^3$ the hyperbolic Gauss map on $\partial\Omega_n$. With the normalization chosen above, where 
$\Lambda_n\to \Lambda$, the images $G_n(x_n^\pm)=G_n\circ\iota^{\pm}_n(0)$ remain in a compact subset of 
$\partial_\infty\HH^3\setminus \Lambda$
again by the fact that the Gauss maps $G_n$ are uniformly quasi-conformal and $\iota^{\pm}_n$ are conformal.

We claim that the points $x^\pm_n$ also remain in a compact region of $\HH^3$. To see this, notice that, after extracting a subsequence, we can assume that $x_n^\pm \to x^\pm\in \bar \HH^3 =\HH^3\cup \partial_\infty \HH^3$. However we cannot have 
$x^+\in \Lambda$, since otherwise it would follow from the convexity of $\Omega_n$ that 
$G_n(x_n^+)\to x^+\in\Lambda$ (which contradicts that $G_n(x_n^+)$ remain in a compact subset of $\partial_\infty\HH^3\setminus \Lambda$), and similarly $x^-\not\in \Lambda$. It remains to show that $x^\pm\not\in \partial_\infty\HH^3\setminus \Lambda$.

Note that the domains $\Omega_n$ are invariant under quasifuchsian groups. Moreover, the induced metrics on $\partial_{\pm}\Omega_n$ are isometric to $h^{\pm}_n$, which are uniformly bounded conformal metrics (in the sense that $h^{\pm}_n\leq e^{2c_{\pm}}h_0$ for all $n$, where $c_{\pm}>0$ depend only on $h_{\pm}$) and have curvature in $[-1+\epsilon/C,C/\epsilon]$ for a constant $C>1$. Apply Lemma \ref{lm:distance} to $\partial_{\pm}\Omega_n$, we have that $\partial_{\pm}\Omega_n$ (and thus $x_n^\pm$) lie at distance $r$ (depending only on $\epsilon$, $C$ and $h_{\pm}$) from the convex hull of $\Lambda_n$. Note also that $\Lambda_n\rightarrow\Lambda$ in the Hausdorff topology. It follows that $x^\pm\not\in \partial_\infty\HH^3\setminus \Lambda$. We can therefore conclude that, after taking a subsequence $x_n^\pm\to x^\pm\in \HH^3$. The claim follows.

By Proposition \ref{prop:convergence-II}, after passing to a subsequence, $(\Omega_n)_{n\in\N^+}$ converges to a convex domain $\Omega\subset\HH^3$ of second type with the induced metric on $\partial_{\pm}\Omega$ is isometric to $h_{\pm}$, and the $I$-gluing map of ${\Omega}$ being $\Phi_C$. Combined with Proposition \ref{prop:equiv}, Statement (iii) follows.
\end{proof}

%% file: short-weyl6.tex
\section{Applications}

\subsection{New parametrizations of Bers slice}
Recall that a Bers slice is the space of quasifuchsian hyperbolic manifolds with a fixed conformal structure at one prescribed boundary at infinity (up to isotopy). We denote by ${\cB}_c$ the Bers slice with a fixed conformal structure $c$ at the prescribed boundary at infinity. By the Bers' simultaneous uniformization theorem, ${\cB}_c$ is parameterized by $\cT_S$ in terms of the conformal structure on the other boundary at infinity. In addition, ${\cB}_c$ is uniquely determined by the bending lamination (resp. the induced metric) on the boundary component of the convex core facing the prescribed boundary at infinity with conformal structure $c$, by the Grafting parametrization theorem given by Scannell-Wolf \cite{scannell-wolf} (resp. Dumas-Wolf \cite{dumas-wolf}).

As an application of the special cases of Corollary \ref{cr:qf-I} and Corollary \ref{cr:qf-III} when the prescribed metrics $h$ and $h^*$ have constant curvature, we obtain a new parametrization of $\cB_c$ in terms of constant curvature surfaces. More precisely, let $K\in(-1,0)$, we define $\Phi_K: \cB_c\rightarrow \cT(S)$ as the map that sends $M\in\cB_c$ to the hyperbolic metric homothetic to the induced metric on the $K$-surface in $M$ on the opposite side of the prescribed boundary at infinity.
Let $K^*\in(-\infty,0)$, we define $\Psi_{K^*}: \cB_c\rightarrow \cT(S)$ as the map that sends $M\in\cB_c$ to the hyperbolic metric homothetic to the third fundamental form on the $K^*/(1-K^*)$-surface in $M$ on the opposite side of the prescribed boundary at infinity. Then we have the following statement.

\begin{corollary}
For each $K\in(-1,0)$ and $K^*\in(-\infty,0)$, the maps $\Phi_K$ and $\Psi_{K^*}$ are both homeomorphisms.
\end{corollary}

\subsection{Proof of Theorem \ref{tm:cc-I}}
Another application of Corollary \ref{cr:qf-I} is Theorem \ref{tm:cc-I}. We recall it here for convenience. Let $h$ be a hyperbolic metric on $S$, and let $c\in \cT_S$ be a conformal structure on $S$. 
We want to show that there is a quasifuchsian hyperbolic manifold with the conformal structure on its upper boundary component at infinity isotopic to $c$, and with the induced metric on the lower boundary component of its convex core isotopic to $h$. To see this, we first show the following compactness statement:

\begin{proposition}\label{prop:appl-convergence}
Let $K_n\in (-1, 0)$ with $K_n$ decreasingly converging to $-1$ and let $h_n=(1/|K_n|)h$. Let $M_n$ be a sequence of quasifuchsian hyperbolic manifolds with the conformal structure on its upper boundary component at infinity isotopic to $c$, and with the induced metric on the $K_n$-surface in the lower side of its convex core isotopic to $h_n$. Then up to extracting a subsequence, $M_n$ converges to a quasifuchsian hyperbolic manifold with the conformal structure on its upper boundary component at infinity isotopic to $c$, and with the induced metric on the lower boundary component of its convex core isotopic to $h$.
\end{proposition}

\begin{proof}
  Let $\Sigma_n$ be the $K_n$-surface in the lower connected component of the complement of the convex core in $M_n$. We denote by 
  $\partial_-C_{M_n}$ the lower boundary of the convex core, and by $m_n^-$ its induced metric.

  Note that the nearest-point projection from $\Sigma_n$ to 
  $\partial_-C_{M_n}$ is contracting. As a consequence, $m_n^-\leq h_n=(1/|K_n|)h$ (in the sense of comparison of the length spectrum) and since $K_n\to -1$, it follows that $m^-_n\to h$ (by the positive definiteness of Thurston's asymmetric metric on $\cT_S$, see \cite[Prop. 2.1]{thurston:minimal}).

  Since $c_n^-$ (the conformal structure at infinity on the lower boundary component) is uniformly quasi-conformal to $m_n^-$ (see e.g. \cite{epstein-marden-markovic}), it follows that $(c_n^-)_{n\in \N}$ remains in a compact subset of $\cT_S$. After extracting a sub-sequence, it converges to a limit $c^-$, and $(M_n)_{n\in \N}$ converges to the quasifuchsian manifold $M$ having conformal structures $c$ and $c^-$ on the upper and lower ideal boundary components. 
  Recall that the map that takes a quasifuchsian hyperbolic manifold to the induced metrics on the boundary of its convex core is continuously differentiable (see \cite[Theorem 1]{bonahon-variations}).  
  $M$ has the induced metric $h$ on the lower boundary component of the convex core. The proposition follows.
\end{proof}

\begin{proof}[\textbf{Proof of Theorem \ref{tm:cc-I}}]
Let $K_n\in (-1, 0)$ with $K_n$ decreasingly converging to $-1$ and let $h_n=(1/|K_n|)h$. By Corollary \ref{cr:qf-I}, there is a unique (up to isometries) quasifuchsian hyperbolic manifold, denoted by $M_n$, with the conformal structure on its upper boundary component at infinity isotopic to $c$, and with the induced metric on the $K_n$-surface in the lower side of its convex core isotopic to $h_n$. It then follows from Proposition \ref{prop:appl-convergence} that the sequence $(M_n)_{n\in\N}$ converges to a quasifuchsian hyperbolic manifold with the conformal structure on its upper boundary component at infinity isotopic to $c$, and with the induced metric on the lower boundary component of its convex core isotopic to $h$, which is a desired manifold. This concludes the proof of Theorem \ref{tm:cc-I}.
\end{proof}